\documentclass[12pt]{book}
\usepackage{umlaut,german,graphicx}
\usepackage{amssymb}
\newcommand{\N}{\mathbb{N}}
\newcommand{\Z}{\mathbb{Z}}
\newcommand{\C}{\mathbb{C}}
\newcommand{\R}{\mathbb{R}}
\newcommand{\Q}{\mathbb{Q}}

\newcommand{\SL}{\mathfrak{sl}_{4}\mathbb{C}}
\newcommand{\g}{\mathfrak{g}}
\newcommand{\h}{\mathfrak{h}}
\newcommand{\B}{\mathfrak{b}}

\newcommand{\p}{\mathfrak{p}}
\newcommand{\tr}{\mathrm{tr}}

\usepackage{amsthm}
\usepackage{amsmath}
\usepackage{amscd}

\usepackage{amstext}

\makeatletter
\def\Ddots{\mathinner{\mkern1mu\raise\p@
\vbox{\kern7\p@\hbox{.}}\mkern2mu
\raise4\p@\hbox{.}\mkern2mu\raise7\p@\hbox{.}\mkern1mu}}
\makeatother

\newtheorem{theorem}{Theorem}
\newtheorem{lemma}{Lemma}
\newtheorem{proposition}{Proposition}
\newtheorem{corollary}{Corollary}
\newtheorem{definition}{Definition}

\newcommand{\on}[2]{\setbox0=\hbox{$#1$}\setbox1=\hbox{$#2$}%
            \dimen0=\wd0\advance\dimen0 by \wd1\divide\dimen0 by 2%
             \ifdim\wd0>\wd1$#1$\hskip-\dimen0$#2$\advance\dimen0 by -\wd1%
              \else$#2$\hskip-\dimen0$#1$\advance\dimen0 by -\wd0%
             \fi%
            \hskip\dimen0}
\newcommand{\onn}[2]{\makebox{\on{#1}{#2}}}
\newcommand{\cross}{\smash\times\vphantom\bullet}
\newcommand{\leftblob}[1]{\hspace*{-1ex}\onn{\raisebox{-.5ex}
     {$\genfrac{}{}{0pt}{}{\hfill\hrulefill}{\hphantom{\hspace*{2ex}#1}}$}}
                {\stackrel{#1}{\bullet}}}
\newcommand{\rightblob}[1]{\onn{\raisebox{-.5ex}
     {$\genfrac{}{}{0pt}{}{\hrulefill\hfill}{\hphantom{\hspace*{2ex}#1}}$}}
                 {\stackrel{#1}{\bullet}}\hspace*{-1ex}}
\newcommand{\middleblob}[1]{\onn{\raisebox{-.5ex}
     {$\genfrac{}{}{0pt}{}{\hrulefill}{\hphantom{\hspace*{2ex}#1}}$}}
                 {\stackrel{#1}{\bullet}}}
\newcommand{\leftcross}[1]{\hspace*{-1ex}\onn{\raisebox{-.5ex}
     {$\genfrac{}{}{0pt}{}{\hfill\hrulefill}{\hphantom{\hspace*{2ex}#1}}$}}
                {\stackrel{#1}{\cross}}}
\newcommand{\rightcross}[1]{\onn{\raisebox{-.5ex}
     {$\genfrac{}{}{0pt}{}{\hrulefill\hfill}{\hphantom{\hspace*{2ex}#1}}$}}
                 {\stackrel{#1}{\cross}}\hspace*{-1ex}}
                
\newcommand{\middlecross}[1]{\onn{\raisebox{-.5ex}
     {$\genfrac{}{}{0pt}{}{\hrulefill}{\hphantom{\hspace*{2ex}#1}}$}}
                 {\stackrel{#1}{\cross}}}
\newcommand{\oo}[2]{\leftblob{#1}\hspace*{-1ex}\rightblob{#2}}

\newcommand{\Dd}[7]{\begin{picture}(5,2)
\put(0,0.2){\makebox(0,0)[r]{$\xoo{#1}{#2}{#3}$}}
\put(0.5,0.2){...}
\put(3,0.2){\makebox(0,0)[r]{$\oo{#4}{#5}$}}
\put(2.7,0.1){\line(1,1){1}}
\put(2.7,0.1){\line(1,-1){1}}
\put(3.7,1.1){\makebox(0,0){$\bullet$}}
\put(3.7,-0.9){\makebox(0,0){$\bullet$}}
\put(4,1.1){\makebox(0,0)[l]{\scriptsize $#6$}}
\put(4,-0.9){\makebox(0,0)[l]{\scriptsize $#7$}}
\end{picture}\vspace{1cm}}
\newcommand{\Ddd}[7]{\begin{picture}(5,2)
\put(0,0.2){\makebox(0,0)[r]{$\xoo{#1}{#2}{#3}$}}
\put(0.5,0.2){...}
\put(3,0.2){\makebox(0,0)[r]{$\oo{#4}{#5}$}}
\put(2.7,0.1){\line(1,1){1}}
\put(2.7,0.1){\line(1,-1){1}}
\put(3.7,1.1){\makebox(0,0){$\bullet$}}
\put(3.7,-0.9){\makebox(0,0){$\bullet$}}
\put(4,1.1){\makebox(0,0)[l]{\scriptsize $#6$}}
\put(4,-0.9){\makebox(0,0)[l]{\scriptsize $#7$}}
\end{picture}}

\newcommand{\Ddo}[7]{\begin{picture}(5,2)
\put(0,0.2){\makebox(0,0)[r]{$\ooo{#1}{#2}{#3}$}}
\put(0.5,0.2){...}
\put(3,0.2){\makebox(0,0)[r]{$\oo{#4}{#5}$}}
\put(2.7,0.1){\line(1,1){1}}
\put(2.7,0.1){\line(1,-1){1}}
\put(3.7,1.1){\makebox(0,0){$\bullet$}}
\put(3.7,-0.9){\makebox(0,0){$\bullet$}}
\put(4,1.1){\makebox(0,0)[l]{\scriptsize $#6$}}
\put(4,-0.9){\makebox(0,0)[l]{\scriptsize $#7$}}
\end{picture}
\vspace{1cm}}
\newcommand{\x}[1]{\stackrel{#1}{\cross}}

\newcommand{\xo}[2]{\leftcross{#1}\hspace*{-1ex}
               \rightblob{#2}}      
\newcommand{\ox}[2]{\leftblob{#1}\hspace*{-1ex}
               \rightcross{#2}}                     
\newcommand{\ooo}[3]{\leftblob{#1}\hspace*{-1ex}
               \middleblob{#2}\hspace*{-1ex}\rightblob{#3}}
\newcommand{\xoo}[3]{\leftcross{#1}\hspace*{-1ex}
               \middleblob{#2}\hspace*{-1ex}\rightblob{#3}}
\newcommand{\oox}[3]{\leftblob{#1}\hspace*{-1ex}
               \middleblob{#2}\hspace*{-1ex}\rightcross{#3}}

\newcommand{\xox}[3]{\leftcross{#1}\hspace*{-1ex}
               \middleblob{#2}\hspace*{-1ex}\rightcross{#3}}
\newcommand{\oxo}[3]{\leftblob{#1}\hspace*{-1ex}
               \middlecross{#2}\hspace*{-1ex}\rightblob{#3}}               
\newcommand{\xxx}[3]{\leftcross{#1}\hspace*{-1ex}
               \middlecross{#2}\hspace*{-1ex}\rightcross{#3}}
\newcommand{\oooo}[4]{\leftblob{#1}\hspace*{-1ex}\middleblob{#2}\hspace*{-1ex}
               \middleblob{#3}\hspace*{-1 ex}\rightblob{#4}}

\newcommand{\onnn}[3]{\onn{\onn{#1}{#2}}{#3}}

\newcommand{\rightdblob}[1]{\onnn{\raisebox{-.7ex}
     {$\genfrac{}{}{0pt}{}{\hrulefill\hfill}{\hphantom{\hspace*{2ex}#1}}$}}
                                  {\raisebox{-.3ex}
     {$\genfrac{}{}{0pt}{}{\hrulefill\hfill}{\hphantom{\hspace*{2ex}#1}}$}}
                 {\stackrel{#1}{\bullet}}\hspace*{-1ex}}

\newcommand{\Btwoo}[2]{\leftdblob{#1}\hspace*{-.5ex}
     \makebox[0pt]{\small $\rangle$}\hspace*{-.5ex}\rightdblob{#2}}

\newcommand{\leftdblob}[1]{\hspace*{-1ex}\onnn{\raisebox{-.7ex}
     {$\genfrac{}{}{0pt}{}{\hfill\hrulefill}{\hphantom{\hspace*{2ex}#1}}$}}
                                  {\raisebox{-.3ex}
     {$\genfrac{}{}{0pt}{}{\hfill\hrulefill}{\hphantom{\hspace*{2ex}#1}}$}}
                 {\stackrel{#1}{\bullet}}}
\newcommand{\btwo}[2]{\leftdblob{#1}\hspace*{-.5ex}
     \makebox[0pt]{\small $\rangle$}\hspace*{-.5ex}\rightdblob{#2}}
\newcommand{\arrowhead}{
                \begin{picture}(0,0)\put(3,.605){\vector(1,0){0}}\end{picture}}
              
\newcommand{\xooo}[4]{\leftcross{#1}\hspace*{-1ex}
               \middleblob{#2}\hspace*{-1 ex}\middleblob{#3}\hspace*{-1ex}\rightblob{#4}}

\newcommand{\Bfive}[5]{\leftcross{#1}\hspace*{-1ex}
               \middleblob{#2}\hspace*{-1ex}\rightblob{#3}\;...\;\leftdblob{#4}\hspace*{-.5ex}
     \makebox[0pt]{\small $\rangle$}\hspace*{-.5ex}\rightdblob{#5}}

\newcommand{\semidownbracefill}{$\m@th\braceld\leaders\vrule\hfill\braceru
  \bracelu\leaders\vrule\hfill\arrowhead$}


\begin{document}
\title{Invariant Bilinear Differential Pairings on Parabolic Geometries}
\author{\textbf{Jens Kroeske}\\ \\ \\ \\
Supervisor: Prof.~Michael Eastwood\\ \\
School of Pure Mathematics\\
University of Adelaide\\ \\
June 2008}
%\address{School of Mathematical Sciences, University of Adelaide,
%SA 5005, Australia}
%\email{jens.kroeske@student.adelaide.edu.au}
%\subjclass{...}
%\dedicatory{...}
%\renewcommand{\subjclassname}{\textup{2000} Mathematics Subject Classification}

\date{}
\maketitle

\setcounter{tocdepth}{2}
\tableofcontents

\newpage
\section{Abstract}
This thesis is concerned with the theory of invariant bilinear differential pairings on parabolic geometries. It introduces the concept formally with the help of  the jet bundle formalism and provides a detailed analysis. More precisely, after introducing the most important notations and definitions, we first of all give an algebraic description for pairings on homogeneous spaces and obtain a first existence theorem. Next, a classification of first order invariant bilinear differential pairings is given under exclusion of certain degenerate cases that are related to the existence of invariant linear differential operators. Furthermore, a concrete formula for a large class of invariant bilinear differential pairings of arbitrary order is given and many examples are computed. The general theory of higher order invariant bilinear differential pairings turns out to be much more intricate and a general construction is only possible under exclusion of finitely many degenerate cases whose significance in general remains elusive (although a result for projective geometry is included). The construction relies on so-called splitting operators examples of which are described for projective geometry, conformal geometry and CR geometry in the last chapter.

\newpage
\section{Thesis declaration}
This work contains no material which has been accepted for the award of any other degree or diploma in any university or other tertiary institution and, to the best of my knowledge and belief, contains no material previously published or written by any other person, except where due reference has been made in the text.
\par
\vspace{1cm}
I give consent to this copy of my thesis, when deposited in the University library, being made available for loan and photocopying, subject to the provisions of the Copyright Act 1968.

\newpage
\section{Acknowledgment}
It is my pleasure to express my deepest gratitude towards my supervisor Prof.~Michael Eastwood who has skillfully guided me through this thesis and without whom none of this would have been possible. He has not only provided me with invaluable help but also kindled my love for this particular field of mathematics. My time in Adelaide has been thoroughly enjoyable.
\par
\vspace{10cm}
{\em This thesis is dedicated to my beautiful daughter Alexis.}

\newpage
Before you can fly, you have to learn how to walk\footnote{Nolan Wallach, Brisbane winter school in mathematics.}

\section{Constants and notation used throughout}
\begin{enumerate}
\item
$\mathcal{M}$: a manifold of dimension $n$ (for CR geometry $\mathcal{M}$ will have real dimension $2n+1$).
\item
$\g$: a semisimple Lie algebra of rank $l$.
\item
$[.,.]$: the bracket in $\g$.
\item
$G$: the simply connected Lie group with Lie algebra $\g$.
\item
$\h$: a fixed Cartan subalgebra of $\g$.
\item
$\p$: a parabolic subalgebra of $\g$.
\item
$k_{0}$: length of the grading of $\g$.
\item
$l_{0}$: number of simple roots in $\mathcal{S}\backslash\mathcal{S}_{\p}$.
\item
$\alpha_{i}$, $i=1,...,l$: simple roots of $\g$.
\item
$\omega_{i}$, $i=1,...,l$: fundamental weights corresponding to the simple roots.
\item
$I\subset\{1,...,l\}$: indices that correspond to simple roots in $\mathcal{S}\backslash\mathcal{S}_{\p}$, i.e.~to crossed through nodes.
\item
$J=\{1,...,l\}\backslash I$: indices that correspond to simple roots in $\mathcal{S}_{\p}$.
\item
$\mathcal{W}$: the Weyl group of $\g$.
\item
$\mathcal{W}^{\p}$: the Hasse diagram of $G/P$.
\item
$\rho=\sum_{i=1}^{l}\omega_{i}$: integral weight in the dominant Weyl chamber which lies closest to the origin.
\item
$B(.,.)$: the Killing form of $\g$.
\item
$T(\mathfrak{a})$: the tensor algebra of a Lie algebra $\mathfrak{a}$.
\item
$\mathfrak{U}(\mathfrak{a})$: the universal enveloping algebra of a Lie algebra $\mathfrak{a}$.
\item
$\mathcal{Z}(\mathfrak{a})$: the center of the universal enveloping algebra $\mathfrak{U}(\mathfrak{a})$.
\item
$\mathfrak{gl}_{l}\C$: the endomorphisms of $\C^{l}$. These can be identified with $\C^{l\times l}$.
\item
$\otimes$: the tensor product symbol.
\item
$\odot$: the symbol for the symmetric tensor product.
\item
$\Lambda$: the symbol for the skew symmetric tensor product.
\item
$\circledcirc$: the Cartan product of representations.
\item
$\mathcal{G}$: the total space of the principal bundle defining a parabolic geometry.
\item
$\omega$: the Cartan connection of $\mathcal{G}$ (the Maurer Cartan form of $G$ is denoted by $\omega_{MC}$). Note that $\omega$ is also used as a symbol for the geometric weight of a representation, but the context should make the meaning clear.
\item
$\mathcal{A}$: the adjoint tractor bundle.
\item
$\mathbb{V}_{\lambda}$: the irreducible finite dimensional representation (of $\g$ or $\p$) that is dual to the irreducible finite dimensional representation of highest weight $\lambda\in\h^{*}$.
\item
$J^{k}V$: $k$-th order jet bundle of a vector bundle $V$.
\item
$\mathcal{J}^{k}V$: $k$-th order weighted jet bundle.
\item
$\bar{J}^{k}V$: $k$-th order semi-holonomic jet bundle.
\item
$\bar{\mathcal{J}}^{k}V$: $k$-th order restricted semi-holonomic jet bundle.
\item
$\mathcal{E}$  or $\mathcal{O}$: the bundle of smooth (or holomorphic) functions of $\mathcal{M}$.
\item
$\mathcal{E}^{a}$: the bundle of tangent vectors.
\item
$\mathcal{E}_{a}$: the bundle of one-forms.
\item
$V^{*}$: the dual of the vector space (representation, bundle, etc.) $V$.
\item
$M_{\p}(\mathbb{V})$: the generalized Verma module associated to a representation $\mathbb{V}$.
\item
$\nabla$: a connection.
\end{enumerate}

\newpage

\section{Introduction}
\subsection{What this is all about}

It is generally known (see~\cite{pr},  p.~202), that on an arbitrary manifold $\mathcal{M}$ one can write down the Lie derivative $\mathcal{L}_{X}\omega_{b}$ of a one-form $\omega_{b}\in\Gamma(T^{*}\mathcal{M})=\Omega^{1}(\mathcal{M})$ with respect 
to a vector field $X^{a}\in \Gamma(T\mathcal{M})=\mathrm{Vect}(\mathcal{M})$ in terms of an arbitrary torsion-free connection $\nabla_{a}$ as
$$\mathcal{L}_{X}\omega_{b}=X^{a}\nabla_{a}\omega_{b}+\omega_{a}\nabla_{b}X^{a},$$
where the indices used are abstract in the sense of~\cite{pr}. This {\bf pairing} is obviously linear in $X^{a}$ and in $\omega_{b}$, i.e.~{\bf bilinear}, {\bf first order} and
{\bf invariant} in the sense that it does not depend upon a specific choice of connection. In~\cite{kms}, 30.1, it is shown that all such bilinear invariant differential pairings 
$$\mathrm{Vect}(\mathcal{M})\times\Omega^{1}(\mathcal{M})\rightarrow\Omega^{1}(\mathcal{M})$$
are given by a two parameter family spanned by $\mathcal{L}_{X}\omega_{b}$ and $X^{a}(d\omega)_{ab}$, where $d$ denotes the exterior derivative. Demanding that a pairing is invariant in the sense that it does not depend upon a specific choice of connection within the class of all torsion-free connections (in~\cite{kms} those pairings are called natural) turns out to be rather restrictive. Instead, one can specify an equivalence class of connections and ask for invariance under change of connection 
within this equivalence class. This is a standard procedure in many different geometries. In conformal geometry, for example, one deals with an equivalence class of connections that consists of the Levi-Civita connections that correspond to metrics in the 
conformal class. In projective geometry, one is given a projective equivalence class of connections that consists of all those torsion-free connections, which induce the same (unparameterised) geodesics.
This is equivalent (see~\cite{e}, p.~2, Proposition~1) to saying that $\nabla_{a}$ and $\hat{\nabla}_{a}$ are in the same equivalence class
if and only if there is a one form $\Upsilon_{a}$, such that
$$\hat{\nabla}_{a}\omega_{b}=\nabla_{a}\omega_{b}-2\Upsilon_{(a}\omega_{b)},$$
where round brackets around indices denote symmetrization, i.e.
$$\Upsilon_{(a}\omega_{b)}=\frac{1}{2}(\Upsilon_{a}\omega_{b}+\Upsilon_{b}\omega_{a}).$$
As a consequence, the difference between the two
connections when acting on sections of any weighted tensor bundle can be deduced (see~\cite{e}, p.~2) and the invariance of any expression can be checked by hand. For vector fields, for example,
 we have
$$\hat{\nabla}_{b}X^{a}=\nabla_{b}X^{a}+\Upsilon_{b}X^{a}+\Upsilon_{c}X^{c}\delta_{b}{}^{a},$$
so the invariance of the Lie derivative can be checked directly.
It is also clear that
$$X^{a}(d\omega)_{ab}=X^{a}\nabla_{[a}\omega_{b]}$$
is a first order bilinear invariant differential pairing in that sense, where square brackets around indices denote skewing, i.e.~$\nabla_{[a}\omega_{b]}=\frac{1}{2}(\nabla_{a}\omega_{b}-\nabla_{b}\omega_{a}$). 
%Therefore there are at least two first order bilinear invariant pairings
%$$Vect(\mathcal{M})\times\Omega^{1}(\mathcal{M})\rightarrow\Omega^{1}(\mathcal{M}).$$
\par
To obtain a more interesting example of a first order bilinear invariant differential pairing in projective geometry, consider pairings
$$\Gamma(\odot^{2}T\mathcal{M})\times\Omega^{1}(\mathcal{M})\rightarrow \mathcal{O},$$
where $\odot^{2}$ denotes the second symmetric product and $\mathcal{O}$ is the sheaf of holomorphic (or smooth) functions. The transformation rule for $V^{ab}\in\odot^{2}T\mathcal{M}$ under change of
connection is given by $\hat{\nabla}_{c}V^{ab}=\nabla_{c}V^{ab}+2\Upsilon_{c}V^{ab}+2\Upsilon_{d}V^{d(a}\delta_{c}{}^{b)}$. This implies
\begin{eqnarray*}
V^{ab}\hat{\nabla}_{(a}\omega_{b)}&=&V^{ab}\nabla_{(a}\omega_{b)}-2V^{ab}\Upsilon_{a}\omega_{b}\quad\text{and}\\
\omega_{b}\hat{\nabla}_{a}V^{ab}&=&\omega_{b}\nabla_{a}V^{ab}+(n+3)\omega_{b}\Upsilon_{a}V^{ab},
\end{eqnarray*}
where~$n=\mathrm{dim}(\mathcal{M})$. Therefore the pairing
$$(n+3)V^{ab}\nabla_{(a}\omega_{b)}+2\omega_{b}\nabla_{a}V^{ab}$$
does not depend upon a choice of connection within the projective class.
\par
\vspace{0.2cm}
It is natural to ask the question of whether these are all first order bilinear invariant differential pairings between those bundles and whether one can classify pairings between arbitrary bundles in general.
\par
In the following we will study exactly this question for a large class of geometries called {\bf parabolic geometries}. These are special cases of Cartan's \lq{\em \'{e}space g\'{e}n\'{e}ralis\'{e}}\rq~which are
geometric structures that have homogeneous spaces $G/P$, where $G$ is a Lie group and $P$ a subgroup, as their models. They are defined by a generalization of the principal $P$ bundle
$$\begin{array}{ccc}
P&\rightarrow&G\\
&&\downarrow\\
&&G/P
\end{array}$$
together with a Cartan connection that generalizes the Maurer Cartan form 
$$\omega:TG\rightarrow Lie(G)=\g.$$
Riemannian geometry, for example, can be defined as a (torsion-free) Cartan geometry modeled on Euclidean space $G/P$ with $G=\mathrm{Euc}_{n}\R$, the group of Euclidean motions, and $P=\mathrm{SO}_{n}(\R)$, see~\cite{sh}. 
In this case there exists a canonical connection, the Levi-Civita connection, on the principal bundle. If, however,
$P$ is a parabolic subgroup of a semisimple Lie group $G$, then the name parabolic geometry is commonly used. For each parabolic geometry there exists an equivalence class of connections (see~\cite{cs2}) and one can study invariant operators and invariant pairings as indicated above.

\subsection{Invariant differential pairings on the Riemann sphere}\label{riemannsphere}
In this section we will classify all bilinear invariant differential pairings on the Riemann sphere $\mathbb{CP}_{1}$. This little warm-up exercise does not only give the reader an idea of how one might
attempt to classify invariant bilinear differential pairings, but also produces a final formula that will turn out to hold in great generality. More specifically, the linear equations that we have
to solve on the Riemann sphere are exactly those that we will have to solve twice again in this thesis in a vastly more general setting. One can see the reason for this phenomena by studying~\cite{g2},
where it is explained how invariant linear differential operators on $\mathbb{CP}_{1}$ give rise to all standard operators on a general four dimensional conformal manifold.
\par
\subsubsection{The setup}
The basic objects that we can pair on $\mathbb{CP}_{1}$ are sections of line bundles, traditionally denoted by $\mathcal{O}(q)$, for $q\in\Z$, where $\mathcal{O}(1)$ is the hyperplane section bundle. For reasons to become clear
later, we will write $\mathcal{O}(\x{q})$ for $\mathcal{O}(q)$. Invariance on $\mathbb{CP}_{1}$ means invariance under Moebius transformations that act not only on $\mathbb{CP}_{1}$, but also on sections of
$\mathcal{O}(\x{q})$ in a way to be described below.
\par
The sections of $\mathcal{O}(\x{q})$ can be described by pairs of functions $\{f_{i}\}_{i=1,2}$ that depend on one variable $z\in\C$ and and are related by $f_{1}(z)=\zeta^{-q}f_{2}(\zeta)$ for 
$\zeta=-\frac{1}{z}$. More precisely, we are really concerned with local sections of these bundles. In that case the $f_{i}$ are defined in some (connected) open subset of $\C$ and are related on the overlap as indicated above. We can identify those 
sections with a function $s_{f}:\mathrm{SL}_{2}\C\rightarrow\C$, such that 
$$s_{f}\left(\left(\begin{array}{cc}
a&b\\
c&d
\end{array}\right)\left(\begin{array}{cc}
x&y\\
0&x^{-1}
\end{array}\right)\right)=x^{q}s_{f}\left(\begin{array}{cc}
a&b\\
c&d
\end{array}\right).$$
Then we have
$$f_{1}(z)=s_{f}\left(\begin{array}{cc}
1&0\\
z&1
\end{array}\right)\quad\text{and}\quad f_{2}(\zeta)=s_{f}\left(\begin{array}{cc}
\zeta&1\\
-1&0
\end{array}\right),$$
or equivalently
$$s_{f}\left(\begin{array}{cc}
a&b\\
c&d
\end{array}\right)=a^{q}f_{1}\left(\frac{c}{a}\right)=(-c)^{q}f_{2}\left(-\frac{a}{c}\right).$$
Again, $f_{1},f_{2}$ and $s_{f}$ need not be defined globally, but may just be given on some appropriate (open and connected) neighborhood of a point. In the sequel we will neglect this technicality for this expository example.\par
$\mathrm{SL}_{2}\C$ acts on the space of sections by
$$(\phi s_{f})(h)=s_{f}(\phi^{-1}h)$$
and hence by
$$(\phi f_{1})(z)=(d-bz)^{q}f_{1}\left(\frac{az-c}{d-bz}\right)$$
and
$$(\phi f_{2})(\zeta)=(c\zeta+a)^{q}f_{2}\left(\frac{d\zeta+b}{c\zeta+a}\right),$$
for
$$\phi=\left(\begin{array}{cc}
a&b\\
c&d
\end{array}\right).$$
It is easy to check that
$$(\phi f_{1})(z)=\zeta^{-q}(\phi f_{2})(\zeta),$$
so $\{\phi f_{1},\phi f_{2}\}$ is again a section of $\mathcal{O}(\x{q})$. The elements in $\mathrm{SL}_{2}\C$ are generated by the following three transformations
\begin{enumerate}
\item
$(\phi f_{1})(z)=f_{1}(z+\mu)$ for $\phi=\left(\begin{array}{cc}1&0\\-\mu&1\end{array}\right)$, 
\item
$(\phi f_{1})(z)=\lambda^{-q}f_{1}(\lambda^{2}z)$ for $\phi=\left(\begin{array}{cc}\lambda&0\\0&\lambda^{-1}\end{array}\right)$ and
\item
$(\phi f_{2})(\zeta)=f_{2}(\zeta+\kappa)$ for $\phi=\left(\begin{array}{cc}1&\kappa\\0&1\end{array}\right)$.
\end{enumerate}
A general differential pairing in the first coordinate chart $z$ has the form:
\begin{eqnarray*}
P:\mathcal{O}(\x{q})\times\mathcal{O}(\x{q'})&\rightarrow&\mathcal{O}(\x{p})\\
P(f,g)_{1}(z)&=&\sum_{i,j}a_{ij}(z)\left(\left(\frac{d}{dz}\right)^{i}f_{1}(z)\right)\left(\left(\frac{d}{dz}\right)^{j}g_{2}(z)\right),
\end{eqnarray*}
for some functions $a_{ij}(z)$. $P(f,g)_{2}(\zeta)$ is defined analogously with $\zeta$ instead of $z$.
In order for this to be an invariant differential pairing, the invariance equation
$$P(\phi f,\phi g)(z)=(\phi P(f,g))(z),$$
for $\phi\in\mathrm{SL}_{2}\C$, has to be satisfied.  Moreover, we must have
$$P(f,g)_{1}(z)=\zeta^{-p}P(f,g)_{2}(\zeta),$$
for $\zeta=-z^{-1}$, in order to obtain a section of $\mathcal{O}(\x{p})$.
The first transformation rule immediately implies that 
\begin{eqnarray*}
&&
\sum_{i,j}a_{ij}(z)\left(\left(\frac{d^{i}f}{dz^{i}}\right)(z+\mu)\right)\left(\left(\frac{d^{j}g}{dz^{j}}\right)(z+\mu)\right)\\
&=&\sum_{i,j}a_{ij}(z+\mu)\left(\left(\frac{d^{i}f}{dz^{i}}\right)(z+\mu)\right)\left(\left(\frac{d^{j}g}{dz^{j}}\right)(z+\mu)\right)
\end{eqnarray*}
and hence that the functions $a_{i,j}(z)$ are all constant. The last transformation rule analogously implies that the $a_{ij}(\zeta)$ are constant.
The second transformation rule implies
\begin{eqnarray*}
&&\sum_{i,j}\lambda^{2(i+j)-(q+q')}a_{ij}\left(\left(\frac{d^{i}f}{dz^{i}}\right)(\lambda^{2}z)\right)\left(\left(\frac{d^{j}g}{dz^{j}}\right)(\lambda^{2}z)\right)\\
&=&\lambda^{-p}\sum_{i,j}a_{ij}\left(\left(\frac{d^{i}f}{dz^{i}}\right)(\lambda^{2}z)\right)\left(\left(\frac{d^{j}g}{dz^{j}}\right)(\lambda^{2}z)\right)
\end{eqnarray*}
and hence $p=q+q'-2M$ for some $M\in\N$ and every term with $i+j\not=M$ vanishes. Thus the general $M$-th order invariant differential pairing looks like this:
\begin{eqnarray*}
P:\mathcal{O}(\x{q})\times\mathcal{O}(\x{q'})&\rightarrow&\mathcal{O}(\x{q+q'-2M})\\
P(f,g)_{1}(z)&=&\sum_{j=0}^{M}\gamma_{M,j}\left(\left(\frac{d}{dz}\right)^{j}f_{1}(z)\right)\left(\left(\frac{d}{dz}\right)^{M-j}g_{2}(z)\right),
\end{eqnarray*}
for some constants $\gamma_{M,j}\in\C$. The equation
$$P(f,g)_{1}(z)=\zeta^{-(q+q'-2M)}P(f,g)_{2}(\zeta),$$
for $\zeta=-z^{-1}$, will determine the constants $\gamma_{M,j}$ uniquely. This is shown by the following two lemmata:

\begin{lemma}
The transformation law 
$$z^{-k-2}\left(z^{2}\frac{d}{dz}\right)^{k+1}z^{-k}\psi(z)=\frac{d^{k+1}}{dz^{k+1}}\psi(z)$$
holds for an arbitrary function $\psi(z)$, $z\not=0$ and $k\geq -1$.
\end{lemma}
\begin{proof}
This lemma is stated in~\cite{g2} and easily proved by induction.
\end{proof}

\begin{lemma}
If $q\not\in\{0,1,...,M-1\}$ or $q'\not\in\{0,1,...,M-1\}$, then the equation
\begin{eqnarray*}
&&\sum_{j=0}^{M}\gamma_{M,j}\left(\frac{d^{j}}{dz^{j}}f_{1}(z)\right)\left(\frac{d^{M-j}}{dz^{M-j}}g_{1}(z)\right)\\
&=&\zeta^{2M-(q+q')}\sum_{j=0}^{M}\gamma_{M,j}\left(\frac{d^{j}}{d\zeta^{j}}f_{2}(\zeta)\right)\left(\frac{d^{M-j}}{d\zeta^{M-j}}g_{2}(\zeta)\right)
\end{eqnarray*}
uniquely determines the constants $\gamma_{M,j}$ as
$$\gamma_{M,j}=(-1)^{j}\binom{M}{j}\prod_{i=j}^{M-1}(q-i)\prod_{i=M-j}^{M-1}(q'-i).$$
\end{lemma}
\begin{proof}
First of all, we use Lemma 1 to compute
\begin{eqnarray*}
&&\sum_{j=0}^{M}\gamma_{M,j}\left(\frac{d^{j}}{dz^{j}}f_{1}(z)\right)\left(\frac{d^{M-j}}{dz^{M-j}}g_{1}(z)\right)\\
&=&\sum_{j=0}^{M}\gamma_{M,j}z^{-M-2}\left(\left(z^{2}\frac{d}{dz}\right)^{j}z^{-j+1}f_{1}(z)\right)\left(\left(z^{2}\frac{d}{dz}\right)^{M-j}z^{-M+j+1}g_{1}(z)\right).
\end{eqnarray*}
Now we change coordinates $\zeta=-\frac{1}{z}$ and therefore $z^{2}\frac{d}{dz}=\frac{d}{d\zeta}$.
%  Moreover let $f_{1}(z)=\omega^{-q}f_{2}(\omega)$ and $g_{1}(z)=\omega^{-q'}g_{2}(\omega)$. 
Then we have
\begin{eqnarray*}
&&\sum_{j=0}^{M}\gamma_{M,j}z^{-M-2}\left(\left(z^{2}\frac{d}{dz}\right)^{j}z^{-j+1}f_{1}(z)\right)\left(\left(z^{2}\frac{d}{dz}\right)^{M-j}z^{-M+j+1}g_{1}(z)\right)\\
&=&\sum_{j=0}^{M}\gamma_{M,j}\zeta^{M+2}\left(\frac{d^{j}}{d\zeta^{j}}\zeta^{j-1-q}f_{2}(\zeta)\right)\left(\frac{d^{M-j}}{d\zeta^{M-j}}\zeta^{M-j-1-q'}g_{2}(\zeta)\right)\\
&=&\sum_{j=0}^{M}\gamma_{M,j}\zeta^{M+2}\left(\sum_{i=0}^{j}\binom{j}{i}\prod_{l=1}^{i}(j-q-l)\zeta^{j-1-q-i}\frac{d^{j-i}}{d\zeta^{j-i}}f_{2}(\zeta)\right)\\
&&\times\left(\sum_{i'=0}^{M-j}\binom{M-j}{i'}\prod_{k=1}^{i'}(M-j-q'-k)\zeta^{M-j-1-q'-i'}\frac{d^{M-j-i'}}{d\zeta^{M-j-i'}}g_{2}(\zeta)\right)\\
&=&\zeta^{2M-q-q'}\sum_{j=0}^{M}\gamma_{M,j}\left(\frac{d^{j}}{d\zeta^{j}}f_{2}(\zeta)\right)\left(\frac{d^{M-j}}{d\zeta^{M-j}}g_{2}(\zeta)\right)\\
&&+\mathrm{Obstructions}.
\end{eqnarray*}
Let us look at a general obstruction term 
$$\zeta^{2M-q-q'-(M-s-t)}\left(\frac{d^{s}}{d\zeta^{s}}f_{2}(\zeta)\right)\left(\frac{d^{t}}{d\zeta^{t}}g_{2}(\zeta)\right),$$
for arbitrary $s,t\leq M$. The constant in front of this term is given by
$$\sum_{j=s}^{M-t}\gamma_{M,j}\binom{j}{j-s}\prod_{l=1}^{j-s}(j-q-l)\binom{M-j}{M-j-t}\prod_{k=1}^{M-j-t}(M-j-q'-k).$$
For $s+t=M-1$, we obtain the following $M$ equations:
\begin{equation}\label{mainequation}
\fbox{$\displaystyle \gamma_{M,s}(M-s)(M-s-q'-1)+\gamma_{M,s+1}(s+1)(s-q)=0,$}
\end{equation}
for $s=0,...,M-1$. If $q\not\in\{0,1,...,M-1\}$ or $q'\not\in\{0,1,...,M-1\}$, this determines uniquely (up to scale): 
\begin{eqnarray*}
\gamma_{M,j}&=&(-1)^{j}\binom{M}{j}\prod_{i=j}^{M-1}(q-i)\prod_{i=1}^{j}(q'-M+i)\\
&=&(-1)^{j}\binom{M}{j}\prod_{i=j}^{M-1}(q-i)\prod_{i=M-j}^{M-1}(q'-i).
\end{eqnarray*}
Having defined $\gamma_{M,j}$ like this, the constants in front of the other obstruction terms are given by
$$\prod_{i=s}^{M-1}(q-i)\prod_{i=t}^{M-1}(q'-i)(-1)^{M-s-t}\frac{M!}{s!t!}\sum_{j=s}^{M-t}(-1)^{j}\frac{1}{(j-s)!(M-j-t)!}.$$
If $t=M-s$, then this constant is exactly $\gamma_{M,s}$. If $s+t\not=M$, then this vanishes due to the fact that
$$\sum_{i=0}^{n}(-1)^{i}\binom{n}{i}=0,$$
for all $n\geq 1$.
So the only terms we are left with are
\begin{eqnarray*}
&&\zeta^{2M-q-q'}\sum_{j=0}^{M}\gamma_{M,j}\left(\frac{d^{j}}{d\zeta^{j}}f_{2}(\zeta)\right)\left(\frac{d^{M-j}}{d\zeta^{M-j}}g_{2}(\zeta)\right).
%&=&z^{q+q'-2M}\sum_{j=0}^{M}\gamma_{M,j}\left(\left(z^{2}\frac{d}{dz}\right)^{j}z^{-q}f(z)\right)\left(\left(z^{2}\frac{d}{dz}\right)^{M-j}z^{-q'}g(z)\right).
\end{eqnarray*}
\end{proof}

Therefore, we have proved:

\begin{theorem}
If $q\not\in\{0,1,...,M-1\}$ or $q'\not\in\{0,1,...,M-1\}$, there is exactly one $M$-th order bilinear differential pairing between section of $\mathcal{O}(\x{q})$ and $\mathcal{O}(\x{q'})$ on the Riemann sphere that is invariant with respect to Moebius transformations. This pairing
is given by:
\begin{eqnarray*}
P:\mathcal{O}(\x{q})\times\mathcal{O}(\x{q'})&\rightarrow&\mathcal{O}(\x{q+q'-2M})\\
P(f,g)(z)&=&\sum_{j=0}^{M}\gamma_{M,j}\left(\left(\frac{d}{dz}\right)^{j}f(z)\right)\left(\left(\frac{d}{dz}\right)^{M-j}g(z)\right),
\end{eqnarray*}
with
$$\gamma_{M,j}=(-1)^{j}\binom{M}{j}\prod_{i=j}^{M-1}(q-i)\prod_{i=M-j}^{M-1}(q'-i).$$
\end{theorem}
The formula in the theorem above also shows a general dichotomy. There is a critical set of weights $K=\{0,1,...,M-1\}$. If $q$ or $q'$ do not lie in $K$, then there exists exactly one pairing as described
above. However, if $q,q'\in K$, then several peculiar things can happen. Firstly, for $q+q'>M-2$, $P(f,g)=0$. However, (\ref{mainequation}) can still be solved and there exists an invariant bilinear differential pairing. Secondly, in~\cite{g2} it is shown that, for $q\in K$, $\mathcal{D}^{q+1}(f)=\frac{d^{q+1}}{dz^{k+1}}f$ is a linear invariant differential operator.
So, for example for $q=q'=M-1$, there are two independent invariant bilinear differential pairings given by $f\mathcal{D}^{M}g$ and $g\mathcal{D}^{M}f$. This problem will accompany us throughout this thesis, where we will always have to exclude certain numbers (weights, representations) in order to obtain a classification of invariant differential pairings.
One main aim will be to associate a certain meaning to these excluded numbers as we have done here: an excluded number corresponds to the existence of a linear invariant differential  operator.

\subsubsection{Remark}
The formula above also appears in~\cite{ol}, Theorem 3.46, as the $M$-th transvectant of two polynomials $f$ and $g$.

\subsection{Outline of this thesis}
In the first chapter we give the basic background to representation theory and parabolic geometry that is needed in order to understand this thesis. The material (especially about Lie algebras and 
representation theory) is fairly standard and may be skipped by anyone who is familiar with the subject. However, many notations are introduced that will be used freely throughout this thesis.
\par
In the second chapter the basic notion of an invariant bilinear differential pairing is introduced. These pairings are the central objects of our study. Firstly, we describe pairings analytically in terms of
(weighted-) jet bundles and define (weighted) {\bf bi-jet} bundles. Then we give an algebraic definition and study pairings algebraically on homogeneous spaces which are the model spaces for the various
types of parabolic geometries. The notion of a {\bf generalized bi-Verma module} is introduced and we explain how for invariant bilinear differential pairings these modules play the same role as generalized Verma modules
play for invariant linear differential operators. More precisely, invariant differential pairings can be described by singular vectors in generalized bi-Verma modules. A first large class of invariant differential pairings is constructed
with this method. Finally, we discuss in detail the notion of invariance for a general curved parabolic geometry and end by describing the geometric structures that underlie parabolic geometries. Apart from 
the material in the last section (which is taken directly from various sources in the literature) the concepts introduced are new. The computations, however, are modelled on those that are used for
describing invariant differential operators.
\par
The third chapter gives a classification of all first order bilinear invariant differential pairings for all non-degenerate cases. It is explained how a degenerate case corresponds to the existence of an invariant linear
differential operator. Several examples are given. The methods used are modeled on the methods of~\cite{css}. A (weaker version) of this chapter has been accepted for publication~\cite{k}.
\par
The fourth chapter deals with higher order invariant differential pairings for which we can write down an explicit formula including curvature correction terms. The remarkable result is that those
pairings only depend on the order and neither on the specific parabolic geometry nor on the vector bundles involved. These pairings are obtained by using Ricci-corrected derivatives as introduced in~\cite{cds}.
\par
In order to construct higher order invariant differential pairings, the fifth chapter firstly reviews some deep results about Lie algebra cohomology.
Essentially using the ideas from~\cite{bceg}, we then define certain tractor bundles (the {\bf $M$-bundles}) that encode the information about the (weighted) $M$-jets of sections of vector bundles.
An easy argument shows that the tensor products of those tractor bundles decompose into irreducible components that are exactly the possible targets for our pairings. The final step is to define
{\bf splitting operators} that include the bundles in question into these $M$-bundles. At this stage, the results cease to be 100\% satisfactory. To be more precise, as in the classification of first order invariant differential pairings, certain representations (or rather certain
geometric weights) have to be excluded. However, we can only conjecture that every excluded weight corresponds to the existence of an invariant linear differential operator.
We will come back to this point in the last section about open problems. However, modulo the problem of excluding too many representation than absolutely necessary, all higher order pairings
with exactly the right multiplicity that we expect to exist and that are curved versions of pairings on the homogeneous model spaces can be constructed this way.
The splitting operators defined in this section are a generalization of the curved Casimir operator defined in~\cite{cs3}. Finally, we examine closely higher order pairings for projective geometry. In this case, all the excluded representations correspond to the existence of invariant linear differential operators. This result has also been accepted for publication in~\cite{k}.
\par
The objective of the last chapter is twofold. Firstly, we take a closer look at three examples of parabolic geometry that we refer to throughout the thesis: projective geometry, conformal 
geometry and CR geometry. Then we describe the tractor calculus for those geometries which enables one to carry out explicit computations with tractors in order to get a better understanding of the abstract theory which is used in Chapter five. 
Secondly, we explicitly construct some special splittings and show how this leads to explicit formulae for higher order pairings. In particular, one can in theory write down all the higher order pairings between sections of those bundles for which we have written down the splitting, even those pairings with multiplicity. In practice this can be extremely tedious due to complicated expressions for some tractors. The splittings that we construct for projective 
geometry were first written down in~\cite{f1} and those for conformal geometry (and conformal weight 0) in~\cite{e2}, but we tried to construct these splittings in a unified manner which is more adapted to
the general flavour of this thesis. 
\par
The appendix describes the BGG sequences of projective, conformal and CR geometry. These sequences are not directly used in any part of the thesis, but they are mentioned every now and then when we
talk about standard operators. For the convenience of the reader, we have included them as an appendix.

\subsubsection{The achievements of this thesis}
The most important new results and concepts that we have introduced and proved in this paper are:
\begin{enumerate}
\item
The conceptual description of bilinear differential pairings in terms of bi-jet bundles, which are constructed out of jet bundles.
\item 
The algebraic criterion for the existence of invariant bilinear differential pairings on homogeneous spaces via singular vectors in generalized bi-Verma modules.
\item
The classification of non-degenerate first order invariant bilinear differential pairings and the characterization of degeneracy in terms of the existence of invariant linear differential operators.
\item
A general formula including curvature correction terms for a certain class of higher order invariant bilinear differential pairings.
\item
The construction of general higher order invariant bilinear differential pairings under the exclusion of certain representations.
\item
A precise analysis of higher order invariant bilinear differential pairings on manifolds with a projective structure including the interpretation of each excluded geometric weight in terms of the existence of an invariant linear differential operator.
%\item
%Explicit construction of splitting operators for totally symmetic tensors in CR geometry.
\end{enumerate}

\subsubsection{A note on the length of this thesis}
The author is aware that this thesis is probably a bit longer than it necessarily needs to be. Several calculations are taken from various sources in the literature and have just been rewritten with our conventions and notations. Wherever this is the case, an explicit reference to the original source will be given. The reason for including all those explicit calculations is simple: this text is supposed to be as self contained as possible and accessible to someone with relatively little background knowledge.  The decisions as to which details to include and which to leave out are obviously based on the background knowledge of the author. For that I apologize.

\chapter{Background}
In this first chapter we will give the necessary background needed to understand the theory of pairings as presented in this thesis. Most definitions are absolutely standard, but they have to be stated in order to introduce notations that will be used freely throughout this thesis. The first part of this chapter deals with Lie algebras and representation theory, whereas the second part is devoted to the introduction of parabolic geometries.

\section{Lie algebras}

\subsection{Root spaces}
Let $\mathfrak{g}$ be a complex semisimple Lie algebra.
% The Cartan subalgebras
%\footnote{A \emph{Cartan subalgebra} of $\g$ is a nilpotent subalgebra that equals its normalizer in $\g$.}
%$\mathfrak{h}$ of $\mathfrak{g}$ are precisely the maximal toral
%subalgebras
%\footnote{A subalgebra of $\g$ is called \emph{toral} if it consists of semisimple elements.} 
%of $\mathfrak{g}$ since $\mathfrak{g}$ is semisimple (see [H], p.~80) and any two of those are
%conjugate under the adjoint action of $G$, where $G$ is the associated simply connected Lie group.
We fix a {\bf Cartan subalgebra} $\h$ of $\g$ and define
$$\g_{\alpha }=\{X\in \g:\;[H,X]=\alpha (H)X\;\forall \;H\in \h\}$$
for all $\alpha \in \h^{*}$. Then we set
$$\Delta=\Delta  (\g,\h)=\{\alpha \in \h^{*}:\;\g_{\alpha }\not= \emptyset\;,\;\alpha \not=0\}$$
and call the elements in $\Delta  (\g,\h)$ {\bf roots} of $\g$ relative to $\h$. According to a standard result in linear algebra we get a root space decomposition
$$\g=\h\oplus \bigoplus_{\alpha \in \Delta  (\g,\h)}\g_{\alpha }.$$

\subsubsection{Example}
\begin{enumerate}
\item
Look at 
$$A_{l}=\mathfrak{sl}_{l+1}\C=\{X\in\mathfrak{gl}_{l+1}\C\;:\; \tr(X)=0\}$$
%with corresponding group $\mathrm{SL}_{n+1}\C=\{g\in \C^{n+1\times n+1}\;:\;\det(g)=1\}$.
with Cartan subalgebra 
$$\mathfrak{h}=\left\{\sum_{i=1}^{n+1}a_{i}H_{i}:\;\sum_{i=1}^{l+1}a_{i}=0\right\},$$
where $H_{i}$ is the diagonal matrix which has a one in the $i$-th entry and zeroes elsewhere. Therefore we find
$$\h^{*}=\mathrm{span}\{\epsilon_{i}\}_{i=1,...,l+1}/\left(\sum_{i=1}^{l+1}\epsilon_{i}=0\right),$$
where $\epsilon_{i}$ is defined by $\epsilon_{i}(H_{j})=\delta _{ij}$. We will denote the equivalence class of $\epsilon_{i}$ in $\h^{*}$ also by $\epsilon_{i}$. If we
denote the matrix which has a $1$ in the $i$-th row and $j$-th
column and zeroes elsewhere by $E_{i,j}$, we can calculate
$$ [H,E_{i,j}]=(a_{i}-a_{j})E_{ij},$$
for $H=\sum_{i=1}^{l+1}a_{i}H_{i}$. Therefore $E_{i,j}\in \g_{\epsilon_{i}-\epsilon_{j}}$ and the set of roots in $\mathfrak{sl}_{l+1}\C$ is just $\{\epsilon_{i}-\epsilon_{j}\}_{i,j=1,...,l+1}$.
\item
The next example is 
$$D_{l}=\mathfrak{so}_{2l}\C=\left\{\left(\begin{array}{cc}
A&B\\
C&D
\end{array}\right)\in \mathfrak{gl}_{2l}\C\;:\;D=-A^{T}\;\text{and}\;B,C\;\text{are skew symmetric}\right\},$$
with Cartan subalgebra 
$$\h=\mathrm{span}\{H_{i}=E_{i,i}-E_{l+i,l+i},\;i=1,...,l\}.$$
This yields
$$\h^{*}=\mathrm{span}\{\epsilon_{i}\;:\;\epsilon_{i}(H_{j})=\delta_{i,j}\}$$
and one can check that the roots are given by $\Delta(\g,\h)=\{\pm \epsilon_{i}\pm\epsilon_{j}\}_{i\not=j}$.
\item
The final example is 
$$B_{l}=\mathfrak{so}_{2l+1}\C=\left\{\left(\begin{array}{cc}
A &v\\
-v^{T}&0
\end{array}\right)\;:\;A\in\mathfrak{so}_{2l}\C,v \in\C^{2l}\right\},$$
where we can take $\h$ to be the Cartan subalgebra of $\mathfrak{so}_{2l}\C$ included in $\mathfrak{so}_{2l+1}\C$ in the obvious manner.
Hence  
$$\h^{*}=\mathrm{span}\{\epsilon_{i}\;:\;\epsilon_{i}(H_{j})=\delta_{i,j}\}$$
and 
$$\Delta(\g,\h)=\{\pm\epsilon_{i}\pm\epsilon_{j}\}_{i\not=j}\cup\{\pm\epsilon_{i}\}_{i=1,...,l}.$$
\end{enumerate}

\subsection{Dynkin diagrams}
As is proved in~\cite{h}, p.~48, every Cartan subalgebra $\h$ has a
{\bf basis} 
$$\mathcal{S}=\{\alpha_{1},...,\alpha_{l}\}\subseteq \Delta  (\g,\h),$$
so that every
root may be written as a linear combination of elements in
$\mathcal{S}$ with all non-negative or all non-positive
coefficients. $\mathcal{S}$ is then called a system of
{\bf simple roots} of $\g$. We will fix such a basis
$\mathcal{S}$, which induces a partial ordering
$$\lambda \succeq \mu \Leftrightarrow \lambda -\mu =\sum_{i}a_{i}\alpha _{i}\; \text{with}\;\alpha _{i}\in \mathcal{S}\;\text{and}\;a_{i}\geq 0,$$
so that we can set $\Delta ^{+}(\g,\h)=\{\alpha \in \Delta
(\g,\h):\;\alpha \succ 0\}$ for the set of {\bf positive
roots}.
\par
For every $X\in\g$, we have a canonical endomorphism $\mathrm{ad}(X):\g\rightarrow\g$ with 
$$\mathrm{ad}(X)(Y)=[X,Y].$$
This allows us to define a symmetric bilinear form, the so-called {\bf Killing form}, by
$$B(X,Y)= \tr (\mathrm{ad}(X)\circ \mathrm{ad}(Y))$$
on $\g$ which is non-degenerate if and only if $\g$ is semisimple
(see~\cite{h}, p.~22). We restrict this symmetric bilinear form to a non-degenerate form on $\h$ (\cite{h}, p.~37) and therefore on $\h^{*}$. The
form on $\h^{*}$ can be defined by
$$B(\alpha ,\beta)=B(h_{\alpha },h_{\beta })\;\forall \; \alpha ,\beta \in \h^{*},$$
where $h_{\alpha }$ is the unique element in $\h$ such that $\alpha (H)=B(h_{\alpha },H)$ for all $H\in \h$. For any root $\alpha $ we define the
co-root $\alpha ^{\vee }=2\alpha /B(\alpha ,\alpha)$ and obtain the {\bf Cartan integers}
$$c_{ij}=B(\alpha _{i},\alpha _{j}^{\vee}),$$
for all pairs of simple roots $\alpha _{i},\alpha _{j}\in
\mathcal{S}$. If $\alpha ,\beta $ are roots, then $B(\alpha,\beta ^{\vee})\in \Z$ (\cite{h}, p.~40). It can easily be seen
that $\mathcal{S}$ spans $\Delta  (\g,\h)$ over the rational
numbers (\cite{h}, p.~37 and p.~39) so that the $\Q$ subspace $E_{\Q}$ of
$\h^{*}$ spanned by all the roots has $\Q$-dimension
$l=\dim_{\C}\h^{*}$. If we also allow real coefficients, we get a
real vector space $E=\R\otimes _{\Q}E_{\Q}$, i.e.~$\h^{*}$ is the
complexification of $E$. In this space $E$ we can look at the
angle $\Theta $ of two roots $\alpha ,\beta $ and obtain $B(\alpha ,\beta ^{\vee})B(\beta ,\alpha^{\vee})=4\cos^{2}\Theta $. Therefore our Cartan integers
can only take the values $0,\pm 1,\pm 2,\pm 3$ and we can draw a
{\bf Dynkin diagram}, which is a graph such that nodes correspond
to simple roots, and edges determine $c_{ij}$ according to
\begin{enumerate}
\item $B(\alpha _{i},\alpha _{i}^{\vee})=2$,
\item $\alpha _{i}\not=\alpha _{j}$ are connected if and only if $B(\alpha _{i},\alpha _{j}^{\vee})\not=0$,
\item 
\begin{enumerate}
\item
$\oo{\alpha}{\beta}
\quad\Leftrightarrow \quad B(\alpha ,\beta^{\vee})=-1,\; B(\beta,\alpha^{\vee})=-1$,
\item
$\Btwoo{\alpha}{\beta}
\quad\Leftrightarrow \quad B(\alpha ,\beta^{\vee})=-2,\; B(\beta,\alpha^{\vee})=-1$ and
\item
$
\begin{picture}(1,1)
\put(0,0){$\bullet$}
\put(1,0){$\bullet$}
\put(0.2,0.3){\line(1,0){1}}
\put(0.2,0.2){\line(1,0){1}}
\put(0.2,0.1){\line(1,0){1}}
\put(0,0.5){\scriptsize $\alpha$}
\put(1,0.5){\scriptsize $\beta$}
\put(0.5,0){$\rangle$}
\end{picture}\hspace{0.6cm}\Leftrightarrow \quad B(\alpha ,\beta^{\vee})=-3,\; B(\beta,\alpha^{\vee})=-1$. 
\end{enumerate}
\end{enumerate}
It is the central theorem of semisimple complex Lie algebras that
the structure of $E$ and the knowledge of the Dynkin diagram
determines and is determined by $\g$ (see~\cite{h}, p.~57,65).

\subsubsection{Example}
\begin{enumerate}
\item
In the case of $A_{l}$ we can take
$$\mathcal{S}=\{\alpha _{i}=\epsilon_{i}-\epsilon_{i+1},\;i=1,...,l\}$$
and calculate the Killing form as
$$B\left(\sum_{i}a_{i}\epsilon_{i},\sum_{j}b_{j}\epsilon_{j}\right)=\frac{1}{2(l+1)}\left(\sum_{i}a_{i}b_{i}-\frac{1}{l+1}\sum_{i,j}a_{i}b_{j}\right).$$
Therefore we obtain the Dynkin diagram of $\mathfrak{sl}_{l+1}\C$:
$$\begin{picture}(5,2)
\put(0,1){\makebox(0,0){$\bullet$}}
\put(0,1){\line(1,0){1}}
\put(1,1){\makebox(0,0){$\bullet$}}
\put(1,1){\line(1,0){1}}
\put(2,1){\makebox(0,0){$\bullet$}}
\put(3.5,1){\makebox(0,0){$\bullet$}}
\put(2.5,1){...}
\put(4.5,1){\makebox(0,0){$\bullet$}}
\put(3.5,1){\line(1,0){1}}
\put(5,1){,}
\end{picture}$$
where there are $l$ nodes.
\item
For $D_{l}$ we can take 
$$\mathcal{S}=\{\alpha_{i}=\epsilon_{i}-\epsilon_{i+1},\;i=1,...,l-1,\;\alpha_{l}=\epsilon_{l-1}+\epsilon_{l}\}$$
and calculate the Killing form as
$$B\left(\sum_{i}a_{i}\epsilon_{i},\sum_{j}b_{j}\epsilon_{j}\right)=\frac{1}{4l-4}\left(\sum_{i}a_{i}b_{i}\right).$$
The Dynkin diagram has the form
$$
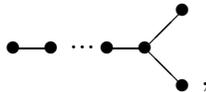
\begin{picture}(5,2)
\put(0,1){\makebox(0,0){$\bullet$}}
\put(1,1){\makebox(0,0){$\bullet$}}
\put(3.5,1){\makebox(0,0){$\bullet$}}
\put(1.5,1){...}
\put(2.5,1){\makebox(0,0){$\bullet$}}
\put(4.5,2){\makebox(0,0){$\bullet$}}
\put(4.5,0){\makebox(0,0){$\bullet$}}
\put(0,1){\line(1,0){1}}
\put(2.5,1){\line(1,0){1}}
\put(3.5,1){\line(1,1){1}}
\put(3.5,1){\line(1,-1){1}}
\end{picture},$$
where there are $l$-nodes.
\item
For $B_{l}$ we can take
$$\mathcal{S}=\{\alpha_{i}=\epsilon_{i}-\epsilon_{i+1},\;i=1,...,l-1,\;\alpha_{l}=\epsilon_{l}\}$$
and calculate the Killing form as
$$B\left(\sum_{i}a_{i}\epsilon_{i},\sum_{j}b_{j}\epsilon_{j}\right)=\frac{1}{4l-2}\left(\sum_{i}a_{i}b_{i}\right).$$
The Dynkin diagram has the form
$$\ooo{\;\;}{\;\;}{\;\;}\;...\;\btwo{\;\;}{\;\;},$$
where there are $l$-nodes.
\end{enumerate}

\subsection{Parabolic subalgebras}
Let $G$ be a complex semisimple and simply connected Lie group with Lie algebra $\g$.
{\bf Borel subalgebras} of $\g$ are the maximal solvable
subalgebras of $\g$ and every Borel subalgebra is $G$-conjugate to
the standard Borel subalgebra (\cite{h}, p.~84)
$$\B=\h\oplus \mathfrak{n}$$
with
$$\mathfrak{n}=\bigoplus_{\alpha \in \Delta ^{+}(\g,\h)}\g_{\alpha }.$$
A {\bf parabolic subalgebra} $\p$ of $\g$ is a subalgebra which contains a Borel subalgebra. According to the standard form for $\B$ we get a standard form for $\p$:
Let $\mathcal{S}_{\p}\subseteq \mathcal{S}$ be any subset. Then we set
$$\Delta(\g_{0},\h)=\mathrm{span}\mathcal{S}_{\p}\cap \Delta (\g,\h),\quad \Delta (\p_{+},\h)=\Delta^{+}(\g,\h)\backslash\Delta(\g_{0},\h)$$
and accordingly
$$\mathfrak{g}_{0}=\h\oplus \bigoplus_{\alpha \in \Delta (\mathfrak{g}_{0},\h)}\g_{\alpha },\quad \mathfrak{p}_{+}=\bigoplus_{\alpha \in \Delta (\mathfrak{p}_{+},\h)}\g_{\alpha }.$$
$\g_{0}$ is reductive and can hence be written as $\g_{0}=\g_{0}^{S}+\mathfrak{z}(\g_{0})$, where $\mathfrak{z}(\g_{0})$ is the center of $\mathfrak{g}_{0}$ and $\mathfrak{g}_{0}^{S}=[\g_{0},\g_{0}]$
is semisimple of rank $\mid \mathcal{S}_{\p}\mid $. So finally we get the {\bf Levi-decomposition}
$$\p=\mathfrak{g}_{0}\oplus \mathfrak{p}_{+}.$$
To represent $\p$ we take the Dynkin diagram for $\g$ and cross through all nodes which lie in $\mathcal{S}\backslash \mathcal{S}_{\p}$.
\par
If we do not have any crossed through nodes, then $\mathcal{S}_{\p}=\mathcal{S}$ and
therefore $\p=\g$. If there are crosses through every node, then $\mathcal{S}_{\p}=\emptyset$ and we get $\p=\B$.
\par
\vspace{0.2cm}

\subsubsection{Example}
For $A_{l}$, $\mathfrak{n}$ consists of strictly upper triangular matrices, because the positive roots are the differences $\epsilon_{i}-\epsilon_{j}$ where $i<j$ and therefore
$\B$ consists of all the upper triangular matrices in $A_{l}$. We will demonstrate one particular example of a parabolic subalgebra, the general case being similar.
\par
\vspace{0.2cm}
Let $l=3$ and $\mathcal{S}_{\p}=\{\alpha_{2},\alpha _{3}\}$, then $\p$ is given by
$$\p=\left\{\left(\begin{array}{cccc} *&*&*&*\\
                                   0&*&*&*\\
                                   0&*&*&*\\
                                   0&*&*&*
\end{array}\right)\in \SL\right\}.$$
In this example we have $E_{3,2}\in \g_{-\alpha _{2}}$, $E_{4,3}\in \g_{-\alpha _{3}}$ and $E_{4,2}\in \g_{-(\alpha _{2}+\alpha _{3})}$ lying in $\p$ in addition to the upper triangular matrices. $\p$ is represented via the Dynkin diagram as follows:
$$\xoo{\;\;}{\;\;}{\;\;}.$$

There is indeed more structure to the parabolic subalgebra $\p$ given by a so-called {\bf grading}.

\subsection{$|k_{0}|$-graded Lie algebras}\label{grading}
Let $\mathcal{S}=\{\alpha_{i}\}_{i=1,...,l} $ be the simple roots of $\g$ with corresponding {\bf fundamental weights} $\{\omega_{j}\}_{j=1,...,l}$, i.e.~$B(\omega_{j},\alpha_{i}^{\vee})=\delta_{i,j}$. The parabolic subalgebra $\p$ is specified by a subset $\mathcal{S}_{\p}$ of the simple roots, so that 
$$\mathcal{S}\backslash\mathcal{S}_{\p}=\{\alpha_{i}\}_{i\in I},$$
where $I=\{i_{1},...,i_{l_{0}}\}$. We set $J=\{j_{1},...,j_{l-l_{0}}\}=\{1,...,l\}\backslash\{i_{1},...,i_{l_{0}}\}$. Then we can define
$$\g_{j}=\bigoplus_{\alpha\in\Delta_{j}}\g_{\alpha},\;\text{where}\;\Delta_{j}=\{\alpha=\sum_{i=1}^{l}n_{i}\alpha_{i}\in\Delta(\g,\h)\;:\;\sum_{i\in I}n_{i}=j\}.$$
This yields a $|k_{0}|$-grading:
$$\g=\g_{-k_{0}}\oplus...\oplus\g_{-1}\oplus\g_{0}\oplus\g_{1}\oplus...\oplus\g_{k_{0}},$$
where $[\g_{i},\g_{j}]\subset \g_{i+j}$ and $\g_{-}=\g_{-k_{0}}\oplus...\oplus\g_{-1}$ is generated by $\g_{-1}$. One can check that the non-negative part of this grading $\g_{0}\oplus...\oplus\g_{k_{0}}=\p$ and $\p_{+}=\g_{1}\oplus...\oplus\g_{k_{0}}$. The integer $k_{0}$ is easily computed: write the highest root $\gamma=\sum_{i=1}^{l}n_{i}\alpha_{i}$ as a sum of simple roots. Then $k_{0}=\sum_{i\in I}n_{i}$.
The following facts are known about this grading (see~\cite{cs}, Proposition 2.2):
\begin{enumerate}
\item
There exists a grading element $E\in\mathfrak{z}(\g_{0})$ such that the decomposition of
$\g$ corresponds to a decomposition under the adjoint action of $E$ into eigenspaces, i.e.~$[E,X]=jX$ if and only if $X\in\g_{j}$. 
\item
$B(\g_{i},\g_{j})=0$ for all $j\not=-i$ and the Killing form $B(.,.)$ can be used to define isomorphisms $\g_{-i}\cong\g_{i}^{*}$ for $i=1,...,k_{0}$.
\end{enumerate}

\subsubsection{Remark}
In the complex setting, a grading is equivalent to having a parabolic subalgebra $\p$, see~\cite{cs}. In the real setting, this is not the case. However, we will start of with a grading of a real Lie algebra $\g_{\R}$ that induces a grading of the complexification $\g$ and hence a parabolic subalgebra $\p\subset\g$. The representation theory will then be applied to the (complex) pair $(\g,\p)$.

\subsubsection{Example}\label{gradingexample}
\begin{description}
\item{(a)}
The grading corresponding to the parabolic subalgebra
$$\begin{picture}(5,2)
\put(0,1){\makebox(0,0){$\cross$}}
\put(1,1){\makebox(0,0){$\bullet$}}
\put(2,1){\makebox(0,0){$\bullet$}}
\put(1,1){\line(1,0){1}}
\put(3.5,1){\makebox(0,0){$\bullet$}}
\put(2.5,1){...}
\put(4.5,1){\makebox(0,0){$\bullet$}}
\put(0,1){\line(1,0){1}}
\put(3.5,1){\line(1,0){1}}
%\put(5,1){,}
\end{picture}$$
of $A_{l}$ is given by $\g=\g_{-1}\oplus\g_{0}\oplus\g_{1}$, with
$$\g_{-1}=\left\{\left(\begin{array}{cc}
0 & 0\\
v & 0
\end{array}\right)\in\g\right\},\;
\g_{0}=\left\{\left(\begin{array}{cc}
x & 0\\
0 & A
\end{array}\right)\in\g\right\}\;\text{and}\;
\g_{1}=\left\{\left(\begin{array}{cc}
0 & \tilde{v}^{T}\\
0 & 0
\end{array}\right)\in\g\right\}.$$
Here we have $v,\tilde{v}\in\C^{l}$, $x\in\C$, $A\in\mathfrak{gl}_{l}\C$ with $\tr(A)+x=0$.
\item{(b)}
The grading corresponding to the parabolic subalgebra
$$
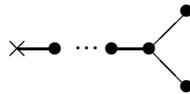
\begin{picture}(5,4)
\put(0,1){\makebox(0,0){$\cross$}}
\put(1,1){\makebox(0,0){$\bullet$}}
\put(3.5,1){\makebox(0,0){$\bullet$}}
\put(1.5,1){...}
\put(2.5,1){\makebox(0,0){$\bullet$}}
\put(4.5,2){\makebox(0,0){$\bullet$}}
\put(4.5,0){\makebox(0,0){$\bullet$}}
\put(0,1){\line(1,0){1}}
\put(2.5,1){\line(1,0){1}}
\put(3.5,1){\line(1,1){1}}
\put(3.5,1){\line(1,-1){1}}
\end{picture}$$
of $D_{l}$ is given by $\g=\g_{-1}\oplus\g_{0}\oplus\g_{1}$, with
$$\g_{-1}=\left\{\left(\begin{array}{cccc}
0 &0&0&0\\
v&0&0&0\\
0&w^{T} & 0&-v^{T}\\
-w&0&0&0
\end{array}\right)\in\g\right\},\;
\g_{0}=\left\{\left(\begin{array}{cccc}
x&0&0&0\\
0&A&0&B\\
0&0 & -x&0\\
0&C&0&-A^{T}
\end{array}\right)\in\g\right\}$$
and
$$\g_{1}=\left\{\left(\begin{array}{cccc}
0 &\tilde{v}^{T}&0&\tilde{w}^{T}\\
0&0&-\tilde{w}&0\\
0&0 & 0&0\\
0&0&-\tilde{v}&0
\end{array}\right)\in\g\right\}.$$
Here we have $v,\tilde{v},w,\tilde{w}\in\C^{l-1}$, $A,B,C\in\mathfrak{gl}_{l-1}\C$, $x\in\C$ and $B,C$ have to be skew symmetric.
\par
The grading corresponding to
$$\xoo{\;\;}{\;\;}{\;\;}\;...\;\btwo{\;\;}{\;\;}$$
can be described similarly with
$$\Delta_{-1}=\{-(\epsilon_{1}\pm\epsilon_{i})\}_{i=2,...,l}\cup\{-\epsilon_{1}\},\;\;\Delta_{1}=-\Delta_{-1}\;\text{and}\;\Delta_{0}=\Delta\backslash\{\Delta_{1}\cup\Delta_{-1}\}.$$
\item{(c)}
The grading corresponding to the parabolic subalgebra 
$$\begin{picture}(5,2)
\put(0,1){\makebox(0,0){$\cross$}}
\put(1,1){\makebox(0,0){$\bullet$}}
\put(2,1){\makebox(0,0){$\bullet$}}
\put(1,1){\line(1,0){1}}
\put(3.5,1){\makebox(0,0){$\bullet$}}
\put(2.5,1){...}
\put(4.5,1){\makebox(0,0){$\cross$}}
\put(0,1){\line(1,0){1}}
\put(3.5,1){\line(1,0){1}}
\end{picture}$$
of $A_{l+1}$ is given by $\g=\g_{-2}\oplus\g_{-1}\oplus\g_{0}\oplus\g_{1}\oplus\g_{2}$, with
$$\g_{-2}=\left\{\left(\begin{array}{ccc}
0 & 0&0\\
0 & 0&0\\
z & 0&0
\end{array}\right)\in\g\right\},\;\g_{-1}=\left\{\left(\begin{array}{ccc}
0 & 0&0\\
v & 0&0\\
0& w^{T}&0
\end{array}\right)\in\g\right\}$$
$$\g_{0}=\left\{\left(\begin{array}{ccc}
x & 0&0\\
0 & A&0\\
0& 0&y
\end{array}\right)\in\g\right\}$$
and 
$$\g_{1}=\left\{\left(\begin{array}{ccc}
0 & \tilde{v}^{T}&0\\
0 & 0&\tilde{w}\\
0& 0&0
\end{array}\right)\in\g\right\},\;\g_{2}=\left\{\left(\begin{array}{ccc}
0 & 0&\tilde{z}\\
0 & 0&0\\
0& 0&0
\end{array}\right)\in\g\right\}.$$
\end{description}
Here we have $v,\tilde{v},w,\tilde{w}\in\C^{l}$, $x,y,z,\tilde{z}\in\C$, $A\in\mathfrak{gl}_{l}\C$ and $x+y+\tr(A)=0$.

\section{Representation theory}

\subsection{Representations of (complex) semisimple Lie algebras}
Let $\mathbb{V}$ be a finite dimensional representation of $\g$. We denote
the action of $\g$ on $\mathbb{V}$ by $x.v$ for all $x\in \g$ and $v\in \mathbb{V}$.
An element $v\in \mathbb{V}\backslash\{0\}$ is called weight vector of
{\bf weight} $\lambda \in \h^{*}$ if
$$H.v=\lambda (H)v\;\forall\; H\in \h.$$
We also denote by $\Delta (\mathbb{V})$ the set of all weights that arise via this construction. As in the last section we write
$$\mathfrak{n}=\bigoplus_{\alpha\in \Delta ^{+}(\g,\h)}\g_{\alpha }\quad\text{and}\quad \mathfrak{n}_{-}=\bigoplus_{\alpha\in \Delta ^{+}(\g,\h)}\g_{-\alpha }$$
for the raising and lowering subalgebras of $\g$. A vector $v\in \mathbb{V}$ is called maximal (resp.~minimal) if $v$ is killed by $\mathfrak{n}$
(resp.~$\mathfrak{n}_{-}$). A maximal (resp. minimal) vector of weight $\lambda $ is called {\bf highest weight vector} (resp.~{\bf lowest weight vector}) if
$\lambda \succeq \lambda '$ (resp.~$\lambda \preceq  \lambda '$ ) for all $\lambda '\in \Delta (\mathbb{V})$. A weight $\lambda $ that satisfies
$$B(\lambda ,\alpha^{\vee})\geq 0\;\forall \;\alpha \in \mathcal{S}$$
is called {\bf dominant} for $\g$. Moreover $\lambda $ is said to be {\bf integral} for $\g$ if and only if
$$B(\lambda ,\alpha^{\vee})\in \Z\;\forall \;\alpha\in \mathcal{S}.$$
We will represent a weight $\lambda $ for $\g$ by writing the coefficients $B(\lambda ,\alpha _{i}^{\vee})$ over the $i$-th node that represents the simple root $\alpha _{i}\in \mathcal{S}$.

\begin{theorem}[Theorem of the highest weight]
There is a one-to-one correspondence between finite dimensional irreducible $\g$-modules and dominant integral weights, i.e.~every finite dimensional irreducible representation of $\g$ has
a unique (up to scale) highest weight vector of weight $\lambda $, which is dominant integral for $\g$ and for every dominant integral weight there exists a finite dimensional irreducible representation
of $\g$ which has a highest weight vector of that weight.
\end{theorem}
\begin{proof}
This is the classification theorem of irreducible finite
dimensional representations of semisimple Lie algebras and the
proof may be found in~\cite{h}, p.~113. The difficult point is to show
that the irreducible representation which we obtain for every
$\lambda \in \h^{*}$ by means of Verma modules (which will be
explained in another section) is finite dimensional if $\lambda $
is dominant integral.
\end{proof}

\subsection{Representations of parabolic subalgebras}
\begin{enumerate}
\item
To denote a representation $(\mathbb{V},\rho)$ of a simple Lie algebra $\g$ or a parabolic subalgebra $\p\subset\g$ we write down the coefficient $B(\lambda,\alpha_{j}^{\vee})$ over the $j$-th node in the Dynkin diagram for $\g$, with $\lambda$ being the highest weight of the dual representation $(\mathbb{V}^{*},\rho^{*})$. The details for this construction and the reason for this slightly odd notation is explained in~\cite{be}.
\item 
It is easy to show that the nilpotent part $\p_{+}$ of $\p$ has to act trivially on any irreducible $\p$-module. $\g_{0}\cong\p/\p_{+}$ is reductive and hence the sum of an abelian algebra $\mathfrak{z}(\g_{0})$, which has to act via a character on any irreducible representation, and a semisimple part $\g_{0}^{S}$. 
Since $\g_{0}^{S}$ is semisimple we can apply the above theorem to $\g_{0}^{S}$ and get: the finite dimensional irreducible representations of $\p$ are in one-to-one
correspondence with $\lambda \in \h^{*}$ which are dominant and integral for $\g_{0}^{S}$. In our Dynkin diagram notation that corresponds to having non-negative integers over the uncrossed nodes. More precisely, a finite dimensional $\g_{0}$-module is completely reducible if and only if $\mathfrak{z}(\g_{0})$ acts diagonalizably. Unless stated otherwise, we will always implicitly assume this. 
%\item
%In order for an irreducible representation of $\p$ to lift to one for $P$, the dominant integral weight $\lambda $ has to be integral for $\g$ and not just for $\p$.
%Therefore the coefficient over every node has to be an integer and the coefficients over uncrossed nodes have to be non-negative to yield a representation for $P$.
\end{enumerate}

\subsubsection{Example}
The $\g_{0}$-module $\g_{1}$ decomposes into irreducible $\g_{0}$-modules
$$\g_{1}=\bigoplus_{i\in I}\mathbb{E}_{-\alpha_{i}},$$
where $\mathbb{E}_{\alpha}$ has lowest weight $-\alpha\in\h^{*}$ (see~\cite{be}, p.~129). We will abbreviate $\mathbb{E}_{-\alpha_{i}}$ by $\g_{1}^{i}$ in the future. Analogously
$\g_{-1}$ decomposes into irreducible factors $\g_{-1}^{i}$, for $i\in I$, where the highest weight of $\g_{-1}^{i}$ is $-\alpha_{i}$.
\par
It is true in general that the dual of an irreducible finite dimensional module of lowest weight  $\lambda$ has highest weight $-\lambda$.

\subsubsection{Example}\label{sltwo}
It can be easily shown (\cite{h}, p.~37) that for every $\alpha \in \Delta (\g,\h)$ and $X_{\alpha }\in \g_{\alpha }\backslash\{0\}$ there exists $X_{-\alpha }\in \g_{-\alpha }$, such that
$X_{\alpha },X_{-\alpha }$ and $H_{\alpha }= [X_{\alpha} ,X_{-\alpha }]$ span a three dimensional simple subalgebra of $\g$ isomorphic to $\mathfrak{sl}_{2}\C$ and that we can
calculate $B(\lambda ,\alpha _{i}^{\vee})=\lambda (H_{\alpha _{i}})$ for every $\alpha _{i}\in \mathcal{S}$ and $\lambda \in \h^{*}$. 
For $A_{l}$, $X_{\alpha _{i}}$ and $X_{-\alpha_{i}}$ correspond to $E_{i,i+1}$ and $E_{i+1,i}$ respectively. We therefore have $H_{\alpha _{i}}=H_{i}-H_{i+1}$, for $i=1,...,l$. So for a weight $\lambda =\sum_{i=1}^{l+1}a_{i}\epsilon_{i}$, we have
$$B(\lambda ,\alpha _{i}^{\vee})=a_{i}-a_{i+1},\; \text{for}\; i=1,...,l.$$

%\subsubsection{Notation}
%Since $\p_{+}$ is nilpotent, it acts trivially on any irreducible representation of $\p$ and so we can denote any irreducible representation of $\p$ by the highest weight of the corresponding representation of $\g_{0}^{S}$ and by specifying how $\mathfrak{z}(\g_{0})$ acts. Since $\mathfrak{z}(\g_{0})$ is abelian it will always act
%via an element in $\mathfrak{z}(\g_{0})^{*}$. This means, that every finite dimensional irreducible representation $\mathbb{E}$ of $\p$ can be written as
%$$\mathbb{E}=\mathbb{E}^{0}\otimes\C_{\lambda'},$$
%where $\mathbb{E}^{0}$ is a finite dimensional irreducible representation of $\g_{0}^{S}$ and $\C_{\mu}$ is the one dimensional representation of $\mathfrak{z}(\g_{0})$ with action given by $X.v=\lambda'(X)v$ for all $X\in \mathfrak{z}(\g_{0})$. Obviously $\lambda'\in \mathfrak{z}(\g_{0})$. 
%The highest weight $\lambda$ of $\mathbb{E}$ can be written as $\lambda=\lambda_{0}+\lambda'$, where $\lambda_{0}\in(\h^{S})^{*}$, the dual of the Cartan subalgebra of $\g_{0}^{S}$.
%\par
%It will be convenient to have the explicit description
%$$\mathfrak{z}(\g_{0})=\{H\in\h\;:\;\alpha_{j}(H)=0\;\forall\;j\in J\}\Rightarrow\mathfrak{z}(\g_{0})^{*}\cong\mathrm{span}\{\omega_{i}\;:\;i\in I\}.$$

\subsection{Composition series and induced modules}
\begin{definition}
{\rm
We will write {\bf composition series} with the help of $+$ signs as explained in~\cite{beg}, p.~1193 and~\cite{es}, p.~11, so a short exact sequence 
$$0\rightarrow \mathbb{A}_{1}\rightarrow \mathbb{A}\rightarrow \mathbb{A}_{0}\rightarrow 0$$
of modules is equivalent to writing a composition series
$$\mathbb{A}=\mathbb{A}_{0}+\mathbb{A}_{1}.$$
If we have two composition series
$$\mathbb{A}=\mathbb{A}_{0}+\mathbb{A}_{1}\quad\text{and}\quad \mathbb{B}=\mathbb{B}_{0}+\mathbb{B}_{1},$$
then we can tensor them together to obtain
$$\mathbb{A}\otimes \mathbb{B}=\mathbb{A}_{0}\otimes \mathbb{B}_{0}+\begin{array}{c}\mathbb{A}_{0}\otimes\mathbb{B}_{1}\\
\oplus\\
\mathbb{A}_{1}\otimes\mathbb{B}_{0}
\end{array}
+\mathbb{A}_{1}\otimes \mathbb{B}_{1}.$$
This can be easily extended to composition series with more composition factors.
\par
In general, a composition series
$$\mathbb{B}=\mathbb{A}_{0}+\mathbb{A}_{1}+...+\mathbb{A}_{N},$$
denotes a filtration
$$\mathbb{A}_{N}=\mathbb{A}^{N}\subseteq \mathbb{A}^{N-1}\subseteq...\subseteq \mathbb{A}^{0}=\mathbb{B},$$
with $\mathbb{A}_{i}=\mathbb{A}^{i}/\mathbb{A}^{i+1}$. We will use the same notation for the composition series of vector bundles.}
\end{definition}

\subsubsection{Remark}
\begin{enumerate}
\item
It can be noted that every composition series $\mathbb{B}=\mathbb{A}_{0}+\mathbb{A}_{1}+...+\mathbb{A}_{N}$ has a projection $\mathbb{B}\rightarrow \mathbb{A}_{0}$ and  injections $\mathbb{A}_{j}+..+\mathbb{A}_{N}\rightarrow \mathbb{B}$ for $j=0,...,N$.
\item
The dual of a composition series is obtained by writing down the dual of the composition factors in opposite order, i.e.
$$\mathbb{B}^{*}=\mathbb{A}_{N}^{*}+\mathbb{A}_{N-1}^{*}+...+\mathbb{A}_{0}^{*}.$$
Composition series will occur in our considerations as so-called {\bf branching rules} that describe how a finite dimensional irreducible representation of $\g$ composes as a representation of $\p$. In particular, the grading can be looked at as a composition series of the adjoint representation when restricted to the parabolic $\p$.
\end{enumerate}

\subsubsection{Example}
The $|k_{0}|$-grading of the Lie algebra $\g$ corresponding to the parabolic subalgebra $\p$ can be understood if we look at $\g$ as a $\p$-module. There exits a $\p$-module filtration
$$\g=\g^{-k_{0}}\supset\g^{-k_{0}+1}\supset\cdots\supset\g^{k_{0}}=\g_{k_{0}},$$
with $\g^{i}=\g_{i}\oplus\g_{i+1}\oplus\cdots\oplus\g_{k_{0}}$. The corresponding composition series is exactly the grading
$$\g=\g_{-k_{0}}+\dots+\g_{k_{0}}.$$

\begin{definition}
{\rm
For any Lie algebra $\mathfrak{a}$, let $\mathfrak{U}(\mathfrak{a})$ denote the {\bf universal enveloping algebra}, which is defined in the following manner:
we denote the tensor algebra of $\mathfrak{a}$ by $T(\mathfrak{a})=\bigoplus_{i=0}^{\infty}\otimes^{i}\mathfrak{a}$, then the universal enveloping algebra is defined to be 
$$\mathfrak{U}(\mathfrak{a})=T(\mathfrak{a})/I,$$
where $I$ is the two-sided ideal spanned by $X\otimes Y- Y\otimes X-[X,Y]$ for all $X,Y\in\mathfrak{a}$.
Every representation of $\mathfrak{a}$ can in a unique manner be extended to an action of $\mathfrak{U}(\mathfrak{a})$.
\par
Let $\mathbb{V}$ be a finite dimensional $\p$-module with dual $\mathbb{V}^{*}$, then
$$M_{\p}(\mathbb{V})=\mathfrak{U}(\g)\otimes_{\mathfrak{U}(\p)}\mathbb{V}^{*}$$
is called an {\bf induced module}. Here $\mathfrak{U}(\g)$ is a right $\mathfrak{U}(\p)$-module, $\mathbb{V}^{*}$ is a left $\mathfrak{U}(\p)$-module and $\otimes_{\mathfrak{U}(\p)}$ is the usual tensor
product for modules over an algebra. More precisely 
$$M_{\p}(\mathbb{V})=(\mathfrak{U}(\g)\otimes_{\C}\mathbb{V}^{*})/J,$$
where $J$ is the subspace generated by all $uv\otimes w-v\otimes v.w$ for $u\in\mathfrak{U}(\g)$, $v\in\mathfrak{U}(\p)$ and $w\in\mathbb{V}^{*}$ and where the dot denotes the given action of $\mathfrak{U}(\p)$ on $\mathbb{V}^{*}$. $\mathfrak{U}(\g)$ acts on $M_{\p}(\mathbb{V})$ by multiplication on the left factor.
\par
If $\mathbb{V}$ is irreducible, then $M_{\p}(\mathbb{V})$ is called {\bf generalized Verma module}. More information about generalized Verma modules can be found in~\cite{l}, p.~500 and its correlation to homogeneous vector bundles on $G/P$ is explained in~\cite{be}, p.~164.}
\end{definition}

\subsubsection{Remark}
Homomorphisms between generalized Verma modules play a decisive role in the classification of invariant differential operators on homogeneous spaces, see~\cite{er}. 
Bernstein-Gelfand-Gelfand resolutions (\cite{bgg}) and generalizations thereof (\cite{l}) have led to the so-called {\bf BGG} machinery that produces resolutions of all finite dimensional irreducible representations of $G$ by invariant differential operators between sections of homogeneous vector bundles.

\subsection{The real case}
So far we have stated everything in the complex category. However, the theory of parabolic geometries applies to the complex and the smooth real category. More specifically, we want our theory to apply to real manifolds with a given structure, e.g.~real conformal manifolds. The theory is set up in a way that applies to the real and complex category at the same time. Although we will, for example, use the usual notation $\mathcal{O}$ for holomorphic objects, we could have also written everything with the symbol $\mathcal{C}^{\infty}$ denoting smooth objects. In fact, this is the usual way the theory is presented. For every real algebra $\g_{\R}$, we will always consider its complexification $\g$ that can be described as in the first two sections. In particular, any representation of a real Lie algebra $\g_{\R}$ on a complex
vector space $\mathbb{V}$ extends to a representation of the complexification $\g$ and can be dealt with as described above. In fact, there is a bijective correspondence between complex representations of real Lie algebras and complex representations of their complexifications making the complex representation theories of $\g_{\R}$ and $\g$ completely equivalent (see~\cite{cs1}). This allows us to deal with (almost all) the representation theory in a complex setting.
\par
If we start with a real representation $\mathbb{V}_{\R}$, we can always form the complexification $\mathbb{V}$ and consider the extension of the representation to $\g$.
There are, however, cases when $\mathbb{V}_{\R}$ is irreducible and $\mathbb{V}$ is not. A very good example of this phenomena will be encountered in Chapter~\ref{tractorchapter}, when we consider the canonical subbundle of the tangent bundle in CR geometry. Apart from this example, we will encounter real representations only in the form of (projective, conformal,..) weights that arise via one dimensional representations on $\R$. These representations give rise to line bundles that we tensor other vector bundles with. This procedure does not affect the representation theory.
\par
The complexification of a $|k_{0}|$-graded real Lie algebra is $|k_{0}|$-graded again and the complexification of the non-negative part of this grading is a parabolic subalgebra as described in~\ref{grading}.
When referring to specific real manifolds, especially in Chapter~\ref{tractorchapter}, we will state everything in the smooth category.

\section{Parabolic geometries}

\begin{definition}
{\rm
We denote the adjoint representation of $G$ on $\g$ by $\mathrm{Ad}$. Then we can define several subgroups of $G$:
\begin{enumerate}
\item
$P=\{g\in G\;:\;\mathrm{Ad}(g)(\g_{i})\subset\g^{i}\;\forall\;i=-k_{0},...,k_{0}\}$ and
\item
$G_{0}=\{g\in G\;:\;\mathrm{Ad}(g)(\g_{i})\subset\g_{i}\;\forall\;i=-k_{0},...,k_{0}\}$.
\end{enumerate}}
\end{definition}
It can be shown (\cite{cs}, Proposition 2.9) that the Lie algebras of $P$ and $G_{0}$ are $\p$ and $\g_{0}$ respectively.

\subsubsection{Remark}
Instead of choosing $G$ to be the simply connected and connected Lie group with Lie algebra $\g$, one can alternatively choose $G=\mathrm{Aut}(\g)$ or $G=\mathrm{Aut}_{0}(\g)=\mathrm{Int}(\g)$ corresponding to a slightly different geometric structure (orientation, spin, etc.), but our analysis of pairings is local, so none of these choices affect it.

\subsubsection{Remark}
In order for an finite dimensional irreducible (complex) representation of $\p$ to lift to one for $P$, the weight $\lambda $ has to be integral for $\g$ and not just for $\p$, see~\cite{be}, Remark 3.1.6 and~\cite{cs1}, 3.2.10. 
Therefore, in our Dynkin diagram notation, the coefficient over every node has to be an integer and the coefficients over uncrossed nodes have to be non-negative to yield a representation for $P$. In the specific examples given in Chapter~\ref{tractorchapter}, we will define real representations that will allow us to have arbitrary real projective or conformal  weights. In other (real) examples, there are different restrictions on the numbers over the crossed through nodes. We will not discuss this in detail, but give the precise statements in every case considered in Chapter~\ref{tractorchapter}.

\subsection{Cartan connection}\label{Cartanconnection}
A {\bf Parabolic geometry} $(\mathcal{M},\mathcal{G},\g,\omega)$ of  type $(G,P)$ consists of
\begin{description}
\item{(a)}
a manifold $\mathcal{M}$,
\item{(b)}
a principal right $P$ bundle $\mathcal{G}$ over $\mathcal{M}$ and
\item{(c)}
a $\mathfrak{g}$-valued $1$-form $\omega$ on $\mathcal{G}$ (the {\bf Cartan connection}) satisfying the following conditions:
\begin{description}
\item{(i)}
the map $\omega _{g}:T_{g}\mathcal{G}\longrightarrow \mathfrak{g}$ is a linear isomorphism for every $g\in \mathcal{G}$,
\item{(ii)}
$r_{p}^{*}\omega =\mathrm{Ad}(p^{-1})\circ\omega $ for all $p\in P$, where $r_{p}$ denotes the natural right action of an element $p\in P$ in the structure group,  and
\item{(iii)}
$\omega (\zeta_{X})=X$ for all $X\in \mathfrak{p}$, where $\zeta_{X}$ is the (vertical) {\bf fundamental} vector field on $\mathcal{G}$ associated to $X\in\p$.
\end{description}
\end{description}

\subsubsection{Remark}
The dimension $n=\mathrm{dim}\mathcal{M}$ of a parabolic geometry of type $(G,P)$ is easily computed to be
$$n=\mathrm{dim}\;\g_{-}=\frac{\mathrm{dim}\;\g-\mathrm{dim}\;\g_{0}^{S}-l_{0}}{2}.$$
Formulae for the dimension of the semisimple Lie algebras $\g$ and $\g_{0}^{S}$ are readily available (see~\cite{h}).

\begin{definition}
{\rm We will call the case where the parabolic subalgebra $\p$ induces a $|1|$-grading, i.e.~where $k_{0}=1$, the {\bf AHS} case and parabolic geometries for $|1|$-gradings are called almost hermitian symmetric spaces. It immediately follows (in the complex setting) that a parabolic subalgebra $\p$ which induces a $|1|$-grading must arise from a subset $\mathcal{S}_{\p}\subset\mathcal{S}$, so that $\mathcal{S}\backslash \mathcal{S}_{\p}=\{\alpha_{i_{0}}\}$. In the Dynkin diagram notation this corresponds to having one cross through the node associated to $\alpha_{i_{0}}$. The reverse statement is obviously not correct.}
\end{definition}

\begin{definition}
{\rm 
The {\bf curvature} of a parabolic geometry is defined to be the curvature function
\begin{eqnarray*}
\kappa:\mathcal{G}&\rightarrow &\Lambda^{2}\g_{-}^{*}\otimes\g\\
g&\mapsto &\kappa_{g}\;\text{with}\;\kappa_{g}(X,Y)=d\omega_{g}(\omega_{g}^{-1}(X),\omega_{g}^{-1}(Y))+[X,Y],\;X,Y\in\g_{-}.
\end{eqnarray*}
One can decompose the curvature in terms of homogeneity degree, i.e.
$$\kappa^{i}(g):\g_{r}\otimes \g_{s}\rightarrow \g_{r+s+i},\; i=-k_{0}+2,...,3k_{0}$$
or split $\kappa$ in terms of the grading, i.e.~$\kappa_{i}(g):\Lambda^{2}\g_{-}\rightarrow\g_{i}$, $i=-k_{0},...,k_{0}$. Moreover one can compose $\kappa$ with a map $\partial^{*}:\Lambda^{2}\g_{-}^{*}\otimes\g\rightarrow\g_{-}^{*}\otimes\g$, which is the adjoint to the Lie algebra differential $\partial$ to be defined in Section~\ref{cohomology}. Following~\cite{cs2}, we call a parabolic geometry 
\begin{enumerate}
\item
{\bf normal}, if $\partial^{*}\circ \kappa=0$,
\item
{\bf regular}, if it is normal and $\kappa^{i}=0$ for all $i\leq 0$,
\item
{\bf torsion-free}, if $\kappa_{i}=0$ for all $i<0$ and
\item
{\bf flat}, if $\kappa=0$.
\end{enumerate} 

}\end{definition}

\subsubsection{Remark}\label{cartanconnection}
The curvature is horizontal, i.e.~$\kappa_{g}(X,Y)=0$ for all $X\in\p$,  $Y\in\g$ and $g\in\mathcal{G}$.
\begin{proof}
By definition,
$$\kappa(X,Y)=-\omega([\omega^{-1}(X),\omega^{-1}(Y)])+[X,Y].$$
Now we use the fact that $[\omega^{-1}(X),\omega^{-1}(Y)]=\omega^{-1}([X,Y])$ for $X\in\p$ which follows from~\ref{Cartanconnection} (c), (ii) and (iii).
\end{proof}

\subsubsection{Example}\label{cartanexample}
\begin{description}
\item{(a)}
Projective geometry in $n$ dimensions can be described as a parabolic geometry with $\g_{\R}=\mathfrak{sl}(n+1,\R)$ and a grading whose complexification corresponds to the grading in Example~\ref{gradingexample} (a).
% 1.1.4.1. 
The notation via Dynkin diagrams is
%$$\begin{picture}(5,2)
%\put(0,1){\makebox(0,0){$\cross$}}
%\put(1,1){\makebox(0,0){$\bullet$}}
%\put(2.5,1){\makebox(0,0){$\bullet$}}
%\put(1.5,1){...}
%\put(3.5,1){\makebox(0,0){$\bullet$}}
%\put(0,1){\line(1,0){1}}
%\put(2.5,1){\line(1,0){1}}
%\end{picture}.$$
$$\begin{picture}(5,2)
\put(0,1){\makebox(0,0){$\cross$}}
\put(1,1){\makebox(0,0){$\bullet$}}
\put(2,1){\makebox(0,0){$\bullet$}}
\put(1,1){\line(1,0){1}}
\put(3.5,1){\makebox(0,0){$\bullet$}}
\put(2.5,1){...}
\put(4.5,1){\makebox(0,0){$\bullet$}}
\put(0,1){\line(1,0){1}}
\put(3.5,1){\line(1,0){1}}
\put(5,1){.}
\end{picture}$$
The first construction of a canonical Cartan connection on projective manifolds is due to Cartan, see~\cite{ca2}. For a more modern treatment the reader is advised to refer to~\cite{css4,kob,sh}.
\item{(b)}
A conformal manifold $(\mathcal{M},[g])$ of signature $(p,q)$ is equivalent to a normal parabolic geometry with $\g_{\R}=\mathfrak{so}(p+1,q+1)$ and $|1|$-grading as in~\cite{css}, the complexification of which is the $|1|$-grading given in Example~\ref{gradingexample} (b). The Dynkin diagram notation is
$$
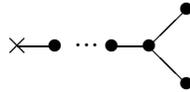
\begin{picture}(5,4)
\put(0,1){\makebox(0,0){$\cross$}}
\put(1,1){\makebox(0,0){$\bullet$}}
\put(1.5,1){...}
\put(2.5,1){\makebox(0,0){$\bullet$}}
\put(3.5,1){\makebox(0,0){$\bullet$}}
\put(4.5,2){\makebox(0,0){$\bullet$}}
\put(4.5,0){\makebox(0,0){$\bullet$}}
\put(0,1){\line(1,0){1}}
\put(2.5,1){\line(1,0){1}}
\put(3.5,1){\line(1,1){1}}
\put(3.5,1){\line(1,-1){1}}
\end{picture}$$
for $p+q$ even and
$$\xoo{\;\;}{\;\;}{\;\;}\;...\;\btwo{\;\;}{\;\;}$$
for $p+q$ odd.
It turns out that this parabolic geometry is automatically regular and torsion-free (see~\cite{cs}). The first construction of a canonical Cartan connection on manifolds with a conformal structure is due to Cartan~\cite{ca1}. See~\cite{css1} and~\cite{css4} for a more modern treatment.
\item{(c)}
A manifold with partially integrable almost CR-structure with non-degenerate Levi-form of signature $(p,q)$ is equivalent to a regular parabolic geometry with $\g_{\R}=\mathfrak{su}(p+1,q+1)$ endowed with a $|2|$-grading as in~\cite{cs}, the complexification of which is the grading given in Example~\ref{gradingexample} (c) with Dynkin diagram
%$$\begin{picture}(5,2)
%\put(0,1){\makebox(0,0){$\cross$}}
%\put(1,1){\makebox(0,0){$\bullet$}}
%\put(2,1){\makebox(0,0){$\bullet$}}
%\put(2,1){\makebox(0,0){$\bullet$}}
%\put(1,1){\line(1,0){1}}
%\put(.5,1){...}
%\put(4.5,1){\makebox(0,0){$\cross$}}
%\put(0,1){\line(1,0){1}}
%\put(3.5,1){\line(1,0){1}}
%
%\end{picture}.$$
$$\begin{picture}(5,2)
\put(0,1){\makebox(0,0){$\cross$}}
\put(1,1){\makebox(0,0){$\bullet$}}
\put(2,1){\makebox(0,0){$\bullet$}}
\put(1,1){\line(1,0){1}}
\put(3.5,1){\makebox(0,0){$\bullet$}}
\put(2.5,1){...}
\put(4.5,1){\makebox(0,0){$\cross$}}
\put(0,1){\line(1,0){1}}
\put(3.5,1){\line(1,0){1}}
\put(5,1){.}
\end{picture}$$
The torsion-free parabolic geometries are exactly the CR-structures (see~\cite{cs}). The original construction of a canonical Cartan connection for CR-manifolds in three dimensions is due to Cartan (\cite{ca}). The general case is due to Tanaka (\cite{t}) and Chern and Moser (\cite{cm}). 
\end{description}

\subsection{The homogeneous model}
The {\bf homogeneous model} for a parabolic geometry of type $(G,P)$ is the principal right $P$-bundle
$$\begin{array}{ccc}
P&\rightarrow&G\\
&&\downarrow\\
&&G/P
\end{array}$$
over the {\bf generalized flag manifold} $G/P$ together with the Maurer-Cartan form
$$\omega_{MC}(X)=(l_{g^{-1}})_{*}X,\;\text{for}\;X\in T_{g}G,$$
where $l_{g}$ denotes left translation and $(l_{g})_{*}$ is the corresponding tangent mapping. The properties of $\omega_{MC}$ are easily computed:
\begin{enumerate}
\item
$(l_{g^{-1}})_{*}:T_{g}G\rightarrow T_{e}G=\g$ is a linear isomorphism with inverse map $(l_{g})_{*}$.
\item
Let $X\in T_{g}G$, then we compute
\begin{eqnarray*}
(r_{p}^{*}\omega_{MC})X&=&\omega_{MC}((r_{p})_{*}X)\\
&=&(l_{(gp)^{-1}})_{*}(r_{p})_{*}X\\
&=&(l_{p^{-1}}\circ l_{g^{-1}}\circ r_{p})_{*}X\\
&=&(\mathrm{Ad}(p^{-1})\circ l_{g^{-1}})_{*}X\\
&=&\mathrm{Ad}(p^{-1})\circ\omega_{MC}(X).
\end{eqnarray*}
\item
We have $(l_{g})_{*}X=(\zeta_{X})_{g}\in T_{g}G$, so obviously $\omega_{MC}(\zeta_{X})=X$ for all $X\in\p$.
\end{enumerate}
The curvature function vanishes identically because of the structural equation 
$$d\omega_{MC}(X,Y)+[\omega_{MC}(X),\omega_{MC}(Y)]=0,$$
see~\cite{sh}, p.~108.
In fact, it is known that the curvature $\kappa$ is a complete obstruction against local flatness, see e.g.~\cite{cd,css2}.  

\subsubsection{Examples}\label{homogeneousexample}
\begin{description}
\item{(a)}
The homogeneous model for {\bf projective geometry}, which we will discuss in some detail in Chapter~\ref{tractorchapter}, is $\mathbb{RP}_{n}=G/P$ with $G=\mathrm{PSL}_{n+1}\R$ and
$$P=\left\{g\in G\;:\;g=\left(\begin{array}{ccccc}
*&*&\cdots&*\\
0&*&\cdots&*\\
\vdots&\vdots&\ddots&\vdots\\
0&*&\cdots&*
\end{array}\right)\right\}.$$
Note that $P$ is the isotropy subgroup of the point $e=[1:0:\cdots:0]\in\mathbb{RP}_{n}$.
%The corresponding grading of the complexified Lie algebra $\g=\mathfrak{sl}_{n+1}\C$ is given in example ??.

\item{(b)}
The homogeneous model for {\bf conformal geometry}, which we will discuss in some detail in Chapter~\ref{tractorchapter}, is the sphere $\mathbb{S}^{n}=G/P$ with $G=SO_{0}(p+1,q+1)$ being the identity connected component of $SO(p+1,q+1)$ corresponding to a non-degenerate bilinear form $\tilde{g}$ of signature $(p+1,q+1)$ on $\R^{n+2}$, for example given by
$$\tilde{g}(x_{0},x,x_{\infty})=2x_{0}x_{\infty}+\langle x,x\rangle,$$
where $\langle x,x\rangle$ is the standard inner product of signature $(p,q)$ in $\R^{n}$.
$P$ is the isotropy subgroup of the point 
$$e=[1:0:\cdots:0]\in\mathbb{S}^{n}=\{[x_{0}:x:x_{\infty}]\in\mathbb{RP}_{n+1}\;:\;\tilde{g}(x_{0},x,x_{\infty})=0\}.$$
%with a corresponding grading of the  complexified Lie algebra $\g=\mathfrak{so}(2n)$ as given in example ?? for $m=2n-2$ even.

\item{(c)}
The homogeneous model for {\bf CR geometry}, which we will discuss in some detail in Chapter~\ref{tractorchapter}, is a real hyperquadric $\mathcal{H}=G/P$ in $\mathbb{CP}_{n+1}$, for example given by
$$\tilde{h}(z_{0},z,z_{\infty})=z_{0}\bar{z}_{\infty}+\bar{z}_{0}z_{\infty}+\langle z,z\rangle,$$
where $\langle z,z\rangle$ is the standard hermitian inner product of signature $(p,q)$ in $\C^{n}$. So $G=PSU(p+1,q+1)$ and $P$ is the isotropy subgroup of the point 
$$e=[1:0:\cdots:0]\in\mathcal{H}=\{[z_{0}:z:z_{\infty}]\in\mathbb{CP}_{n+1}\;:\;\tilde{h}(z_{0},z,z_{\infty})=0\}.$$
%The corresponding grading of the complexified Lie algebra $\g=A_{n+1}$ is given in example ??.
\end{description}

\subsection{Associated vector bundles}
For every representation $\rho:P\rightarrow \mathrm{End}(\mathbb{V})$ we can define a corresponding {\bf associated vector bundle}
$$V=(V,\rho)=\mathcal{G}\times_{P}\mathbb{V}=(\mathcal{G}\times\mathbb{V})/\sim,$$
where
$$(gp,v)\sim(g,\rho(p)v)\;\forall\;g\in \mathcal{G},p\in P\;\text{and}\;v\in\mathbb{V}.$$
Sections of this bundle can be identified with maps
$$s:\mathcal{G}\rightarrow\mathbb{V}\,\text{s.t.}\;s(gp)=\rho(p^{-1})s(g)$$
for all $g\in\mathcal{G}$ and $p\in P$. Equivalently, we can differentiate this requirement to obtain
$$(\zeta_{X}s)(g)=-\rho(X)s(g)\;\forall\;X\in\p,g\in\mathcal{G}.$$
We will write $\Gamma(V)=\mathcal{O}(\mathcal{G},\mathbb{V})^{P}$ for the space of sections of $V$. The tangent bundle $T\mathcal{M}$ and cotangent bundle $T^{*}\mathcal{M}$,
for example, arise via the adjoint representation of $P$ on $\g/\p\cong\g_{-}=\g_{-k}\oplus...\oplus\g_{-1}$ and its dual $(\g/\p)^{*}\cong\p_{+}$.
We will denote the bundles and sections of these bundles by the Dynkin diagram notation for the representation that induces them.

\subsubsection{Example}
\begin{enumerate}
\item
On a projective manifold $\mathcal{M}$ we write
$$\mathrm{Vect}(\mathcal{M})\;=\;\xoo{1}{0}{0}\;...\;\oo{0}{1}=\;\text{and}\;\Omega^{1}(\mathcal{M})\;=\;\xoo{-2}{1}{0}\;...\;\oo{0}{0}.$$
\item
On a conformal manifold $\mathcal{M}$ we write
%$$\xoo{0}{1}{0}\;...\;\Dfive{0}{0}{0}{0}=\mathrm{Vec}(\mathcal{M})\;\text{and}\;\xoo{-2}{1}{0}\;...\;\Dfive{0}{0}{0}{0}=\Omega^{1}(\mathcal{M}).$$
$$\mathrm{Vect}(\mathcal{M})=\hspace{1.5cm}\Dd{0}{1}{0}{0}{0}{0}{0}\;\text{or}\;\mathrm{Vect}(\mathcal{M})\;=\;\xoo{0}{1}{0}\;...\;\btwo{0}{0}$$
and
$$\Omega^{1}(\mathcal{M})=\hspace{1.5cm}\Dd{-2}{1}{0}{0}{0}{0}{0}\;\text{or}\;\Omega^{1}(\mathcal{M})\;=\;\xoo{-2}{1}{0}\;...\;\btwo{0}{0}.$$

\item
For CR geometry, the tangent and cotangent bundle are no longer irreducible. Instead they (or more precisely their complexification) have a filtration given by
$$\mathrm{Vect}(\mathcal{M})=\;\xoo{1}{0}{0}\;...\;\oox{0}{0}{1}+
\begin{array}{c}
\xoo{1}{0}{0}\;...\;\oox{0}{1}{-1}\\
\oplus\\
\xoo{-1}{1}{0}\;...\;\oox{0}{0}{1}
\end{array}
$$
and
$$\Omega^{1}(\mathcal{M})=\begin{array}{c}
\xoo{-2}{1}{0}\;...\;\oox{0}{0}{0}\\
\oplus\\
\xoo{0}{0}{0}\;...\;\oox{0}{1}{-2}
\end{array}+\xoo{-1}{0}{0}\;...\;\oox{0}{0}{-1}.$$

\end{enumerate}

%\subsubsection{Remark}
%For real almost hermitian symmetric spaces it is possible (see~\cite{cg}, 4.15) to define line bundles $\mathcal{E}[w]$ for all $w\in \R$, which correspond to having an arbitrary real number over the crossed through node in our Dynkin diagram notation, see remark 1.3. Examples of these line bundles are described for projective and conformal manifolds in the last chapter.

\subsection{Invariant differential}\label{invariantdifferential}
The Cartan connection does not yield a connection on these associated vector bundles, but we can still define the {\bf invariant differential}
$$\nabla^{\omega}:\mathcal{O}(\mathcal{G},\mathbb{V})\rightarrow \mathcal{O}(\mathcal{G},\g_{-}^{*}\otimes\mathbb{V})$$
with
$$\nabla^{\omega}s(g)(X)=\nabla^{\omega}_{X}s(g)=[\omega^{-1}(X)s](g),\;\forall\;X\in \g_{-},g\in\mathcal{G}\;\text{and}\;s\in\mathcal{O}(\mathcal{G},\mathbb{V}).$$
It has to be noted that it does not take $P$-equivariant sections to $P$-equivariant sections. In order to ensure that, we can define the  {\bf fundamental derivative},
\begin{eqnarray*}
D:\mathcal{O}(\mathcal{G},\mathbb{V})^{P}&\rightarrow &\mathcal{O}(\mathcal{G},\g^{*}\otimes \mathbb{V})^{P}\\
s&\mapsto&X\mapsto \nabla^{\omega}_{X}s=\omega^{-1}(X)s.
\end{eqnarray*}
The distinction between the fundamental derivative and the invariant differential is a bit artificial and solely done for notational convenience. 
The bundle $\mathcal{G}\times_{P}\g=\mathcal{A}$ is called the 
{\bf adjoint tractor bundle}. The invariant derivative and the fundamental derivative will be of vital importance in later chapters. Important properties of the fundamental derivative can be found in~\cite{cg}, Proposition 3.1. We will just prove that $D$ takes $P$-equivariant functions to $P$-equivariant functions.
\begin{proof}
Let $f\in\mathcal{O}(\mathcal{G},\mathbb{V})^{P}$ (where the bundle $V$ is an associated bundle for a representation $\rho:P\rightarrow\mathrm{Aut}(\mathbb{V})$), $s\in\mathcal{O}(\mathcal{G},\g)^{P}$, $g\in\mathcal{G}$ and $p\in P$, then we compute
\begin{eqnarray*}
D_{s}f(gp)&=&\omega^{-1}_{gp}(s(gp))f\;\text{by definition}\\
&=&\omega^{-1}_{gp}(\mathrm{Ad}(p^{-1})s(g))f\;\text{by the equivariance of}\; s\\
&=&(r_{p})_{*}\omega^{-1}_{g}(s(g))f\;\text{by the property (c) of the Cartan connection}\\
&=&d(f\circ r_{p})\omega^{-1}_{g}(s(g))\\
&=&\rho(p^{-1})df\omega^{-1}_{g}(s(g))\;\text{by the equivariance of}\; f\\
&=&\rho(p^{-1})D_{s}f(g),
\end{eqnarray*}
so $D_{s}f\in\mathcal{O}(\mathcal{G},\mathbb{V})^{P}$.
\end{proof}

\subsubsection{Remark}
We have
$$\nabla^{\omega}_{X}\nabla^{\omega}_{Y}-\nabla^{\omega}_{Y}\nabla^{\omega}_{X}=\nabla^{\omega}_{[X,Y]}-\nabla^{\omega}_{\kappa(X,Y)}$$
for all $X,Y\in\g$.
\begin{proof}
By definition
\begin{eqnarray*}
\kappa(X,Y)&=&d\omega(\omega^{-1}(X),\omega^{-1}(Y))+[X,Y]\\
&=&-\omega([\omega^{-1}(X),\omega^{-1}(Y)])+[X,Y]
\end{eqnarray*}
and therefore
$$\nabla^{\omega}_{[X,Y]}-\nabla^{\omega}_{\kappa(X,Y)}=[\omega^{-1}(X),\omega^{-1}(Y)]=\nabla^{\omega}_{X}\nabla^{\omega}_{Y}-\nabla^{\omega}_{Y}\nabla^{\omega}_{X}.$$
\end{proof}

\subsubsection{Example}\label{homogeneousdiff}
In the flat homogeneous model spaces $G/P$ we have for all $X\in\g$, $f\in\mathcal{O}(G,\mathbb{V})^{P}$ and $g\in G$:
\begin{eqnarray*}
(Df(g))(X)&=&\omega_{MC}^{-1}(X)_{g}f\\
&=&((l_{g})_{*}X)f\\
&=&\frac{d}{dt}f(g\exp(tX))|_{t=0}\\
&=&\frac{d}{dt}f(\exp(t\mathrm{Ad}(g)X)g)|_{t=0}\\
&=&((r_{g})_{*}(\mathrm{Ad}(g)X))f.
\end{eqnarray*}
Furthermore, the map $G\rightarrow \g$, $g\mapsto \mathrm{Ad}(g^{-1})X$, can be looked at as a section $s_{X}\in\Gamma(\mathcal{A})$, so we can write
$$D_{s_{X}}f=-R_{X}f,$$
where $(R_{X})_{g}=-(r_{g})_{*}X$ is the (right invariant) vector field on $G$ that is derived from the action of $G$ on $\mathcal{O}(G,\mathbb{V})^{P}$ via $(g.s)(h)=s(g^{-1}h)$ for all $g,h\in G$ and $s\in\mathcal{O}(G,\mathbb{V})^{P}$. $R_{X}$ takes $P$ equivariant functions to $P$ equivariant functions.

\begin{definition}
{\rm The filtration 
$$\g=\g^{-k_{0}}\supset \g^{-k_{0}+1}\supset\cdots\supset\g^{k_{0}}\supset\{0\}$$
induces a filtration
$$\mathcal{A}=\mathcal{A}^{-k_{0}}\supset\mathcal{A}^{-k_{0}+1}\supset\cdots\supset \mathcal{A}^{k_{0}}$$
on the adjoint tractor bundle $\mathcal{A}=\mathcal{G}\times_{P}\g$. For $s\in\mathcal{A}^{0}=\mathcal{G}\times_{P}\p$, we have
$$(D_{s}f)(g)=\omega^{-1}_{g}(s(g))f=(\zeta_{s(g)}f)(g)=-\lambda(s(g))f(g)$$
for all $f\in\mathcal{O}(\mathcal{G},\mathbb{V})^{P}$, where $\lambda:\p\rightarrow\mathfrak{gl}(\mathbb{V})$ is the derived action of $\p$. Following~\cite{cg}, we will denote this action by $D_{s}f=-s\bullet f$.
}\end{definition}

\subsubsection{Remark}\label{tractorconnection}
If $\mathbb{V}$ is a representation of $\g$, then we can define $\nabla_{s}^{V}f=D_{s}f+s\bullet f$ for $s\in\Gamma(\mathcal{A})$ and $f\in\mathcal{O}(\mathcal{G},\mathbb{V})^{P}$. Since $\nabla^{V}_{s}f=0$ for $s\in\Gamma(\mathcal{A}^{0})$, this descends to a mapping
$$\nabla^{V}:\mathcal{O}(\mathcal{G},\mathbb{V})^{P}\rightarrow \mathcal{O}(\mathcal{G},\g_{-}^{*}\otimes\mathbb{V})^{P},$$
which is a linear connection on $V$, see~\cite{cg}, Proposition 3.2. This connection is called {\bf tractor connection} and we will make use of it in Chapter~\ref{tractorchapter}.

\chapter{Pairings}
In this second chapter we will define the notion of an invariant bilinear differential pairing of a certain weighted order. This will firstly be done analytically with the help of (weighted) jets. Then we give an algebraic description of pairings on homogeneous spaces and derive the first important theorem.

\section{Analytic description}
This section is a generalization of the description of invariant differential pairings in~\cite{k} to the case of arbitrary parabolic geometries. In particular, a $|k_{0}|$-grading with $k_{0}>1$ gives rise to the notion of {\bf weighted differential operators} as detailed in~\cite{m}.

\subsection{Filtered manifolds}

\subsubsection{Weighted modules}\label{filtered}
Let us examine the universal enveloping algebras of $\p_{+}$ and $\g_{-}$ more closely. We will discuss only the case for $\p_{+}$ since $\g_{-}$ can be dealt with by duality. We denote the tensor algebra of $\p_{+}$ by $T(\p_{+})=\bigoplus_{i=0}^{\infty}\otimes^{i}\p_{+}$ and the universal enveloping algebra by $\mathfrak{U}(\p_{+})=T(\p_{+})/J$, where $J$ is the two-sided ideal spanned by $X\otimes Y- Y\otimes X-[X,Y]$ for all $X,Y\in\p_{+}$. We can define a {\bf weighted grading} on those algebras via the following construction.
Let $i\in\N$ and 
$$T_{i}(\p_{+})=\{u=\sum_{j}X_{1,j}\otimes...\otimes X_{s_{j},j}\in T(\p_{+})\;:\;X_{l,j}\in\g_{t_{l,j}}\;\text{and}\;\sum_{m=1}^{s_{j}}t_{m,j}=i\;\forall\;j\}.$$
Since $J$ behaves well with respect to the grading, we can define
$$\mathfrak{U}_{i}(\p_{+})=\{u=\sum_{j}X_{1,j}...X_{s_{j},j}\in \mathfrak{U}(\p_{+})\;:\;X_{l,j}\in\g_{t_{l,j}}\;\text{and}\;\sum_{m=1}^{s_{j}}t_{m,j}=i\;\forall\;j\}.$$

\begin{definition}
{\rm
Let $\mathcal{M}$ be a (smooth) complex manifold and 
$$T\mathcal{M}=\mathfrak{f}^{-k_{0}}\supset \mathfrak{f}^{-k_{0}+1}\supset...\supset \mathfrak{f}^{-1}\supset\mathfrak{f}^{0}=0$$
a filtration of the tangent bundle by  subbundles. If $[s^{p},s^{q}]$ is a section of $\mathfrak{f}^{p+q}$ for all sections $s^{p}\in\Gamma(\mathfrak{f}^{p})$, $s^{q}\in\Gamma(\mathfrak{f}^{q})$ and all integers $p,q$, then the filtration is called {\bf tangential filtration} and $\mathcal{M}$ is called {\bf filtered manifold}.
}
\end{definition}

\subsubsection{Example}\label{tangentialfiltration}
The grading of $\g$ induces a filtration of the tangent bundle on the underlying manifold $\mathcal{M}$ of a parabolic geometry $(\mathcal{M},\mathcal{G},\g,\omega)$:
$$T\mathcal{M}=\mathfrak{f}^{-k_{0}}\supset \mathfrak{f}^{-k_{0}+1}\supset...\supset \mathfrak{f}^{-1}\supset\mathfrak{f}^{0}=0$$
with
$$\mathfrak{f}^{-i}=\mathcal{G}\times_{P}(\g_{-i}+...+\g_{-1}).$$
If the parabolic geometry is regular, then this is a tangential filtration (see~\cite{slov}, Proposition 2.2) and the concept of weighted order is well defined.

\subsection{Weighted jet bundles}\label{weightedjetbundles}
Let $\mathcal{M}$ be a filtered manifold with tangential filtration
$$T\mathcal{M}=\mathfrak{f}^{-k_{0}}\supset \mathfrak{f}^{-k_{0}+1}\supset...\supset \mathfrak{f}^{-1}\supset\mathfrak{f}^{0}=0.$$
Following~\cite{m}, we say that a local vector field $X$ on $\mathcal{M}$ has {\bf weighted order} $\leq i$ if $X$ is a section of $\mathfrak{f}^{-i}$. The minimum of such $i$ is called the {\bf weighted order} of $X$ and is denoted by $\mathrm{w-ord} X$.
A differential operator $P$ is said to be of weighted order $\leq\mu$ if $P$ can locally be written as $P=\sum_{j} X^{j}_{1}...X^{j}_{r_{j}}$ for local vector fields $X^{j}_{1},...,X^{j}_{r_{j}}$ and if $\mathrm{max}_{j}\left\{\sum_{i=1}^{r_{j}} \mathrm{w-ord}X^{j}_{i}\right\}=\mu$. 
\par
Let $E$ be a vector bundle (with trivial filtration) over $\mathcal{M}$ and denote by $\underline{E}$ the sheaf of sections of $E$. For every $x\in \mathcal{M}$, let $\mathfrak{f}_{x}^{k}E$ denote the subspace of 
$\underline{E}_{x}$ consisting of those germs of sections $s\in\underline{E}_{x}$, such that
$$ (P\langle \alpha,s\rangle)(x)=0,$$
for all $\alpha\in \Gamma(E^{*})$ and all differential operators $P$ of weighted order $<k$.
The {\bf weighted jet bundle} is then defined to be
$$\mathcal{J}^{k}E=\bigcup_{x\in \mathcal{M}}\mathcal{J}^{k}_{x}E,\quad \mathcal{J}^{k}_{x}E=\underline{E}_{x}/\mathfrak{f}_{x}^{k+1}E.$$
We can identify $\mathcal{J}^{0}E=E$ and obtain exact sequences
$$0\rightarrow \mathfrak{f}^{k}\mathcal{J}^{k}E\rightarrow\mathcal{J}^{k}E\rightarrow \mathcal{J}^{k-1}E\rightarrow 0,$$
for $k\geq 1$. A linear differential operator $d:\Gamma(E)\rightarrow\Gamma(F)$ of {\bf weighted order $k$} is a map
$$D:\mathcal{J}^{k}E\rightarrow F$$
with {\bf symbol} given by the composition map $\mathfrak{f}^{k}\mathcal{J}^{k}E\rightarrow F$. Moreover, we can identify
\begin{equation}\label{symbol}
\mathfrak{f}^{k}\mathcal{J}^{k}E=(\mathfrak{U}_{-k}(grT\mathcal{M}))^{*}\otimes E,
\end{equation}
where $\mathfrak{U}(grT_{x}\mathcal{M})$ is the universal enveloping algebra of the graded algebra 
$$grT_{x}\mathcal{M}=\bigoplus_{p=-k_{0}}^{-1}\mathfrak{f}_{x}^{p}/\mathfrak{f}_{x}^{p+1}$$ 
and $\mathfrak{U}_{-k}(grT_{x}\mathcal{M})$ denotes the subset of all homogeneous elements of degree $-k$ (see~\cite{m}, p.~237) as defined in~\ref{filtered}. We will give an inductive definition of this space in~\ref{inductivefiltered}.
\par
Note that there is a canonical projection $J^{1}E\rightarrow\mathcal{J}^{1}E$ from the first order jet bundle $J^{1}E$ (as defined, for example, in~\cite{s}) of $E$ to the weighted first order jet bundle $\mathcal{J}^{1}E$.

\subsubsection{Remark}
We can write down more explicitly
\begin{eqnarray*}
\mathcal{J}^{k}_{x}E=\underline{E}_{x}/\mathfrak{f}^{k+1}_{x}E&\cong&\bigoplus_{l\leq k}\mathrm{Hom}(\mathfrak{U}_{-l}(grT_{x}\mathcal{M}),E_{x})\\
\underline{E}_{x}/\mathfrak{f}^{k+1}_{x}E\ni f&\mapsto&F_{f}\;\text{such that}\;F_{f}(u)=u^{0}f,
\end{eqnarray*}
where ${}^{0}$ is defined by $(Z_{1}...Z_{l})^{0}=(-1)^{l}Z_{l}...Z_{1}$ for all $Z_{i}\in grT_{x}\mathcal{M}$. Under this isomorphism~(\ref{symbol}) becomes apparent.

\subsubsection{Example}
Let $(\mathcal{M},\mathcal{G},\g,\omega)$ be a parabolic geometry of type $(G,P)$ with the tangential filtration as in Example~\ref{tangentialfiltration}. Then we have
$$grT_{x}\mathcal{M}\cong gr\g_{-}=\g_{-k_{0}}\oplus\cdots\oplus\g_{-1}.$$
The first order jet bundle $J^{1}E$ of every associated bundle $E$ is an associated bundle again (see~\cite{sl}, p.~196), which induces the structure of an associated bundle on $\mathcal{J}^{1}E$. The corresponding $\p$-module 
$$\mathcal{J}^{1}\mathbb{E}=\mathbb{E}+\g_{1}\otimes\mathbb{E}$$
has a $\p$-module structure that is induced by the $\p$-module structure of $J^{1}\mathbb{E}$.

\subsection{Weighted bi-jet bundles and pairings}\label{pairings}
For every complex (or smooth) filtered manifold $\mathcal{M}$ and holomorphic vector bundle $V$ over $\mathcal{M}$, we denote by $\mathcal{J}^{k}V$ the weighted jet bundle over $\mathcal{M}$ as described in the last section.
A {\bf bilinear} {\bf differential} {\bf pairing} between sections of the bundle $V$ and sections of the bundle $W$ to sections of a bundle $U$ is a homomorphism
$$d:\mathcal{J}^{k}V\otimes \mathcal{J}^{l}W\rightarrow U.$$
This pairing is of {\bf weighted order $M$} if and only if 
\begin{description}
\item{(a)}
$k=l=M$,
\item{(b)}
there is a subbundle $B$ of $\mathcal{J}^{M}V\otimes \mathcal{J}^{M}W$, so that there is a commutative diagram
$$\begin{array}{ccc}
\mathcal{J}^{M}V\otimes \mathcal{J}^{M}W &\\
\downarrow &\searrow d\\
 (\mathcal{J}^{M}V\otimes \mathcal{J}^{M}W)/B& \stackrel{\phi}{\rightarrow}& U
 \end{array}$$
and
\item{(c)}
the map $\phi$ induces a formula that consist of terms in which derivatives of sections of $V$ are combined with derivatives of sections of $W$ in such a way that the total weighted order is $M$ (i.e.~a term may consist of a differential operator of weighted order $k$ applied a section of $V$ combined with a differential operator of weighted order $(M-k)$ applied to a section of $W$, for~$k=0,...,M$).
\end{description}
In fact, $B$ is characterized by (c) as detailed in the next subsection. We will therefore write $\mathcal{J}^{M}(V,W)=(\mathcal{J}^{M}V\otimes \mathcal{J}^{M}W)/B$ for this canonical choice of $B$.
This is not to be confused with the set of all $M$-jets of $V$ into $W$ as defined in~\cite{kms}, p.~117, Definition~12.2.
\par
If $\mathcal{M}=G/P$ is a homogeneous space, then a pairing is called {\bf invariant} (some authors use the term equivariant) if it commutes with the action of $G$ on local sections of the involved homogeneous vector bundles, which is given by $(g.s)(h)=s(g^{-1}h)$ for all $g,h\in G$ and $s\in\Gamma(F)$ (see also Example~\ref{homogeneousdiff} and the Introduction~\ref{riemannsphere}).
\par
In general, there is no commonly accepted notion of invariance for manifolds with a parabolic structure (see~\cite{sl}, p.~193, Section 2). We will deal with this issue by taking a pragmatic point of view: first of all, every manifold with a parabolic geometry is equipped with a distinguished class of connections (Weyl connections), as detailed in~\cite{css1}, p.~42 and~\cite{css}, p.~54. A pairing is then called {\bf invariant}, if $\phi$ induces a formula that consists of terms involving an arbitrary connection from the distinguished equivalence class, but that as a whole does not depend on its choice. We will discuss this slightly subtle point further in Section~\ref{invariance}.

\subsection{Description of $\mathcal{J}^{M}(V,W)$}\label{bijet}
\begin{definition}
{\rm Let $\mathcal{M}$ be a (smooth) complex filtered manifold and let $V,W$ be holomorphic vector bundles (with trivial filtrations) over $\mathcal{M}$.  For every holomorphic vector bundle $U$ (with trivial filtration) over $\mathcal{M}$ and every integer $k\in\N$, there exists the weighted jet bundle $\mathcal{J}^{k}U$ and for every $0\leq l\leq k$ there is a projection $\pi^{k}_{l}:\mathcal{J}^{k}U\rightarrow \mathcal{J}^{l}U$ (see~\cite{m}, p.~237). Let $\mathfrak{U}_{-k}(grT\mathcal{M})$ denote the bundle over $\mathcal{M}$ whose fibre at $x\in \mathcal{M}$ is $\mathfrak{U}_{-k}(gr T_{x}\mathcal{M})$ as defined in~\ref{weightedjetbundles}.
The projections can be put into an exact sequence
$$0\rightarrow (\mathfrak{U}_{-k}(grT\mathcal{M}))^{*}\otimes U\rightarrow \mathcal{J}^{k}U\rightarrow \mathcal{J}^{k-1}U\rightarrow 0$$
as described in~\ref{weightedjetbundles}. This exact sequence induces a filtration
\begin{eqnarray*}
\mathcal{J}^{k}U&=&\sum_{l=0}^{k}\mathfrak{U}_{l}(grT^{*}\mathcal{M})\otimes U\\
&=&U+\mathfrak{U}_{1}(grT^{*}\mathcal{M})\otimes U+\mathfrak{U}_{2}(grT^{*}\mathcal{M})\otimes U+...+\mathfrak{U}_{k}(grT^{*}\mathcal{M})\otimes U
\end{eqnarray*}
on the jet bundle. The mapping
$$\varphi_{M}=\oplus_{k+l=M}\pi_{k}^{M}\otimes\pi^{M}_{l}:\mathcal{J}^{M}V\otimes \mathcal{J}^{M}W\rightarrow\bigoplus_{k+l=M}\mathcal{J}^{k}V\otimes \mathcal{J}^{l}W$$
defines a canonical subbundle $B=\mathrm{ker}\; \varphi_{M}$ in $\mathcal{J}^{M}V\otimes \mathcal{J}^{M}W$, so that
$$\mathcal{J}^{M}(V,W)=(\mathcal{J}^{M}V\otimes \mathcal{J}^{M}W)/\mathrm{ker}\; \varphi_{M}.$$
We will call the bundles $\mathcal{J}^{M}(V,W)$ {\bf weighted bi-jet bundles}.}
\end{definition}

\subsubsection{Remark}
It is easy to see that the vector bundle $\mathcal{J}^{M}(V,W)$ has a filtration
$$\mathcal{J}^{M}(V,W)=\sum_{k=0}^{M}\bigoplus_{l=0}^{k}\mathfrak{U}_{l}(grT^{*}\mathcal{M}) \otimes V\otimes\mathfrak{U}_{k-l}(grT^{*}\mathcal{M})\otimes W,$$
which is equivalent to a series of exact sequences
$$0\rightarrow \bigoplus_{l=0}^{k}\mathfrak{U}_{l}(grT^{*}\mathcal{M})\otimes V\otimes\mathfrak{U}_{k-l}(grT^{*}\mathcal{M})\otimes W\stackrel{\iota}{\rightarrow} \mathcal{J}^{k}(V,W)\rightarrow \mathcal{J}^{k-1}(V,W)\rightarrow 0,$$
for $0\leq k\leq M$. The exact sequence
$$0\rightarrow \bigoplus_{l=0}^{M}\mathfrak{U}_{l}(grT^{*}\mathcal{M})\otimes V\otimes\mathfrak{U}_{M-l}(grT^{*}\mathcal{M})\otimes W\stackrel{\iota}{\rightarrow} \mathcal{J}^{M}(V,W)\rightarrow \mathcal{J}^{M-1}(V,W)\rightarrow 0$$
gives rise to a {\bf symbol} $\sigma=\phi\circ\iota$ for every homomorphism $\phi:\mathcal{J}^{M}(V,W)\rightarrow E$, i.e.~for every weighted $M$-th order bilinear differential pairing. As for ordinary jet bundles, we can define the {\bf formal (weighted) bi-jet bundle} $\mathcal{J}^{\infty}(V,W)$ as the projective limit 
$$\mathcal{J}^{\infty}(V,W)\rightarrow...\rightarrow\mathcal{J}^{M}(V,W)\rightarrow \mathcal{J}^{M-1}(V,W)\rightarrow...\rightarrow \mathcal{J}^{1}(V,W)\rightarrow V\otimes W\rightarrow 0.$$

\par
In the homogeneous case $\mathcal{J}^{M}(V,W)$ is a homogeneous bundle with a $\p$-module structure on the standard fibre $\mathcal{J}^{M}(\mathbb{V},\mathbb{W})$ that is induced by the $\p$-module structures of $\mathcal{J}^{M}\mathbb{V}$ and $\mathcal{J}^{M}\mathbb{W}$, the standard fibres of $\mathcal{J}^{M}V$ and $\mathcal{J}^{M}W$.

\section{Algebraic description}\label{algebraicdescription}
This section provides an algebraic description for invariant bilinear differential pairings on homogeneous spaces. Similar to the situation for invariant linear differential operators, invariant differential pairings can be described by homomorphism between certain algebraic objects. There is, however, one major difference. An invariant differential operator between irreducible homogeneous vector bundles is the same as a homomorphism between generalized Verma modules and the existence of such a homomorphism restricts the possible candidates quite severely. In fact, both generalized Verma modules have to have the same central character and the theorem of Harish-Chandra implies that their highest weights have to be related by the affine action of the Weyl group. For pairings there is no such restriction. In fact, Theorem~\ref{mainresultone} shows that there are infinitely many pairings between arbitrary irreducible bundles.
\par
The procedure in this section follows the classification of invariant linear differential operators on homogeneous spaces as presented by L.~Barchini and R.~Zierau in the ICE-EM Australian Graduate School in Mathematics in Brisbane, July 2-20, 2007 and the author would like to express his gratitude for the the inspiration that these lectures provided.
\par
For this section define $\mathbb{K}=\R$ or $\mathbb{K}=\C$.

\subsection{Pairings on homogeneous spaces}
Let $\mathbb{V},\mathbb{W}$ and $\mathbb{E}$ be $\mathbb{K}$-vector spaces and denote by $\mathcal{C}^{\infty}(\mathbb{K}^{n},\mathbb{F})$ the space of smooth ($\mathbb{K}=\R$) or holomorphic ($\mathbb{K}=\C$) $\mathbb{F}$-valued functions on $\mathbb{K}^{n}$ for any vector space $\mathbb{F}$. A bilinear map
$$P:\mathcal{C}^{\infty}(\mathbb{K}^{n},\mathbb{V})\times\mathcal{C}^{\infty}(\mathbb{K}^{n},\mathbb{W})\rightarrow\mathcal{C}^{\infty}(\mathbb{K}^{n},\mathbb{E})$$
is called a {\bf bilinear differential pairing} if it is of the form
\begin{equation}\label{canonicalform}
P(\phi,\psi)=\sum_{\alpha,\beta}a_{\alpha,\beta}\left(\frac{\partial^{|\alpha|}}{\partial x^{\alpha}}\phi\otimes\frac{\partial^{|\beta|}}{\partial x^{\beta}}\psi\right),\;\forall\;(\phi,\psi)\in\mathcal{C}^{\infty}(\mathbb{K}^{n},\mathbb{V})\times\mathcal{C}^{\infty}(\mathbb{K}^{n},\mathbb{W}),
\end{equation}
where $a_{\alpha,\beta}\in\mathcal{C}^{\infty}(\mathbb{K}^{n},\mathrm{Hom}(\mathbb{V}\otimes\mathbb{W},\mathbb{E}))$ and $x$ is a coordinate system. The indices $\alpha,\beta$ are multiindices in the usual sense, i.e.~$x^{\alpha}=x_{1}^{\alpha_{1}}...x_{n}^{\alpha_{n}}$ for $\alpha=(\alpha_{1},...,\alpha_{n})\in\N^{n}$. 
%The smallest integer $M$ with $|\alpha|,|\beta|\leq M$ if $a_{\alpha,\beta}\not=0$ is called the {\bf order} of $P$.
\par
If $\mathcal{M}$ is a manifold and $V,W,E$ are vector bundles over $\mathcal{M}$, then a mapping 
$$P:\Gamma(V)\times\Gamma(W)\rightarrow\Gamma(E)$$
is a differential pairing if, in each local trivialization, $P$ is of the form~(\ref{canonicalform}). The space of all such pairings is denoted by $\mathbb{P}(V\times W,E)$. 
\par
Let us now assume that $\mathcal{M}=G/P$ is a homogeneous space and that $V$, $W$ and $E$ are associated bundles that arise from representations $\lambda$, $\nu$ and $\mu$ of $P$ on the vector spaces $\mathbb{V},\mathbb{W}$ and $\mathbb{E}$ respectively. We will denote the derived representations of $\p$ by the same symbols.
A pairing is called {\bf $G$-invariant}, if 
$$P(l_{g}\phi,l_{g}\psi)=l_{g}P(\phi,\psi)$$
for all $g\in G$, $\phi\in\Gamma(V)$ and $\psi\in\Gamma(W)$ and where $l_{g}$ denotes left translation, i.e.
$$(l_{g}\phi)(h)=(\phi\circ l_{g^{-1}})(h)=\phi(g^{-1}h)\;\forall\;g,h\in G.$$
The space of all $G$ invariant bilinear differential pairings is denoted by 
$$\mathbb{P}_{G}(V\times W,E).$$
We will denote right translation by 
$$(r(X)f)(x)=\frac{d}{dt}f(x\exp(tX))|_{t=0}$$
for $f\in\mathcal{C}^{\infty}(G,\mathbb{F})$ and $X\in\g$, where $\mathbb{F}$ denotes an arbitrary vector space and $\g$ is the Lie algebra of $G$.
It is easy to see that $r(X)_{g}=(l_{g})_{*}X\in T_{g}G$ is a left invariant vector field on $G$. We will write the Lie algebra of $G$ as $\g=\g_{-}+\p$, where $\p$ is the Lie algebra of $P$ and $\g_{-}$ is any vector space complement. Let $\{X_{j}\}_{j=1,...,n}$ be a basis of $\g_{-}$, then we can introduce local coordinates on $G/P$ around $gP$, $g\in G$:
\begin{eqnarray*}
\varphi:U &\rightarrow & G/P\\
\varphi(x_{1},...,x_{n})&=&g\exp(x_{1}X_{1}+...+x_{n}X_{n})\;\mathrm{mod}\;P,
\end{eqnarray*}
where $U\subset\mathbb{K}^{n}$ is an appropriate neighborhood of $0$. With respect to these coordinates, we see that
$$(r(X_{j})f)(g)=\frac{\partial f}{\partial x_{j}}|_{x=0}.$$
We can canonically extend the  action of $\g$ via $r$ to the universal enveloping algebra $\mathfrak{U}(\g)$.

\begin{proposition}\label{canonic}
In each local trivialization, an invariant differential pairing can be written as
\begin{equation*}
P(\phi,\psi)=\sum_{\alpha,\beta}T_{\alpha,\beta}(r(X^{\alpha})\phi\otimes r(X^{\beta})\psi),
\end{equation*}
where $T_{\alpha,\beta}\in\mathrm{Hom}(\mathbb{V}\otimes\mathbb{W},\mathbb{E}))$, $r(X^{\alpha})=r(X_{1})^{\alpha_{1}}\cdots r(X_{n})^{\alpha_{n}}$ and $\{X_{i}\}_{i=1,...,n}$ is a basis of $\g_{-}$.  
\end{proposition}
\begin{proof}
Using the definition of $\mathbb{P}(V\times W,E)$ and the local trivialization from above, we can write each bilinear pairing as
$$P(\phi,\psi)=\sum_{\alpha,\beta}a_{\alpha,\beta}(r(X^{\alpha})\phi\otimes r(X^{\beta})\psi),$$
for $a_{\alpha,\beta}\in\mathcal{C}^{\infty}(G,\mathrm{Hom}(\mathbb{V}\otimes\mathbb{W},\mathbb{E}))$. If the pairing is to be $G$-invariant, we can compute:
\begin{eqnarray*}
P(\phi,\psi)(g)&=&(l_{g^{-1}}P(\phi,\psi))(e)\\
&=&P(l_{g^{-1}}\phi,l_{g^{-1}}\psi)(e)\\
&=&\sum_{\alpha,\beta}a_{\alpha,\beta}(e)(r(X^{\alpha})l_{g^{-1}}\phi\otimes r(X^{\beta})l_{g^{-1}}\psi)(e)\\
&=&\sum_{\alpha,\beta}a_{\alpha,\beta}(e)(l_{g^{-1}}r(X^{\alpha})\phi \otimes l_{g^{-1}}r(X^{\beta})\psi)(e)\\
&=&\sum_{\alpha,\beta}a_{\alpha,\beta}(e)(r(X^{\alpha})\phi \otimes r(X^{\beta})\psi)(g)
\end{eqnarray*}
and take $T_{\alpha,\beta}=a_{\alpha,\beta}(e)$.
\end{proof}

\subsection{Generalized bi-Verma modules}

\begin{definition}\label{biverma}
{\rm
Let $Y=\sum_{i}Y_{i,1}\otimes Y_{i,2}$ be an arbitrary element in $\mathfrak{U}(\p)\otimes\mathfrak{U}(\p)$. Then we define a left (resp.~right) $\mathfrak{U}(\p)\otimes\mathfrak{U}(\p)$-module structure on $\mathfrak{U}(\g)\otimes\mathfrak{U}(\g)$ and $\mathrm{Hom}(\mathbb{V}\otimes\mathbb{W},\mathbb{E})$ by
$$(u_{1}\otimes u_{2}).Y=\sum_{i}u_{1}Y_{i,1}\otimes u_{2}Y_{i,2}$$
and
$$(Y.T)(v\otimes w)=T\left(\sum_{i}\left(\lambda(Y_{i,1}^{0})v\otimes \nu(Y_{i,2}^{0})w\right)\right)$$
respectively. The antiautomorphism ${}^{0}$ of $\mathfrak{U}(\p)$ is defined by $(Z_{1}...Z_{l})^{0}=(-1)^{l}Z_{l}...Z_{1}$ for all $Z_{i}\in\p$.
Let us denote by 
$$\left(\mathfrak{U}(\g)\otimes\mathfrak{U}(\g)\right)\otimes_{\mathfrak{U}(\p)\otimes\mathfrak{U}(\p)}\mathrm{Hom}(\mathbb{V}\otimes\mathbb{W},\mathbb{E})$$
the tensor product of the two $\mathfrak{U}(\p)\otimes\mathfrak{U}(\p)$-modules and abbreviate $\otimes_{\mathfrak{U}(\p)\otimes\mathfrak{U}(\p)}$ by $\hat{\otimes}$.
Furthermore we define an action of $P$ on $\left(\mathfrak{U}(\g)\otimes\mathfrak{U}(\g)\right)\otimes_{\mathfrak{U}(\p)\otimes\mathfrak{U}(\p)}\mathrm{Hom}(\mathbb{V}\otimes\mathbb{W},\mathbb{E})$ by
$$p.(u_{1}\otimes u_{2}\hat{\otimes}T)=\mathrm{Ad}(p)u_{1}\otimes \mathrm{Ad}(p)u_{2}\hat{\otimes}\mu(p)\circ T\circ \lambda(p)^{-1}\otimes\nu(p)^{-1}$$
and denote by
$$\left(\left(\mathfrak{U}(\g)\otimes\mathfrak{U}(\g)\right)\otimes_{\mathfrak{U}(\p)\otimes\mathfrak{U}(\p)}\mathrm{Hom}(\mathbb{V}\otimes\mathbb{W},\mathbb{E})\right)^{P}$$
the space of elements which are fixed under this action.
}
\end{definition}

\begin{theorem}\label{three}
There is an isomorphism
$$\mathbb{P}_{G}(V\times W,E)\cong \left(\left(\mathfrak{U}(\g)\otimes\mathfrak{U}(\g)\right)\otimes_{\mathfrak{U}(\p)\otimes\mathfrak{U}(\p)}\mathrm{Hom}(\mathbb{V}\otimes\mathbb{W},\mathbb{E})\right)^{P}.$$
\end{theorem}
\begin{proof}
Let us define a mapping
\begin{eqnarray*}
\left(\mathfrak{U}(\g)\otimes\mathfrak{U}(\g)\right)\otimes_{\mathfrak{U}(\p)\otimes\mathfrak{U}(\p)}\mathrm{Hom}(\mathbb{V}\otimes\mathbb{W},\mathbb{E})&\rightarrow&\mathbb{P}_{G}(V\times W,E)\\
\sum_{j}\left(\sum_{i}u_{ij,1}\otimes u_{ij,2}\right)\hat{\otimes} T_{j}&\mapsto&  \left(\sum_{j}\left(\sum_{i}u_{ij,1}\otimes u_{ij,2}\right)\hat{\otimes} T_{j}\right)^{\vee},
\end{eqnarray*}
where
$$\left(\sum_{j}\left(\sum_{i}u_{ij,1}\otimes u_{ij,2}\right)\hat{\otimes} T_{j}\right)^{\vee}(\phi,\psi)=\sum_{j}T_{j}\left(\sum_{i}\left(r(u_{ij,1})\phi\otimes r(u_{ij,2})\psi\right)\right).$$
We will proceed in several steps:
\begin{enumerate}
\item
We will first show that this is well defined. Using the fact that
$$r(uY)f=\sigma(Y^{0})r(u)f$$
for all $u\in\mathfrak{U}(\g)$, $Y\in\mathfrak{U}(\p)$ and $f\in\mathcal{C}^{\infty}(G,\mathbb{F})^{P}$ (and we have denoted the action of $\p$ on $\mathbb{F}$ by $\sigma$), we can prove that the mapping is well defined for the tensor product $\hat{\otimes}$, i.e.
$$\left((u_{1}\otimes u_{2}).Y\otimes T\right)^{\vee}=\left(u_{1}\otimes u_{2}\otimes Y.T\right)^{\vee}.$$
To keep notation simple, let $Y=Y_{1}\otimes Y_{2}\in\mathfrak{U}(\p)\otimes\mathfrak{U}(\p)$. Then we compute
\begin{eqnarray*}
(((u_{1}\otimes u_{2}).Y)\otimes T)^{\vee}(\phi,\psi)&=&(u_{1}Y_{1}\otimes u_{2}Y_{2}\otimes T)^{\vee}(\phi,\psi)\\
&=&T(r(u_{1}Y_{1})\phi\otimes r(u_{2}Y_{2})\psi)\\
&=&T(\lambda(Y_{1}^{0})r(u_{1})\phi\otimes \nu(Y_{2}^{0})r(u_{2})\psi)\\
&=&Y.T(r(u_{1})\phi\otimes  r(u_{2})\psi)\\
&=&(u_{1}\otimes u_{2}\otimes Y.T)^{\vee}(\phi,\psi).
\end{eqnarray*}
\item
Let us first compute for $X\in\g$, $\phi\in\mathcal{C}^{\infty}(G,\mathbb{V})^{P}$, $g\in G$ and $p\in P$: 
\begin{eqnarray*}
(r(X)\phi)(gp)&=&\frac{d}{dt}\phi(gp\exp(tX))|_{t=0}\\
&=&\frac{d}{dt}\phi(g\exp(t\mathrm{Ad}(p)X)p)|_{t=0}\\
&=&\lambda(p^{-1})(r(\mathrm{Ad}(p)X)\phi)(g).
\end{eqnarray*}
The same is obviously also true for $\psi\in\mathcal{C}^{\infty}(G,\mathbb{W})^{P}$ and we see that
\begin{eqnarray*}
(u_{1}\otimes u_{2}\hat{\otimes}T)^{\vee}(\phi,\psi)(gp)&=&T((r(u_{1})\phi)(gp)\otimes (r(u_{2})\psi)(gp))\\
&=&T(\lambda(p)^{-1}(r(\mathrm{Ad}(p)u_{1})\phi)(g)\otimes \nu(p)^{-1}(r(\mathrm{Ad}(p)u_{2})\psi)(g))\\
&=&\mu(p)^{-1}(p.(u_{1}\otimes u_{2}\hat{\otimes}T))^{\vee}(\phi,\psi)(g).
\end{eqnarray*}
So $(u_{1}\otimes u_{2}\hat{\otimes}T)^{\vee}(\phi,\psi)\in\mathcal{C}^{\infty}(G,\mathbb{E})^{P}$ if and only if $(u_{1}\otimes u_{2}\hat{\otimes}T)$ is fixed by the action of $P$ as defined above.
\item
It is clear that this defines a $G$-invariant pairing since $r$ and $l$ mutually commute, see Proposition~\ref{canonic}.
\item
Injectivity follows from the fact that, as vector spaces,
$$\left(\mathfrak{U}(\g)\otimes\mathfrak{U}(\g)\right)\otimes_{\mathfrak{U}(\p)\otimes\mathfrak{U}(\p)}\mathrm{Hom}(\mathbb{V}\otimes\mathbb{W},\mathbb{E})\cong
\left(\mathfrak{U}(\g_{-})\otimes\mathfrak{U}(\g_{-})\right)\otimes_{\C}\mathrm{Hom}(\mathbb{V}\otimes\mathbb{W},\mathbb{E})$$
and the differential operators $r(X^{\alpha})$ are linearly independent for $X^{\alpha}\in\mathfrak{U}(\g_{-})$.
\item
The mapping is surjective by Proposition~\ref{canonic}.
\end{enumerate}
\end{proof}

\begin{corollary}\label{one}
Let us now specialize to the case that $G$ is semisimple and $P$ is a parabolic subgroup as described in the last chapter, i.e.~$G/P$ is a generalized flag manifold. Moreover, we assume that the representations of $P$ come from representations of the parabolic subalgebra $\p\subset \g$. Then we have an isomorphism
$$\mathbb{P}_{G}(V\times W,E)\cong \mathrm{Hom}_{\mathfrak{U}(\g)}\left(M_{\p}(\mathbb{E}),\left(\mathfrak{U}(\g)\otimes\mathfrak{U}(\g)\right)\otimes_{\mathfrak{U}(\p)\otimes\mathfrak{U}(\p)}\mathbb{V}^{*}\otimes\mathbb{W}^{*}\right).$$
\end{corollary}
\begin{proof}
We have the following isomorphisms:
\begin{eqnarray*}
\mathfrak{U}(\g)\otimes\mathfrak{U}(\g)\otimes_{\mathfrak{U}(\p)\otimes\mathfrak{U}(\p)}\mathrm{Hom}(\mathbb{V}\otimes\mathbb{W},\mathbb{E})&\cong&
\mathfrak{U}(\g)\otimes\mathfrak{U}(\g)\otimes_{\mathfrak{U}(\p)\otimes\mathfrak{U}(\p)}(\mathbb{V}\otimes\mathbb{W})^{*}\otimes\mathbb{E}\\
&\cong&\left(\mathfrak{U}(\g)\otimes\mathfrak{U}(\g)\otimes_{\mathfrak{U}(\p)\otimes\mathfrak{U}(\p)}(\mathbb{V}\otimes\mathbb{W})^{*}\right)\otimes\mathbb{E},
\end{eqnarray*}
since the $\mathfrak{U}(\p)\otimes\mathfrak{U}(\p)$ action on $(\mathbb{V}\otimes\mathbb{W})^{*}\otimes\mathbb{E}$ as defined in Definition~\ref{biverma} is given by the derived action of $\p$ on $(\mathbb{V}\otimes\mathbb{W})^{*}$ and the trivial action on $\mathbb{E}$. This isomorphism is an isomorphism of $P$ modules with the action given in Definition~\ref{biverma}.
Therefore
\begin{eqnarray*}
&&\left(\mathfrak{U}(\g)\otimes\mathfrak{U}(\g)\otimes_{\mathfrak{U}(\p)\otimes\mathfrak{U}(\p)}\mathrm{Hom}(\mathbb{V}\otimes\mathbb{W},\mathbb{E})\right)^{P}\\
&=&\mathrm{Hom}_{P}(\mathbb{E}^{*},\mathfrak{U}(\g)\otimes\mathfrak{U}(\g)\otimes_{\mathfrak{U}(\p)\otimes\mathfrak{U}(\p)}(\mathbb{V}\otimes\mathbb{W})^{*})\\
&=&\mathrm{Hom}_{\mathfrak{U}(\p)}(\mathbb{E}^{*},\mathfrak{U}(\g)\otimes\mathfrak{U}(\g)\otimes_{\mathfrak{U}(\p)\otimes\mathfrak{U}(\p)}(\mathbb{V}\otimes\mathbb{W})^{*})\\
&=&\mathrm{Hom}_{\mathfrak{U}(\g)}(\mathfrak{U}(\g)\otimes_{\mathfrak{U}(\p)}\mathbb{E}^{*},\mathfrak{U}(\g)\otimes\mathfrak{U}(\g)\otimes_{\mathfrak{U}(\p)\otimes\mathfrak{U}(\p)}(\mathbb{V}\otimes\mathbb{W})^{*})\\
&=&\mathrm{Hom}_{\mathfrak{U}(\g)}(M_{\p}(\mathbb{E}),\mathfrak{U}(\g)\otimes\mathfrak{U}(\g)\otimes_{\mathfrak{U}(\p)\otimes\mathfrak{U}(\p)}(\mathbb{V}\otimes\mathbb{W})^{*})).
\end{eqnarray*}
The second but last equality follows from the algebraic Frobenius reciprocity, see~\cite{v}.

\end{proof}

\begin{definition}
%As vector spaces (and even as $\g_{0}$-modules) we have an isomorphism
%$$\left(\mathfrak{U}(\g)\otimes\mathfrak{U}(\g)\right)\otimes_{\mathfrak{U}(\p)\otimes\mathfrak{U}(\p)}\mathbb{V}^{*}\otimes\mathbb{W}^{*}\cong\mathfrak{U}(\g_{-})\otimes\mathfrak{U}(\g_{-})\otimes\mathbb{V}^{*}\otimes\mathbb{W}^{*}.$$
%The filtration that we obtain is dual to the one for $J^{\infty}(\mathbb{V},\mathbb{W})$. 
{\rm Let us define
$$M_{\p}(\mathbb{V},\mathbb{W})=\left(\mathfrak{U}(\g)\otimes\mathfrak{U}(\g)\right)\otimes_{\mathfrak{U}(\p)\otimes\mathfrak{U}(\p)}\mathbb{V}^{*}\otimes\mathbb{W}^{*}$$
and call this module, following~\cite{d}, a {\bf generalized bi-Verma module}.
}\end{definition}

\begin{proposition}
$M_{\p}(\mathbb{V},\mathbb{W})$ is a $\mathfrak{U}(\g)$-module with action given by
$$X.(u_{1}\otimes u_{2}\hat{\otimes} v\otimes w)=(Xu_{1}\otimes u_{2}+u_{1}\otimes Xu_{2})\hat{\otimes}v\otimes w.$$
Note that the restriction of this action to $\p$ is exactly the derived action of $P$ from above.
\end{proposition}
\begin{proof}
The action of $P$ as defined on 
$$\left(\mathfrak{U}(\g)\otimes\mathfrak{U}(\g)\right)\hat{\otimes}\mathrm{Hom}(\mathbb{V}\otimes\mathbb{W},\mathbb{E})
\cong\left(\mathfrak{U}(\g)\otimes\mathfrak{U}(\g)\right)\hat{\otimes}\mathbb{V}^{*}\otimes\mathbb{W}^{*}\otimes\mathbb{E}$$
is just the tensorial one. This implies that the action of $P$ on 
$$\left(\mathfrak{U}(\g)\otimes\mathfrak{U}(\g)\right)\hat{\otimes}\mathbb{V}^{*}\otimes\mathbb{W}^{*}$$
as used in Corollary~\ref{one} is also just the tensorial one. For simplicity, let $X,Y\in\g$ and $v^{*}\in\mathbb{V}^{*}$, $w^{*}\in\mathbb{W}^{*}$ and look at the derived action of $\xi\in\p$:
\begin{eqnarray*}
\xi.(X\otimes Y\hat{\otimes}v^{*}\otimes w^{*})&=& [\xi,X]\otimes Y\hat{\otimes}v^{*}\otimes w^{*}+X\otimes [\xi, Y]\hat{\otimes}v^{*}\otimes w^{*}\\
&&+X\otimes Y\hat{\otimes}\xi.v^{*}\otimes w^{*}+X\otimes Y\hat{\otimes}v^{*}\otimes \xi.w^{*}\\
&=&\xi X\otimes Y\hat{\otimes}v^{*}\otimes w^{*}+X\otimes \xi Y\hat{\otimes}v^{*}\otimes w^{*}.
\end{eqnarray*}
This action can be extended to $\g$ (as in Corollary~\ref{one}) and it is exactly as described above.
\end{proof}

\subsubsection{Remark}
The module $M_{\p}(\mathbb{V},\mathbb{W})$ can be defined as a direct limit as follows. As a $\g_{0}$-module, we have the isomorphism
$$M_{\p}(\mathbb{V},\mathbb{W})\cong\mathfrak{U}(\g_{-})\otimes_{\C}\mathfrak{U}(\g_{-})\otimes_{\C}\mathbb{V}^{*}\otimes_{\C}\mathbb{W}^{*}$$
that allows us to define 
$$M_{\p}(\mathbb{V},\mathbb{W})_{k}=\{u_{1}\otimes u_{2}\otimes v^{*}\otimes w^{*}\in \mathfrak{U}(\g_{-})\otimes\mathfrak{U}(\g_{-})\otimes\mathbb{V}^{*}\otimes\mathbb{W}^{*}\;:\;u_{i}\in\mathfrak{U}_{-l_{i}}(\g_{-})\;\text{and}\;l_{1}+l_{2}\leq k\}.$$
This induces a filtration 
$$0\subset \mathbb{V}^{*}\otimes\mathbb{W}^{*}\subset M_{\p}(\mathbb{V},\mathbb{W})_{1}\subset\cdots\subset M_{\p}(\mathbb{V},\mathbb{W})_{k}\subset M_{\p}(\mathbb{V},\mathbb{W})_{k+1}\subset\cdots\subset M_{\p}(\mathbb{V},\mathbb{W}).$$
%\subsubsection{Remark}
%We have a $\mathfrak{U}(\g)$-module inclusion
%\begin{eqnarray*}
%\mathfrak{S}_{2}\hat{\otimes}v_{0}&\hookrightarrow& M_{\mathfrak{b}}(\mathbb{V},\mathbb{V})\\
%\{u_{1}\otimes u_{2}\}\hat{\otimes}v_{0}&\mapsto& \{u_{1}\otimes u_{2}\} \hat{\otimes}v_{0}\otimes v_{0}.
%\end{eqnarray*}

\begin{proposition}\label{propthree}
The notion of invariant bilinear differential pairings on generalized flag manifolds $G/P$ as given in Section~\ref{pairings} is equivalent to the definition given in this section.
\end{proposition}
\begin{proof} 
Let
$$P:\Gamma(V)\times\Gamma(W)\rightarrow \Gamma(E)$$
be an invariant bilinear differential pairing. According to the last section, this is equivalent to having a homomorphism
$$\mathcal{J}^{M}(V,W)\rightarrow E,$$
for some $M$.
%that we can extend to a homomorphism $\mathcal{J}^{\infty}(V,W)\rightarrow E$ (by the definition of projective limit).
Since the pairing is invariant, it is determined by  its action at the identity coset $eP$ where it determines a $\p$-module homomorphism 
$$\mathcal{J}^{M}(\mathbb{V},\mathbb{W})\rightarrow \mathbb{E}.$$
%We claim that  $M_{\p}(\mathbb{V},\mathbb{W})$ is the dual of $\mathcal{J}^{\infty}(\mathbb{V},\mathbb{W})$. 
We assert that $\mathcal{J}^{k}(\mathbb{V},\mathbb{W})^{*}\cong M_{\p}(\mathbb{V},\mathbb{W})_{k}$ for every $k\geq 0$ and prove the assertion by induction.
%, where $M_{\p}(\mathbb{V},\mathbb{W})_{k}$ is the space of element in $M_{\p}(\mathbb{V},\mathbb{W})$ that have a  
%$$M_{\p}(\mathbb{V},\mathbb{W})\cong\mathfrak{U}(\g_{-})\otimes_{\C}\mathfrak{U}(\g_{-})\otimes_{\C}\mathbb{V}^{*}\otimes_{\C}\mathbb{W}^{*}$$
%(as a $\g_{0}$-module). 
%More precisely, $M_{\p}(\mathbb{V},\mathbb{W})$ is the direct limit
%$$0\subset \mathbb{V}^{*}\otimes\mathbb{W}^{*}\subset M_{\p}(\mathbb{V},\mathbb{W})_{1}\subset\cdots\subset M_{\p}(\mathbb{V},\mathbb{W})_{k}\subset M_{\p}(\mathbb{V},\mathbb{W})_{k+1}\subset\cdots\subset M_{\p}(\mathbb{V},\mathbb{W}).$$
For $k=0$ the assertion is trivial. Let us assume that the assertion is true for $k-1$ and use the exact sequence
$$0\rightarrow \bigoplus_{l=0}^{k}\mathfrak{U}_{l}(\p_{+})\otimes \mathbb{V}\otimes\mathfrak{U}_{k-l}(\p_{+})\otimes \mathbb{W}\stackrel{\iota}{\rightarrow} \mathcal{J}^{k}(\mathbb{V},\mathbb{W})\rightarrow \mathcal{J}^{k-1}(\mathbb{V},\mathbb{W})\rightarrow 0.$$
Taking duals of this sequence together with the isomorphism $\p_{+}\cong\g_{-}^{*}$  proves the assertion. Now we can use the dual map 
$$\mathbb{E}^{*}\rightarrow \mathcal{J}^{M}(\mathbb{V},\mathbb{W})^{*}\cong M_{\p}(\mathbb{V},\mathbb{W})_{M}\subset M_{\p}(\mathbb{V},\mathbb{W})$$
%Taking the direct limit yields 
%$$M_{\p}(\mathbb{V},\mathbb{W})\cong \;\stackrel{\mathrm{lim}}{\rightarrow}\mathcal{J}^{k}(\mathbb{V},\mathbb{W})^{*}=\left(\stackrel{\mathrm{lim}}{\leftarrow}\mathcal{J}^{k}(\mathbb{V},\mathbb{W})\right)^{*}=\mathcal{J}^{\infty}(\mathbb{V},\mathbb{W})^{*}.$$
%\par
together with the Frobenius reciprocity 
$$\mathrm{Hom}_{\p}(\mathbb{E}^{*},M_{\p}(\mathbb{V},\mathbb{W}))\cong\mathrm{Hom}_{\mathfrak{U}(\g)}(M_{\p}(\mathbb{E}),M_{\p}(\mathbb{V},\mathbb{W}))$$
and Corollary~\ref{one} to prove the claim.
\end{proof}

\subsection{Singular vectors in $M_{\p}(\mathbb{V},\mathbb{W})$}

\begin{definition}
{\rm A vector $\Theta\in M_{\p}(\mathbb{V},\mathbb{W})$ is called {\bf singular vector} of weight $\mu\in\h^{*}$ iff
\begin{enumerate}
\item
$X.\Theta=0$ for all $X\in\mathfrak{n}$ and
\item
$H.\Theta=\mu(H)\Theta$ for all $H\in\h$.
\end{enumerate}}
\end{definition}

\subsubsection{Remark}
Note that we allow singular vectors to be of the form $1\otimes 1\hat{\otimes}f_{0}^{*}$, where $f_{0}^{*}$ is a highest weight vector of an irreducible component of $\mathbb{V}^{*}\otimes\mathbb{W}^{*}$. These weight vectors correspond to zero order invariant bilinear differential pairings. So, for example, contraction of a $k$-form with a vector field is included.

\begin{proposition}
Singular vectors in $M_{\p}(\mathbb{V},\mathbb{W})$ are in 1-1 correspondence with invariant bilinear differential pairings, where we assume all modules to be irreducible.
\end{proposition}
\begin{proof}
Using Theorem~\ref{three} it is clear that each invariant  bilinear differential pairing $\Gamma(V)\times\Gamma(W)\rightarrow \Gamma(E)$ induces a singular vector by looking at the image of a highest weight vector of $M_{\p}(\mathbb{E})$ in $M_{\p}(\mathbb{V},\mathbb{W})$.
\par
Conversely, for every singular vector $\Theta\in M_{\p}(\mathbb{V},\mathbb{W})$ of weight $\mu$, it is enough to define a $\mathfrak{U}(\g)$-module homomorphism
$$M_{\p}(\mathbb{E})\rightarrow M_{\p}(\mathbb{V},\mathbb{W}).$$
This is easily done by
$$Y.(1\hat{\otimes}e_{0}^{*})\mapsto Y.\Theta,$$
for $Y\in\mathfrak{U}(\g)$ and $e_{0}^{*}$ a highest weight vector in $\mathbb{E}^{*}$ of weight $\mu$. It only remains to show that there is a finite dimensional irreducible $\p$-module of highest weight $\mu$, i.e.~that $\mu$ is dominant integral for $\p$. The following lemma ensures this.
\end{proof}

\begin{lemma}
$M_{\p}(\mathbb{V},\mathbb{W})$ is the direct sum of finite dimensional irreducible $\g_{0}$ modules.
\end{lemma}
\begin{proof}
As a $\g_{0}$-module 
\begin{eqnarray*}
M_{\p}(\mathbb{V},\mathbb{W})&\cong&\mathfrak{U}(\g_{-})\otimes\mathfrak{U}(\g_{-})\otimes\mathbb{V}^{*}\otimes\mathbb{W}^{*}\\
&\cong& M_{\p}(\mathbb{V})\otimes M_{\p}(\mathbb{W}).
\end{eqnarray*}
Now we can use the fact that $M_{\p}(\mathbb{F})$ is the direct sum of finite dimensional irreducible $\g_{0}$-modules (see~\cite{l}, p.~500) for every finite dimensional irreducible representation $\mathbb{F}$ of $\p$.
\end{proof}
Note that the above homomorphism is not a homomorphism of $\g$-modules. In particular, the sum of the central character of $M_{\p}(\mathbb{V})$ and $M_{\p}(\mathbb{W})$
cannot be used as the central character of $M_{\p}(\mathbb{V},\mathbb{W})$.
% It shows, however, that the filtrations of $M_{\p}(\mathbb{V},\mathbb{W})$ and
%$M_{\p}(\mathbb{V})\otimes M_{\p}(\mathbb{W})$ are equivalent. In fact, the filtration of $M_{\p}(\mathbb{V},\mathbb{W})$ is dual to the filtration of the fibre of $\mathcal{J}^{\infty}(V,W)$ at the identity. For the explicit description of the fibration of $\mathcal{J}^{\infty}(V,W)$ is given in the appendix.
\par
\vspace{0.5cm}

\begin{theorem}[Main result 1]\label{mainresultone}
Let $G/P$ be a generalized flag manifold with a filtration of the Lie algebra $\g$ as in~\ref{grading}. Fix two irreducible finite dimensional representations $\mathbb{V}$ and $\mathbb{W}$ and let $v_{0}^{*}$ (resp.~$w_{0}^{*}$) be a highest weight vector of $\mathbb{V}^{*}$ (resp.~$\mathbb{W}^{*}$) with highest weight $\lambda$ (resp.~$\nu$). 
Moreover let $\alpha_{i}$, $i\in I$, be a simple root.
Then there exists an invariant bilinear differential pairing
$$\Gamma(V)\times\Gamma(W)\rightarrow\Gamma(E_{M,i})$$
of weighted order $M$ for all $M\in\N$ and all $i\in I$, where $\mathbb{E}_{M,i}$ is the finite dimensional irreducible representation of $\p$ that is dual to the finite dimensional irreducible representation of highest weight $\lambda+\nu-M\alpha_{i}$. 
\end{theorem}
\begin{proof}
For $i\in I$ as above choose elements $X_{\alpha_{i}}\in\g_{\alpha_{i}}$, $X_{-\alpha_{i}}\in\g_{-\alpha_{i}}$ and $H_{\alpha_{i}}\in\h$ such that $[X_{\alpha_{i}},X_{-\alpha_{i}}]=H_{\alpha_{i}}$ as in Example~\ref{sltwo}.
Let us define a vector $\Theta^{\alpha_{i}}_{M}$ for every $M\geq 0$ and $i\in I$ by
\begin{equation}
\Theta^{\alpha_{i}}_{M}=\left(\sum_{j=0}^{M}\gamma_{M,j}X_{-\alpha_{i}}^{j}\otimes X_{-\alpha_{i}}^{M-j}\right)\hat{\otimes}v_{0}^{*}\otimes w_{0}^{*},
\end{equation}
for some constants $\gamma_{M,j}$.
%If $X\in\mathfrak{n}_{S}=\bigoplus_{\alpha\in\Delta^{+}\cap\mathrm{span}\mathcal{S}_{\p}}\g_{\alpha}$, then $[X,X_{-\alpha_{i}}]=0$ since $\alpha-\alpha_{i}\not\in\Delta$ for all
%$\alpha\in\Delta^{+}\cap\mathrm{span}\mathcal{S}_{\p}$. So $\mathfrak{n}_{S}$ annihilates $\Theta_{M}^{\alpha_{i}}$. If $X\in\mathfrak{p}_{+}$, then $[X,X_{-\alpha_{i}}]\in\mathfrak{n}\cup\{0\}$ for $X\not=X_{\alpha_{i}}$. Now use the fact that the only scalar multiples of $\alpha_{i}$ that are roots are $\alpha_{i}$ and $-\alpha_{i}$. This implies that $\Theta_{\alpha_{i}}^{M}$ is annihilated by  $\mathfrak{n}\backslash\g_{\alpha_{i}}$. 
To show that $\Theta_{\alpha_{i}}^{M}$ is annihilated by $\mathfrak{n}$ it is sufficient to show that $X_{\alpha}$ annihilates $\Theta_{\alpha_{i}}^{M}$ for all $\alpha\in\mathcal{S}$. If $\alpha\not=\alpha_{i}$, then $[X_{\alpha},X_{-\alpha_{i}}]\in\g_{\alpha-\alpha_{i}}=0$. Moreover each $X_{\alpha}$ annihilates $v_{0}^{*}$ and $w_{0}^{*}$. This implies that $\Theta_{\alpha_{i}}^{M}$ is annihilated by all $X_{\alpha}$, $\alpha\in\mathcal{S}\backslash\{\alpha_{i}\}$. 
In $\mathfrak{U}(\g)$ we have the equality
$$X_{\alpha_{i}}X_{-\alpha_{i}}^{k}=X_{-\alpha_{i}}^{k}X_{\alpha_{i}}+kX_{-\alpha_{i}}^{k-1}H_{\alpha_{i}}+a_{k}X_{-\alpha_{i}}^{k-1},$$
where 
%$a_{k}$ is define recursively as
%$$a_{0}=0,\;a_{j}=-2(j-1)+a_{j-1},\;j=1,...,k.$$
$$a_{k}=k(1-k).$$
Here we have used that $[H_{\alpha_{i}},X_{-\alpha_{i}}]=-\alpha_{i}(H_{\alpha_{i}})X_{-\alpha_{i}}=-2X_{-\alpha_{i}}$.
Using this, we can compute
\begin{eqnarray*}
X_{\alpha_{i}}.\Theta^{\alpha_{i}}_{M}&=&
(\sum_{j=0}^{M}\gamma_{M,j}(j\lambda(H_{\alpha_{i}})+a_{j})X_{-\alpha_{i}}^{j-1}\otimes X^{n-j}_{-\alpha_{i}}\\
&&+\gamma_{M,j}((M-j)\nu(H_{\alpha_{i}})+a_{M-j})X_{-\alpha_{i}}^{j}\otimes X_{-\alpha_{i}}^{M-j-1})\hat{\otimes}v_{0}^{*}\otimes w_{0}^{*}.
\end{eqnarray*}
So we obtain $M$ equations that have to be satisfied for $\Theta_{M}^{\alpha_{i}}$ to be a singular vector:
\begin{equation}\label{mainequationtwo}
\fbox{$\displaystyle\gamma_{M,j+1}(j+1)\left(q-j\right)+\gamma_{M,j}(M-j)\left(q'-M+j+1\right)=0,$}
\end{equation}
for $j=0,1,...,M-1$ and $q=\lambda(H_{\alpha_{i}})=B(\lambda,\alpha_{i}^{\vee})$ and $q'=\nu(H_{\alpha_{i}})=B(\nu,\alpha_{i}^{\vee})$. These equations are exactly the equations that we had to consider in~(\ref{mainequation}) in the Introduction.
Since there are more unknowns than equations, this system can be solved for $\gamma_{M,j}$ to obtain an (at least) one-paramter family of solutions. The corresponding $\Theta_{M}^{\alpha_{i}}$ has
weight
$$H.\Theta_{M}^{\alpha_{i}}=(\lambda+\nu-M\alpha_{i})(H)\Theta_{M}^{\alpha_{i}},$$
which is exactly the highest weight of $\circledcirc^{M}\g_{-1}^{i}\circledcirc\mathbb{V}^{*}\circledcirc\mathbb{W}^{*}$, where $\circledcirc$ denotes the Cartan product of representations (see~\cite{e1}).
We can now define an inclusion
\begin{eqnarray*}
M_{\p}(\mathbb{E}_{M,i})&\rightarrow& M_{\p}(\mathbb{V},\mathbb{W})\\
Y.(1\hat{\otimes}e^{*})&\mapsto& Y.\Theta_{n}^{\alpha_{0}}\quad\forall\;Y\in\mathfrak{U}(\g),
\end{eqnarray*}
where $\mathbb{E}_{M,i}$ is dual to an irreducible finite dimensional representation of $\p$ with highest weight $\lambda+\nu-M\alpha_{i}$. This defines an invariant differential pairing of weighted order $M$
$$\Gamma(V)\times\Gamma(W)\rightarrow\Gamma(E_{M,i}).$$
\end{proof}

\begin{corollary}
\begin{enumerate}
\item
With the situation as above. If
$$q=B(\lambda,\alpha_{i}^{\vee})\not\in\{0,...,M-1\}$$
or
$$q'=B(\nu,\alpha_{i}^{\vee})\not\in\{0,...,M-1\},$$
then there exists (up to scalars) exactly one singular vector $\Theta$ of the form $\Theta_{M}^{\alpha_{i}}$. 
\item
If $q\geq 0$, then there exists an invariant linear differential operator 
$$\Gamma(V)\rightarrow \Gamma(E_{\lambda-(q+1)\alpha_{i}}).$$
Analogously if $q'\geq 0$, then there exists an invariant linear differential operator $\Gamma(W)\rightarrow \Gamma(E_{\nu-(q'+1)\alpha_{i}})$.
\end{enumerate}
\end{corollary}
\begin{proof}
The first statement follows immediately by looking at the rank of the matrix that describes the linear equations~(\ref{mainequationtwo})
from the proof of Theorem~\ref{mainresultone}.

\par
Say $q\geq 0$, then 
$$\Psi_{q+1}^{i}=X_{-\alpha_{i}}^{q+1}\hat{\otimes}v_{0}^{*}\in M_{\p}(\mathbb{V})$$
is a singular vector of highest weight $\lambda-(q+1)\alpha_{i}$. This vector can be used to define a $\g$-module homomorphism
$$M_{\p}(\mathbb{E}_{\lambda-(q+1)\alpha_{i}})\rightarrow M_{\p}(\mathbb{V}),$$
which in turn defines an invariant linear differential operator $\Gamma(V)\rightarrow \Gamma(E_{\lambda-(q+1)\alpha_{i}})$ (see~\cite{be}, Theorem 11.2.1).
The situation for $\mathbb{W}$ is exactly analogous.
\end{proof}

\subsubsection{Remark}
A short examination of the linear equations~(\ref{mainequationtwo}) shows that if $q\in\{0,1,...,M-1\}$ and $q'\in\{0,1,...,M-1\}$, then there could be a one parameter family or a two parameter family of singular vectors depending on the relation of $q$ to $q'$.

\subsection{Examples}
%Let us suppose that $\mathbb{V}^{*}$ and $\mathbb{W}^{*}$ are one dimensional representations and that $\h$ acts as $\lambda$ and $\nu$ respectively.
\begin{enumerate}
\item
$M=1$: 
We define for $a,b\in\C$:
$$\Theta_{1}^{\alpha_{i}}=(a1\otimes X_{-\alpha_{i}}+bX_{-\alpha_{i}}\otimes 1)\hat{\otimes}v_{0}^{*}\otimes w_{0}^{*},$$
then
$$X_{\alpha_{i}}.\Theta_{1}^{\alpha_{i}}=(a\nu(H_{\alpha_{i}})+b\lambda(H_{\alpha_{i}}))1\otimes 1\hat{\otimes} v_{0}^{*}\otimes w_{0}^{*}.$$
Choosing $a=\lambda(H_{\alpha_{i}})$ and $b=-\nu(H_{\alpha_{i}})$ yields a singular vector of highest weight $\lambda+\nu-\alpha_{i}$ which is exactly the highest weight of $\mathbb{V}^{*}\circledcirc\mathbb{W}\circledcirc \g_{-1}^{i}$. But in case $\lambda(H_{\alpha_{i}})=\nu(H_{\alpha_{i}})=0$, we can choose $a,b$ arbitrarily and obtain two independent singular vectors. Note that 
$$\lambda(H_{\alpha_{i}})=B(\lambda,\alpha_{i}^{\vee}),$$
so these numbers are exactly the numbers that we write over the $i$-th (crossed through) node in the Dynkin diagram notation for $\mathbb{V}$.
\par
Of course, the singular vectors described above correspond to the pairings like 
\begin{eqnarray*}
\xoo{v+1}{0}{0}\;\cdots\;\oo{0}{1}\;\times \xoo{w}{0}{0}\;\cdots\;\oo{0}{0}&\rightarrow &\xoo{v+w-1}{1}{0}\;\cdots\;\oo{0}{1}\\
(X^{a},f)&\mapsto&wf(\nabla_{a}X^{b}-\frac{1}{n}\delta_{a}{}^{b}\nabla_{c}X^{c})\\
&&                          -(v+1)(X^{b}\nabla_{a}f-\frac{1}{n}\delta_{a}{}^{b}X^{c}\nabla_{c}f),
\end{eqnarray*}
\begin{eqnarray*}
\quad\Dd{v}{1}{0}{0}{0}{0}{0}\;\times\hspace{1.5cm}\;\Dd{w}{0}{0}{0}{0}{0}{0}\;&\rightarrow&\hspace{2cm}\;\Dd{v+w-2}{2}{0}{0}{0}{0}{0}\;\\
&&\\
(X^{a},f)&\mapsto&wf(\nabla_{(a}X_{b)}-\frac{1}{n}g_{ab}\nabla_{c}X^{c})\\
&&                         -v(X_{(a}\nabla_{b)}f-\frac{1}{n}g_{ab}X^{c}\nabla_{c}f)
\end{eqnarray*}
or
\begin{eqnarray*}
\xo{w}{0}\;\cdots\;\oox{0}{0}{w'}\;\times\;\xo{v}{0}\;\cdots\;\oox{0}{0}{v'}\;&\rightarrow&\;\xo{w+v}{0}\;\cdots\;\oox{0}{1}{-2+w'+v'}\;\\
(f,g)&\mapsto & w'f\nabla_{\bar{\alpha}}g-v'g\nabla_{\bar{\alpha}}f.
\end{eqnarray*}
For the notational conventions used here, we refer to Chapter~\ref{tractorchapter}.
Note that the pairing $P(f,g)=wf\nabla_{a}g-vg\nabla_{a}f$ is skew in $f,g$ for $v=w$, so the corresponding (non-linear) homogeneous operator of degree two given by $f\mapsto P(f,f)$ vanishes. This is an example of a situation where the setup in~\cite{d} differs from the setup presented here.
\item
We can even go further and define
$$\Theta_{2}^{\alpha_{i}}=(a 1\otimes X_{-\alpha_{i}}^{2}+bX_{-\alpha_{i}}\otimes X_{-\alpha_{i}}+c X^{2}_{-\alpha_{i}}\otimes 1)\hat{\otimes}v_{0}^{*}\otimes w_{0}^{*},$$
for some constants $a,b,c\in\C$. This yields the following equations
\begin{eqnarray*}
2a(\nu(H_{\alpha_{i}})-1)+b\lambda(H_{\alpha_{i}})&=&0\\
2c(\lambda(H_{\alpha_{i}})-1)+b\nu(H_{\alpha_{i}})&=&0.
\end{eqnarray*}
On $\mathbb{RP}_{n}$, for example, this corresponds to pairings like
\begin{eqnarray*}
\xoo{w}{0}{0}\;\cdots\;\oo{0}{0}\;\times\;\xoo{v}{0}{0}\;\cdots\;\oo{0}{0}\;&\rightarrow&\;\xoo{-4+v+w}{2}{0}\;\cdots\;\oo{0}{0}\;\\
(f,g)&\mapsto&\underbrace{(w-1)w}_{a}f\nabla_{a}\nabla_{b}g\\
&&\underbrace{-2(w-1)(v-1)}_{b}\nabla_{(a}f\nabla_{b)}g\\
&&+\underbrace{(v-1)v}_{c}g\nabla_{a}\nabla_{b}f.
\end{eqnarray*}
In this case $\alpha_{i}=\alpha_{1}$ with $\lambda(H_{\alpha})=w$ and $\nu(H_{\alpha})=v$. The general pattern for these pairings will be determined in Chapter~\ref{higherorderone}.
\end{enumerate}

\section{Invariance}\label{invariance}

\subsection{The homogeneous case}\label{thehomogeneouscase}
So far we have only properly defined invariance on the homogenous model spaces $G/P$. Invariance here means that a pairing is invariant with respect to the action of $G$ on the sections of the vector bundles involved as in Section~\ref{pairings}. In Proposition~\ref{propthree} it was proved that there is a bijective correspondence between invariant bilinear differential pairings
$$\Gamma(V)\times\Gamma(W)\rightarrow\Gamma(E)$$
and $P$-module homomorphisms  
$$\mathcal{J}^{M}(\mathbb{V},\mathbb{W})\rightarrow \mathbb{E},$$
for some $M$.

\subsection{Weyl connections}

\begin{definition}
{\rm Let $P_{+}=\exp(\p_{+})$, then it is easy to see (\cite{cs}, Proposition 2.10) that $P/P_{+}\cong G_{0}$ and we can consider the following two principal bundles
$$\begin{array}{ccc}
P&\rightarrow&\mathcal{G}\\
&&\downarrow\\
&&\mathcal{M}
\end{array}\quad\text{and}\quad  \begin{array}{cccc}
P/P_{+}=G_{0}&\rightarrow&\mathcal{G}_{0}&=\mathcal{G}/P_{+}\\
&&\downarrow&\\
&&\mathcal{M}&
\end{array}.$$
In fact, $\mathcal{G}_{0}$ is the frame bundle of $gr T\mathcal{M}$.\par
A {\bf Weyl structure} is given by a $G_{0}$-equivariant section $\sigma: \mathcal{G}_{0}\rightarrow\mathcal{G}$. Writing the $\g_{0}$ component of the Cartan connection on $\mathcal{G}$ as $\omega_{0}$, one can define the pullback
$$\gamma^{\sigma}=\sigma^{*}(\omega_{0}),$$
which is a principal connection, the so-called {\bf Weyl connection}, on $\mathcal{G}_{0}$. This defines, for each choice of Weyl structure, a linear connection on $\mathcal{M}$.}
\end{definition}

\subsubsection{Remark}
Similarly to the remark made in the Introduction, we again have to issue a warning at this stage. Our considerations are of a completely local nature, so we can always restrict ourselves to an appropriate open set in $\mathcal{M}$ and consider {\bf local} Weyl structures. This is especially important in the holomorphic setting where a global Weyl structure might not exist (due to the lack of partitions of unity that exist in the smooth category). We will ignore this problem throughout the thesis and always implicitly restrict our considerations to an appropriate coordinate patch.

\begin{definition}
{\rm
Let $\mathcal{A}=\mathcal{G}\times_{P}\g$ be the {\bf adjoint tractor bundle} induced by the adjoint representation of $P$ on $\g$. The grading of $\g$ induces a grading on $\mathcal{A}$:
$$\mathcal{A}=\mathcal{A}_{-k_{0}}+...+\mathcal{A}_{k_{0}}.$$
}\end{definition}

\begin{lemma}\label{four}
For each parabolic geometry $(\mathcal{M},\mathcal{G},\g,\omega)$ of type $(G,P)$ the space of global (possibly after restriction to an appropriate open coordinate patch) $G_{0}$ equivariant sections $\sigma :\mathcal{G}_{0}\rightarrow\mathcal{G}$ is a non-empty affine space modeled over the space of all one-forms on $\mathcal{M}$. More precisely, let $\sigma$ and $\hat{\sigma}$ be two such sections, then
$$\hat{\sigma}(u)=\sigma(u)\exp(\Upsilon_{1}(u))\cdots\exp(\Upsilon_{k_{0}}(u)),$$
where $\Upsilon=(\Upsilon_{1},...,\Upsilon_{k_{0}})\in\Gamma(\mathcal{A}_{1}\oplus\cdots\oplus\mathcal{A}_{k_{0}})=\Gamma(gr T^{*}\mathcal{M})$.
\end{lemma}
\begin{proof}
A proof may be found in~\cite{cs2}, Proposition 3.2.
\end{proof}

Let $\rho:P\rightarrow\mathrm{GL}(\mathbb{V})$ be an irreducible representation, which is therefore determined by $\rho|_{G_{0}}:G_{0}\rightarrow\mathrm{GL}(\mathbb{V})$. Then each Weyl structure induces a connection on the associated bundle $V=\mathcal{G}\times_{P}\mathbb{V}=\mathcal{G}_{0}\times_{G_{0}}\mathbb{V}$. Let $\sigma$ and $\hat{\sigma}$ be two Weyl structures related by $\Upsilon$ as in Lemma~\ref{four}, then we denote the corresponding connections on $V$ by $\nabla$ and $\hat{\nabla}$. An explicit formula for the difference $\nabla-\hat{\nabla}$ is given in~\cite{cs2}, Proposition 3.9:
\begin{equation}\label{diffconnection}
\hat{\nabla}_{\xi}s=\nabla_{\xi}s+\sum_{\Vert \underline{j}\Vert+l=0}\frac{(-1)^{\underline{j}}}{\underline{j}!}(\mathrm{ad}(\Upsilon_{k_{0}})^{j_{k_{0}}}\circ\cdots\circ\mathrm{ad}(\Upsilon_{1})^{j_{1}}(\xi_{l}))\bullet s,
\end{equation}
where $\underline{j}=(j_{1},...,j_{k_{0}})$, $(-1)^{\underline{j}}=(-1)^{j_{1}+\cdots+j_{k_{0}}}$, $\Vert \underline{j}\Vert=j_{1}+2j_{2}+\cdots+k_{0}j_{k_{0}}$, $\underline{j}!=j_{1}! \cdots j_{k_{0}}!$, $\xi=(\xi_{-k_{0}},...,\xi_{-1})\in\Gamma(gr T\mathcal{M})$ and where $\bullet:\mathcal{A}^{0}\times V\rightarrow V$ is induced by the representation $\p\times\mathbb{V}\rightarrow\mathbb{V}$.
We will give an explicit version of this formula for projective, conformal and CR geometry in Chapter~\ref{tractorchapter}.

\subsubsection{Remark}\label{semiinvariant}
Let us choose a Weyl structure with corresponding Weyl connection $\nabla$ on every irreducible associated bundle $V$ as above. Let $E$ and $F$  be associated bundles that are induced by representations of $\p$ which, as $\g_{0}$-modules, are completely reducible. Examples are irreducible $\p$-modules or $\p$-modules that are the restrictions of representations of $\g$, as we shall see later. The Weyl structure induces a trivialization of these bundles into a direct sum of irreducible subbundles (see the last chapter about trivializations of tractor bundles for each choice of Weyl structure and how these trivializations vary when changing from one Weyl structure to another Weyl structure).
A {\bf semi-invariant} differential operator is given by a formula that consists of terms, each of which is induced by a mapping
$$E\stackrel{\nabla^{N}}{\longrightarrow}\otimes^{N}\Lambda^{1}\otimes E\stackrel{\phi}{\longrightarrow} F,$$
where $\phi$ is induced by a homomorphism $\Phi:\otimes^{N}\p_{+}\otimes\mathbb{E}\rightarrow\mathbb{F}$ of $\g_{0}$-modules. A semi-invariant differential pairing is defined analogously. For the classical groups, semi-invariance is usually defined as in~\cite{gs} via local coordinates. However, Weyl's classical invariant theory implies that these two definitions are equivalent for the classical groups, see~\cite{gs}, Section 4. Hence our definition extends the notion of semi-invariant operators (pairings) to arbitrary parabolic geometries.

\begin{definition}\label{definvariance}
{\rm A bilinear differential pairing (linear differential operator) on a parabolic geometry is called {\bf invariant}, if it can be written as a semi-invariant formula as described in Remark~\ref{semiinvariant} for any choice of (local) Weyl connection in such a way that the formula as a whole does not change (i.e.~does not involve any $\Upsilon$-terms) when changing from one Weyl connection $\nabla$ to another $\hat{\nabla}$. In other words, it has to be independent of the choice of Weyl structure.
}\end{definition}

\subsection{Weyl structures and filtrations}\label{filtrations}
Let $\mathbb{V}=\mathbb{V}^{0}\supset\mathbb{V}^{1}\supset\cdots\supset\mathbb{V}^{s}\supset\{0\}$ be a filtered $\p$-module, so that
\begin{enumerate}
\item
$\mathfrak{z}(\g_{0})$ acts diagonalizably,
\item
the associated graded module
$$gr\mathbb{V}=gr_{0}\mathbb{V}\oplus gr_{1}\mathbb{V}\oplus\cdots\oplus gr_{s}\mathbb{V},$$
with $gr_{i}\mathbb{V}=\mathbb{V}^{i}/\mathbb{V}^{i+1}$, has the property that $\g_{i}.gr_{j}\mathbb{V}\subset gr_{i+j}\mathbb{V}$ and
\item
the action of $\p$ exponentiates to an action of $P$.
\end{enumerate}
Then $gr\mathbb{V}$ is a decomposition into completely reducible $\g_{0}$-modules and we can look at the associated bundles 
$$V=\mathcal{G}\times_{P}\mathbb{V}\;\text{and}\; gr V=\mathcal{G}_{0}\times_{G_{0}}gr\mathbb{V}=\mathcal{G}_{0}\times_{G_{0}}\mathbb{V}.$$
\par
Every Weyl structure is given by a $G_{0}$-equivariant section $\sigma:\mathcal{G}_{0}\rightarrow\mathcal{G}$. A choice of such a section gives an identification 
$\sigma_{V}:gr V\rightarrow V$ by
$$\sigma_{V}(u,v)=(\sigma(u),v).$$
Since $\sigma$ is $G_{0}$-equivariant, this is well defined. Let $\hat{\sigma}(u)=\sigma(u)\exp(\Upsilon_{1}(u))\cdots\exp(\Upsilon_{k_{0}}(u))$ be a different Weyl structure. Then
$$\hat{\sigma}_{V}(u,v)=(\sigma(u),(\exp(\Upsilon_{1}(u))\cdots\exp(\Upsilon_{k_{0}}(u))).v)$$ 
and hence
$$\hat{\sigma}_{V}^{-1}\circ\sigma_{V}(u,v)=(u,(\exp(-\Upsilon_{k_{0}}(u))\cdots\exp(-\Upsilon_{1}(u))).v).$$
This implies that we can write $v\in V$ as 
$$\left(\begin{array}{c}
v_{0}\\
v_{1}\\
\vdots\\
v_{s}
\end{array}\right)$$
for any choice of Weyl structure. Under change of Weyl structure this changes to
$$\widehat{\left(\begin{array}{c}
v_{0}\\
v_{1}\\
\vdots\\
v_{s}
\end{array}\right)}=\left(\begin{array}{c}
v_{0}\\
v_{1}-\Upsilon_{1}\bullet v_{0}\\
\vdots\\
\sum_{\Vert \underline{j}\Vert+i=s}\frac{(-1)^{\underline{j}}}{\underline{j}!}(\Upsilon_{k_{0}}^{j_{k_{0}}}\circ\cdots\circ\Upsilon_{1}^{j_{1}})\bullet v_{i}
\end{array}\right).$$

\subsubsection{Example 1}
Let us look at an irreducible representation $\mathbb{V}$ of $\p$. As discussed in the last section $gr \mathcal{J}^{1}\mathbb{V}=\mathbb{V}\oplus\g_{-1}^{*}\otimes\mathbb{V}$. According to the above discussion, for every choice of Weyl structure, we can write 
$$\left(\begin{array}{c}
v\\
\varphi
\end{array}\right)$$
for the elements in $\mathcal{J}^{1}V$. Under change of connection, these elements transform as
$$\widehat{\left(\begin{array}{c}
v\\
\varphi
\end{array}\right)}=\left(\begin{array}{c}
v\\
\varphi-\Upsilon_{1}\bullet v
\end{array}\right).$$
The action of $Z\in\g_{1}$ on $gr_{0}\mathcal{J}^{1}\mathbb{V}=\mathbb{V}$ will be determined in the next chapter and can be written as $Z.v=([Z,.]).v\in gr_{1}\mathcal{J}^{1}\mathbb{V}=\g_{-1}^{*}\otimes\mathbb{V}$. Therefore
$$\widehat{\left(\begin{array}{c}
v\\
\varphi
\end{array}\right)}=\left(\begin{array}{c}
v\\
\varphi-(\mathrm{ad}(\Upsilon_{1})(.))\bullet v
\end{array}\right).$$
For $\xi_{-1}\in gr_{-1}T\mathcal{M}$, the transformation law for the Weyl connection is given by
$$\hat{\nabla}_{\xi_{-1}}s=\nabla_{\xi_{-1}}s-(\mathrm{ad}(\Upsilon_{1})(\xi_{-1}))\bullet s.$$
One has to be careful about the abuse of notation here: the actions of $\p$ on the various vector spaces are all denoted by a dot \lq.\rq,~although we refer to different actions. The corresponding actions of $\mathcal{A}^{0}$ are also all denoted by a bullet \lq$\bullet$\rq.~The above implies that the mapping
$$s\mapsto \left(\begin{array}{c}
s\\
\nabla s
\end{array}\right),$$
which is written with respect to a specific Weyl structure (more precisely, both the identification of $\mathcal{J}^{1}V$ with $gr \mathcal{J}^{1}V$ and the connection $\nabla$ are chosen with respect to the same Weyl structure), is in fact independent of the Weyl structure and hence an invariant differential operator
$$\mathcal{O}(\mathcal{G},\mathbb{V})^{P}\rightarrow\mathcal{O}(\mathcal{G},\mathcal{J}^{1}\mathbb{V})^{P}.$$
We will see in the next chapter that the invariant derivative really takes $P$-equivariant sections to $P$-equivariant sections. 

\subsubsection{Example 2}
Let $\mathbb{V}$ be an irreducible $\p$-module, so that the action of $\p$ exponentiates to an action of $P$. The fundamental derivative $D$ can be written as a pairing
$$D:V\times\mathcal{A}\rightarrow V.$$
The grading of the vector bundle $\mathcal{A}$ can be written as $gr \mathcal{A}=\mathcal{A}_{-}\oplus\mathcal{A}_{0}\oplus\mathcal{A}_{+}$ according to the filtration of $\g=\g_{-}+\g_{0}+\g_{+}$ as a $\p$-module. For every choice of Weyl structure we can therefore write sections of $\mathcal{A}$ as a 3-tuple. Under change of Weyl structure, this 3-tuple changes as
$$\widehat{\left(\begin{array}{c}
\xi\\
X_{0}\\
\mu
\end{array}\right)}=\left(\begin{array}{c}
\xi\\
X_{0}+\sum_{\Vert \underline{j}\Vert+i=0}\frac{(-1)^{\underline{j}}}{\underline{j}!}(\mathrm{ad}(\Upsilon_{k_{0}})^{j_{k_{0}}}\circ\cdots\circ\mathrm{ad}(\Upsilon_{1})^{j_{1}})(\xi_{i})\\
*\\
%\sum_{\Vert \underline{j}\Vert+i=k_{0}}\frac{(-1)^{\underline{j}}}{\underline{j}!}(\mathrm{ad}(\Upsilon_{k_{0}})^{j_{k_{0}}}\circ\cdots\circ\mathrm{ad}(\Upsilon_{1})^{j_{1}})(X_{i})\\
\end{array}\right).$$
Now $D$ is given by
\begin{eqnarray*}
\Gamma(V)\times \Gamma(gr\mathcal{A})&\mapsto&\Gamma(V)\\ 
(v,(\xi,X_{0},\mu))&\mapsto& \nabla_{\xi}v-X_{0}\bullet v.
\end{eqnarray*} 
Using equation~(\ref{diffconnection}) to determine how $\nabla$ changes and the description of how the splitting changes shows that this definition is independent of the Weyl structure.
%\subsubsection{Remark}
%In the literature one frequently encounters the term {\bf natural} operator (pairing, etc.~) to describe a family of operators
%$$D_{(\mathcal{G},\omega)}:\Gamma(V)\rightarrow\Gamma(W)$$
%for each parabolic geometry $(\mathcal{G},\omega)$, such that
%$$\Phi^{*}\circ D_{(\mathcal{G}',\omega')}=D_{(\mathcal{G},\omega)}\circ \Phi^{*}$$
%for any morphism $\Phi:(\mathcal{G},\omega)\rightarrow(\mathcal{G}',\omega')$ of parabolic geometries. It can be shown (see, for example,~\cite{cs2}) that natural operators (pairings) on the category of flat parabolic geometries (which are locally isomorphic to the homogeneous model, see~\cite{cs}, proposition 4.12) are exactly the (translation-) invariant operators (pairings) as defined above.
%\par
%On a general parabolic geometry, all natural pairings are invariant in the sense of definition ??.
%\begin{proof}
%Let $\sigma,\hat{\sigma}$ be two Weyl structures and consider the projection $p:\mathcal{G}\rightarrow\mathcal{G}_{0}$.  We can define
%$$\Phi=\sigma\circ p\;:\;\mathcal{G}\rightarrow\mathcal{G}.$$
%Then we have $\Phi\circ \hat{\sigma}=\sigma\circ p \circ \hat{\sigma}=\sigma$. We can define $\omega'=\Phi^{*}\omega$ (IS THIS A CARTAN CONNECTION???), then 
%$$\gamma^{\sigma}=\sigma^{*}\omega_{0}=\hat{\sigma}^{*}\omega'_{0}=\gamma'^{\hat{\sigma}}.$$
%Now naturality implies that the pairing can be computed using $\nabla$ or $\hat{\nabla}$.
%\end{proof}

\subsubsection{Remark}\label{pmodulehom}
Let $\mathbb{V}$ and $\mathbb{W}$ be two filtered $\p$-modules satisfying the conditions in Section~\ref{filtrations} and let $gr \Phi:gr\mathbb{V}\rightarrow gr\mathbb{W}$ be a $\g_{0}$-module homomorphism. We denote the corresponding mapping between vector bundles by $gr\phi$. Then a choice of Weyl structure determines a mapping
$$\phi_{\sigma}=\sigma_{W}\circ gr \phi\circ \sigma_{V}^{-1}: V\rightarrow W.$$
It immediately follows that $\phi$ is independent of the Weyl structure (i.e.~$\phi_{\hat{\sigma}}=\phi_{\sigma}$) if and only if $gr \Phi$ is a $\p$-module homomorphism.

%A Weyl structure can equivalently (see~\cite{cds}, proposition A.1) be given by an {\bf algebraic Weyl structure} which is a series of splittings of the exact sequences
%$$0\rightarrow\g^{i+1}\rightarrow\g^{i}\rightarrow\g_{i}\rightarrow 0,$$
%for $i=-k_{0},...,k_{0}$. An algebraic Weyl structure gives an isomorphism $\epsilon_{V}:\mathbb{V}\rightarrow gr\mathbb{V}$ for every filtered $P$-module $\mathbb{V}$. Given a $G_{0}$-module homomorphism $\tilde{\Phi}:gr\mathbb{V}\rightarrow gr\mathbb{W}$ for two filtered $P$-modules $\mathbb{V}$ and $\mathbb{W}$, we can form
%$$\Phi_{\epsilon}=\epsilon_{\mathbb{W}}^{-1}\circ\tilde{\Phi}\circ\epsilon_{\mathbb{V}}:\mathbb{V}\rightarrow\mathbb{W}.$$
%$\Phi_{\epsilon}$ is independent of the algebraic Weyl structure if and only if it is a $P$-module homomorphism, see~\cite{cds}, 5.1.

\begin{lemma}\label{five}
\begin{enumerate}
\item
If there is a $P$-module map $\Phi:\mathbb{V}\rightarrow\mathbb{W}$, then the corresponding map
\begin{eqnarray*}
\mathcal{O}(\mathcal{G},\mathbb{V})^{P}&\rightarrow&\mathcal{O}(\mathcal{G},\mathbb{W})^{P}\\
s&\mapsto &\Phi\circ s
\end{eqnarray*}
is invariant.
\item
The invariant differential
\begin{eqnarray*}
\mathcal{O}(\mathcal{G},\mathbb{V})^{P}&\rightarrow &\mathcal{O}(\mathcal{G},\mathcal{J}^{1}\mathbb{V})^{P}\\
s&\mapsto& (s,\nabla^{\omega} s)
\end{eqnarray*}
and the iterated fundamental derivative
\begin{eqnarray*}
\mathcal{O}(\mathcal{G},\mathbb{V})^{P}&\rightarrow &\mathcal{O}(\mathcal{G},\otimes^{k}\g^{*}\otimes\mathbb{V})^{P}\\
s&\mapsto& D^{k}s
\end{eqnarray*}
are invariant.
\end{enumerate}
\end{lemma}
\begin{proof}
The first statement follows by taking $N=0$ in~\ref{semiinvariant} and making use of Remark~\ref{pmodulehom}. 
%The naturality of the fundamental derivative and the invariant derivative follows from the fact that both operators are defined with the help of the Cartan connection of the parabolic geometry and every morphism of parabolic geometries has to respect this structure. For more details the reader is advised to refer to~\cite{cg}, proposition 3.1 and~\cite{css2}. 
%\par
%Alternatively
%\par
%The invariance of the invariant derivative and the fundamental derivative follows from the fact that they are constructed by using the Cartan connection $\omega$, which is intrinsic to the parabolic geometry and they do not make any choice of specific Weyl structure. 
%For more details on the fundamental derivative, the reader is advised to refer to~\cite{cg}, proposition 3.1.
%More precisely, for each choice of Weyl connection $\nabla$, they can be written as
The invariance of the invariant derivative and the fundamental derivative was proved in Example 1 and 2 for irreducible $\p$-modules. The general case follows from the fact that they are constructed by using the Cartan connection $\omega$, which is intrinsic to the parabolic geometry, and they do not make any choice of specific Weyl structure. 
In order to obtain explicit formulae like in the examples, one can write down the corresponding formulae for each $\g_{0}$ irrreducible component and keep in mind that the individual components are not invariant, just the expression as a whole. We will  see many examples of this subtle point in Chapter~\ref{tractorchapter}.
Finally note that the symbol of $D^{k}$ is given by
%The corresponding mappings are given by
%\begin{eqnarray*}
%V &\stackrel{\nabla^{1}}{\longrightarrow}&\Lambda^{1}\otimes V\stackrel{\phi}{\longrightarrow}J^{1}V\\
%&\text{with}& \Phi:\g_{-1}^{*}\otimes\mathbb{V}\rightarrow\mathbb{V}\oplus\g_{-1}^{*}\otimes\mathbb{V}
%\end{eqnarray*}
%and
$$\begin{array}{cc}
                &\otimes^{k}\Lambda^{1}\otimes V\stackrel{\phi}{\longrightarrow}\otimes^{k}\mathcal{A}\otimes V\\
\text{with}& \Phi:\otimes^{k}(\g/\p)^{*}\otimes\mathbb{V}\hookrightarrow\otimes^{k}\g^{*}\otimes\mathbb{V}.
\end{array}$$
\end{proof}

%Let $\mathbb{V}$ be any filtered $\p$-module satisfying the conditions in ??. For every choice of Weyl structure, let us look at the mapping
%$$\left(\begin{array}{c}
%v_{0}\\
%\ddots\\
%v_{s}
%\end{array}\right)\mapsto \left(\begin{array}{c}
%\nabla v_{0}\\
%\ddots\\
%\nabla v_{s}
%\end{array}\right),$$
%where we have used the Weyl structure to define an identification of $V$ and $gr V$ and we have used the Weyl structure to define the connection $\nabla$. We can assume that each $v_{i}$ is a section of $V_{i}$, where $\mathbb{V}_{i}$ is an irreducible representation of $\g_{0}$ that lies in some $gr_{j}\mathbb{V}$.
%Let $\mathbb{V}$ be an irreducible representation of $\p$. Then the fundamental derivative can be written as
%$$V\ni s\mapsto\left(\begin{array}{c}
%(.)\bullet s\\
%\\
%\nabla s
%\end{array}
%\right)\in\begin{array}{c}
%(\mathcal{A}^{0})^{*}\otimes V\\
%\oplus\\
%(T\mathcal{M})^{*}\otimes V
%\end{array}=gr \mathcal{A}^{*}\otimes V.$$

\subsubsection{Remark}
\begin{enumerate}
\item
All the pairings and operators to be defined in this thesis are combinations of the invariant operations defined above and are hence invariant the sense of Definition~\ref{definvariance}.
\item
In the literature (\cite{cg,css,css2}) it is quite common to use the term {\bf natural} to describe the operators in Lemma~\ref{five}. Those natural operators are defined to be  systems of operators $D_{(\mathcal{G},\omega)}$ for a certain category of parabolic geometries that behave well with respect to morphisms of that category. It can be shown (see, for example,~\cite{cs2}) that natural operators (pairings) on the category of flat parabolic geometries (which are locally isomorphic to the homogeneous model, see~\cite{cs}, Proposition 4.12) are exactly the (translation-) invariant operators (pairings) as defined in Section~\ref{algebraicdescription}. 
\end{enumerate}

%\subsubsection{Remark}
%We will encounter invariant differential operators (pairings) in two different situations:
%\item
%In the form of {\bf strongly invariant} first order operators (pairings) $\Gamma(V)\rightarrow \Gamma(W)$ (\Gamma(V)\times \Gamma(W)\rightarrow \Gamma(E)$), that are induces by $P$-module homomorphisms $J^{1}\mathbb{V}\rightarrow \mathbb{W}$ ($\mathcal{J}^{1}(\mathbb{V},\mathbb{W})\rightarrow\mathbb{E}$). Combined with the invariant derivative $\Gamma(V)\ni s\mapsto (s,\nabla^{\omega}s)\in \Gamma(J^{1}V)$ ($\Gamma(V)\times\Gamma(W)\ni (s,t)\mapsto (s\otimes t, \nabla^{\omega}s\otimes t, s\otimes\nabla^{\omega}t)\in\Gamma(J^{1}(V,W))$, these operators (pairings) are independent of the choice of Weyl structure, see~\cite{cds}, 5.1.
%\item
%\end{enumerate}

\subsection{Geometric structures}
The description of certain geometric structures on manifolds as parabolic geometries is a complicated issue and it is not trivial to show that various familiar structures (projective, conformal, CR) are equivalent to parabolic geometries of a certain type. The paper~\cite{cs} is dedicated to this problem and we will state (following~\cite{cs2}) the upshot of these considerations without going into too much detail.

\begin{definition}
{\rm
\begin{enumerate}
\item
An {\bf infinitesimal flag structure} of type $(\g,\p)$ on a smooth manifold $\mathcal{M}$ is given by a filtration
$$T\mathcal{M}=T^{-k_{0}}\mathcal{M}\supset\cdots\supset T^{-1}\mathcal{M},$$
such that the rank of $T^{i}\mathcal{M}$ equals the dimension of $\g^{i}/\p$, and a reduction of the associated graded vector bundle
$$gr T\mathcal{M}=gr_{-k_{0}}(T\mathcal{M})\oplus\cdots\oplus gr_{-1}(T\mathcal{M}),$$
with $gr_{i}(T\mathcal{M})=T^{i}\mathcal{M}/T^{i+1}\mathcal{M}$, to the structure group $G_{0}$. Since we can take $\g_{-}=\g_{-k_{0}}\oplus\cdots\oplus\g_{-1}$ as the modeling vector space for $gr(T\mathcal{M})$, the reduction is defined via $\mathrm{Ad}:G_{0}\rightarrow \mathrm{GL}_{gr}(\g_{-})$.
\item
Let 
$$T\mathcal{M}=T^{-k_{0}}\mathcal{M}\supset\cdots\supset T^{-1}\mathcal{M}$$
be an infinitesimal flag structure of type $(\g,\p)$ that makes $\mathcal{M}$ into a filtered manifold. 
There are two ways of defining a bracket $gr(T\mathcal{M})\times  gr(T\mathcal{M})\rightarrow gr(T\mathcal{M})$: firstly, we can use the reduction of $grT\mathcal{M}$ to the structure group $G_{0}$ and the (algebraic) Lie bracket on $\g_{-}$. Secondly, we can use the fact that $\mathcal{M}$ has a tangential filtration to define the {\bf Levi-bracket} that is induced by the usual bracket of vector fields. The infinitesimal flag structure is called {\bf regular}, if those two brackets coincide. 
\end{enumerate}
}\end{definition}
\par
Let us assume that no simple factor of $\g$ lies in $\g_{0}$ and that $\g$ does not contain any simple factors of type $A_{1}$. Then the following theorem can be proved.

\begin{theorem}[\cite{cs}]
\begin{enumerate}
\item
If $(\g,\p)$ does not contain any simple factor of the form 
$$\begin{picture}(12,2)
\put(0,1){\makebox(0,0){$\cross$}}
\put(1,1){\makebox(0,0){$\bullet$}}
\put(2,1){\makebox(0,0){$\bullet$}}
\put(1,1){\line(1,0){1}}
\put(3.5,1){\makebox(0,0){$\bullet$}}
\put(2.5,1){...}
\put(4.5,1){\makebox(0,0){$\bullet$}}
\put(0,1){\line(1,0){1}}
\put(3.5,1){\line(1,0){1}}
\put(5.3,1){$\text{or}$}
\put(7,1){\makebox(0,0){$\cross$}}
\put(8,1){\makebox(0,0){$\bullet$}}
\put(9,1){\makebox(0,0){$\bullet$}}
\put(8,1){\line(1,0){1}}
\put(10.5,1){\makebox(0,0){$\bullet$}}
\put(9.5,1){...}
\put(10.5,1.1){\line(1,0){1}}
\put(11.5,1){\makebox(0,0){$\bullet$}}
\put(7,1){\line(1,0){1}}
\put(10.5,0.9){\line(1,0){1}}
\put(10.8,0.8){$\langle$}
\put(12,1){,}
\end{picture}$$
then there is an equivalence of categories between regular parabolic geometries and regular infinitesimal flag structures.
\item
If $(\g,\p)$ contains a simple factor of the form given above, then there is a bijective correspondence between regular parabolic geometries and underlying $P$-frame bundles of degree two (reductions of the second order frame bundle).
\end{enumerate}
%Let $\mathcal{M}$ be a manifold with a tangential filtration and a reduction of the associated graded vector bundle $gr T\mathcal{M}$ to the structure group $G_{0}$. If the structure equation holds and $H^{1}_{k}(\g_{-},\g)$ vanishes for all $k>0$, then (up to isomorphisms) there is a unique Cartan connection $\omega$ on a unique principal $P$-bundle $\mathcal{G}\rightarrow \mathcal{M}$ such that is curvature is $\partial^{*}$-closed. 
\end{theorem}

\subsubsection{Remarks}
\begin{enumerate}
\item
Infinitesimal flag structures can equivalently (see~\cite{cs2}) be described in terms of {\bf frame forms of length 1} in the  sense of~\cite{cs}, Definition 3.2. From this description the theorem above is proved by a {\bf prolongation procedure}.
\item
The condition that the algebraic bracket is equivalent to the Levi-bracket can equally be stated in terms of structure equations as in~\cite{cs}.
\item
The conditions in the theorem above are equivalent to demanding that certain (Lie-algebra-) cohomology groups $H^{1}_{l}(\g_{-},\g)$ vanish (see~\cite{cs,y}).
\item
The second excluded case in the theorem above corresponds to so-called contact projective structures. A comprehensive treatment of contact projective structures can be found in~\cite{f2}.
\end{enumerate}

\subsubsection{Example}
If the grading of $\g$ is of length one, then the filtration of the tangent bundle is trivial and an infinitesimal flag structure is equivalent to a reduction of $T\mathcal{M}$ to the structure group $G_{0}$.
\begin{enumerate}
\item
For the choice of $G$ and $P$ as in Example~\ref{cartanexample} (b) and~\ref{homogeneousexample} (b), we obtain $G_{0}=CO(p,q)$, so an infinitesimal  flag structure is equivalent to the choice of a conformal class of metrics (with signature $(p,q)$). For $n=4$, for example, a different choice of $G$ leads to $G_{0}=S(\mathrm{GL}(2,\C)\times\mathrm{GL}(2,\C))$, which is a $4-1$ covering of $CO(4,\C)$. Reductions of the structure group of $T\mathcal{M}$ to this group correspond to {\bf spin structures}, see~\cite{er}.
\item 
For the choice of $G$ and $P$ as in Example~\ref{cartanexample} (a) and~\ref{homogeneousexample} (a), we obtain $G_{0}=\mathrm{GL}(n,\R)$, so it is obvious that an infinitesimal flag structure does not carry any information at all. 
Rather, we have to look at the underlying $P$-frame bundle of degree two. This is the first excluded case in the theorem above. 
\end{enumerate}

\chapter{The first order case}\label{thefirstordercase}
This chapter describes a classification of all (non-degenerate) weighted first order invariant bilinear differential pairings on a general regular curved parabolic geometry $(\mathcal{M},\mathcal{G},\g,\omega)$ of type $(G,P)$ that we consider fixed throughout the rest of this thesis. Furthermore we characterize degenerate pairings via the existence of invariant linear differential operators.

\section{The obstruction term}
\subsection{The possible candidates}
In the following we will fix two finite dimensional irreducible representations given by $\tilde{\lambda}:\p\rightarrow\mathfrak{gl}(\mathbb{V})$ and $\tilde{\nu}:\p\rightarrow\mathfrak{gl}(\mathbb{W})$.
Moreover denote the highest weights of $\mathbb{V}^{*}$ and $\mathbb{W}^{*}$ by $\lambda\in \h^{*}$ and $\nu\in\h^{*}$ respectively. 
Sometimes we will want to include this additional information about a representation in the notation: for any irreducible representation $\mathbb{E}$ of $\p$ we will write $\mathbb{E}_{\mu}$ to record that $\mathbb{E}_{\mu}^{*}$ has highest weight $\mu\in\h^{*}$, 
i.e.~$\mathbb{V}=\mathbb{V}_{\lambda}$ and $\mathbb{W}=\mathbb{W}_{\nu}$. 
\par
There are two exact sequences associated to the first  weighted jet bundles of $V$ and $W$:
$$0\rightarrow (\mathfrak{U}_{-1}(grT\mathcal{M}))^{*}\otimes V\rightarrow \mathcal{J}^{1}V \rightarrow V\rightarrow 0$$
and 
$$0\rightarrow (\mathfrak{U}_{-1}(grT\mathcal{M}))^{*}\otimes W\rightarrow \mathcal{J}^{1}W\rightarrow W\rightarrow 0,$$
which are the \emph{weighted-jet exact sequences} as described in the last chapter. All these bundles are associated bundles, so on the level of $\p$ representations we have
$$0\rightarrow \g_{1}\otimes \mathbb{V}\rightarrow \mathcal{J}^{1}\mathbb{V} \rightarrow \mathbb{V}\rightarrow 0$$
and 
$$0\rightarrow \g_{1}\otimes \mathbb{W}\rightarrow \mathcal{J}^{1}\mathbb{W}\rightarrow \mathbb{W}\rightarrow 0.$$
In other words, there are two filtered modules
$$\mathcal{J}^{1}\mathbb{V}=\mathbb{V}+\g_{1}\otimes \mathbb{V}$$
and
$$\mathcal{J}^{1}\mathbb{W}=\mathbb{W}+\g_{1}\otimes \mathbb{W}.$$
Therefore the tensor product has a filtration
$$\mathcal{J}^{1}\mathbb{V}\otimes \mathcal{J}^{1}\mathbb{W}=\mathbb{V}\otimes \mathbb{W}+\begin{array}{c}
\g_{1}\otimes \mathbb{V}\otimes \mathbb{W}\\
\oplus\\
\mathbb{V}\otimes\g_{1}\otimes \mathbb{W}
\end{array}
+\g_{1}\otimes \mathbb{V}\otimes\g_{1}\otimes \mathbb{W}.$$
The $\p$-module structure of
%\begin{eqnarray*}
$$\mathcal{J}^{1}(\mathbb{V},\mathbb{W})=\mathcal{J}^{1}\mathbb{V}\otimes \mathcal{J}^{1}\mathbb{W}/(\g_{1}\otimes \mathbb{V}\otimes\g_{1}\otimes \mathbb{W})
=\mathbb{V}\otimes\mathbb{W}+\begin{array}{c}
\mathbb{V}\otimes\g_{1}\otimes\mathbb{W}\\
\oplus\\
\g_{1}\otimes\mathbb{V}\otimes\mathbb{W}\end{array}$$
%\end{eqnarray*}
is such that
$$\mathcal{O}(\mathcal{G},\mathbb{V})^{P}\otimes \mathcal{O}(\mathcal{G},\mathbb{W})^{P}\ni(s,t)\mapsto
(s\otimes t,s\otimes \nabla^{\omega}t,\nabla^{\omega}s\otimes t)\in\mathcal{O}(\mathcal{G},\mathcal{J}^{1}(\mathbb{V},\mathbb{W}))^{P}$$
is well defined, i.e.~maps $P$-equivariant sections to a $P$-equivariant section. In order to see this, we introduce dual (with respect to the Killing form) linear basis $\{\xi_{\alpha}\}_{\alpha=1,...,n}$ and $\{\eta^{\alpha}\}_{\alpha=1,...,n}$ of $\g_{-}$ and $\p_{+}$ respectively. Since $\g_{-1}^{*}\cong\g_{1}$, we can restrict those basis to basis $\{\xi_{\alpha'}\}_{\alpha'=1,...,n'}$ and $\{\eta^{\alpha'}\}_{\alpha'=1,...,n'}$ of $\g_{-1}$ and $\g_{1}$ respectively.
Then the following lemma holds.

\begin{lemma}\label{lemmasix}
Let 
$$p:\mathcal{J}^{1}\mathbb{V}\otimes\mathcal{J}^{1}\mathbb{W}\rightarrow\mathcal{J}^{1}(\mathbb{V},\mathbb{W})$$
be the canonical projection, let 
$Z\in\p$ and let $(v_{0},X\otimes v)\otimes (w_{0}, Y\otimes w)\in \mathcal{J}^{1}\mathbb{V}\otimes\mathcal{J}^{1}\mathbb{W}$. Furthermore define the action of
$Z$ on $\mathcal{J}^{1}(\mathbb{V},\mathbb{W})$ by
\begin{eqnarray*}
&j^{1}(\tilde{\lambda},\tilde{\nu})(Z)p((v_{0},X\otimes v)\otimes (w_{0}, Y\otimes w))=\\
&\left(\begin{array}{c}
\tilde{\lambda}(Z)v_{0}\otimes w_{0}+v_{0}\otimes\tilde{\nu}(Z)w_{0}\\
\tilde{\lambda}(Z)v_{0}\otimes Y\otimes w+v_{0}\otimes\left(Y\otimes \tilde{\nu}(Z)w+[Z,Y]_{\g_{1}}\otimes w+\sum_{\alpha'}\eta^{\alpha'}\otimes\tilde{\nu}([Z,\xi_{\alpha'}]_{\p})w_{0}\right)\\
X\otimes v \otimes \tilde{\nu}(Z)w_{0}+\left( X\otimes \tilde{\lambda}(Z)v+[Z,X]_{\g_{1}}\otimes v+\sum_{\alpha'}\eta^{\alpha'}\otimes\tilde{\lambda}([Z,\xi_{\alpha'}]_{\p})v_{0}  \right)\otimes w_{0}
\end{array}\right),&
\end{eqnarray*}
where $[.,.]_{\mathfrak{a}}$ denotes the bracket in $\g$ followed by the projection onto a subspace $\mathfrak{a}$ of $\g$. Then the mapping
 $$\mathcal{O}(\mathcal{G},\mathbb{V})^{P}\otimes \mathcal{O}(\mathcal{G},\mathbb{W})^{P}\ni(s,t)\mapsto
(s\otimes t,s\otimes \nabla^{\omega}t,\nabla^{\omega}s\otimes t)\in\mathcal{O}(\mathcal{G},\mathcal{J}^{1}(\mathbb{V},\mathbb{W}))^{P}$$
is well defined, i.e.~maps $P$-equivariant sections to a $P$-equivariant section and therefore defines an isomorphism
$$\mathcal{J}^{1}(V,W)\cong\mathcal{G}\times_{P}\mathcal{J}^{1}(\mathbb{V},\mathbb{W}).$$
\end{lemma}
\begin{proof}
We have the following canonical projections:
$$J^{1}\mathbb{V}\rightarrow\mathcal{J}^{1}\mathbb{V},\;J^{1}\mathbb{W}\rightarrow\mathcal{J}^{1}\mathbb{W}$$
and
$$\mathcal{J}^{1}\mathbb{V}\otimes\mathcal{J}^{1}\mathbb{W}\rightarrow \mathcal{J}^{1}(\mathbb{V},\mathbb{W}).$$
It is therefore sufficient to define the $\p$-module structure of $J^{1}\mathbb{U}$ for an arbitrary representation $\mathbb{U}$ of $\p$
and to check that $s\mapsto (s,\nabla^{\omega}s)$ is a well defined mapping 
$$\mathcal{O}(\mathcal{G},\mathbb{U})^{P}\rightarrow \mathcal{O}(\mathcal{G},J^{1}\mathbb{U})^{P}$$
that takes $P$ equivariant sections to $P$ equivariant sections.
% Exactly the same calculation shows (see~\cite{slov}, 3.1) that $J^{1}\rho$ is also the representation that induces the homogeneous bundle $J^{1}U$ on the homogeneous model space $G/P$. 
In other words, that the mapping $s\mapsto (s,\nabla^{\omega}s)$ defines an isomorphism $J^{1}U\cong\mathcal{G}\times_{P}J^{1}\mathbb{U}$.
\par
So let $\rho:\p\rightarrow\mathfrak{gl}(\mathbb{U})$ be a representation of $\p$ and denote the representation of $\p$ on $J^{1}\mathbb{U}$ by $j^{1}\rho$. First of all note that
$$-\zeta_{Z}s=-\omega^{-1}(Z)s=-\nabla^{\omega}_{Z}s=\rho(Z)s$$
for every $Z\in\p$ and $s\in\mathcal{O}(\mathcal{G},\mathbb{U})^{P}$. This shows us exactly how to define the representation $j^{1}\rho$.
\begin{eqnarray*}
-\zeta_{Z}(s,\nabla^{\omega}_{X}s)&=&(-\nabla^{\omega}_{Z}s,-\nabla^{\omega}_{Z}\nabla^{\omega}_{X}s)\\
&=&(\rho(Z)s,-\nabla^{\omega}_{X}\nabla^{\omega}_{Z}s-\nabla^{\omega}_{[Z,X]}s+\nabla^{\omega}_{\kappa(Z,X})\\
&=&(\rho(Z)s,\rho(Z)\nabla_{X}^{\omega}s+\rho([Z,X]_{\p})s-\nabla^{\omega}_{[Z,X]_{\g_{-}}}s)\\
&\stackrel{!}{=}&j^{1}\rho(Z)(s,\nabla^{\omega}_{X}s),
\end{eqnarray*}
for all $Z\in\p$ and $X\in\g_{-}$. Note that we have used the fact that the curvature of any Cartan connection is horizontal (see Remark~\ref{cartanconnection}). 
%For a more detailed exposition of how to determine the $\p$-module structure of $J^{1}\mathbb{U}$ see~\cite{slov}, 3.1.
This determines the action $j^{1}\rho$ on an arbitrary element $(u,\varphi)\in J^{1}\mathbb{U}=\mathbb{U}\oplus\mathrm{Hom}(\g_{-},\mathbb{U})$ via
$$J^{1}\rho(Z)(u,\varphi)=(\rho(Z)u,\rho(Z)\circ\varphi+\rho(\mathrm{ad}_{\p}(Z)(.))v-\varphi\circ\mathrm{ad}_{\g_{-}}(Z)).$$
Using the isomorphism $\mathrm{Hom}(\g_{-},\mathbb{U})\cong\p_{+}\otimes\mathbb{U}$, we can write elements in $J^{1}\mathbb{U}=\mathbb{U}\oplus\p_{+}\otimes\mathbb{U}$
as linear combinations of simple elements of the form $(u_{0},Y\otimes u)$. Then the action of $Z\in\p$ is given by:
$$j^{1}\rho(Z)(u_{0},Y\otimes u)=(\rho(Z)u_{0},Y\otimes \rho(Z)u+[Z,Y]\otimes u+\sum_{\alpha}\eta^{\alpha}\otimes\rho([Z,\xi_{\alpha}]_{\p})u_{0}).$$
This follows from
$$(Y\otimes u)(\mathrm{ad}_{\g_{-}}(Z)(.))=B(\mathrm{ad}_{\g_{-}}(Z)(.),Y)u=-B(.,[Z,Y])u=-([Z,Y]\otimes u)$$
as a map $\g_{-}\rightarrow \mathbb{U}$, where we have used that $B(\g_{i},\g_{j})=0\;\forall\;i\not=-j$,
and from
$$\rho(\mathrm{ad}_{\p}(Z)(.))u=\sum_{\alpha}\eta^{\alpha}\otimes\rho([Z,\xi_{\alpha}]_{\p})u.$$
%The representation $J^{1}\rho$ of $\p$ on $J^{1}\mathbb{U}$ defined like this ensures that $s\mapsto (s,\nabla^{\omega}s)$ is a well defined mapping 
%$$\mathcal{O}(\mathcal{G},\mathbb{U})^{P}\rightarrow \mathcal{O}(\mathcal{G},J^{1}\mathbb{U})^{P}$$
%that takes $P$ equivariant sections to $P$ equivariant sections. Exactly the same calculation shows (see~\cite{slov}, 3.1) that $J^{1}\rho$ is also the representation that induces the homogeneous bundle $J^{1}U$ on the homogeneous model space $G/P$. In other words, the mapping $s\mapsto (s,\nabla^{\omega}s)$ defines an isomorphism $J^{1}U\cong\mathcal{G}\times_{P}J^{1}\mathbb{U}$.
%\par
It can now be easily checked that $j^{1}(\tilde{\lambda},\tilde{\nu})$ is exactly the induced representation under the projections and tensor products given above. 
\end{proof}

\begin{corollary}\label{corthree}
$\g_{0}$ acts tensorally on $\mathcal{J}^{1}(\mathbb{V},\mathbb{W})$.
If $Z\in\g_{2}\oplus...\oplus\g_{k}$, then 
$$j^{1}(\tilde{\lambda},\tilde{\nu})(Z)p((v_{0},X\otimes v)\otimes (w_{0}, Y\otimes w))=0.$$
If $Z\in\g_{1}$, then
$$j^{1}(\tilde{\lambda},\tilde{\nu})(Z)p((v_{0},X\otimes v)\otimes (w_{0}, Y\otimes w))=\left(\begin{array}{c}
0\\
v_{0}\otimes\sum_{\alpha'}\eta^{\alpha'}\otimes\tilde{\nu}([Z,\xi_{\alpha'}])w_{0}\\
\sum_{\alpha'}\eta^{\alpha'}\otimes\tilde{\lambda}([Z,\xi_{\alpha'}])v_{0} \otimes w_{0}
\end{array}\right).$$
We will call this term the {\bf obstruction term}. 
\end{corollary}
\begin{proof}
Using $[\g_{i},\g_{j}]\subset\g_{i+j}$ for all $i,j$, these considerations follow easily from the fact that $\p_{+}$ acts trivially on $\mathbb{W}$ and $\mathbb{V}$.
\end{proof}

The reason for this setup is given in the following lemma.

\begin{lemma}
Weighted first order bilinear invariant differential pairings
$$\Gamma(V)\times\Gamma(W)\rightarrow \Gamma(E)$$
in the flat homogeneous case $G/P$ are in one-to-one correspondence with $\p$-module homomorphisms
$$\mathcal{J}^{1}(\mathbb{V},\mathbb{W})\rightarrow \mathbb{E}.$$
In the general curved case, these homomorphisms yield (modulo scalars or curvature correction terms) all weighted first order invariant bilinear differential pairings.
\end{lemma}\begin{proof}
%A weighted first order bilinear differential pairing $\Gamma(V)\times\Gamma(W)\rightarrow \Gamma(E)$ corresponds to a homomorphism  $\mathcal{J}^{1}(V,W)\rightarrow E$ (by definition of the weighted jet bundles and weighted first order pairings).
%In the flat homogeneous case $G/P$, a pairing is called invariant if it commutes with the action of $G$, hence those pairings are uniquely determined by their action at the identity coset $P$, where they induce $\p$-module homomorphisms
%$$\mathcal{J}^{1}(\mathbb{V},\mathbb{W})\rightarrow \mathbb{E}.$$
%Conversely, every such $\p$-module homomorphism induces a weighted first order invariant differential pairing. 
The first  part of the lemma concerning pairings on homogeneous spaces was discussed in~\ref{thehomogeneouscase}.
In the general curved case, as mentioned above, the $\p$-module structure of $\mathcal{J}^{1}(\mathbb{V},\mathbb{W})$ ensures that the mapping
$$\mathcal{O}(\mathcal{G},\mathbb{V})^{P}\otimes \mathcal{O}(\mathcal{G},\mathbb{W})^{P}\ni(s,t)\mapsto
(s\otimes t,s\otimes \nabla^{\omega}t,\nabla^{\omega}s\otimes t)\in\mathcal{O}(\mathcal{G},\mathcal{J}^{1}(\mathbb{V},\mathbb{W}))^{P}$$
is well defined. The composition of this mapping with $\p$-module homomorphisms $\mathcal{J}^{1}(\mathbb{V},\mathbb{W})\rightarrow \mathbb{E}$ yield, following~\cite{css2}, {\bf strongly invariant} differential pairings of weighted order one. These are independent of the choice of Weyl structure, see~\cite{cds}, Section 5.1, and Lemma~\ref{five}. An arbitrary invariant differential pairing of weighted order one can be restricted to the flat model, where it is given by the above construction. The original pairing on a general manifold with parabolic structure can then differ only by scalars or curvature correction terms from the strongly invariant operator. 
\end{proof}\par
\vspace{0.3cm}

Looking at the exact sequence of $\p$-modules
$$0\rightarrow\begin{array}{c}
\g_{1}\otimes \mathbb{V}\otimes \mathbb{W}\\
\oplus\\
\mathbb{V}\otimes\g_{1}\otimes \mathbb{W}
\end{array}\rightarrow \mathcal{J}^{1}(\mathbb{V},\mathbb{W})\rightarrow  \mathbb{V}\otimes\mathbb{W}\rightarrow 0,$$
it is clear that a $\p$-module homomorphism $\mathcal{J}^{1}(\mathbb{V},\mathbb{W})\rightarrow\mathbb{E}$ onto an irreducible $\p$-module $\mathbb{E}$ induces a 
$\g_{0}$-module homomorphism
$$\begin{array}{c}
\g_{1}\otimes \mathbb{V}\otimes \mathbb{W}\\
\oplus\\
\mathbb{V}\otimes\g_{1}\otimes \mathbb{W}
\end{array}\stackrel{\pi}{\rightarrow} \mathbb{E}$$
and so the only candidates for $\mathbb{E}$ are the irreducible components of
$\g_{1}\otimes\mathbb{V}\otimes\mathbb{W}$ viewed as $\g_{0}$-modules (or as $\g_{0}^{S}$ modules, since $\mathfrak{z}(\g_{0})$ acts by a character).
However, not every projection $\pi$ is a $\p$-module homomorphism. In order to determine which $\pi$ are allowed, we use Corollary~\ref{corthree} to  note that the action of $\g_{0}$ on $\mathcal{J}^{1}(\mathbb{V},\mathbb{W})$ is just the tensorial one, so $\mathcal{J}^{1}(\mathbb{V},\mathbb{W})$ can be split as a $\g_{0}$-module. But $\p_{+}$ does not act trivially as on any irreducible $\p$-module, so in order to check that a specific projection is indeed a $\p$-module homomorphism and not just
a $\g_{0}$-module homomorphism, the image of the action of $\p_{+}$, when acting in $\mathcal{J}^{1}(\mathbb{V},\mathbb{W})$, has to vanish under~$\pi$.  This is exactly the obstruction term from above. On the other hand this is obviously sufficient for $\pi$ to be a $\p$-module homomorphism.
\par
\vspace{0.2cm} 
%Following~\cite{ss}, we will in the sequel use a normalization $(.,)$ of the Killing form $B(.,.)$ given by the requirement $(E,E)=1$, where $E$ is the grading element. This will make the treatment coherent with the AHS case considered in $~\cite{k}$. 

\subsection{Casimir computations}

\begin{lemma}\label{lemmaeight}
Let $\mathbb{V}$ be a finite dimensional irreducible representation of $\g_{0}$ and let $\lambda$ be the highest weight of $\mathbb{V}^{*}$. Moreover let $\{Y_{a}\}$, $\{Y^{a}\}$ be dual basis of $\g_{0}$ with respect to $B(.,.)$ and denote by $\rho_{0}=\frac{1}{2}\sum_{\alpha\in\Delta^{+}(\g_{0})}$ half the sum over all positive roots in $\g_{0}$. Then the {\bf Casimir operator} $c=\sum_{a}Y_{a}Y^{a}$ acts by
$$c(\lambda)=B(\lambda,\lambda+2\rho_{0})$$
on $\mathbb{V}$.
\end{lemma}
\begin{proof}
It is well known that the lowest weight of $\mathbb{V}$ is $-\lambda$. Now let $v_{-\lambda}\in\mathbb{V}$ be a lowest weight vector and take 
$$c=\sum_{i}h_{i}h^{i}+\sum_{\alpha\in\Delta^{+}(\g_{0})}(x_{\alpha}x_{-\alpha}+x_{-\alpha}x_{\alpha}),$$
where $\{h_{i}\}$, $\{h^{i}\}$ are dual basis of $\h$ and $x_{\alpha}\in\g_{\alpha}$, $x_{-\alpha}\in\g_{-\alpha}$ are dual with respect to $B(.,.)$ for all $\alpha\in\Delta^{+}(\g_{0})$. Since $v_{-\lambda}$ is a lowest weight vector, it is killed by all $x_{-\alpha}$, so the action of the second term is given by
\begin{eqnarray*}
\sum_{\alpha\in\Delta^{+}(\g_{0})}(x_{\alpha}x_{-\alpha}+x_{-\alpha}x_{\alpha}).v_{-\lambda}&=&\sum_{\alpha\in\Delta^{+}(\g_{0})}x_{-\alpha}x_{\alpha}.v_{-\lambda}\\
&=&\sum_{\alpha\in\Delta^{+}(\g_{0})}-[x_{\alpha},x_{-\alpha}].v_{-\lambda}\\
&=&\sum_{\alpha\in\Delta^{+}(\g_{0})}-h_{\alpha}.v_{-\lambda}\\
&=&\sum_{\alpha\in\Delta^{+}(\g_{0})}\lambda(h_{\alpha})v_{-\lambda}\\
&=&\sum_{\alpha\in\Delta^{+}(\g_{0})}B(h_{\lambda},h_{\alpha})v_{-\lambda}\\
&=&\sum_{\alpha\in\Delta^{+}(\g_{0})}B(\lambda,\alpha)v_{-\lambda}\\
&=&B(\lambda,2\rho_{0})v_{-\lambda}.
\end{eqnarray*}
Here, for every $\mu\in\h^{*}$, $h_{\mu}\in\h$ is the unique element with $\mu(H)=B(h_{\mu},H)$ for all $H\in\h$ and $[x_{\alpha},x_{-\alpha}]=B(x_{\alpha},x_{-\alpha})h_{\alpha}=h_{\alpha}$ as in~\cite{h}, Proposition 8.3.
The first term obviously yields $\sum_{i}\lambda(h_{i})\lambda(h^{i})v_{-\lambda}=B(\lambda,\lambda)v_{-\lambda}$, see~\cite{h}, 22.3. By the Schur lemma, we deduce that the action of $c$ on the whole representation $\mathbb{V}$ is given by $B(\lambda,\lambda+2\rho_{0})$. Finally we note that $c$ is independent of the choice of basis.
\end{proof}
%Let $\tilde{\lambda}$ be the highest weigh of $\mathbb{V}$, then it is well known (~\cite{h}, 22.1-22.3 and ~\cite{ss}, lemma 4.2) that $c$ acts by $(\tilde{\lambda},\tilde{\lambda}+2\rho_{0}). Now let $\omega_{0,\p}}$ be the longest element in the Weyl group of $\g_{0}^{S}$ and note that $-w_{0,\p}\tilde{\lambda}=\lambda$ and $w_{0,\p}\rho_{0}=-\rho_{0}$. Now the claim follows from the fact that $(.,.)$ is invariant under the action of the Weyl group.

\begin{lemma}[\cite{css}]\label{lemmanine}
The obstruction term can be written as
$$\left(\begin{array}{c}
0\\
v_{0}\otimes\left(\sum_{i\in I}\sum_{j=1}^{m_{i}}c_{\nu\sigma_{i,j}}\pi_{\nu\sigma_{i,j}}(Z\otimes w_{0})\right)\\
\left(\sum_{i\in I}\sum_{j=1}^{n_{i}}c_{\lambda\tau_{i,j}}\pi_{\lambda\tau_{i,j}}(Z\otimes v_{0})\right) \otimes w_{0}
\end{array}\right),$$
where
$$\g_{1}^{i}\otimes\mathbb{V}=\mathbb{V}_{\tau_{i,1}}\oplus...\oplus\mathbb{V}_{\tau_{i,n_{i}}}\;\text{and}\;\g_{1}^{i}\otimes\mathbb{W}=\mathbb{W}_{\sigma_{i,1}}\oplus...\oplus
\mathbb{W}_{\sigma_{i,m_{i}}},$$
for $i=1,...,l_{0}$, are the decompositions into irreducible $\g_{0}$-modules with corresponding projections $\pi_{\lambda\tau_{i,j}}$ and $\pi_{\nu\sigma_{i,j}}$ and where
$$c_{\gamma\kappa_{i,j}}=\frac{1}{2}\left(c(\kappa_{i,j})-c(\gamma)-c(-\alpha_{i})\right).$$
%Here we have used the notation $\mathbb{V}(\mu)$ for a finite dimensional irreducible representation which is dual to the finite dimensional irreducible representation of $\p$ with highest weight $\mu$.
\end{lemma}
\begin{proof}
Writing down the obstruction term as a mapping
$$\Phi:\g_{1}\otimes\mathbb{V}\otimes\mathbb{W}\rightarrow
\begin{array}{c}
\mathbb{V}\otimes\g_{1}\otimes\mathbb{W}\\
\oplus\\
\g_{1}\otimes\mathbb{V}\otimes\mathbb{W}\\
\end{array},$$
with
$$\Phi(Z\otimes v\otimes w)=\left(
\begin{array}{c}
v_{0}\otimes \sum_{\alpha'}\eta_{\alpha'}\otimes\tilde{\nu}([Z,\xi_{\alpha'}])w_{0}\\
\sum_{\alpha'}\eta_{\alpha'}\otimes\tilde{\lambda}([Z,\xi_{\alpha'}])v_{0}\otimes w_{0}
\end{array}\right),$$
allows one to use the Casimir operator to turn this into an easier expression.
\par
Let $\{Y_{a}\}$, $\{Y^{a}\}$ be basis of $\g_{0}$, orthonormal with respect to the form $B(.,.)$ as in Lemma~\ref{lemmaeight}. The Casimir operator of $\g_{0}$
is $c=\sum_{a}Y_{a}Y^{a}$ and
% This operator acts on each irreducible $\g_{0}^{S}$-module of highest weight $\lambda$ by a scalar $c(\lambda)=(\lambda,\lambda+2
%\rho_{0})$, where 
%$$\rho_{0}=\frac{1}{2}\sum_{\alpha\in\Delta^{+}(\g_{0}^{S})}\alpha\footnote{$\Delta^{+}(\g_{0}^{S})$ denotes the set of positive weights of $\g_{0}^{S}$}.$$
we compute
\begin{eqnarray*}
[Z,\xi_{\alpha'}] & = & \sum_{a}B(Y^{a},[Z,\xi_{\alpha'}])Y_{a}\\
                        & = & \sum_{a}B([Y^{a},Z],\xi_{\alpha'})Y_{a},
\end{eqnarray*}
where in the first equality we have written the element $[Z,\xi_{\alpha'}]\in\g_{0}$ in the basis $\{Y_{a}\}$ and the second equality follows from the associativity of the Killing form.
Furthermore
\begin{eqnarray*}
\sum_{\alpha'}\eta_{\alpha'}\otimes\tilde{\lambda}([Z,\xi_{\alpha'}])v & = &  \sum_{\alpha'}\eta_{\alpha'}\otimes
\tilde{\lambda}\left(\sum_{a}([Y^{a},Z],\xi_{\alpha'})Y_{a}\right)v_{0}\\
& = & \sum_{a}\sum_{\alpha'}\eta_{\alpha'}([Y^{a},Z],\xi_{\alpha'})\otimes \tilde{\lambda}(Y_{a})v_{0}\\
& = &  \sum_{a}[Y^{a},Z]\otimes \tilde{\lambda}(Y_{a})v_{0},
\end{eqnarray*}
where this time we have written the element $[Y^{a},Z]\in\g_{1}$ in the basis $\{\eta_{\alpha'}\}$. Interchanging $\{Y_{a}\}$ with $\{Y^{a}\}$, the same calculation can be done to obtain an analogous expression 
$$ \sum_{a}[Y^{a},Z]\otimes \tilde{\lambda}(Y_{a})v_{0}=\sum_{\alpha'}\eta_{\alpha'}\otimes\tilde{\lambda}([Z,\xi_{\alpha'}])v = \sum_{a}[Y_{a},Z]\otimes \tilde{\lambda}(Y^{a})v_{0}.$$
This  yields
\begin{eqnarray*}
\sum_{a}[Y^{a},Z]\otimes \tilde{\lambda}(Y_{a})v & = & \frac{1}{2}\sum_{a}Y_{a}Y^{a}(Z\otimes v_{0})\\
&  & -\frac{1}{2}\sum_{a}(Y_{a}Y^{a}.Z)\otimes v -\frac{1}{2}\sum_{a}Z\otimes(Y_{a}Y^{a}.v_{0})\\
& = & \frac{1}{2}\sum_{i\in I}\sum_{j=1}^{n_{i}}\left(c(\tau_{i,j}) -c(-\alpha_{i})-c(\lambda)\right) \pi_{\lambda\tau_{i,j}}(Z\otimes v_{0}),
\end{eqnarray*}
where $-\alpha_{i}$ is the the highest weight of $(\g^{i}_{1})^{*}=\g^{i}_{-1}$, $\tau_{i,j}$ ranges over the highest weights of the duals of the irreducible components of $\g^{i}_{1}\otimes\mathbb{V}$ and
$\pi_{\lambda\tau_{i,j}}$ denotes the corresponding projection.
\par
Exactly the same calculation can be done 
for $\sum_{\alpha'}\eta_{\alpha'}\otimes[Z,\xi_{\alpha'}].w_{0}$ 
\end{proof}

\begin{lemma}\label{lemmaten}
With the conventions as above,
$$2c_{\lambda\tau_{i,j}}=\Vert\tau_{i,j}+\rho\Vert^{2}-\Vert \lambda+\rho\Vert^{2}\;\text{and}\;2c_{\nu\sigma_{i,j}}=\Vert\sigma_{i,j}+\rho\Vert^{2}-\Vert \nu+\rho\Vert^{2},$$
where $\rho=\sum_{k=1}^{n}\omega_{k}=\frac{1}{2}\sum_{\alpha\in\Delta^{+}(\g,\h)}\alpha$ is the half sum over all positive roots in $\g$ and
$$\Vert\mu\Vert^{2}=B(\mu,\mu)\;\forall\;\mu\in\h^{*}.$$
\end{lemma}
\begin{proof} 
We will do the computation for $\mathbb{V}$, the case of $\mathbb{W}$ being analogous. Let $\mu$ be the highest weight of $\mathbb{E}^{*}$, where $\mathbb{E}\subset\g_{1}^{i}\otimes\mathbb{V}$ is an irreducible component.
First of all, we note that $\mathbb{E}^{*}\subset\g_{-1}^{i}\otimes\mathbb{V}^{*}$, so $\mu$ can be written as $\lambda$ plus a weight of $\g_{-1}^{i}$. But all weights of $\g_{-1}^{i}$ have the form $-\alpha_{i}-\sum_{j\in J}n_{j}\alpha_{j}$, so $\lambda-\mu-\alpha_{i}=\sum_{j\in J}n_{j}\alpha_{j}$.
\par
Now use the equation
$$2B(\rho,\alpha_{i})=B(\rho,\alpha_{i}^{\vee})\Vert\alpha_{i}\Vert^{2}=\Vert\alpha_{i}\Vert^{2}$$
to deduce
\begin{eqnarray*}
2c_{\lambda\mu}&=&B(\mu,\mu+2\rho_{0})-B(\lambda,\lambda+2\rho_{0})-B(-\alpha_{i},-\alpha_{i}+2\rho_{0})\\
&=&B(\mu,\mu+2\rho)-B(\lambda,\lambda+2\rho)-2B(\rho-\rho_{0},\mu-\lambda)-\Vert\alpha_{i}\Vert^{2}+2B(\alpha_{i},\rho_{0})\\
&=&\Vert \mu+\rho\Vert^{2}-\Vert\lambda+\rho\Vert^{2}+2B(\rho-\rho_{0},\lambda-\mu-\alpha_{i})\\
&=&\Vert \mu+\rho\Vert^{2}-\Vert\lambda+\rho\Vert^{2},
\end{eqnarray*}
because 
$$B(\rho-\rho_{0},\lambda-\mu-\alpha_{i})=B\left(\sum_{i\in I}\omega_{i},\sum_{j\in J}n_{j}\alpha_{j}\right)=0.$$
\end{proof}

If $\mathbb{E}_{\mu}$ is one of the irreducible components of $\g_{1}\otimes\mathbb{V}\otimes\mathbb{W}$, then we denote by
$\pi_{\tau_{i,j}\mu}^{k}$ the projection $\mathbb{V}_{\tau_{i,j}}\otimes\mathbb{W}\rightarrow \mathbb{E}^{(k)}_{\mu}$ onto the $k$-th copy of $\mathbb{E}_{\mu}$ in
the decomposition. $\pi_{\sigma_{i,j}\mu}^{k}$ is defined analogously as the projection onto the $k$-th copy of $\mathbb{E}_{\mu}$ in $\mathbb{V}\otimes\mathbb{W}_{\sigma_{i,j}}$.
Every projection 
$$\pi:
\begin{array}{c}
\mathbb{V}\otimes\g_{1}\otimes\mathbb{W}\\
\oplus\\
\g_{1}\otimes\mathbb{V}\otimes\mathbb{W}
\end{array}
\rightarrow \mathbb{E}_{\mu}$$
can be written as
\begin{eqnarray*}
\pi \left(
\begin{array}{c}
v_{1}\otimes Z_{1}\otimes w_{1}\\
Z_{2}\otimes v_{2}\otimes w_{2} \\
\end{array}\right)
& = &
\sum_{i,j,k}a_{\tau_{i,j},k}\pi^{k}_{\tau_{i,j}\mu}\left(\pi_{\lambda\tau_{i,j}}(Z_{1}\otimes v_{1})\otimes w_{1}\right)\\
&   & +
\sum_{i,j,k}b_{\sigma_{i,j},k}\pi_{\sigma_{i,j}\mu}^{k}\left(v_{2}\otimes \pi_{\nu\sigma}(Z_{2}\otimes w_{2})\right),
\end{eqnarray*}
for some constants $a_{\tau_{i,j},k}$ and $b_{\sigma_{i,j},k}$.
In order for a projection $\pi$ to be a $\p$-module homomorphism, $\pi\circ\Phi(Z\otimes v\otimes w))=0$ has to hold for all
$Z\in\g_{1}$, $v\in\mathbb{V}$ and $w\in\mathbb{W}$. This reads
\begin{eqnarray*}
\pi\circ\Phi (Z\otimes v\otimes w)& = &
\sum_{i,j,k}a_{\tau_{i,j},k}c_{\lambda\tau_{i,j}}\pi^{k}_{\tau_{i,j}\mu}(\pi_{\lambda\tau_{i,j}}(Z\otimes v)\otimes w)\\
&   & +\sum_{i,j,k}b_{\sigma_{i,j},k}c_{\nu\sigma_{i,j}}\pi^{k}_{\sigma_{i,j}\mu}(v\otimes\pi_{\nu\sigma_{i,j}}(Z\otimes w))\\
& = & 0.
\end{eqnarray*}
Let $x$ denote the number of copies of $\mathbb{E}_{\mu}$ in $\g_{1}\otimes \mathbb{V}\otimes\mathbb{W}$, then there are 
$2x$ unknowns and $x$ equations. Since $Z$, $v$ and $w$ are to be arbitrary and all $\pi_{\tau_{i,j}\mu}^{k}(\pi_{\lambda\tau_{i,j}}(Z\otimes v)\otimes w)$ lie in different
copies of $\mathbb{E}_{\mu}$, we can think of those elements as constituting a basis $\{e_{i}\}$ of~$\oplus^{x}\mathbb{E}_{\mu}$. The same is true for the different
$\pi_{\sigma_{i,j}\mu}^{k}(v\otimes\pi_{\nu\sigma_{i,j}}(Z\otimes w))$, which constitute a different basis~$\{f_{j}\}$. Hence there is a linear isomorphism 
$f_{j}=\sum_{i}A_{ij}e_{i}$ connecting those two basis and we obtain $x$ equations
$$a_{i}c_{\lambda\tau(i)}+\sum_{j}b_{j}c_{\nu\sigma(j)}A_{ij}=0,\; i=1,..,x,$$
where $\tau(i)$ (resp.~$\sigma(j)$) denotes the representation corresponding to the index $i$ (resp.~$j$), i.e.~the $i$-th (resp.~$j$-th) copy of $\mathbb{E}_{\mu}$ lies in 
$\mathbb{V}_{\tau(i)}\otimes\mathbb{W}$ (resp.~in $\mathbb{V}\otimes\mathbb{W}_{\sigma(j)}$). If all $c_{\lambda\tau(i)}\not=0$, then the constants $a_{i}$ are uniquely determined by the~$b_{j}$'s. 
\par
This yields an $x$-parameter family of invariant bilinear differential pairings if $c_{\lambda\tau_{i,j}}\not=0$ for all $i,j$ such that $\mathbb{E}_{\mu}\subset\mathbb{V}_{\tau_{i,j}}\otimes\mathbb{W}$.
If $c_{\lambda\tau_{i,j}}=0$, then there exists an invariant linear differential operator 
$$\Gamma(V)\rightarrow\Gamma(V_{\tau_{i,j}}),$$
where $V_{\tau_{i,j}}$ is the associated bundle to the representation $\mathbb{V}_{\tau_{i,j}}$. The roles of the $a_{i}$ and $b_{j}$ can, of course, be interchanged, so that we can alternatively exclude 
the situation where $c_{\nu\sigma_{i,j}}=0$, which corresponds to the existence of first order invariant differential operators $\Gamma(W)\rightarrow\Gamma(W_{\sigma_{i,j}})$. Thus we have proved:

\section{Classification}

\subsection{The main result}

\begin{theorem}[Main result 2]\label{mainresulttwo}
Let $\mathbb{V}$ and $\mathbb{W}$ be two finite dimensional irreducible $\p$-modules, so that $\mathbb{V}^{*}$ and $\mathbb{W}^{*}$ have highest weights $\lambda$ and $\nu$ respectively. 
%If $\mathbb{V}(\tau)$ (resp.~$\mathbb{W}(\sigma)$) denotes the representation which is dual to a finite dimensional irreducible representation of highest weight $\tau$ (resp.~$\sigma$), then
Furthermore denote the decomposition of the tensor products by
$$\g_{1}^{i}\otimes\mathbb{V}=\mathbb{V}_{\tau_{i,1}}\oplus...\oplus\mathbb{V}_{\tau_{i,n_{i}}},\;i=1,...,l_{0}$$
and
$$\g_{1}^{i}\otimes\mathbb{W}=\mathbb{W}_{\sigma_{i,1}}\oplus...\oplus\mathbb{W}_{\sigma_{i,m_{i}}},\;i=1,...,l_{0}.$$
If $c_{\lambda\tau_{i,j}}\not=0$ for all $i,j$ such that $\mathbb{E}_{\mu}\subset\mathbb{V}_{\tau_{i,j}}\otimes\mathbb{W}$ or $c_{\nu\sigma_{i,j}}\not=0$ for all $i,j$ such that $\mathbb{E}_{\mu}
\subset\mathbb{V}\otimes\mathbb{W}_{\sigma_{i,j}}$, then
there exists an $x$-parameter family of first order invariant bilinear differential pairings
$$\Gamma(V)\times\Gamma(W)\rightarrow \Gamma(E_{\mu}),$$
where $x$ is the number of copies of $\mathbb{E}_{\mu}$ in~$\g_{1}\otimes\mathbb{V}\otimes\mathbb{W}$. Modulo curvature terms, all invariant bilinear differential pairings of weighted order one on regular 
parabolic geometries are obtained in such a way.
\end{theorem}

%\begin{corollary}
%According to $\h^{*}=(\h_{0}^{S})^{*}\oplus\mathfrak{z}(\g_{0})^{*}$, we can decompose each element $\mu\in\h$ as $\mu=\mu_{0}+\mu'$. Then the condition $c_{\mu\kappa_{i,j}}=0$ is equivalent to
%$$(\mu',\alpha_{j}')=c^{0}_{\mu_{0}(\kappa_{i,j})_{0}}=\frac{1}{2}\left(((\kappa_{i,j})_{0},(\kappa_{i,j})_{0}+2\rho_{0})-(\mu_{0},\mu_{0}+2\rho_{0})-((\alpha_{j})_{0},(\alpha_{j})_{0}+2\rho_{0})\right).$$
%\begin{enumerate}
%\item
%The Borel case:
%\par
%We have $\g_{1}^{j}\otimes\mathbb{V}=\mathbb{V}(\lambda+\alpha_{j})$ and the condition is
%$$(\lambda,\alpha_{j})=0.$$
%\item
%The AHS case:
%\par
%We have $\mathfrak{z}(\g_{0})^{*}=\C\omega_{i_{0}}$, where $i_{0}$ is the unique crossed through node. We can  normalize the Killing form, so that $(\omega_{i_{0}},\omega_{i_{0}})=1$ and obtain the condition
%$$c_{\lambda,\tau_{i}}=0\Leftrightarrow\omega=c^{0}_{\lambda_{0},(\tau_{i})_{0}},$$
%where $\omega$ is the geometric weigh of $\mathbb{V}$. This follows from the fact that $\lambda=\omega\omega_{i_{0}}+\lambda_{0}$, $\alpha_{i_{0}}=\omega_{i_{0}}+(\alpha_{i_{0}})_{0}$ and $\tau=(\omega+1)\omega_{i_{0}}+\tau_{0}$.
%\end{enumerate}
%\end{corollary}

\begin{corollary}
The situation is considerably simplified if there is only one copy of $\mathbb{E}_{\mu}$ in $\g_{1}\otimes\mathbb{V}\otimes\mathbb{W}$. More precisely, let
$$\mathbb{E}_{\mu}\subset\mathbb{V}_{\tau_{i,j_{1}}}\otimes\mathbb{W}\;\text{and}\;\; \mathbb{E}_{\mu}\subset\mathbb{V}\otimes\mathbb{W}_{\sigma_{i,j_{2}}}.$$ 
Then we can choose $a=c_{\nu\sigma_{i,j_{2}}}$ and $b=-c_{\lambda\tau_{i,j_{1}}}$ if we normalize the projections correctly. Every multiple of this pairing is obviously invariant as well.
It also shows what happens if weights are excluded:
\begin{enumerate}
\item If $c_{\lambda\tau_{i,j_{1}}}=0$, then we must take $b=0$ and $a$ is arbitrary. This corresponds to the existence of an invariant first order linear differential operator $\Gamma(V)\rightarrow 
\Gamma(V_{\tau_{i,j_{1}}})$ combined with a unique projection $\Gamma(V_{\tau_{i,j_{1}}})\otimes\Gamma(W)\rightarrow \Gamma(E_{\mu})$.
\item If $c_{\nu\sigma_{i,j_{2}}}=0$, then there exists an invariant first order linear differential operator $\Gamma(W)\rightarrow\Gamma(W_{\sigma_{i,j_{2}}})$. This operator can be combined with the unique projection
 $\Gamma(W_{\sigma_{i,j_{2}}})\otimes\Gamma(V)\rightarrow\Gamma(E_{\mu})$, i.e. we must take $a=0$ and $b$ is arbitrary. 
\item If $c_{\lambda\tau_{i,j_{1}}}=c_{\nu\sigma_{i,j_{2}}}=0$, then the statement of the main theorem is not true anymore. 
We obtain two independent pairings corresponding to the two invariant linear differential operators and the projections mentioned above.
\end{enumerate}
\end{corollary}

\begin{corollary}
Let us briefly examine the condition $c_{\lambda\tau}=0$ for the special case that $\tau=\lambda-\alpha_{i},\; i\in I$. This implies
$$c_{\lambda\tau}=-B(\lambda,\alpha_{i}).$$
%There is an element $w_{\p,0}\in\mathcal{W}$ in the Weyl group (to be defined in the next chapter), so that
%$$w_{\p,0}\lambda=\tilde{\lambda},\;\text{and}\;w_{\p,0}\tilde{\alpha}_{j}=\alpha_{j},$$
%where $\tilde{\lambda}$ is the highest weight of $\mathbb{V}^{*}$.
%Since $B(.,.)$ is invariant under the Weyl group $\mathcal{W}$, we can apply the element $w_{\p,0}$ to the above equation to obtain
This implies that $c_{\lambda\tau}=0$ if and only if the number over the $i$-th node (which is crossed through) corresponding to the simple root $\alpha_{i}$ in the Dynkin diagram notation for $\mathbb{V}$ is zero.
This equation is in accordance with the situation considered in Theorem~\ref{mainresultone} and the Introduction.
\end{corollary}

\subsubsection{Remark}
Let us write $\h=\h^{S}\oplus\mathfrak{z}(\g_{0})$ for the orthogonal decomposition of $\h$ into
$$\h^{S}=\mathrm{span}\{h_{\alpha_{j}}\;:\;j\in J\}\;\text{and}\;\mathfrak{z}(\g_{0})=\{H\in \h\;:\; \alpha_{j}(H)=0\;\forall\; j\in J\}.$$
The duals can be characterized by
$$(\h^{S})^{*}=\mathrm{span}\{\alpha_{j}\;:\;j\in J\}\;,\;\mathfrak{z}(\g_{0})^{*}=\mathrm{span}\{\omega_{i}\;:\; i\in I\}.$$
If we write $\lambda=\lambda_{0}+\lambda'$ for the decomposition of an arbitrary element $\lambda\in\h^{*}$ into $\lambda_{0}\in(\h^{S})^{*}$ and $\lambda'\in\mathfrak{z}(\g_{0})^{*}$, then the following proposition holds.

\begin{proposition}
The equation $c_{\lambda\tau_{i,j}}=0$ is equivalent to  
\begin{eqnarray*}
B(\lambda',(-\alpha_{i})')&=&c^{0}_{\lambda_{0}(\tau_{i,j})_{0}}\\
&=&-\frac{1}{2}B((\tau_{i,j})_{0},(\tau_{i,j})_{0}+2\rho_{0})\\
&&-B(\lambda_{0},\lambda_{0}+2\rho_{0})-B((-\alpha_{i})_{0},(-\alpha_{j})_{0}+2\rho_{0})),
\end{eqnarray*}
where we have written $\lambda=\lambda_{0}+\lambda'$ for $\lambda_{0}\in(\mathfrak{h}^{S})^{*}$ and $\lambda'\in\mathfrak{z}(\g_{0})^{*}$ and correspondingly
$$-\alpha_{i}=(-\alpha_{i})_{0}+(-\alpha_{i})',\;\tau_{i,j}=(\tau_{i,j})_{0}+\lambda'+(-\alpha_{i})'.$$
Note that the tensor product decomposition of $(\g_{1}^{i})^{*}\otimes\mathbb{V}^{*}$ only depends on $\lambda_{0}$ and $(-\alpha_{i})_{0}$. This is a linear equation on $\lambda'$, so in Theorem~\ref{mainresulttwo}
we have to exclude a codimension one subspace of weights in $\mathfrak{z}(\g_{0})^{*}$ for every $\tau_{i,j}$ with $\mathbb{E}_{\mu}\subset\mathbb{V}_{\tau_{i,j}}\otimes\mathbb{W}$.
\end{proposition}
\begin{proof}
This is a direct computation using the fact that $\mathfrak{z}(\g_{0})^{*}$ and $(\h^{S})^{*}$ are orthogonal with respect to $B(.,.)$. In order to see this take
$\lambda\in\mathfrak{z}(\g_{0})^{*}$ and $\mu\in(\h^{S})^{*}$. Then $h_{\lambda}\in\mathfrak{z}(\g_{0})$, i.e.~$\alpha(h_{\lambda})=0$ for all $\alpha\in\mathcal{S}_{\p}$.
Now use the fact that $(\h^{S})^{*}$ is spanned by $\mathcal{S}_{\p}$ to deduce that 
$$B(\lambda,\mu)=\mu(h_{\lambda})=0.$$ 
\end{proof}

\begin{corollary}
We will treat two special cases separately. The case where $\p$ induces a $|1|$-grading and the case where $\p=\B$ is a Borel subalgebra.
\begin{enumerate}
\item
If $\p$ induces a $|1|$-grading, then $\mathfrak{z}(\g_{0})$ is one-dimensional and spanned by the grading element $E$. The action of $E$ on any irreducible representation of $\p$ is given by a scalar which we call {\bf geometric weight}.  $\mathfrak{z}(\g_{0})^{*}$ is then spanned by $\omega_{i_{0}}$, where $I=\{i_{0}\}$ and we scale the inner product $B(.,.)$ to an inner product $(.,.)$ on $\h^{*}$, so that $(\omega_{i_{0}},\omega_{i_{0}})=1$. 
%Moreover we note that the casimir operator (computed with respect to this inner product) acts by the scalar $(\tilde{\mu},\tilde{\mu}+2\rho_{0})$ on a finite dimensional irreducible module of highest weight $\tilde{\mu}$. In fact, this is the usual way, that it is computed, see~\cite{ss}. Finally, denote the highest weight of $\g_{1}$ by $\tilde{-\alpha_{i_{0}}}$.
\par
The highest weight $\lambda$ of $\mathbb{V}^{*}$ can be written as $\lambda=\lambda_{0}-\omega\omega_{i_{0}}$, where $\omega$ is the geometric weight of $\mathbb{V}$ and $\lambda_{0}\in(\h^{S})^{*}$. The geometric weight of $\g_{1}$ is obviously 1, so $-\alpha_{i_{0}}=-\omega_{i_{0}}+(-\alpha_{i_{0}})_{0}$. If $\mathbb{F}\subset\g_{1}^{*}\otimes\mathbb{V}^{*}$ has highest weight $\tau$, then $\mathbb{F}$ has geometric weight $\omega+1$ and we can write $\tau=\tau_{0}-(\omega+1)\omega_{i_{0}}$. Then we obtain 
$$c_{\lambda\tau}=\omega-c^{0}_{\lambda_{0}\tau_{0}},$$
where 
$$c^{0}_{\lambda_{0}\tau_{0}}=-\frac{1}{2}\left((\tau_{0},\tau_{0}+2\rho_{0})-(\lambda_{0},\lambda_{0}+2\rho_{0})-((-\alpha_{i_{0}})_{0},(-\alpha_{i_{0}})_{0}+2\rho_{0})\right),$$
in accordance with~\cite{css}.
\item
If $\p=\B$ is a Borel subalgebra, then $\mathfrak{z}(\g_{0})=\g_{0}=\h$. In this case $\g_{1}^{i}\otimes\mathbb{V}=\mathbb{V}_{\tau_{i}}$, with $\tau_{i}=\lambda-\alpha_{i}$  for $i\in I=\{1,...,n\}$. It immediately follows that
$$c_{\lambda\tau_{i}}=-B(\lambda,\alpha_{i}).$$
\end{enumerate}
\end{corollary}

\subsubsection{Example}
Let us look at the $|1|$-grading given by
$$
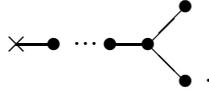
\begin{picture}(5,4)
\put(0,1){\makebox(0,0){$\cross$}}
\put(1,1){\makebox(0,0){$\bullet$}}
\put(1.5,1){...}
\put(2.5,1){\makebox(0,0){$\bullet$}}
\put(3.5,1){\makebox(0,0){$\bullet$}}
\put(4.5,2){\makebox(0,0){$\bullet$}}
\put(4.5,0){\makebox(0,0){$\bullet$}}
\put(0,1){\line(1,0){1}}
\put(2.5,1){\line(1,0){1}}
\put(3.5,1){\line(1,1){1}}
\put(3.5,1){\line(1,-1){1}}
\end{picture}.$$
Then $\omega_{i_{0}}=\epsilon_{1}$ and $-\alpha_{i_{0}}=\epsilon_{2}-\epsilon_{1}$. This implies 
$$(-\alpha_{i_{0}})_{0}=\epsilon_{2}=\sum_{j=2}^{l-2}\alpha_{j}+\frac{1}{2}(\alpha_{l-1}+\alpha_{l})\in\mathrm{span}\{\alpha_{j}\;:\; j=2,...,l\}=(\h^{S})^{*}.$$
Moreover $\rho_{0}=\sum_{j=2}^{l}\omega_{j}=(l-2)\epsilon_{1}+(l-2)\epsilon_{2}+(l-3)\epsilon_{3}+...$ and hence
$$((-\alpha_{i_{0}})_{0},(-\alpha_{i_{0}})_{0}+2\rho_{0})=2l-3.$$
Let us be even more concrete and set
$$\lambda=\hspace{2cm}\Dd{v}{1}{0}{0}{0}{0}{0}=(v+1)\epsilon_{1}+\epsilon_{2}$$
and look at $\tau=v\epsilon_{1}$, i.e.~$\tau_{0}=0$, $\omega=-(v+1)$ and $\lambda_{0}=\epsilon_{2}$. Then
$$\omega=c^{0}_{\lambda_{0}\tau_{0}}\Leftrightarrow (2l-2)+v=0.$$ 
Note that $n=2l-2$ is the dimension of the conformal manifold under consideration (the flat model being the sphere $\mathbb{S}^{n}$).
Of course, we could have also calculated
$$\Vert\tau+\rho\Vert^{2}-\Vert\lambda+\rho\Vert^{2}=\frac{-2}{4(l-1)}(2l-2+v).$$

\subsection{Examples}\label{examplesfirstorder}
We will look at pairings on projective manifolds of dimension $n$. In this case invariance means that a given formula does not depend on the choice of connection in the projective equivalence class. More precisely, recall from the Introduction that if $\nabla$ and $\hat{\nabla}$ are two connections that induce the same unparametrized geodesics, then
there is a one form $\Upsilon_{a}$ such that
$$\hat{\nabla}_{a}\omega_{b}=\nabla_{a}\omega_{b}-\Upsilon_{a}\omega_{b}-\Upsilon_{b}\omega_{a}$$
for every one form $\omega_{a}\in\Omega^{1}(\mathcal{M})$. This formula can be used to deduce the difference of the two connections when acting on any  weighted tensor bundle as shown in~\cite{e} and from this the invariance of any differential operator or pairing can be checked by hand.

\begin{enumerate}
\item
Let
$$V=\;\xoo{w}{0}{0}\;...\;\oo{0}{0}\quad\text{and}\quad W=\;\xoo{1+v}{0}{0}\;...\;\oo{0}{1},$$
so we look at pairings between weighted functions and weighted vector fields. We have
$$\g_{1}\otimes\mathbb{V}\otimes\mathbb{W}=\;\oo{1}{0}\;...\;\oo{0}{1}\;\oplus\;\oo{0}{0}\;...\;\oo{0}{0}$$
as $\g_{0}^{S}$-modules and geometric weights $\omega_{1}=-w\frac{n}{n+1}$ for $V$ and $\omega_{2}=-\frac{nv+n+1}{n+1}$ for $W$. Taking $\mu_{0}=0$ yields $c^{0}_{\lambda_{0}\tau_{0}} =0$ and  $c^{0}_{\nu_{0}\sigma_{0}}=n-1$. This corresponds to
the invariant pairing (where we have multiplied everything by $-\frac{n+1}{n}$):
$$(n+v+1)X^{a}\nabla_{a}f-w(\nabla_{a}X^{a})f.$$
\item
Quite similarly we obtain an invariant paring
\begin{eqnarray*}
\;\xoo{w}{0}{0}\;...\;\oo{0}{0}\;\times\;\;\xoo{v-2}{1}{0}\;...\;\oo{0}{0}\;&\rightarrow&\xoo{v+w-4}{2}{0}\;...\;\oo{0}{0}\\
(f,\sigma_{b}) &\mapsto& (v-2)\sigma_{(a}\nabla_{b)}f-w( \nabla_{(a}\sigma_{b)})f
\end{eqnarray*}
from the fact that in this case $\omega_{2}-c^{0}_{\nu_{0}\sigma_{0}}=-\frac{n}{n+1}(v-2)$ for $W=\Omega^{1}(v)$.
\item
A more sophisticated example can be obtained by taking
$$V=\;\xoo{1+v}{0}{0}\;...\;\oo{0}{1},\quad W=\underbrace{\;\xoo{w-(k+1)}{0}{0}\;...\;\ooo{0}{1}{0}\;...\;\oo{0}{0}}_{1\;\text{is in the}\;(k+1)\text{th position}}$$
and
$$E=\;\underbrace{\xoo{v+w-(k+1)}{0}{0}\;...\;\ooo{0}{1}{0}\;...\;\oo{0}{0}}_{1\;\text{is in the}\;(k+1)\text{th position}},$$
i.e.~we pair weighted vector fields with weighted $k$-forms to obtain weighted $k$-forms again. This time the multiplicity is two and indeed, for non-excluded geometric weights, there is a two parameter family of invariant bilinear differential pairings given by
$$X^{a}\nabla_{a}\omega_{bc...d}+\frac{n+v-w-vw+vk+1}{(n+v+1)(v+1)}(\nabla_{a}X^{a})\omega_{bc...d}-\frac{k+1}{v+1}(\nabla_{[a}X^{a})\omega_{bc...d]}$$
and
$$X^{a}\nabla_{[a}\omega_{bc...d]}+\frac{(n-k)w}{(n+v+1)(v+1)(k+1)}(\nabla_{a}X^{a})\omega_{bc...d}-\frac{w}{v+1}(\nabla_{[a}X^{a})\omega_{bc...d]}.$$
It can be seen that the denominators can only be zero when excluded weights are encountered, because $\omega_{1}-c^{0}_{\lambda_{0}(\tau_{1})_{0}}=-\frac{n}{n+1}(n+v+1)$ and
$\omega_{1}-c^{0}_{\lambda_{0}(\tau_{2})_{0}}=-\frac{n}{n+1}(v+1)$ for $V$. If one of these is zero, then the corresponding operator
$X^{a}\mapsto \nabla_{a}X^{a}$ or $X^{a}\mapsto\nabla_{b}X^{a}-\frac{1}{n}\nabla_{c}X^{c}\delta_{b}{}^{a}$ is projectively invariant.

\end{enumerate}

\section{The Problem with higher order pairings}\label{theproblemwithhigherorderpairings}
For higher order pairings the reasoning in the last section quickly gets out of hand. In the second order case for $|1|$-graded Lie algebras, for example, we have the following problem:
the possible symbols are given by mappings onto irreducible components of 
$$\begin{array}{c}
\odot^{2}\g_{1}\otimes\mathbb{V}\otimes\mathbb{W}\\
\oplus\\
\g_{1}\otimes\mathbb{V}\otimes\g_{1}\otimes\mathbb{W}\\
\oplus\\
\mathbb{V}\otimes\odot^{2}\g_{1}\otimes\mathbb{W}
\end{array}.$$
Therefore we have
$$2\times|\{\mathbb{E}\subset \odot^{2}\g_{1}\otimes\mathbb{V}\otimes\mathbb{W}\}|+|\{\mathbb{E}\subset \g_{1}\otimes\mathbb{V}\otimes\g_{1}\otimes\mathbb{W}\}|$$
unknowns corresponding to the terms which are second order in $V$, those which are second order in $W$ and those which are first order in both. However, there are
$$2\times|\{\mathbb{E}\subset \g_{1}\otimes\mathbb{V}\otimes\g_{1}\otimes\mathbb{W}\}|$$
obstruction terms. So it is not clear that we should obtain any pairings at all if there are more obstruction terms than unknowns. In the homogeneous case of $\mathcal{M}=G/P=\mathbb{CP}_{n}$, for example, one can look at all the pairings between
$$V=\;\xoo{w}{0}{0}\;...\;\oo{0}{0}\quad\text{and}\quad W=\;\xoo{1+v}{0}{0}\;...\;\oo{0}{1}$$
that land in
$$\xoo{v+w-2}{1}{0}\;...\;\oo{0}{0}.$$
The terms at disposal are
$$f\nabla_{a}\nabla_{b} X^{b},\;\nabla_{a}f\nabla_{b}X^{b},\;\nabla_{b}f(\nabla_{a}X^{b}-\frac{1}{n}\nabla_{c}X^{c}\delta_{a}^{\quad b}),\; X^{b}\nabla_{b}\nabla_{a}f$$
and there are four obstruction terms
$$f\Upsilon_{a}\nabla_{b}X^{b},\;f\Upsilon_{b}\nabla_{a}X^{b},\;(\nabla_{a}f)\Upsilon_{b}X^{b},\;(\nabla_{b}f)\Upsilon_{a}X^{b}.$$
So one might expect that only the zero paring would be invariant. But, somehow miraculously from this point of view, this is not the case and we obtain a one parameter family of invariant pairings spanned by
\begin{eqnarray*} 
X^{b}\nabla_{b}\nabla_{a}f&-&\frac{(w-1)(n+1)}{(v+n+1)n}\nabla_{a}f\nabla_{b}X^{b}\\
-\frac{w-1}{v+1}\nabla_{b}f(\nabla_{a}X^{b}-\frac{1}{n}\nabla_{c}X^{c}\delta_{a}{}^{b})&+&\frac{w(w-1)}{(v+1)(v+n+1)}f\nabla_{a}\nabla_{b} X^{b}.
\end{eqnarray*}
This formula even has a curved version that describes the invariant bilinear differential pairing for a manifold with a general projective structure.  One only has to add the curvature correction term 
$$\frac{w(v+w)}{v+1}P_{ab}X^{b}f,$$
where $P_{ab}$ is the Schouten tensor to be defined in Chapter~\ref{tractorchapter}.

\chapter{Higher order pairings 1}\label{higherorderone}

In this chapter we will study certain higher order invariant bilinear differential pairings for which we can write down explicit formulae for the pairings and the curvature correction terms. One can only hope for an explicit formula if there is
a one-parameter family of such pairings. It turns out that a certain class of those allows a unified description which is completely independent of the specific geometry and even of the bundles 
involved. It only depends on the order of the pairing in the same spirit as the differential operators described in~\cite{cds}.

\section{Semi-holonomic jet bundles}
In order to study higher order pairings on general (curved) parabolic geometries, it turns out that we need to consider semi-holonomic jet bundles instead of the usual jet bundles. The reason for this is that the latter cannot be described as associated bundles for some representation of $P$. 
%Starting with semi-holonomic jet bundles we will then define restricted semi-holonimic jet bundles and finally restricted semi-holonomic bi-jet bundles, which
%are the objects that need to be considered.

\subsection{Restricted semi-holonomic bi-jets}
%In this section we will describe the restricted semi-holonomic jet bundles as defined, for example, in~\cite{sl}.
%Then we will define restricted semi-holonomic bi-jet bundles that are constructed from the restricted semi-holonimic jet bundles. 

\begin{definition}
{\rm We will define the $k$-th {\bf restricted semi-holonomic jet prolongation} $\bar{\mathcal{J}}^{k}\mathbb{V}$ of a $\p$-module $\mathbb{V}$ as in~\cite{slov} inductively. Firstly, 
$\bar{\mathcal{J}}^{1}\mathbb{V}=\mathcal{J}^{1}\mathbb{V}$ is the usual first weighted jet prolongation as defined in~\ref{weightedjetbundles}. Having constructed $\bar{\mathcal{J}}^{k-1}\mathbb{V}$,
there are two canonical projections 
$$\mathcal{J}^{1}(\bar{\mathcal{J}}^{k-1}\mathbb{V})\rightarrow \mathcal{J}^{1}(\bar{\mathcal{J}}^{k-2}\mathbb{V}).$$
The first one is the usual projection $\mathcal{J}^{1}\tilde{\mathbb{V}}\rightarrow\tilde{\mathbb{V}}$ for any $\p$-module $\tilde{\mathbb{V}}$ followed by the inclusion $\bar{\mathcal{J}}^{k-1}\mathbb{V}\subset \mathcal{J}^{1}(\bar{\mathcal{J}}^{k-2}\mathbb{V})$ that is assumed to exist by the induction hypothesis. The second projection is induced by the first weighted prolongation of the map $\bar{\mathcal{J}}^{k-1}\mathbb{V}\rightarrow \bar{\mathcal{J}}^{k-2}\mathbb{V}$. Their equalizer is the submodule $\bar{\mathcal{J}}^{k}\mathbb{V}$. The filtration of this module can be written as
$$\bar{\mathcal{J}}^{k}\mathbb{V}=\sum_{i=0}^{k}\otimes^{i}\g_{1}\otimes\mathbb{V}.$$
There are canonical maps $\pi^{k}_{l}:\bar{\mathcal{J}}^{k}\mathbb{V}\rightarrow \bar{\mathcal{J}}^{l}\mathbb{V}$ for all $k\geq l$, so we can define the  
{\bf restricted semi-holonomic bi-jet prolongation} $\bar{\mathcal{J}}^{k}(\mathbb{V},\mathbb{W})$ of two $\p$-modules $\mathbb{V}$ and $\mathbb{W}$ as in~\ref{bijet}. The filtration of this $\p$-module is given by
$$\bar{\mathcal{J}}^{k}(\mathbb{V},\mathbb{W})=\sum_{i=0}^{k}\oplus_{j=0}^{i}(\otimes^{j}\g_{1}\otimes\mathbb{V})\otimes(\otimes^{i-j}\g_{1}\otimes\mathbb{W})$$
and the $\p$-module structure of $\bar{\mathcal{J}}^{k}(\mathbb{V},\mathbb{W})$ is induced by the $\p$-module structures of $\bar{\mathcal{J}}^{k}\mathbb{V}$ and $\bar{\mathcal{J}}^{k}\mathbb{W}$ as described 
in~\cite{slov}.
}\end{definition}

\subsubsection{Remark}
The associated bundles $\bar{\mathcal{J}}^{k}V$ and $\bar{\mathcal{J}}^{k}(V,W)$ are called the restricted $k$-th semi-holonomic jet bundle and the restricted 
$k$-th semi-holonomic bi-jet bundle respectively. In contrast to $\mathcal{J}^{k}V$ the (restricted) semi-holonomic jet bundle $\bar{\mathcal{J}}^{k}V$ is an associated bundle for each parabolic geometry. This is because $\bar{\mathcal{J}}^{k}V$ is defined by iterating the functor $\mathcal{J}^{1}$ which maps an associated bundle $U$ to an associated bundle $\mathcal{J}^{1}U$. The iteration of this functor yields an associated bundle $\underbrace{\mathcal{J}^{1}\cdots\mathcal{J}^{1}}_{k-\text{times}}V$ and $\bar{\mathcal{J}}^{k}V$ as an associated subbundle by using the $P$-module homomorphism 
$$\bar{\mathcal{J}}^{k}\mathbb{V}\hookrightarrow\underbrace{\mathcal{J}^{1}\cdots\mathcal{J}^{1}}_{k-\text{times}}\mathbb{V}.$$

\begin{proposition}\label{propositionsix}
Let $\mathbb{V}_{\lambda}$, $\mathbb{W}_{\nu}$ and $\mathbb{E}_{\mu}$ be three irreducible $\p$-modules and let $\Phi$ be a $\g_{0}$-module homomorphism
$$\oplus_{j=0}^{M}(\otimes^{j}\g_{1}\otimes\mathbb{V}_{\lambda})\otimes(\otimes^{M-j}\g_{1}\otimes\mathbb{W}_{\nu})\rightarrow \mathbb{E}_{\mu}.$$
Then $\Phi$ extends trivially to a $\p$-module homomorphism $\tilde{\Phi}:\bar{\mathcal{J}}^{k}(\mathbb{V}_{\lambda},\mathbb{W}_{\nu})\rightarrow\mathbb{E}_{\mu}$ if and only if
%\begin{eqnarray*}
%\Phi&\left(\rule{0pt}{20pt}\right.&v\otimes Z\star(Y_{1}\otimes\cdots\otimes Y_{M-1}\otimes w),\\
%&&Z\star(v)\otimes (Y_{1}\otimes\cdots\otimes Y_{M-1}\otimes w)+ (X_{1}\otimes v)\otimes Z\star(Y_{1}\otimes\cdots\otimes Y_{M-2}\otimes w),\\
%&&\vdots\\
%&&Z\star(X_{1}\otimes\cdots\otimes X_{j-1}\otimes v)\otimes (Y_{1}\otimes\cdots\otimes Y_{M-j}\otimes w)+(X_{1}\otimes\cdots\otimes X_{j}\otimes v)\otimes Z\star(Y_{1}\otimes\cdots\otimes Y_{M-j-1}\otimes w),\\
%&&\vdots\\
%&&Z\star(X_{1}\otimes\cdots\otimes X_{M-2}\otimes v)\otimes (Y_{1}\otimes w)+(X_{1}\otimes\cdots\otimes X_{M-1} \otimes v)\otimes Z\star(w),\\
%&&Z\star(X_{1}\otimes\cdots\otimes X_{M-1})\otimes w\left.\rule{0pt}{20pt}\right)=0,
%\end{eqnarray*}
\begin{eqnarray*}
\Phi&\left(\rule{0pt}{20pt}\right.&\phi_{0}\otimes Z\star\psi_{M-1},\\
&&Z\star\phi_{0}\otimes \psi_{M-1}+\phi_{1}\otimes Z\star\psi_{M-2},\\
&&\vdots\\
&&Z\star\phi_{j-1}\otimes \psi_{M-j}+\phi_{j}\otimes Z\star\psi_{M-j-1},\\
&&\vdots\\
&&Z\star\phi_{M-2}\otimes \psi_{1}+\phi_{M-1}\otimes Z\star\psi_{0},\\
&&Z\star\phi_{M-1}\otimes\psi_{0}\left.\rule{0pt}{20pt}\right)=0,
\end{eqnarray*}
for $\phi=(\phi_{0},...,\phi_{M})\in\bar{\mathcal{J}}^{M}(\mathbb{V}_{\lambda})$ and $\psi=(\psi_{0},...,\psi_{M})\in\bar{\mathcal{J}}^{M}(\mathbb{W}_{\nu})$, i.e.~$\phi_{j}\in\otimes^{j}\g_{1}\otimes\mathbb{V}_{\lambda}$
and $\psi_{j}\in\otimes^{j}\g_{1}\otimes\mathbb{W}_{\nu}$ for $ j=0,...,M$. The action $\star$ is defined by
$$Z\star (U_{1}\otimes\cdots\otimes U_{j}\otimes u)=\sum_{0\leq i\leq j}\sum_{\alpha}U_{1}\otimes\cdots\otimes U_{i}\otimes\eta^{\alpha}\otimes [Z,\xi_{\alpha}].(U_{i+1}\otimes\cdots\otimes U_{j}\otimes u),$$
for $Z\in\g_{1}$ and dual basis $\{\eta^{\alpha}\}$ and $\{\xi_{\alpha}\}$ of $\g_{1}$ and $\g_{-1}$.
\end{proposition}
\begin{proof}
Let us write
$$\bar{\mathcal{J}}^{M}(\mathbb{V}_{\lambda},\mathbb{W}_{\nu})=\bar{\mathcal{J}}^{M}(\mathbb{V}_{\lambda})\otimes\bar{\mathcal{J}}^{M}(\mathbb{W}_{\nu})/\mathbb{B},$$
then the action of $\p_{+}$ on $\bar{\mathcal{J}}^{M}(\mathbb{V}_{\lambda},\mathbb{W}_{\nu})$ is defined by
$$Z\star[\phi\otimes\psi]=[(Z\star\phi)\otimes\psi+\phi\otimes(Z\star\psi)],$$
where the bracket $[.]$ denotes the equivalence class modulo $\mathbb{B}$ and the stars inside the bracket are the actions of $\p_{+}$ on $\bar{\mathcal{J}}^{M}(\mathbb{V}_{\lambda})$ and 
$\bar{\mathcal{J}}^{M}(\mathbb{W}_{\nu})$ as defined in~\cite{sl}, Proposition 3.9.
Furthermore we note that $\Phi$ extends to a $\p$-module homomorphism if and only if $\Phi$ annihilates the image of the action of $\p_{+}$ on
$\bar{\mathcal{J}}^{k}(\mathbb{V}_{\lambda},\mathbb{W}_{\nu})$ that lies in the module $\oplus_{j=0}^{M}(\otimes^{j}\g_{1}\otimes\mathbb{V}_{\lambda})\otimes(\otimes^{M-j}\g_{1}\otimes\mathbb{W}_{\nu})$.
Finally it has to be noted that $\p_{+}$ is generated by $\g_{1}$, so we can restrict our attention to the action of $\g_{1}$. 
\end{proof}

%\begin{definition}
%{\rm We will defined the $k$-th {\bf weighted semi-holonomic jet prolongation} $\bar{\mathcal{J}}^{k}\mathbb{V}$ of a $\p$-module $\mathbb{V}$ inductively. Firstly, $\bar{\mathcal{J}}^{1}\mathbb{V}=\mathcal{J}^{1}\mathbb{V}$ is the usual first weighted jet prolongation as defined in ??. 
%\par
%The filtration of this module can be written as
%$$\bar{\mathcal{J}}^{k}\mathbb{V}=\sum_{i=0}^{k}T_{i}(\p_{+})\otimes\mathbb{V}.$$
%There are natural maps $\pi^{k}_{l}:\bar{\mathcal{J}}^{k}\mathbb{V}\rightarrow \bar{\mathcal{J}}^{l}\mathbb{V}$, so we can define the  {\bf weighted semi-holonomic bi-jet prolongation} $\bar{\mathcal{J}}^{k}(\mathbb{V},\mathbb{W})$ of two $\p$-modules $\mathbb{V}$ and $\mathbb{W}$ as in ??. The filtration is given by
%$$\bar{\mathcal{J}}^{k}(\mathbb{V},\mathbb{W})=\sum_{i=0}^{k}\oplus_{j=0}^{i}(T_{j}(\p_{+})\otimes\mathbb{V})\otimes(T_{i-j}(\p_{+})\otimes\mathbb{W}).$$
%}\end{definition}

\subsection{Extremal roots}

\begin{definition}[The Weyl group]
{\rm In the first chapter we have introduced the real vector space $E$ for a semisimple complex Lie algebra $\g$ which is equipped with a positive definite bilinear form $B(.,.)$. For every $\alpha \in \Delta (\g,\h)$ we will denote by
$$\sigma _{\alpha }(\lambda )=\lambda -B(\lambda ,\alpha ^{\vee})\alpha \;\forall \;\lambda \in E$$
the reflection on the hyperplane perpendicular to $\alpha
$. The {\bf Weyl group} $\mathcal{W}$ is the group generated by
those reflections. In fact, $\mathcal{W}$ is generated by
simple reflections, i.e.~by $\sigma _{\alpha }$ for $\alpha
\in \mathcal{S}$, as can be seen in~\cite{h}, p.~51. Every $w\in
\mathcal{W}$ has a unique {\bf length} $l(w)$ which denotes the
minimal number of simple reflections necessary to generate
$w$.
\begin{enumerate}
\item
%It can also be proven, [H], p.~51, that $\mathcal{W}$ permutes the set of roots and acts faithfully on the Weyl chambers.
For simple reflections we can write down an explicit formula for the node coefficients for
any weight $\lambda $ using the Cartan integers $c_{ij}$:
$$B(\sigma _{\alpha _{i}}(\lambda ),\alpha _{j}^{\vee})=B(\lambda ,\alpha _{j}^{\vee})-B(\lambda ,\alpha _{i}^{\vee})c_{ij}.$$
\item
The most significant action of the Weyl group on weights, however, is the {\bf affine action} given by
$$w.\lambda =w(\lambda +\rho )-\rho \;\forall \;\lambda \in E,\;w\in \mathcal{W},$$
where $\rho $ is given by $\rho =\sum_{i=1}^{l}\omega _{i}$.
%The weight $\lambda $ is called \emph{singular} if $\lambda $ has a non-trivial stabilizer under this action. Otherwise $\lambda $ is called \emph{non-singular}.
\item
The Weyl group has an important subgraph, the {\bf Hasse diagram} $\mathcal{W}^{\p}$ which is associated to a parabolic subalgebra $\p$ of $\g$. It is defined to be the subset  of elements in $\mathcal{W}$ whose action sends a weight $\lambda $, dominant for $\g$, to a weight dominant for $\p$. For more information and an easy algorithm to determine $\mathcal{W}^{\p}$ see~\cite{be}, p.~39--43. 
\item
For every $\lambda\in\h^{*}$, let $\Vert\lambda\Vert^{2}=B(\lambda,\lambda)$. Since $B(.,.)$ is invariant under the Weyl group (an easy exercise), the Weyl group acts by isometries with respect to this norm.
\end{enumerate}
}\end{definition}

%\begin{lemma}
%\begin{enumerate}
%\item
%There exists an unique element $w_{0}\in\mathcal{W}$ that transforms $\mathcal{S}$ to $-\mathcal{S}$. This element has the property that it sends the highest weight of any $\g$-module to the lowest weight an that $w_{0}\rho=-\rho$.
%\item
%There exists a unique element $w_{\p,0}$ in the  Hasse diagram $\mathcal{W}^{\p}$ that sends the highest weight of a finite dimensional irreducible $\p$-module to the lowest weight.
%\end{enumerate}
%\end{lemma}
%\begin{proof}
%The existence and uniqueness of $w_{0}$ follows from the fact that $-\mathcal{S}$ is also a basis of $\Delta(\g,\h)$ and that $\mathcal{W}$ acts simply transitively on bases, see~\cite{h}, theorem 10.3.
%\end{proof}

\subsubsection{Example}
An easy calculation shows that the Weyl group of $\mathfrak{sl}_{n+1}\C$ is isomorphic to the symmetry group $\mathbb{S}_{n+1}$ and acts on a weight $\lambda=\sum_{i=1}^{n+1}b_{i}\epsilon_{i}$ by permuting the $b_{i}$'s. More precisely, $\sigma_{\alpha_{i}}$ exchanges $b_{i}$ and $b_{i+1}$ and leaves the other numbers unchanged.

\begin{definition}
{\rm Let $\theta$ be a root of $\g$ such that $\g_{\theta}\in\g_{-1}$. Then there exists an $i\in I$ such that $\g_{\theta}\in\g_{-1}^{i}$. Let us call all such roots that in addition lie in the same orbit under the action of $\mathcal{W}$ as $-\alpha_{i}$ {\bf extremal} if $\alpha_{i}$ is a long simple root.
}\end{definition}

Let $\mathbb{V}_{\lambda}$ and $\mathbb{W}_{\nu}$ be two irreducible $\p$-modules so that $\mathbb{V}_{\lambda}^{*}$ has highest weight $\lambda\in\h^{*}$ and $\mathbb{W}_{\nu}^{*}$ has highest weight
$\nu\in\h^{*}$. Moreover suppose that $\alpha,\beta\in\h^{*}$ are extremal roots such that $\g_{\alpha},\g_{\beta}\in\g_{-1}^{i}$ and $\lambda+k\alpha$ and $\nu+k\beta$ are dominant for $\p$ for 
$k=0,...,M$. For the rest of this chapter we will implicitly assume that this setup is given. Finally, suppose that 
%$\mathbb{H}_{\kappa}\subset\mathbb{V}_{\lambda}\otimes\mathbb{W}_{\nu}$ is an irreducible component of multiplicity one such that $\mathbb{H}^{*}_{\kappa}$ has highest weight $\kappa\in\h^{*}$ and there is a long root $\gamma$ such that $\g_{\gamma}\in\g_{-1}$ and $\kappa+k\gamma$ is dominant for $\p$ for $k=0,...,M$. Finally suppose that
there is an irreducible component (of the $\g_{0}$-module tensor product)
$$\mathbb{E}_{\mu}\subset \otimes^{M}\g_{1}\otimes\mathbb{V}_{\lambda}\otimes\mathbb{W}_{\nu}$$
of multiplicity 1 so that
$$\mathbb{E}_{\mu}\subset \mathbb{V}_{\lambda+j\alpha}\otimes\mathbb{W}_{\nu+(M-j)\beta}$$
is an irreducible component of multiplicity one for $j=0,...,M$.
%\par
%Note that $ \mathbb{V}_{\lambda+j\alpha}$ is the unique irreducible component of $\otimes^{j}\p_{+}\otimes\mathbb{V}_{\lambda}$ that is dual to a module of highest weight $\lambda+j\alpha$, since $\alpha$ is a long root (triangle inequality!!), see~\cite{cds}. The same is true for $\mathbb{W}_{\nu+j\beta}$ lying in $\otimes^{j}\p_{+}\otimes\mathbb{W}_{\nu}$.

\subsubsection{Remark}
This setup excludes exactly those problematic pairings that were considered in~\ref{theproblemwithhigherorderpairings}. At the end of this chapter we will comment on the scope of the construction showing that it includes a wide class of pairings.

\subsubsection{Remark}\label{Raoconjecture}
Let $\mathbb{V}_{\kappa}$ be an irreducible $\p$-module and let $\theta$ be an extremal root such that $\kappa+k\theta$ is $\p$-dominant.
Then $k\theta$ is an extremal weight of
$\circledcirc^{k}\g_{-1}$ and hence the Parthasarathy-Ranga-Rao-Varadarajan conjecture, that was proved in~\cite{ku}, shows that there is an irreducible component of highest weight 
$\kappa+k\theta$ in $\otimes^{k}\g_{-1}\otimes\mathbb{V}_{\kappa}^{*}$ as long as $\kappa+k\theta$ is $\p$-dominant. The triangle inequality  shows that $\kappa+k\theta$ appears with multiplicity one,
since $\theta$ is a long root.

\subsection{Obstruction terms}\label{obstructionterms2}

For some constants $\gamma_{M,j}$ and the setup as above, let us define 
\begin{eqnarray*}
\Phi:\oplus_{j=0}^{M}\left(\otimes^{j}\g_{1}\otimes\mathbb{V}_{\lambda}\right)\otimes\left(\otimes^{M-j}\g_{1}\otimes\mathbb{W}_{\nu}\right)&\rightarrow& \mathbb{E}_{\mu}\\
\sum_{j=0}^{M}\tilde{\phi_{j}}\otimes\tilde{\psi}_{M-j}&\mapsto&\sum_{j=0}^{M}\gamma_{M,j}\pi_{j}\left(\pi_{\lambda_{j}}(\tilde{\phi_{j}})\otimes\pi_{\nu_{M-j}}(\tilde{\psi}_{M-j})\right),
\end{eqnarray*}
where 
$$\pi_{\lambda_{i}}:\otimes^{i}\g_{1}\otimes\mathbb{V}\rightarrow\mathbb{V}_{\lambda_{i}},\;\pi_{\nu_{i}}:\otimes^{i}\g_{1}\otimes\mathbb{W}\rightarrow\mathbb{W}_{\nu_{i}}\;
\text{and}\;\pi_{j}:\mathbb{V}_{\lambda_{j}}\otimes\mathbb{W}_{\nu_{M-j}}\rightarrow\mathbb{E}_{\mu}
$$
are the canonical projections and we have used the abbreviations $\lambda_{i}=\lambda+i\alpha$ and $\nu_{i}=\nu+i\beta$.

\begin{lemma}
With the setup as above,
$$\pi_{\lambda_{j}}(Z\star\phi_{j-1})=\frac{1}{2}\left(\Vert\lambda+j\alpha+\rho\Vert^{2}-\Vert\lambda+\rho\Vert^{2}\right)\pi_{\lambda_{j}}(Z\otimes\phi_{j-1})$$
and
$$\pi_{\nu_{j}}(Z\star\psi_{j-1})=\frac{1}{2}\left(\Vert\nu+j\beta+\rho\Vert^{2}-\Vert\nu+\rho\Vert^{2}\right)\pi_{\nu_{j}}(Z\otimes\psi_{j-1}).$$
\end{lemma}
\begin{proof}
This follows directly from Lemma~\ref{lemmanine}, Lemma~\ref{lemmaten} and Proposition~\ref{propositionsix} since
\begin{eqnarray*}
&&\pi_{\lambda_{j}}(Z\star(X_{1}\otimes\cdots\otimes  X_{j-1}\otimes v)\\
&=&\pi_{\lambda_{j}}\left(\sum_{0\leq i\leq j-1}\sum_{\alpha}X_{1}\otimes\cdots\otimes X_{i}\otimes\eta^{\alpha}\otimes [Z,\xi_{\alpha}].(X_{i+1}\otimes\cdots\otimes X_{j-1}\otimes v)\right)\\
&=&\frac{1}{2}\left(\sum_{i=0}^{j-1}(\Vert\lambda+(i+1)\alpha+\rho\Vert^{2}-\Vert\lambda+i\alpha+\rho\Vert^{2}\right)\pi_{\lambda_{j}}(Z\otimes X_{1}\otimes\cdots\otimes X_{j-1}\otimes v)\\
&=&\frac{1}{2}\left(\Vert\lambda+j\alpha+\rho\Vert^{2}-\Vert\lambda+\rho\Vert^{2}\right)\pi_{\lambda_{j}}(Z\otimes X_{1}\otimes\cdots\otimes X_{j-1}\otimes v).
\end{eqnarray*}
Here we have used that $\mathbb{V}_{\lambda+j\alpha}\subset \odot^{j}\g_{1}\otimes\mathbb{V}_{\lambda}$. The calculation for $\psi$ is analogous.
\end{proof}

$\Phi$ is a $\p$-module homomorphism if and only if the obstruction term from above given by $\Phi(Z\star[\phi\otimes\psi])$ vanishes. This can be computed as
\begin{eqnarray*}
\Phi(Z\star[\phi\otimes\psi])&=&\frac{1}{2}\sum_{j=1}^{M}\left(\rule{0pt}{20pt}\right.(\gamma_{M,j}\left(\Vert\lambda+j\alpha+\rho\Vert^{2}-\Vert\lambda+\rho\Vert^{2}\right)\\
&&+\gamma_{M,j-1}\left(\Vert\nu+(M-j+1)\beta+\rho\Vert^{2}-\Vert\nu+\rho\Vert^{2}\right))\pi(Z\otimes\phi_{j-1}\otimes\psi_{M-j})\left.\rule{0pt}{20pt}\right),
\end{eqnarray*}
where $\pi$ is the projection onto $\mathbb{E}_{\mu}$.
Now we note that
$$\Vert\lambda+j\alpha+\rho\Vert^{2}-\Vert\lambda+\rho\Vert^{2}=\Vert\alpha\Vert^{2}j(B(\lambda+\rho,\alpha^{\vee})+j).$$
Thus the vanishing of the obstruction term is equivalent to the following $M$ equations:
$$\gamma_{M,j+1}(j+1)\Vert\alpha\Vert^{2}(q-j)+\gamma_{M,j}(M-j)\Vert\beta\Vert^{2}(q'+j+1-M)=0,\;j=0,...,M-1,$$
where $q=-B(\lambda+\rho,\alpha^{\vee})-1$ and $q'=-B(\nu+\rho,\beta^{\vee})-1$. 
%This is exactly the condition
%$$\sigma_{\alpha}.\lambda=\lambda+(q+1)\alpha\;\;\text{and}\;\;\sigma_{\beta}.\nu=\nu+(q'+1)\beta.$$
Moreover we know that $\alpha$ and $\beta$ lie in the same Weyl orbit and therefore have the same length. Hence the equations reduce to
\begin{equation}\label{mainequationthree}
\fbox{$\displaystyle\gamma_{M,j+1}(j+1)(q-j)+\gamma_{M,j}(M-j)(q'+j+1-M)=0,\;j=0,...,M-1.$}
\end{equation}
This is the third (and by far the most general) situation in which we encounter these equations. They are formally equivalent to the ones that we considered in the Introduction and Section~\ref{algebraicdescription}. If $q\not\in\{0,1,...,M-1\}$ or $q'\not\in\{0,1,...,M-1\}$, then~(\ref{mainequationthree}) determines $\gamma_{M,j}$ up to scale:
$$\gamma_{M,j}=(-1)^{j}\binom{M}{j}\prod_{i=j}^{M-1}(q-i)\prod_{i=M-j}^{M-1}(q'-i).$$

\section{Ricci corrected derivatives}
In this section we will formally define the (weighted) $M$-th order bilinear differential pairing $\Gamma(V_{\lambda})\times\Gamma(W_{\nu})\rightarrow\Gamma(E_{\mu})$ that is induced by the mapping $\Phi$ described above.

\subsection{Formal definition of the pairing}

Using the full first jet-bundle $J^{1}V$ rather than the weighted first order jet-bundle $\mathcal{J}^{1}V$ of an associated vector bundle $V$, we can inductively construct the {\bf semi-holonomic jet bundle} $\bar{J}^{k}V$. This is the next definition.

\begin{definition}
{\rm
Let $\mathbb{V}$ be a $\p$-module. The $k$-th order semi-holonomic jet prolongation $\bar{J}^{k}\mathbb{V}$ is defined inductively as follows: $\bar{J}^{1}\mathbb{V}=J^{1}\mathbb{V}$ the usual first jet
prolongation. The $k$-th order {\bf semi-holonomic jet prolongation} $\bar{J}^{k}\mathbb{V}$ is defined to be the subbundle of $J^{1}(\bar{J}^{k-1}\mathbb{V})$ where the two canonical maps to 
$J^{1}(\bar{J}^{k-2}\mathbb{V})$ coincide. The associated bundle $\bar{J}^{k}V=\mathcal{G}\times_{P}\bar{J}^{k}\mathbb{V}$ is the $k$-th order semi-holonomic jet bundle. Moreover the iterated invariant differential
$$\Gamma(V)\ni s\mapsto \hat{j}^{k}s=(s,\nabla^{\omega}s,...,(\nabla^{\omega})^{k}s)\in\Gamma(\bar{J}^{k}V)$$
defines an embedding of $J^{k}V$ in $\bar{J}^{k}V$. 
}\end{definition}

\subsubsection{Remark}
The fact that the iterated invariant differential takes $P$-equivariant sections to $P$ equivariant sections can be checked by an inductive procedure based on the calculations in
Lemma~\ref{lemmasix}, see~\cite{slov}.

\subsubsection{Remark}
A Weyl structure $\sigma:\mathcal{G}_{0}\rightarrow\mathcal{G}$ induces an isomorphism between filtered and graded vector bundles and hence an isomorphism $\sigma_{\mathcal{A}}:gr\mathcal{A}\rightarrow \mathcal{A}$ as detailed in~\ref{filtrations}.
%T\mathcal{M}\oplus \mathcal{G}\times_{P}\g_{0}\oplus T^{*}\mathcal{M}\rightarrow \mathcal{A}$.
Thus, following~\cite{cds}, one can use the fundamental derivative to define the first {\bf Ricci-corrected derivative}
$$D^{(1)}_{X}s=\nabla_{\sigma_{\mathcal{A}}(X)}^{\omega}s,$$
for $X\in T\mathcal{M}$ and $s$ a section of an arbitrary associated bundle $V$. The higher order Ricci-corrected derivatives are defined analogously, using the Weyl connection to define an isomorphism 
between the semi-holonomic jet bundle $\bar{J}^{k}V$ and the associated graded bundle $gr\bar{J}^{k}V$. Under this isomorphism $\hat{j}^{k}s$ is mapped to $\hat{j}^{k}_{D}s$ with components denoted by $D^{(j)}s\in\otimes^{j}T^{*}\mathcal{M}\otimes V$, for $j=0,...,k$. A pairing (differential operator) can then be constructed by a simple projection from the graded vector bundle and it is invariant 
if and only if the projection of the image of the action of $\p_{+}$ vanishes, exactly as described above. The following lemma relates the Ricci-corrected derivative to the Weyl connection
$$D_{X}s=\sigma_{V}D_{X}^{(1)}(\sigma_{V}^{-1}s),$$
where $\sigma_{V}:gr V \rightarrow V$ is the isomorphism induced by the choice of Weyl structure.

\begin{lemma}\label{lemmatwelve}
$$D^{(1)}_{X}s=D_{X}s+r^{D}(X)\bullet s,$$
where $r^{D}$ is a $T^{*}\mathcal{M}$-valued one form on $\mathcal{M}$ and the bullet denotes the action of $T^{*}\mathcal{M}$ on $V$ induced by the action of $\p_{+}$ on $\mathbb{V}$.
\end{lemma}
\begin{proof}
This lemma is taken from~\cite{cds}, Proposition 4.2. Together with Appendix A in~\cite{cds}, this proposition also contains the proof that $D$ is the Weyl connection that is induced by $\sigma^{*}\omega_{0}$ and that $r^{D}$ is the Rho-tensor induced by $\sigma^{*}\omega_{+}$.
\end{proof}

\subsubsection{Construction of the pairings}
The formal construction of the pairing is carried out in several steps:
we take a section $s\in\Gamma(V_{\lambda})$, map it via the iterated invariant differential to $\hat{j}^{M}s\in\bar{J}^{M}V_{\lambda}$,
use the Weyl connection to map it to $\hat{j}^{M}_{D}s=(s,D^{(1)}s,...,D^{(M)}s)\in gr\bar{J}^{M}V_{\lambda}$ and then project onto $gr\bar{\mathcal{J}}^{M}V_{\lambda}$. The same is done for
$t\in\Gamma(W_{\nu})$. Then we can tensor $\hat{j}^{M}_{D}s$ and $\hat{j}^{M}_{D}t$ together and project onto $gr\bar{\mathcal{J}}^{M}(V_{\lambda},W_{\nu})$.
Using $\Phi$ as defined above ensures that the procedure is independent of the choice of Weyl structure involved.
Note that the obvious projection $\bar{J}^{M}\mathbb{V}\rightarrow \bar{\mathcal{J}}^{M}\mathbb{V}$ is a $\p$-module homomorphism.

\subsection{Explicit formulae}
In order to write down explicit formulae in terms of a Weyl connection, we can use the Ricci-corrected derivatives as in~\cite{cds}, where a recurrence formula for $D^{(k)}$ of 
$\hat{j}^{k}_{D}s=(s,D^{(1)}s,...,D^{(k)}s)$ is given. In the following $\nabla$ denotes a choice of Weyl connection and $\mathcal{D}_{k}$ denotes the $k$-th Ricci-corrected derivative. Moreover
$\Gamma=-\frac{1}{2}\Vert\alpha_{i}\Vert^{2}r^{D}$ with $r^{D}$ as in Lemma~\ref{lemmatwelve}. More precisely, we write down a symbolic formula where all indices are suppressed.
To obtain an actual formula, one has to include all indices and combine them as prescribed by the projection $\Phi$. The recursion formula from~\cite{cds}, Theorem 6.2, takes the form
$$\mathcal{D}_{k}s=\nabla\mathcal{D}_{k-1}s+(k-1)(q-k+2)\Gamma\mathcal{D}_{k-2}s$$
for $s\in\Gamma(V_{\lambda})$, $\mathcal{D}_{j}=\pi_{\lambda_{j}}\circ D^{(j)}$ and $q=-B(\lambda+\rho,\alpha^{\vee})-1$.
This yields
\begin{eqnarray*}
\mathcal{D}_{0}s&=&s\\
\mathcal{D}_{1}s&=&\nabla s\\
\mathcal{D}_{2}s&=&\nabla^{2}s+q\Gamma s\\
\mathcal{D}_{3}s&=&\nabla^{3}s+2(q-1)\Gamma\nabla s+q\nabla(\Gamma s)\\
\mathcal{D}_{4}s&=&\nabla^{4}s+q\nabla^{2}(\Gamma s)+2(q-1)\nabla(\Gamma\nabla s)+3(q-2)\Gamma\nabla^{2}s+3(q-2)q\Gamma^{2}s\\
\mathcal{D}_{5}s&=&\nabla^{5}s+q\nabla^{3}(\Gamma s)+2(q-1)\nabla^{2}(\Gamma\nabla s)+3(q-2)\nabla(\Gamma\nabla^{2}s)+3(q-2)q\nabla(\Gamma^{2}s)\\
&&+4(q-3)\Gamma\nabla^{3}s+4q(q-3)\Gamma\nabla(\Gamma s)+8(q-3)(q-1)\Gamma^{2}\nabla s\\
\mathcal{D}_{6}s&=&\nabla^{6}s+q\nabla^{4}(\Gamma s)+2(q-1)\nabla^{3}(\Gamma\nabla s)+3(q-2)\nabla^{2}(\Gamma\nabla^{2}s)+3(q-2)q\nabla^{2}(\Gamma^{2}s)\\
&&+4(q-3)\nabla(\Gamma\nabla^{3}s)+4q(q-3)\nabla(\Gamma\nabla(\Gamma s))+8(q-3)(q-1)\nabla(\Gamma^{2}\nabla s)\\
&&+5(q-4)\Gamma\nabla^{4}s+5(q-4)q\Gamma\nabla^{2}(\Gamma s)+10(q-4)(q-1)\Gamma\nabla(\Gamma\nabla s)\\
&&+15(q-4)(q-2)\Gamma^{2}\nabla^{2}s+15(q-4)(q-2)q\Gamma^{3}s.
\end{eqnarray*}
Putting $q=j-1$ in $\mathcal{D}_{j}$ yields the formulae as in~\cite{cds}.

\subsubsection{Remark}
With the right scaling, $\Gamma$ is the Rho-tensor in conformal geometry or the Schouten tensor in projective geometry. It is easy to determine what is the tensor $\Gamma$ in the various geometries. One just has to write down the simplest invariant differential operator of order two (for example when acting on weighted functions) and compute the curvature correction term. The general form will be $\nabla^{2}s+Ts$, where $T$ is some tensor. Then it follows from the above discussion that $T=\Gamma$. We will demonstrate this for our usual three examples (and refer the reader to Chapter~\ref{tractorchapter} for the specific notations):
\begin{enumerate}
\item
In projective geometry, the differential operator
\begin{eqnarray*}
\mathcal{E}(1)&\rightarrow &\mathcal{E}_{(ab)}(1)\\
f&\mapsto &\nabla_{a}\nabla_{b}f+P_{ab}f
\end{eqnarray*}
is invariant, where $P_{ab}$ is the Schouten tensor to be defined in Chapter~\ref{tractorchapter}. Hence $\Gamma=P_{ab}$.
\item
In conformal geometry in $n$ dimensions, the differential operator
\begin{eqnarray*}
\mathcal{E}[1]&\rightarrow &\mathcal{E}_{(ab)}[1]\\
f&\mapsto &\nabla_{a}\nabla_{b}f-\frac{1}{n}g_{ab}\nabla^{c}\nabla_{c}f+(P_{ab}-\frac{1}{n}g_{ab}P^{c}{}_{c})f
\end{eqnarray*}
is invariant, where $P_{ab}$ is the Rho-tensor to be defined in Chapter~\ref{tractorchapter}. Hence $\Gamma=P_{ab}$.

\item
In CR geometry, the differential operator
\begin{eqnarray*}
\mathcal{E}(w,1)&\rightarrow &\mathcal{E}^{(\alpha\beta)}(w-2,-1)\\
f&\mapsto &\nabla^{\alpha}\nabla^{\beta}f-iA^{\alpha\beta}f
\end{eqnarray*}
is invariant, where $A_{\alpha\beta}$ is the pseudohermitian torsion tensor to be defined in Chapter~\ref{tractorchapter}. Hence $\Gamma=-iA^{\alpha\beta}$ (for complex conjugate operators including derivatives $\nabla_{\alpha}$ the right tensor is $iA_{\alpha\beta}$).
\end{enumerate}

%The procedure is quite simple. We take a section $s\in\mathcal{C}^{\infty}(\mathcal{G},\mathbb{V})^{P}$, map it via the iterated invariant differential to $\hat{j}^{M}_{\omega}\in\hat{J}^{M}V$,
%use the Weyl connection to map it to $\hat{j}^{M}_{D}s=(s,D^{(1)}s,...,D^{(M)}s)\in\hat{J}^{M}_{E}V$ and then project onto $\hat{J}^{M}_{\mathcal{R}}V$. The same is done for
%$t\in\mathcal{C}^{\infty}(\mathcal{G},\mathbb{V})^{P}$ and we tensor $\hat{j}^{M}_{D}s\otimes\hat{j}^{M}_{D}t$ together and project onto $\hat{J}^{M}_{\mathcal{R}}(V,W)$.
%Then we use $\Phi$ as defined above to make sure that the procedure is independent of the choice of Weyl structure involved.
\subsubsection{Remark}
It is quite easy to write down a general formula. $\mathcal{D}_{k}$ has 
$$\sum_{j=0}^{\left[\frac{k}{2}\right]}\binom{k-j}{j}$$
curvature correction terms. Each term is determined by a sequence of $i$ $\nabla$'s and $j$ $\Gamma$'s, so that $i+2j=k$, in a precise order. This looks like this:
$$\nabla^{l_{t}}\Gamma^{s_{r}}\nabla^{l_{t-1}}\cdots\Gamma^{s_{1}}\nabla^{l_{1}}s,$$
so that the $\Gamma$'s are in position $i_{1},...,i_{j}$ ($j=s_{1}+...+s_{r}$) counting from the right and counting each $\Gamma$ twice (taking the leftmost position of each $\Gamma$ as $i_{m}$).
Then the constant is in front of this term is:
\begin{equation}\label{cct}
\prod_{m=1}^{j}(i_{m}-1)(q-i_{m}+2).
\end{equation}
Note that this formula is purely algebraic. One can use the Leibniz rule to rearrange terms, but this only leads to a more complicated expression.

\subsubsection{Examples 1}
\begin{enumerate}
\item
The term
$$\nabla\Gamma\nabla\Gamma s,$$
that occurs in $\mathcal{D}_{6}s$, has $i_{1}=2$ and $i_{2}=5$, so that the constant is given by $4q(q-3)$. 
\item
$$\Gamma\nabla\Gamma^{2}\nabla s$$
occurs in $\mathcal{D}_{8}s$ and has $i_{1}=3,i_{2}=5$ and $i_{3}=8$, so the constant is given by $56(q-1)(q-3)(q-6)$.
The invariant operator corresponding to $q=7$ has a correction term, where this term appears with constant $56\times6\times4\times 1=1344$.
\end{enumerate}

\subsubsection{Examples 2}
Let us write down the full formula for second and third order pairings:
\begin{enumerate}
\item
\begin{eqnarray*}
\mathcal{P}^{2}(s,t)&=&q(q-1)s(\mathcal{D}_{2}t)-2(q-1)(q'-1)(\mathcal{D}_{1}t)(\mathcal{D}_{1}s)+q'(q'-1)(\mathcal{D}_{2}s)t\\
&=&q(q-1)s(\nabla^{2}t+q'\Gamma t)-2(q-1)(q'-1)(\nabla t)(\nabla s)\\
&&+q'(q'-1)(\nabla^{2}s+q\Gamma s)t\\
&=&q(q-1)s(\nabla^{2}t)-2(q-1)(q'-1)(\nabla t)(\nabla s)+q'(q'-1)(\nabla^{2}s)t\\
&&+qq'(q+q'-2)st\Gamma
\end{eqnarray*}
\item
\begin{eqnarray*}
\mathcal{P}^{3}(s,t)&=&q(q-1)(q-2)s(\mathcal{D}_{3}t)-3(q-1)(q-2)(q'-2)(\mathcal{D}_{1}s)(\mathcal{D}_{2}t)\\
&&+3(q-2)(q'-1)(q'-2)(\mathcal{D}_{2}s)(\mathcal{D}_{1}t)-q'(q'-1)(q'-2)(\mathcal{D}_{3}s)t\\
&=&q(q-1)(q-2)s(\nabla^{3}t+2(q'-1)\Gamma\nabla t+q'\nabla\Gamma t)\\
&&-3(q-1)(q-2)(q'-2)(\nabla s)(\nabla^{2} t+q'\Gamma t)\\
&&+3(q-2)(q'-1)(q'-2)(\nabla^{2}s+q\Gamma s)(\nabla t)\\
&&-q'(q'-1)(q'-2)(\nabla^{3}s+2(q-1)\Gamma\nabla s+q\nabla\Gamma s)t\\
&=&q(q-1)(q-2)s(\nabla^{3}t)-3(q-1)(q-2)(q'-2)(\nabla s)(\nabla^{2} t)\\
&&+3(q-2)(q'-1)(q'-2)(\nabla^{2}s)(\nabla t)-q'(q'-1)(q'-2)(\nabla^{3}s)t\\
&&+q(q-2)(q'-1)(2q+3q'-8)s\Gamma\nabla t\\
&&-q'(q'-2)(q-1)(3q+2q'-8)t\Gamma\nabla s\\
&&+qq'(q-1)(q-2)s\nabla\Gamma t-qq'(q'-1)(q'-2)t\nabla\Gamma s.
\end{eqnarray*}
\end{enumerate}
\par
To summerize, we have shown:

\begin{theorem}[Main result 3]\label{mainresultthree}
Let $\mathbb{V}_{\lambda}$ and $\mathbb{W}_{\nu}$ be two irreducible $\p$-modules and let $\alpha$ and $\beta$ be two extremal roots so that $\lambda+j\alpha$ and $\nu+j\beta$ are $\p$-dominant for $j=1,...,M$.
Furthermore let $\mathbb{E}_{\mu}\subset\otimes^{M}\g_{1}\otimes\mathbb{V}_{\lambda}\otimes\mathbb{W}_{\nu}$ be an irreducible component of multiplicity one that lies in
$$\mathbb{V}_{\lambda+j\alpha}\otimes\mathbb{W}_{\nu+(M-j)\beta}$$
for $j=0,...,M$. 
If 
$$q=-B(\lambda+\rho,\alpha^{\vee})-1\not\in\{0,1,...,M-1\}$$
or
$$q'=-B(\lambda+\rho,\alpha^{\vee})-1\not\in\{0,1,...,M-1\},$$
then there exists a weighted $M$-th order invariant bilinear differential pairing
$$\mathcal{P}^{M}:\Gamma(V_{\lambda})\times\Gamma(W_{\nu})\rightarrow \Gamma(E_{\mu})$$
and this pairing is given by
\begin{equation*}
\fbox{$\displaystyle
\mathcal{P}^{M}(s,t)=\sum_{j=0}^{M}(-1)^{j}\binom{M}{j}\prod_{i=j}^{M-1}(q-i)\prod_{i=M-j}^{M-1}(q'-i)\pi(\mathcal{D}_{j}s\otimes\mathcal{D}_{M-j}t),$}
\end{equation*}
where $\mathcal{D}_{j}$ is defined as above. If $q\in\{0,1,...,M-1\}$, then the linear differential operator $s\mapsto \mathcal{D}_{q+1}s$ is invariant. The analogous statement holds for $q'$.
\end{theorem}

\section{Examples and scope of the construction}

\subsection{Examples}\label{examplepairings}
\begin{description}
\item{(a)}
We can always choose $\alpha=\beta=-\alpha_{i}$, $i\in I$, as extremal roots and take $\mu=\lambda+\nu-M\alpha_{i}$. The pairings thus constructed are the curved analogues of the pairings described in
Theorem~\ref{mainresultone}. 
%In particular, these are all possible possible pairings for manifolds $G/B$, where $P=B$ is the Borel. Thus all invariant differential pairings on $G/B$ allow curved analogues.
\item{(b)}
Let $\mathbb{V}_{\lambda}$ be an irreducible $\p$-module and let $\alpha$ be an extremal root that lies in the same Weyl orbit as $-\alpha_{i}$, $i\in I$. Assume that $\lambda+j\alpha$ is $\p$ dominant for
$j=1,...,M$, but that $q=-B(\lambda+\rho,\alpha^{\vee})-1\not\in\{0,1,...,M-1\}$. 
%We can then look at pairings
 Then there is a unique irreducible component $\mathbb{E}_{\lambda+M\alpha}$ in $\otimes^{M}\g_{1}\otimes\mathbb{V}_{\lambda}$, see Remark~\ref{Raoconjecture}. Furthermore let $\mathbb{W}_{\nu}$ be a one dimensional
representation that is given by a character in $\mathfrak{z}(\g_{0})^{*}$ and a trivial representation of $\g_{0}^{S}$, i.e.~sections of $W_{\nu}$ are weighted functions. Then 
$\mathbb{E}_{\lambda+M\alpha+\nu}$ satisfies the requirements of Theorem~\ref{mainresultthree} with $\beta=-\alpha_{i}$
and we can define an invariant bilinear differential pairing
$$\Gamma(V_{\lambda})\times\Gamma(W_{\nu})\rightarrow \Gamma(E_{\lambda+M\alpha+\nu}).$$
%If $q\in\{0,...,M-1\}$, then $\mathcal{D}_{q+1}:\Gamma(V_{\lambda})\rightarrow \Gamma(E_{\lambda+q\rho})$ is an invariant linear differential operator. 

%The Ricci-corrected derivative $\mathcal{D}_{M}$ induces an 
%invariant differetial operator $D_{M}:\Gamma(V_{\lambda})\rightarrow \Gamma(E_{\lambda+M\alpha})$ if $q=-B(\lambda+\rho,\alpha^{\vee})-1=M-1$. 
%We can then look at pairings
%$$\Gamma(V_{\lambda})\times\Gamma(W_{nu})\rightarrow\Gamma(E_{\mu})$$
%for $\mathbb{W}_{\nu}$ induces by a trivial representation of $\g_{0}^{S}$.
%If we pair sections of an arbitrary bundle $V$ with weighted functions (that are induced by characters in $\mathfrak{z}(\g_{0})^{*}$ together with the trivial representation of $\g_{0}^{S}$), then we can take
%every standard operator that can be described by a $\mathcal{D}_{j}$ (in the last chapter we will determine for 
%conformal, projective and CR geometry exactly which operators can be described like that), take the corresponding extremal root $\alpha$ and choose $\beta=-\alpha_{i}$. Then there is only
\item{(c)}
Projective geometry in $n$ dimensions: for every $k\geq 0$, there exists an invariant bilinear differential pairing
$$\mathcal{E}^{(i_{1}i_{2}...i_{k})}(v)\times\mathcal{E}(w)\rightarrow\mathcal{E}(v+w)$$
given by
$$
\sum_{j=0}^{k}(-1)^{j}\binom{M}{j}\prod_{i=j}^{k-1}(v+n+2k-1-i)\prod_{i=k-j}^{k-1}(w-i)(\nabla_{i_{1}}...\nabla_{i_{j}}V^{i_{1}...i_{k}})(\nabla_{i_{j+1}}...\nabla_{i_{k}}f)$$
$$+\;\;C.C.T.$$
The curvature correction terms (C.C.T.) are given by the combinatorial formula~(\ref{cct}), where $\Gamma=P_{ab}$ is the Schouten tensor.
\item{(d)}
Conformal geometry in $n$-dimensions: there exists a $k$-th order invariant bilinear differential pairing
$$\mathcal{E}^{(i_{1}i_{2}...i_{k})}_{0}[v]\times\mathcal{E}[w]\rightarrow\mathcal{E}[v+w]$$
given by
$$
\sum_{j=0}^{k}(-1)^{j}\binom{M}{j}\prod_{i=j}^{k-1}(v+n+2k-2-i)\prod_{i=k-j}^{k-1}(w-i)(\nabla_{i_{1}}...\nabla_{i_{j}}V^{i_{1}...i_{k}})(\nabla_{i_{j+1}}...\nabla_{i_{k}}f)$$
$$+\;\;C.C.T.$$
The formula for each curvature correction term of $\mathcal{D}_{j}V$ and $\mathcal{D}_{M-j}f$ is given by the combinatorial formula~(\ref{cct}), where $\Gamma=P_{ab}$ is the Rho-tensor.
\item{(e)}
CR geometry in dimension $2n+1$: there exists an invariant bilinear differential pairing
$$\mathcal{E}_{(\alpha_{1}\alpha_{2}...\alpha_{k})}(w,w')\times\mathcal{E}(v,v')\rightarrow\mathcal{E}(v'+w')$$
for every $k\geq 0$ given by
$$
\sum_{j=0}^{k}(-1)^{j}\binom{M}{j}\prod_{i=j}^{k-1}(w'+n+k-1-i)\prod_{i=k-j}^{k-1}(v'-i)(\nabla^{\alpha_{1}}...\nabla^{\alpha_{j}}v_{\alpha_{1}...\alpha_{k}})(\nabla^{\alpha_{j+1}}...\nabla^{\alpha_{k}}f)$$
$$+\;\;C.C.T.$$
The curvature correction terms are again given by the combinatorial formula~(\ref{cct}), where $\Gamma=-iA^{\alpha\beta}$ is the pseudohermitian torsion tensor.

\end{description}

\subsection{Scope of the constructuion}
The symbols of the differential pairings described above are linear combinations of terms $\mathcal{D}_{j}s\otimes\mathcal{D}_{M-j}t$ and the operators $\mathcal{D}_{i}s$ are invariant if $q=i-1$ 
(for $\mathcal{D}_{i}t$ the analogous statement holds). These operators are all standard, but do not inculde every standard operator. But in certain cases, like conformal geometry in even dimensions,
all standard differential operators are of this type. In the Appendix, we have written down the BGG sequence for our standard three examples of projective, conformal and CR geometry. Those BGG sequences
clearly show which operators are constructed with this method and which are not. For more exotic geometric structures the reader is advised to consult~\cite{cds}.

\subsubsection{Remark}
The fact that the linear equations (\ref{mainequationthree}) that we had to solve in this chapter are exactly the ones that had to be solved in the Introduction when we dealt with invariant bilinear differential pairings on the Riemann sphere can be understood as follows: in the homogeneous case invariant differential operators on $G/B$, where $B$ is the Borel subgroup,
give rise to invariant differential operators on $G/P$ via direct images. The same is true for invariant differential pairings. Let us demonstrate this for compactified complexified
Minkowski space $\mathcal{M}=\oxo{\;\;}{\;\;}{\;\;}$. There is a double fibration
$$\begin{array}{ccc}
&\xxx{\;\;}{\;\;}{\;\;}&\\
\mu\swarrow&&\searrow\nu\\
\xox{\;\;}{\;\;}{\;\;}&&\oxo{\;\;}{\;\;}{\;\;}
\end{array},$$
where $\xxx{\;\;}{\;\;}{\;\;}=G/B$. The maps $\mu$ and $\nu$ commute with the action of $G$, so in particular the fibres of $\mu$, which are isomorphic to the Riemann sphere $\cross$, are permuted by $G$. Hence an invariant differential pairing on $\cross$ gives rise to an invariant differential pairing on $G/B$. Then one can use the fact that direct images of line bundles on $G/B$ are vector bundles on $G/P$, see~\cite{be}. Which direct images to take is determined by the extremal roots $\alpha$, $\beta$. More precisely, 
let $\lambda$ be a dominant integral weight for $\p$ and let $\theta$ be an extremal root with $w(-\alpha_{i})=\theta$, $w\in\mathcal{W}$.
Moreover let $w.\kappa=\lambda$, then
\begin{eqnarray*}
B(\kappa,\alpha_{i}^{\vee})&=&B(w^{-1}.\lambda,\alpha_{i}^{\vee})\\
&=&B(w^{-1}(\lambda+\rho),\alpha_{i}^{\vee})-1\\
&=&B(\lambda+\rho,w(\alpha_{i})^{\vee})-1\\
&=&-B(\lambda+\rho,\theta^{\vee})-1\\
&=&q.
\end{eqnarray*}
This implies that $q$ is the number over the node corresponding to $\alpha_{i}$ in the Dynkin diagram notation for $\kappa$. $\kappa$ induces a line bundle $V_{\kappa}$ on $G/B$ whose $l(w)$-th direct image $\nu_{*}^{l(w)}(V_{\kappa})$ is $V_{\lambda}$, the vector bundle over $G/P$ induced by $\mathbb{V}_{\lambda}$.

\chapter{Higher order pairings 2}
This chapter deals with general higher order bilinear invariant differential pairings. The strategy employed is to define a linear invariant differential operator that includes an arbitrary irreducible associated bundle in some other associated bundle, called {\bf $M$-bundle} (which is in fact a {\bf tractor} {\bf bundle}, see~\cite{g}, p.~7), that encodes all the possible differential operators up to order $M$ emanating from this bundle. We will then tensor two of those $M$-bundles together and project onto irreducible components. First of all, we have to define the $M$-bundles:

\section{The $M$-module}\label{mmodule}
\subsection{Formal definition}

Throughout this section, we will write $\underline{s}$ for a vector $(s_{1},...,s_{l_{0}})\in\R^{l_{0}}$.
\par
\vspace{0.3cm}
We will define a representation $\mathbb{V}_{\underline{M}}(\mathbb{E}^{0})$ of $\g$ that is induced from a finite dimensional irreducible representation $\mathbb{E}^{0}$ of $\g_{0}^{S}$ in the following way: $(\mathfrak{h}^{S})^{*}$ can be considered as a subspace of $\h^{*}$ in such a way that    
 $\{\alpha_{j}\}_{j\in J}$ are the simple roots of $\g_{0}^{S}$ with corresponding fundamental weights $\{\omega_{j}\}_{j\in J}$. The highest weight $\lambda_{0}$ of  $(\mathbb{E}^{0})^{*}$ can then be written as $\lambda_{0}=\sum_{j\in J}a_{j}\omega_{j}$ with $a_{j}\geq 0$. $\mathbb{V}_{\underline{M}}(\mathbb{E}^{0})$ is then defined to be the finite dimensional irreducible representation of $\g$ which is dual to the representation of highest weight 
$$\Lambda=\sum_{i\in I}M_{i}\omega_{i}+\lambda_{0}\in\h^{*},$$
where $\underline{M}=(M_{i_{1}},...,M_{i_{l_{0}}})\in\N^{l_{0}}$.
In the Dynkin diagram notation this is easily described. There are $l_{0}$ nodes in the Dynkin diagram for $\g$ which denote the simple roots $\alpha_{i},\;i\in I$. If we erase those nodes and adjacent edges, we obtain the Dynkin diagram for $\g_{0}^{S}$. An irreducible representation $\mathbb{E}^{0}$ of $\g_{0}^{S}$ is denoted by writing non-negative integers over the nodes of this new diagram corresponding to the highest weight of $(\mathbb{E}^{0})^{*}$. $\mathbb{V}_{\underline{M}}(\mathbb{E}^{0})$ is then denoted by writing those numbers over the uncrossed nodes and $M_{i}$ over the node that corresponds to $\alpha_{i},\; i\in I$, in the Dynkin diagram for $\g$.

\subsubsection{Example}
For $CR$ geometry, the Lie algebra of interest is $\g=\mathfrak{sl}_{n+2}\C$ with $\g_{0}^{S}\cong\mathfrak{sl}_{n}\C$. For every representation 
$$\mathbb{E}^{0}=\;\oo{a_{1}}{a_{2}}\;...\;\oo{a_{n-2}}{a_{n-1}}$$
of $\g_{0}^{S}$ and constants $M_{1},M_{2}\geq 0$ we define
$$\mathbb{V}_{\underline{M}}(\mathbb{E}^{0})=\;\ooo{M_{1}}{a_{1}}{a_{2}}\;...\;\oo{a_{n-1}}{M_{2}},$$
a representation of $\g$.

\begin{lemma}\label{lemmathirteen}
$\mathbb{V}_{\underline{M}}(\mathbb{E}^{0})$ has a composition series  
$$\mathbb{V}_{\underline{M}}(\mathbb{E}^{0})=\mathbb{V}_{0}+\mathbb{V}_{1}+...+\mathbb{V}_{N}$$
as a $\p$-module, so that $\g_{j}\mathbb{V}_{i}\subseteq \mathbb{V}_{i+j}$ and $\mathbb{V}_{0}\cong\mathbb{E}^{0}$ as a $\g_{0}^{S}$-module. 
\end{lemma}
\begin{proof}
As a highest weight module, $\mathbb{V}_{\underline{M}}(\mathbb{E}^{0})^{*}$ is a direct sum of its weight spaces and every weight has the form $\Lambda-\sum_{i=1}^{n}k_{i}\alpha_{i}$ with $k_{i}\in\N$ (see~\cite{h}, Theorem 20.2). Let $\Delta(\Lambda)$ be the set of all weights of $\mathbb{V}_{\underline{M}}(\mathbb{E}^{0})^{*}$ and define
$$N=\mathrm{max}_{\mu\in\Delta(\Lambda)}\left\{\sum_{i\in I}k_{i}^{\mu}\;:\;\mu=\Lambda-\sum_{i=1}^{n}k_{i}^{\mu}\alpha_{i}\right\}.$$
For every element $v$ in the weight space of weight $\mu$ and every $X\in\g_{\alpha}$, the element $Xv$ lies in the weight space of weight $\mu+\alpha$. This observation allows us to define a filtration by $\p$  submodules 
$$0\subset(\mathbb{V}^{0})^{*}\subset (\mathbb{V}^{1})^{*}\subset\cdots\subset (\mathbb{V}^{N})^{*}=\mathbb{V}_{M}(\mathbb{E})^{*},$$
where $(\mathbb{V}^{m})^{*}$ is defined to be the sum of all weight spaces whose weight is of the  form 
$$\Lambda-\sum_{j\in J}k_{j}\alpha_{j}-\sum_{i\in I}k_{i}\alpha_{i}\;\text{with}\;\sum_{i\in I}k_{i}\leq m.$$
The sub-quotients $\mathbb{V}_{m}^{*}=(\mathbb{V}^{m})^{*}/(\mathbb{V}^{m-1})^{*}$ give rise to a composition series
$$\mathbb{V}_{\underline{M}}(\mathbb{E}^{0})^{*}=\mathbb{V}^{*}_{N}+...+\mathbb{V}_{0}^{*},$$
whose dual
$$\mathbb{V}_{\underline{M}}(\mathbb{E}^{0})=\mathbb{V}_{0}+\mathbb{V}_{1}+...+\mathbb{V}_{N}$$
is the desired composition series.
\par
Let us examine  those sub-quotients further:
for $m=0,1,...,N$, $\mathbb{V}_{m}^{*}$ is defined to consist of those weight spaces of $\mathbb{V}_{\underline{M}}(\mathbb{E}^{0})^{*}$ whose weight is of the form
$$\Lambda-\sum_{j\in J}k_{j}\alpha_{j}-\sum_{i\in I}k_{i}\alpha_{i}\;\text{with}\;\sum_{i\in I}k_{i}=m.$$
The action of $\g_{j}$ maps $\mathbb{V}^{*}_{i+j}$ to $\mathbb{V}^{*}_{i}$. Dually we obtain a mapping $\g_{j}: \mathbb{V}_{i}\rightarrow\mathbb{V}_{i+j}$ as desired.
This also shows that the composition series can alternatively be looked at as a decomposition into $\g_{0}$-modules and since $\g_{0}$ is reductive (and $\mathfrak{z}(\g_{0})\subset\h$ acts diagonalizably), every $\mathbb{V}_{i}$ decomposes into irreducible components.
\par
Finally one can observe that $\mathbb{V}_{0}$ has acquired the structure of a $\p$-module by being the unique irreducible quotient of $\mathbb{V}_{\underline{M}}(\mathbb{E}^{0})$. Therefore it is dual to a representation of highest weight $\Lambda$. In particular, $\mathbb{V}_{0}$ is isomorphic to $\mathbb{E}^{0}$ as a $\g_{0}^{S}$-module.
\end{proof}

From now on we will write $\mathbb{V}$ instead of $\mathbb{V}_{\underline{M}}(\mathbb{E}^{0})$ assuming that $\underline{M}$ and $\mathbb{E}^{0}$ are fixed. Note that we consider $\mathbb{E}^{0}$ as a $\g_{0}$-module via the isomorphism $\mathbb{V}_{0}\cong\mathbb{E}^{0}$.

\subsection{Lie algebra cohomology}\label{cohomology}
There is a standard complex of $\g_{0}$-modules associated to a finite dimensional irreducible $\g$-module $\mathbb{V}$:
$$0\rightarrow\mathbb{V}\stackrel{\partial}{\rightarrow}\p_{+}\otimes\mathbb{V}\stackrel{\partial}{\rightarrow}
\Lambda^{2}\p_{+}\otimes\mathbb{V}\stackrel{\partial}{\rightarrow}...,$$
where we can identify $\Lambda^{p}\p_{+}\otimes\mathbb{V}\cong\mathrm{Hom}(\Lambda^{p}\g_{-},\mathbb{V})$ and therefore write the differential $\partial$ as
\begin{eqnarray*}
\partial\phi(X_{0},...,X_{p})&=&\sum_{i=0}^{p}(-1)^{i}X_{i}.\phi(X_{0},...,\hat{X}_{i},...,X_{p})\\
&&+\sum_{i<j}(-1)^{i+j}\phi([X_{i},X_{j}],X_{0},...,\hat{X}_{i},...,\hat{X}_{j},...,X_{n}),
\end{eqnarray*}
for $X_{i}\in \g_{-}$ (the hat denotes the element which is left out). One can check that $\partial^{2}=0$, see~\cite{kn}, Proposition 4.1, which allows us to define 
$$H^{p}(\g_{-},\mathbb{V})=\frac{\mathrm{ker}\;\partial:\Lambda^{p}\p_{+}\otimes\mathbb{V}\rightarrow\Lambda^{p+1}\p_{+}\otimes\mathbb{V}}{\mathrm{im}\;\partial:\Lambda^{p-1}\p_{+}\otimes\mathbb{V}\rightarrow\Lambda^{p}\p_{+}\otimes\mathbb{V}}.$$
In~\cite{k} Kostant provides an algorithm to compute these cohomology groups in terms of the Hasse diagram $\mathcal{W}^{\p}$ associated to $\p$.
% More information about Hasse diagrams can be found in~\cite{be}, p.~39.

\begin{theorem}[Kostant]
Let $\mathbb{F}$ be the dual  of a finite dimensional irreducible $\g$-module of highest weight $\lambda$. 
% and denote the dual of an irreducible $\g_{0}$-module of highest weight $\mu$ by $\mathbb{F}(\mu)$. 
Then as $\g_{0}$-modules
$$H^{p}(\g_{-},\mathbb{F})=\bigoplus_{\stackrel{w\in\mathcal{W}^{\p}}{l(w)=p}}\mathbb{F}_{w.\lambda},$$
where $\mathbb{F}_{w.\lambda}$ denotes the dual of the representation of highest weight $w.\lambda$.
\end{theorem}
\begin{proof}
This theorem was originally proved in~\cite{ko} and formulated in a suitable way for our purposes in~\cite{b}.
\end{proof} 

\begin{corollary}\label{corseven}
In the situation described above, we have
$$H^{0}(\g_{-},\mathbb{V})=\mathbb{E}^{0}\quad\text{and}\quad H^{1}(\g_{-},\mathbb{V})=\bigoplus_{i\in I}\circledcirc^{M_{i}+1}\g_{1}^{i}\circledcirc\mathbb{E}^{0},$$
where $\circledcirc$ denotes the Cartan product.
\end{corollary}
\begin{proof}
The first statement is true in general and obvious. The second statement follows from the fact that the elements of length one in the Hasse diagram $\mathcal{W}^{\p}$ are exactly those simple reflections that correspond to the simple roots $\alpha_{i},\; i\in I$ and $(\g_{1}^{i})^{*}$ has highest weight $-\alpha_{i}$. Then we compute
$$\alpha_{i}.\Lambda=\Lambda-(M_{i}+1)\alpha_{i}$$
and note that this is exactly the highest weight of the dual representation given by $(\mathbb{E}^{0})^{*}\circledcirc\circledcirc^{M_{i}+1}(\g_{1}^{i})^{*}$.
\end{proof}
Let us have a closer look at the differential $\partial$. Since it is a $\g_{0}$ module homomorphism it preserves the grading and so we must have 
$$\partial:\mathbb{V}_{i}\rightarrow\bigoplus_{j=1}^{k_{0}}\g_{j}\otimes\mathbb{V}_{i-j},$$
where we set $\mathbb{V}_{l}=0$ for $l<0$.
The kernel of the first differential is $\mathbb{E}^{0}$, so the mapping $\partial:\mathbb{V}_{i}\rightarrow\bigoplus_{j=1}^{k_{0}}\g_{j}\otimes\mathbb{V}_{i-j}$ is injective for
all $i\geq 1$. The cohomology of the second differential is 
$$\bigoplus_{i\in I}\circledcirc^{M_{i}+1}\g_{1}^{i}\circledcirc\mathbb{E}^{0}$$
with
$$\circledcirc^{M_{i}+1}\g_{1}^{i}\circledcirc\mathbb{E}^{0}\subset\bigoplus_{j=1}^{k_{0}}\g_{j}\otimes\mathbb{V}_{M_{i}+1-j}$$
and all those have multiplicity one.
Schematically this looks like
%$$\begin{array}{ccccccccccc}
%\mathbb{V}&=& \mathbb{V}_{0}&+&\mathbb{V}_{1}&+&\mathbb{V}_{2}&+&...&+&\mathbb{V}_{N}\\
%\partial \downarrow &&&\partial \swarrow&&\partial\swarrow&&&&\partial\swarrow&\\
%\p_{+}\otimes\mathbb{V}&=& \g_{1}\otimes\mathbb{V}_{0}&+&\begin{array}{c}
%\g_{1}\otimes\mathbb{V}_{1}\\
%\oplus\\
%\g_{2}\otimes \mathbb{V}_{0}\end{array}&+&\begin{array}{c}
%\g_{1}\otimes\mathbb{V}_{2}\\
%\oplus\\
%\g_{2}\otimes\mathbb{V}_{1}\\
%\oplus\\
%\g_{3}\mathbb{V}_{0}\end{array}&+&...&+&\g_{k_{0}}\otimes\mathbb{V}_{N}\\
%\partial \downarrow &&&\partial \swarrow&&\partial\swarrow&&&&\partial\swarrow&\\
%\Lambda^{2}\p_{+}\otimes\mathbb{V}&=& \Lambda^{2}\g_{1}\otimes\mathbb{V}_{0}&+&\begin{array}{c}
%\g_{1}\otimes\g_{2}\otimes\mathbb{V}_{0}\\
%\oplus\\
%\Lambda^{2}\g_{1}\otimes\mathbb{V}_{1}\end{array}&&+&...&+&\Lambda^{2}\g_{k_{0}}\otimes\mathbb{V}_{N}
%\end{array}.$$
$$\begin{array}{ccccccccccc}
\mathbb{V}&=& \mathbb{V}_{0}&+&\mathbb{V}_{1}&+&\mathbb{V}_{2}&+&\mathbb{V}_{3}&+&\cdots\\
\partial \downarrow &&&\partial \swarrow&&\partial\swarrow&&\partial\swarrow&&&\\
\p_{+}\otimes\mathbb{V}&=& \g_{1}\otimes\mathbb{V}_{0}&+&\begin{array}{c}
\g_{1}\otimes\mathbb{V}_{1}\\
\oplus\\
\g_{2}\otimes \mathbb{V}_{0}\end{array}&+&\begin{array}{c}
\g_{1}\otimes\mathbb{V}_{2}\\
\oplus\\
\g_{2}\otimes\mathbb{V}_{1}\\
\oplus\\
\g_{3}\otimes\mathbb{V}_{0}\end{array}&+&\cdots&&\\
\partial \downarrow &&&\partial \swarrow&&\partial\swarrow&&&&&\\
\Lambda^{2}\p_{+}\otimes\mathbb{V}&=& \Lambda^{2}\g_{1}\otimes\mathbb{V}_{0}&+&\begin{array}{c}
\g_{1}\wedge\g_{2}\otimes\mathbb{V}_{0}\\
\oplus\\
\Lambda^{2}\g_{1}\otimes\mathbb{V}_{1}\end{array}&+&\cdots&&&
\end{array}.$$

\subsubsection{Weighted gradings}\label{inductivefiltered}
Before we proceed let us give an inductive construction of $\mathfrak{U}_{p}(\p_{+})$ as defined in~\ref{tangentialfiltration}. 
\begin{description}
\item{(1)}
$\mathfrak{U}_{0}(\p_{+})=\C$
\item{(2)}
$\mathfrak{U}_{1}(\p_{+})=\g_{1}$
\item{(3)}
$$\mathfrak{U}_{2}(\p_{+})=\left(\begin{array}{c}
\g_{1}\otimes\mathfrak{U}_{1}(\p_{+})\\
\oplus\\
\g_{2}\otimes\mathfrak{U}_{0}(\p_{+})
\end{array}\right)/J_{2},$$
where 
$$J_{2}=\{ X\otimes Y- Y\otimes X-[X,Y]\;:\; X,Y\in\g_{1}\}.$$
\item{}\vdots
\item{($i$+1)}
$$\mathfrak{U}_{i}(\p_{+})=\left(\bigoplus_{j=1}^{k_{0}}\g_{j}\otimes\mathfrak{U}_{i-j}(\p_{+})\right)/J_{i},$$
where 
$$J_{i}=\left\{X\otimes Yu- Y\otimes Xu-[X,Y]\otimes u\;:\begin{tabular}{c}
$X\in\g_{r},Y\in\g_{s},u\in\mathfrak{U}_{i-(r+s)}(\p_{+})$\\
$\text{and}\; 1\leq r,s\leq k_{0}$
\end{tabular}\right\}.$$
\end{description}
The construction for $\g_{-}$ is exactly analogous with all integers being negative.
We will also need some notation for the grading of $\Lambda^{2}\g_{-}$ which is given by 
$$\Lambda^{2}\g_{-}=\sum_{j<-1}\Lambda^{2}_{j}\g_{-},$$
with 
$$\Lambda^{2}_{j}\g_{-}=\bigoplus_{p+q=j}\g_{p}\wedge \g_{q},$$
see~\cite{y}, 2.4. Schematically this looks like
$$\Lambda^{2}\g_{-}=\underbrace{\Lambda^{2}\g_{-1}}_{\Lambda^{2}_{-2}\g_{-}}+\underbrace{\g_{-1}\wedge\g_{-2}}_{\Lambda_{-3}^{2}\g_{-}}+\underbrace{\begin{array}{c}
\g_{-1}\wedge\g_{-3}\\
\oplus\\
\Lambda^{2}\g_{-2}
\end{array}}_{\Lambda^{2}_{-4}\g_{-}}+\underbrace{\begin{array}{c}
\g_{-1}\wedge\g_{-4}\\
\oplus\\
\g_{-2}\wedge\g_{-3}
\end{array}}_{\Lambda^{2}_{-5}\g_{-}}+...+\underbrace{\Lambda^{2}\g_{-k_{0}}}_{\Lambda_{-2k_{0}}^{2}\g_{-}}\;.$$

\begin{proposition}\label{propseven}
We have an isomorphism
$$\mathbb{V}_{i}\cong\mathfrak{U}_{i}(\p_{+})\otimes\mathbb{E}^{0}$$
for all $0\leq i\leq \mathrm{min}_{i\in I}\{M_{i}\}$. Moreover $\mathbb{V}_{j}\subset \mathfrak{U}_{j}(\p_{+})\otimes\mathbb{E}^{0}$ for all $j$.

\end{proposition}
\begin{proof}
We will prove this proposition by induction. The case $i=0$ is trivial. Let  us assume that the proposition is true for all integers smaller than $i$. Then we can use the induction hypothesis to construct the following commutative diagram.
$$\begin{CD}
\mathbb{V}_{i} @>{\partial}>> \bigoplus_{j=-1}^{-k_{0}}\mathrm{Hom}(\g_{j},\mathbb{V}_{i+j})@>{\partial}>>\bigoplus_{l=-2}^{-2k_{0}}\mathrm{Hom}(\Lambda_{l}^{2}\g_{-},\mathbb{V}_{i+l})\\
@V{\mathrm{id}}VV @V{\iota}VV @V{\tau}VV\\
\mathbb{V}_{i}@>{\partial}>>\bigoplus_{j=-1}^{-k_{0}}\mathrm{Hom}(\g_{j}\otimes\mathfrak{U}_{-(i+j)}(\g_{-}),\mathbb{V}_{0})
@>{\partial}>>\bigoplus_{l=-2}^{-2k_{0}}\mathrm{Hom}(\Lambda^{2}_{l}\g_{-}\otimes\mathfrak{U}_{-(i+l)}(\g_{-}),\mathbb{V}_{0})
\end{CD}$$
We can make this diagram commute by setting
\begin{eqnarray*}
(\iota(\phi))_{j}(X\otimes u)&=&u^{0}.\phi_{j}(X)\;\text{for}\;X\otimes u\in\g_{j}\otimes\mathfrak{U}_{-(i+j)}(\g_{-}),\\
(\tau(\lambda))_{l}(X\wedge Y\otimes u)&=&u^{0}.\lambda_{l}(X\wedge Y)\;\text{for}\;X\wedge Y\otimes u\in\Lambda_{l}^{2}\g_{-}\otimes\mathfrak{U}_{-(i+l)}(\g_{-}),\\
(\partial\phi)_{l}(X\wedge Y)&=&X.\phi_{s}(Y)-Y.\phi_{r}(X)-\phi_{l}([X,Y])\;\text{for}\;X\wedge Y\in\g_{r}\wedge\g_{s}\subset\Lambda_{l}^{2}\g_{-}\\
&\text{and}&\\
(\partial\psi)_{l}(X\wedge Y\otimes u)&=& \psi_{r}(X\otimes Yu)-\psi_{s}(Y\otimes Xu)-\psi_{l}([X,Y]\otimes u)
\end{eqnarray*}
for $X\wedge Y\otimes u\in\g_{r}\wedge\g_{s}\otimes\mathfrak{U}_{-(i+l)}(\g_{-})\subset\Lambda_{l}^{2}\g_{-}\otimes\mathfrak{U}_{-(i+l)}(\g_{-})$. It is straightforward to check that
$$\partial\circ\iota \phi=\tau\circ\partial\phi.$$
Note that $(Xu)^{0}=-u^{0}X$ for all $u\in\mathfrak{U}(\g_{-})$ and $X\in \g_{-}$. Corollary~\ref{corseven} shows that the sequence
$$0\rightarrow \mathbb{V}_{i}\stackrel{\partial}{\rightarrow}\bigoplus_{j=-1}^{-k_{0}}\mathrm{Hom}(\g_{j},\mathbb{V}_{i+j})\stackrel{\partial}{\rightarrow}\bigoplus_{l=-2}^{-2k_{0}}\mathrm{Hom}(\Lambda^{2}_{l}\g_{-},\mathbb{V}_{i+l})$$
is exact for $i\leq \mathrm{min}_{i\in I}\{M_{i}\}$. Then we can use the description of $\mathfrak{U}_{i}(\p_{+})$ from above to deduce that for $i\leq \mathrm{min}_{i\in I}\{M_{i}\}$
$$\mathbb{V}_{i}\cong\mathrm{ker}\;\partial\cong\mathrm{Hom}(\mathfrak{U}_{-i}(\g_{-}),\mathbb{V}_{0})\cong\mathfrak{U}_{i}(\p_{+})\otimes\mathbb{E}^{0}.$$
The second statement follows from the fact that 
\begin{eqnarray*}
\mathbb{V}_{i}&\cong& \mathrm{im}\;\partial:\mathbb{V}_{i}\rightarrow \bigoplus_{j=-1}^{-k_{0}}\mathrm{Hom}(\g_{j},\mathbb{V}_{i+j})\\
&\subset&\mathrm{ker}\;\partial:\bigoplus_{j=-1}^{-k_{0}}\mathrm{Hom}(\g_{j},\mathbb{V}_{i+j})\rightarrow\bigoplus_{l=-2}^{-2k_{0}}\mathrm{Hom}(\Lambda^{2}_{l}\g_{-},\mathbb{V}_{i+l})\\
&\cong&\mathfrak{U}_{i}(\p_{+})\otimes\mathbb{E}^{0}
\end{eqnarray*}
for all $i\geq 1$.
\end{proof}

\subsection{Tensor products}
The next step is to look at a tensor product $\mathbb{V}_{\underline{M}}(\mathbb{E}^{0})\otimes\mathbb{V}_{\underline{M}}(\mathbb{F}^{0})$ and decompose it into irreducible $\g$-modules that themselves will have a composition series as $\p$-modules. The composition factors of all the irreducible components will then make up the composition factors of the tensor product. \par
We will use the following colloquialism: if 
$$\mathbb{V}=\mathbb{V}_{0}+...+\mathbb{V}_{k}+...+\mathbb{V}_{N}$$
is the composition series of a $\g$-module, then we will refer to elements in $\mathbb{V}_{k}$ as lying in the $k$-th {\bf  slot}.

\subsubsection{Remark}
Let $\Lambda$ and $\lambda_{0}$ be defined as above and define $\mathbb{V}_{\underline{M}}(\mathbb{F}^{0})$ analogously with highest weights $\Sigma$ and $\sigma_{0}$ of $\mathbb{V}_{\underline{M}}(\mathbb{F}^{0})^{*}$ and $(\mathbb{F}^{0})^{*}$ respectively.
%All irreducible components of $\mathbb{V}_{0}\otimes\mathbb{W}_{0}$ have a highest weight that is of the form $\mu=\Lambda+\Sigma-\sum_{s\in J}k_{s}\alpha_{s}$ for $k_{s}\geq 0$, i.e.~the number over the $j$-th crossed through node will be
%$$(\mu,\alpha_{i_{j}}^{\vee})=2M_{i_{j}}-\sum_{s\in J}k_{s}(\alpha_{s},\alpha_{i_{j}}^{\vee})\geq 2M_{i_{j}},$$
%since $(\alpha_{s},\alpha_{i_{j}}^{\vee})\leq 0$ for all $i=1,...,l$. 
All irreducible components in the $k$-th slot of $\mathbb{V}_{\underline{M}}(\mathbb{E}^{0})\otimes\mathbb{V}_{\underline{M}}(\mathbb{F}^{0})$ are dual to modules with a highest weight of the form $$\mu=\Lambda+\Sigma-\sum_{i\in I}k_{i}\alpha_{i}-\sum_{j\in J}k_{j}\alpha_{j},$$
with $\sum_{i\in I}k_{i}=k$ and $k_{i},k_{j}\geq 0.$ Thus the number over the $I\ni i_{s}$-th node, $s=1,...,l_{0}$, (which is crossed through) will be
\begin{eqnarray*}
B(\mu,\alpha_{i_{s}}^{\vee})&=&2M_{i_{s}}-\sum_{i\in I}k_{i}B(\alpha_{i},\alpha_{i_{s}}^{\vee})-\sum_{j\in J}k_{j}B(\alpha_{j},\alpha_{i_{s}}^{\vee})\\
&=&2M_{i_{s}}-2k_{i_{s}}-\sum_{i_{s}\not=r=1}^{l}k_{r}B(\alpha_{r},\alpha_{i_{s}}^{\vee})\\
&\geq& 2(M_{i_{s}}-k).
\end{eqnarray*}

\begin{proposition}\label{propeight}
Let $\mathbb{E}^{0}$ and $\mathbb{F}^{0}$ be two irreducible representations of $\g_{0}^{S}$.
If we have $k\leq M=\mathrm{min}_{i\in I}\{M_{i}\}$, then for every irreducible component $\mathbb{H}$ of $\mathfrak{U}_{k}(\p_{+})\otimes\mathbb{E}^{0}\otimes\mathbb{F}^{0}$ there is a ($\p$-module) projection
$$\mathbb{V}_{\underline{M}}(\mathbb{E}^{0})\otimes\mathbb{V}_{\underline{M}}(\mathbb{F}^{0})\rightarrow \mathbb{H}.$$

\end{proposition}\begin{proof}
Schematically, we have the following situation
$$\mathbb{V}_{\underline{M}}(\mathbb{E}^{0})=\mathbb{E}^{0}+\mathfrak{U}_{1}(\p_{+})\otimes \mathbb{E}^{0}+\mathfrak{U}_{2}(\p_{+})\otimes \mathbb{E}^{0}+...$$
and
$$\mathbb{V}_{\underline{M}}(\mathbb{F}^{0})=\mathbb{F}^{0}+\mathfrak{U}_{1}(\p_{+})\otimes \mathbb{F}^{0}+\mathfrak{U}_{2}(\p_{+})\otimes \mathbb{F}^{0}+...\;.$$
Now the tensor product looks like
$$\mathbb{E}^{0}\otimes\mathbb{F}^{0}+\begin{array}{c}
\mathbb{E}^{0}\otimes\mathfrak{U}_{1}(\p_{+})\otimes \mathbb{F}^{0}\\
\oplus\\
\mathfrak{U}_{1}(\p_{+})\otimes \mathbb{E}^{0}\otimes \mathbb{F}^{0}
\end{array}
+\begin{array}{c}
\mathbb{E}^{0}\otimes\mathfrak{U}_{2}(\p_{+})\otimes \mathbb{F}^{0}\\
\oplus\\
\mathfrak{U}_{1}(\p_{+})\otimes \mathbb{E}^{0}\otimes\mathfrak{U}_{1}(\p_{+})\otimes \mathbb{F}^{0}\\
\oplus\\
\mathfrak{U}_{2}(\p_{+})\otimes \mathbb{E}^{0}\otimes\mathbb{F}^{0}
\end{array}+...\;.$$
Every irreducible component $\mathbb{G}$ of $\mathbb{E}^{0}\otimes\mathbb{F}^{0}$ (as $\g_{0}^{S}$-modules) corresponds to an irreducible component $\mathbb{U}$ of 
$\mathbb{V}_{\underline{M}}(\mathbb{E}^{0})\otimes\mathbb{V}_{\underline{M}}(\mathbb{F}^{0})$ that has a composition series that starts with $\mathbb{G}$ and then continues with $\mathfrak{U}_{1}(\p_{+})\otimes\mathbb{G}+\mathfrak{U}_{2}(\p_{+})\otimes \mathbb{G}+...$. We will say that the composition series is predictable up to the $x$-th slot, if $\mathbb{U}_{j}\cong \mathfrak{U}_{j}(\p_{+})\otimes\mathbb{G}$ for all $j\leq x$, as $\g_{0}^{S}$-modules. Using Proposition~\ref{propseven} we know that $\mathbb{U}$ composes predictably up to the $x$-th slot if the minimum of the numbers over the crossed through nodes in $\mathbb{G}$ is $x$.
\par
Removing all those composition factors that correspond to irreducible components of $\mathbb{E}^{0}\otimes\mathbb{F}^{0}$ from the composition series of the two $\g$-modules leaves nothing in the zeroth slot, exactly one copy of $\mathbb{E}^{0}\otimes\mathfrak{U}_{1}(\p_{+})\otimes \mathbb{F}^{0}$ in the first slot, one copy of each $\mathfrak{U}_{1}(\p_{+})\otimes \mathbb{E}^{0}\otimes\mathfrak{U}_{1}(\p_{+})\otimes \mathbb{F}^{0}$ and $\mathfrak{U}_{2}(\p_{+})\otimes \mathbb{E}^{0}\otimes\mathbb{F}^{0}$
in the second slot and so forth. Therefore the next irreducible components of $\mathbb{V}_{\underline{M}}(\mathbb{E}^{0})\otimes\mathbb{V}_{\underline{M}}(\mathbb{F}^{0})$ all have a composition series that
starts with an irreducible component of $\mathbb{E}^{0}\otimes\mathfrak{U}_{1}(\p_{+})\otimes \mathbb{F}^{0}$. Removing the corresponding composition factors again leaves nothing in the first two slots, exactly one copy of
$\mathfrak{U}_{2}(\p_{+})\otimes \mathbb{E}^{0}\otimes\mathbb{F}^{0}$ in the second slot and so forth.
So the next irreducible components of the $\g$-module tensor product have a composition series that starts with an irreducible component of the $\g_{0}^{S}$-module tensor product $\mathfrak{U}_{2}(\p_{+})\otimes \mathbb{E}^{0}\otimes\mathbb{F}^{0}$. All this goes well as long as all the compositions series are predictable. This is the case exactly up to the $M$-th slot.
\par
The remark above ensures that in the $k\leq M$-th slot of $\mathbb{V}_{\underline{M}}(\mathbb{E}^{0})\otimes\mathbb{V}_{\underline{M}}(\mathbb{F}^{0})$ the lowest number over the $I\ni i_{s}$-th node, $s=1,...,l_{0}$, (which is crossed through) is bigger or equal to
$2(M_{i_{s}}-k)$. Some of the factors in here correspond to irreducible components of  $\mathbb{V}_{\underline{M}}(\mathbb{E}^{0})\otimes\mathbb{V}_{\underline{M}}(\mathbb{F}^{0})$  as $\g$-representations that themselves decompose according to plan up to the $2(M_{i_{s}}-k)$th slot, which corresponds in the big composition series to the $2M_{i_{s}}-k$-th slot. Therefore the predicted decomposition is alright up to $k=M$. There could be (and in general this happens) more composition factors of $\mathbb{V}_{\underline{M}}(\mathbb{E}^{0})\otimes\mathbb{V}_{\underline{M}}(\mathbb{F}^{0})$, but those correspond to higher order pairings. 
\par
Note that $\mathbb{H}$ is defined to be the first composition factor of a composition series and hence acquires the structure of a $\p$-module. 
\end{proof}

\subsubsection{Remark}\label{geometricweights}
Each of the $\g$-modules and $\p$ modules considered above induces an associated vector bundle on our manifold $\mathcal{M}$. These may be tensored by a line bundle that is induced by the irreducible representation which is dual to the one dimensional representation of $\mathfrak{z}(\g_{0})$ with highest weight $\sum_{i\in I}(k_{i}-M_{i})\omega_{i}$. In the Dynkin diagram notation this corresponds to having $k_{i}-M_{i}$ over the $I\ni i$-th node (which is crossed through) and zeros over the uncrossed nodes. Tensoring with this representation does not change the overall structure of the composition series. It just changes the character by which $\mathfrak{z}(\g_{0})$ acts. Let us denote this representation and the associated bundle by $\mathcal{O}(\underline{k}-\underline{M})$ and denote the tensor product of an arbitrary representation $\mathbb{U}$ with this representation by $\mathbb{U}(\underline{k}-\underline{M})$.
\par
The $\p$ module $\mathbb{V}_{0}(\underline{k}-\underline{M})$ is dual to a representation of highest weight 
$$\sum_{i\in I}k_{i}\omega_{i}+\sum_{j\in J}a_{j}\omega_{j}$$
and by choosing $\mathbb{E}^{0}$ and $\{k_{i}\}_{i\in I}$ correctly, we can write every finite dimensional irreducible $\p$-module $\mathbb{E}$ as $\mathbb{E}=\mathbb{V}_{0}(\underline{k}-\underline{M})$ of some module $\mathbb{V}_{\underline{M}}(\mathbb{E}^{0})(\underline{k}-\underline{M})$. The idea is now to define invariant linear differential splitting operators 
$$E\rightarrow V_{\underline{M}}(\mathbb{E}^{0})(\underline{k}-\underline{M})\;\text{and}\;F\rightarrow V_{\underline{M}}(\mathbb{F}^{0})(\underline{l}-\underline{M}).$$
Once we have such splitting operators, we can tensor 
$$\mathbb{V}_{\underline{M}}(\mathbb{E}^{0})(\underline{k}-\underline{M})\otimes \mathbb{V}_{\underline{M}}(\mathbb{F}^{0})(\underline{l}-\underline{M})$$
%=\mathbb{V}_{\underline{M}}(\mathbb{E}^{0})\otimes \mathbb{V}_{\underline{M}}(\mathbb{F}^{0})(\underline{k+l}-\underline{2M})$$
together and project onto the first composition factor of every irreducible component. This is clearly an invariant bilinear differential pairing between sections of 
$E$ and $F$.

\section{The splitting operator}
In this section we will define splitting operators
$$E\rightarrow V_{\underline{M}}(\mathbb{E}^{0})(\underline{k}-\underline{M})$$
for an arbitrary irreducible finite dimensional $\p$-module $\mathbb{E}=\mathbb{V}_{0}(\underline{k}-\underline{M})$ .

\subsection{Higher order curved Casimir operators}
This subsection introduces splitting operators for general curved parabolic geometries. These splitting operators are defined with the help of general invariant operators, the higher order curved Casimir operators. These operators are a generalization of the curved Casimir operator introduced by A.~\v{C}ap and V.~Sou\v{c}ek in~\cite{cs3}. Throughout this section we will implicitly assume that $\g$ is a classical simple Lie algebra $A_{l},B_{l},C_{l},D_{l}$ or $G_{2}$.

\begin{definition}\label{generatorsone}
{\rm
Let $\g$ be a simple Lie algebra of rank $l$ with basis $\{X_{\mu}\}$. The Lie algebra is determined by the structure equation
$$[X_{\mu},X_{\nu}]=C^{\lambda}_{\mu\nu}X_{\lambda}.$$
Following~\cite{o}, we will call a set $S=\{S_{\mu}\}$ in $T(\g)$ a {\bf vector operator}, if
$$[X_{\mu},S_{\nu}]=C^{\lambda}_{\mu\nu}S_{\lambda}.$$
Let 
$$\bar{g}_{\mu\nu}=\mathrm{tr}(\mathrm{ad}(X_{\mu})\circ\mathrm{ad}(X_{\nu}))$$
be the Killing form with inverse $\bar{g}^{\mu\nu}$.
Fix a reference representation $\phi:\g\rightarrow \mathfrak{gl}(\mathbb{A})$ and define
$$g_{\mu_{1}...\mu_{p}}=\mathrm{tr}(\phi(X_{\mu_{1}})\circ\cdots\circ\phi(X_{\mu_{p}}))$$
and
$$g^{\mu_{1}...\mu_{p}}=\mathrm{tr}(\phi(X^{\mu_{1}})\circ\cdots\circ\phi(X^{\mu_{p}})),$$
where $X^{\mu}=\bar{g}^{\mu\nu}X_{\nu}$. Then we define elements in $T(\g)$ by
\begin{eqnarray*}
K^{(p)}&=&g^{\mu_{1}...\mu_{p}}X_{\mu_{1}}\otimes\cdots\otimes X_{\mu_{p}}\;\;\text{and}\\
S^{(p)}_{\mu}&=&g_{\mu\mu_{1}...\mu_{p}}X^{\mu_{1}}\otimes\cdots\otimes X^{\mu_{p}}.
\end{eqnarray*}
}\end{definition}

\begin{lemma}
$S^{(p)}=\{S^{(p)}_{\mu}\}$ is a vector operator and $K^{(p)}$ is an element in $\tilde{\mathcal{Z}}(\g)$, the center of $T(\g)$.
\end{lemma}
\begin{proof}
Using the cyclic property of the trace operator, it is easy to see that
\begin{eqnarray*}
0&=&\sum_{j=1}^{p}\mathrm{tr}(\phi(X_{\mu_{1}})\circ\cdots\circ\phi(X_{\mu_{j-1}})\circ\phi([X_{\nu},X_{\mu_{j}}])\circ\phi(X_{\mu_{j+1}})\circ\cdots\circ\phi(X_{\mu_{p}}))\\
&=&\sum_{j=1}^{p}C_{\nu\mu_{j}}^{\lambda}\mathrm{tr}(\phi(X_{\mu_{1}})\circ\cdots\circ\phi(X_{\mu_{j-1}})\circ\phi(X_{\lambda})\circ\phi(X_{\mu_{j+1}})\circ\cdots\circ\phi(X_{\mu_{p}}))\\
&=&\sum_{j=1}^{p}C_{\nu\mu_{j}}^{\lambda}g_{\mu_{1}...\mu_{j-1}\lambda\mu_{j+1}...\mu_{p}}.
\end{eqnarray*}
In other words
$$-\sum_{j=1}^{p}C^{\lambda}_{\nu\mu_{j}}g_{\kappa\mu_{1}...\mu_{j-1}\lambda\mu_{j+1}...\mu_{p}}=C^{\lambda}_{\nu\kappa}g_{\lambda\mu_{1}...\mu_{p}}.$$
The same equation holds for the adjoint representation, i.e.
$$\bar{g}_{\lambda\mu}C^{\lambda}_{\nu\kappa}+\bar{g}_{\kappa\lambda}C^{\lambda}_{\nu\mu}=0\;\Rightarrow\;\bar{g}_{\mu\lambda}C^{\lambda}_{\nu\kappa}=\bar{g}_{\kappa\lambda}C^{\lambda}_{\mu\nu}.$$
This yields
\begin{equation*}
[X_{\nu},X^{\mu}]=\bar{g}^{\mu\kappa}C_{\nu\kappa}^{\lambda}X_{\lambda}=-C_{\nu\lambda}^{\mu}X^{\lambda}.
\end{equation*}
%which follows from
%$$\mathrm{tr}(\mathrm{ad}(X_{\mu})\circ\mathrm{ad}([X_{\nu},X_{\kappa}]))=\mathrm{tr}(\mathrm{ad}(X_{\kappa})\circ\mathrm{ad}([X_{\mu},X_{\nu}])),$$
%i.e.~$\bar{g}_{\mu\lambda}C^{\lambda}_{\nu\kappa}=\bar{g}_{\kappa\lambda}C^{\lambda}_{\mu\nu}$.
Then we compute
\begin{eqnarray*}
[X_{\nu},S^{(p)}_{\kappa}]&=& -\sum_{j=1}^{p}C_{\nu\lambda}^{\mu_{j}}g_{\kappa\mu_{1}...\mu_{p}}X^{\mu_{1}}\otimes\cdots\otimes X^{\mu_{j-1}}\otimes X^{\lambda}\otimes X^{\mu_{j+1}}\otimes\cdots\otimes X^{\mu_{p}}\\
&=&-\sum_{j=1}^{p}C_{\nu\mu_{j}}^{\lambda}g_{\kappa\mu_{1}...\mu_{j-1}\lambda\mu_{j+1}...\mu_{p}}X^{\mu_{1}}\otimes\cdots\otimes X^{\mu_{j-1}}\otimes X^{\mu_{j}}\otimes X^{\mu_{j+1}}\otimes\cdots\otimes X^{\mu_{p}}\\
&=&C_{\nu\kappa}^{\lambda}S^{(p)}_{\lambda}.
\end{eqnarray*}
This proves the first claim. The second claim follows immediately from the fact that
$$K^{(p+1)}=X^{\lambda}\otimes S^{(p)}_{\lambda},$$
so
\begin{eqnarray*}
[X_{\nu},K^{(p+1)}]&=&[X_{\nu},X^{\lambda}]\otimes S^{(p)}_{\lambda}+X^{\lambda}\otimes[X_{\nu},S^{(p)}_{\lambda}]\\
&=&-C_{\nu\mu}^{\lambda}X^{\mu}\otimes S^{(p)}_{\lambda}+X^{\lambda}\otimes C_{\nu\lambda}^{\kappa}S^{(p)}_{\kappa}\\
&=&0.
\end{eqnarray*}
\end{proof}

%\subsubsection{Remark}
%In fact $K^{(p+1)}$ lies in the center of $\otimes^{p}\g$, because we have not used the relation $X\otimes Y-Y\otimes X-[X,Y]$ in the proof of the previous lemma.

\subsubsection{Example 1}
If $\g$ is simple, we can take $\mathbb{A}=\g$ and the adjoint representation as our reference representation. Any other choice of representation would lead to $g_{\mu_{1}\mu_{2}}=C\bar{g}_{\mu_{1}\mu_{2}}$ for some non-zero constant $C$, see~\cite{o}. For $\mathfrak{sl}_{2}\C$ with standard basis $(x,h,y)$ and the standard representation on $\C^{2}$ we obtain
$$g_{\mu\nu}=\left(\begin{array}{ccc}
0&0&1\\
0&2&0\\
1&0&0
\end{array}\right),\;g^{\mu\nu}=\left(\begin{array}{ccc}
0&0&1\\
0&\frac{1}{2}&0\\
1&0&0
\end{array}\right)$$
and hence $K^{(2)}=x\otimes y+\frac{1}{2}h\otimes h+y\otimes x$. The usual Casimir operator, see~\cite{h}, can be computed by projecting $\tilde{c}=\bar{g}^{\mu_{1}\mu_{2}}X_{\mu_{1}}\otimes X_{\mu_{2}}$ onto $c\in\mathfrak{U}(\g)$. For $\mathfrak{sl}_{2}\C$ one easily obtains
$$\bar{g}_{\mu\nu}=\left(\begin{array}{ccc}
0&0&4\\
0&8&0\\
4&0&0
\end{array}\right),\;\bar{g}^{\mu\nu}=\left(\begin{array}{ccc}
0&0&\frac{1}{4}\\
0&\frac{1}{8}&0\\
\frac{1}{4}&0&0
\end{array}\right)$$
and hence $\tilde{c}=\frac{1}{4}x\otimes y+\frac{1}{8}h\otimes h+\frac{1}{4}y\otimes x$. This  yields $c=\frac{1}{2}xy+\frac{1}{8}h^{2}-\frac{1}{4}h$.

\subsubsection{Example 2}
In order to obtain a non-trivial example, one has to go through a rather lengthly calculation. So the patient reader is kindly asked to bear with us for a while (while the impatient reader may skip this part):
\par
We will explicitly compute the element $K^{(3)}$ for $\mathfrak{sl}_{3}\C$ and see how it acts on a generalized Verma module $M_{\p}(\mathbb{V}_{\lambda})$ of highest weight $\lambda\in\h^{*}$. Firstly, we take a basis
$$h_{1}=\left(\begin{array}{ccc}
1&0&0\\
0&-1&0\\
0&0&0
\end{array}\right),\;h_{2}=\left(\begin{array}{ccc}
0&0&0\\
0&1&0\\
0&0&-1
\end{array}\right)$$
$$x_{1}=\left(\begin{array}{ccc}
0&1&0\\
0&0&0\\
0&0&0
\end{array}\right),\;x_{2}=\left(\begin{array}{ccc}
0&0&0\\
0&0&1\\
0&0&0
\end{array}\right)$$
$$y_{1}=\left(\begin{array}{ccc}
0&0&0\\
1&0&0\\
0&0&0
\end{array}\right),\;y_{2}=\left(\begin{array}{ccc}
0&0&0\\
0&0&0\\
0&1&0
\end{array}\right)$$
$$z_{1}=\left(\begin{array}{ccc}
0&0&1\\
0&0&0\\
0&0&0
\end{array}\right),\;z_{2}=\left(\begin{array}{ccc}
0&0&0\\
0&0&0\\
1&0&0
\end{array}\right)$$
of $\mathfrak{sl}_{3}\C$ and compute
$$\bar{g}^{\mu\nu}=\frac{1}{18}\left(\begin{array}{cccccccc}
2&1&0&0&0&0&0&0\\
1&2&0&0&0&0&0&0\\
0&0&0&0&3&0&0&0\\
0&0&0&0&0&3&0&0\\
0&0&3&0&0&0&0&0\\
0&0&0&3&0&0&0&0\\
0&0&0&0&0&0&0&3\\
0&0&0&0&0&0&3&0
\end{array}\right).$$
Thus the dual basis to
$$\{h_{1},h_{2},x_{1},x_{2},y_{1},y_{2},z_{1},z_{2}\}$$
is 
$$\frac{1}{18}\{2h_{1}+h_{2},h_{1}+2h_{2},3y_{1},3y_{2},3x_{1},3x_{2},3z_{2},3z_{1}\}.$$
Now we compute $g_{\mu\nu\kappa}$ with respect to the standard representation. The only non-zero elements in $g_{\mu\nu\kappa}$ are
$$\begin{array}{cccccc}
g_{112}=1&g_{121}=1&g_{122}=-1&g_{135}=1&g_{146}=-1&g_{153}=-1\\
g_{178}=1&g_{211}=1&g_{212}=-1&g_{221}=-1&g_{246}=1&g_{253}=1\\
g_{264}=-1&g_{287}=-1&g_{315}=-1&g_{325}=1&g_{348}=1&g_{351}=1\\
g_{426}=-1&g_{461}=-1&g_{462}=1&g_{483}=1&g_{513}=1&g_{531}=-1\\
g_{532}=1&g_{576}=1&g_{614}=-1&g_{624}=1&g_{642}=-1&g_{657}=1\\
g_{728}=-1&g_{765}=1&g_{781}=1&g_{817}=1&g_{834}=1&g_{872}=-1.
\end{array}$$
This allows us to compute $18K^{(3)}$ to be
$$\begin{array}{c}
 (2h_{1}+h_{2})\otimes  (2h_{1}+h_{2})\otimes (h_{1}+2h_{2})+ (2h_{1}+h_{2})\otimes (h_{1}+2h_{2})\otimes  (2h_{1}+h_{2})\\
 - (2h_{1}+h_{2})\otimes (h_{1}+2h_{2})\otimes (h_{1}+2h_{2})\\
+ (2h_{1}+h_{2})\otimes 3y_{1}\otimes 3x_{1}- (2h_{1}+h_{2})\otimes 3y_{2}\otimes 3x_{2}-(2h_{1}+h_{2})\otimes 3x_{1}\otimes 3y_{1}\\
+ (2h_{1}+h_{2})\otimes 3z_{2}\otimes 3z_{1}+(h_{1}+2h_{2})\otimes  (2h_{1}+h_{2})\otimes  (2h_{1}+h_{2})\\
-(h_{1}+2h_{2})\otimes  (2h_{1}+h_{2})\otimes (h_{1}+2h_{2})-(h_{1}+2h_{2})\otimes (h_{1}+2h_{2})\otimes  (2h_{1}+h_{2})\\
+(h_{1}+2h_{2})\otimes 3y_{2}\otimes 3x_{2}+(h_{1}+2h_{2})\otimes 3x_{1}\otimes 3y_{1}\\
-(h_{1}+2h_{2})\otimes 3x_{2}\otimes 3y_{2}-(h_{1}+2h_{2})\otimes 3z_{1}\otimes 3z_{2}-3y_{1}\otimes  (2h_{1}+h_{2})\otimes 3x_{1}\\
+3y_{1}\otimes (h_{1}+2h_{2})\otimes 3x_{1}+3y_{1}\otimes 3y_{2}\otimes 3z_{1}+3y_{1}\otimes 3x_{1}\otimes  (2h_{1}+h_{2})\\
-3y_{2}\otimes (h_{1}+2h_{2})\otimes 3x_{2}-3y_{2}\otimes 3x_{2}\otimes  (2h_{1}+h_{2})+3y_{2}\otimes 3x_{2}\otimes (h_{1}+2h_{2})\\
+3y_{2}\otimes 3z_{1}\otimes 3y_{1}+3x_{1}\otimes  (2h_{1}+h_{2})\otimes 3y_{1}-3x_{1}\otimes 3y_{1}\otimes  (2h_{1}+h_{2})\\
+3x_{1}\otimes 3y_{1}\otimes (h_{1}+2h_{2})+3x_{1}\otimes 3z_{2}\otimes 3x_{2}-3x_{2}\otimes  (2h_{1}+h_{2})\otimes 3y_{2}\\
+3x_{2}\otimes (h_{1}+2h_{2})\otimes 3y_{2}-3x_{2}\otimes 3y_{2}\otimes (h_{1}+2h_{2})+3x_{2}\otimes 3x_{1}\otimes 3z_{2}\\
-3z_{2}\otimes (h_{1}+2h_{2})\otimes 3z_{1}+3z_{2}\otimes 3x_{2}\otimes 3x_{1}+3z_{2}\otimes 3z_{1}\otimes  (2h_{1}+h_{2})\\
+3z_{1}\otimes  (2h_{1}+h_{2})\otimes 3z_{2}+3z_{1}\otimes 3y_{1}\otimes 3y_{2}-3z_{1}\otimes 3z_{2}\otimes (h_{1}+2h_{2})
\end{array}$$

$$\begin{array}{c}
=\\
6h_{1}\otimes h_{1}\otimes h_{1}-6h_{2}\otimes h_{2}\otimes h_{2}\\
+3h_{1}\otimes h_{1} \otimes h_{2}+3 h_{1}\otimes h_{2} \otimes h_{1}+3h_{2}\otimes h_{1}\otimes h_{1}\\
-3h_{2}\otimes h_{1}\otimes h_{2}-3h_{2}\otimes h_{2}\otimes h_{1}-3h_{1}\otimes h_{2}\otimes h_{2}\\
+18 h_{1}\otimes y_{1}\otimes x_{1}-9h_{1}\otimes y_{2}\otimes x_{2} -9h_{1}\otimes x_{1}\otimes y_{1}-9h_{1}\otimes x_{2}\otimes y_{2}\\
-18 h_{2}\otimes x_{2}\otimes y_{2}+9h_{2}\otimes y_{2}\otimes x_{2}+9h_{2}\otimes x_{1}\otimes y_{1}+9h_{2}\otimes y_{1}\otimes x_{1}\\
+18h_{1}\otimes z_{2}\otimes z_{1}-18 h_{2}\otimes z_{1}\otimes z_{2}+9h_{2}\otimes z_{2}\otimes z_{1}-9h_{1}\otimes z_{1}\otimes z_{2}\\
-9 y_{1}\otimes h_{1}\otimes x_{1}+9 y_{1}\otimes h_{2}\otimes x_{1}\\
+18 y_{1}\otimes x_{1}\otimes h_{1}+9y_{1}\otimes x_{1}\otimes h_{2}+27y_{1}\otimes y_{2}\otimes z_{1} \\
-18 y_{2}\otimes h_{2}\otimes x_{2}+9 y_{2}\otimes x_{2}\otimes h_{2}-9y_{2}\otimes h_{1}\otimes x_{2}\\
-9 y_{2}\otimes x_{2}\otimes h_{1}+27y_{2}\otimes z_{1}\otimes y_{1} \\
+18 x_{1}\otimes h_{1}\otimes y_{1}+9 x_{1}\otimes y_{1}\otimes h_{2}\\
-9 x_{1}\otimes y_{1}\otimes h_{1}+9x_{1}\otimes h_{2}\otimes y_{1}+27x_{1}\otimes z_{2}\otimes x_{2}\\
-9 x_{2}\otimes h_{1}\otimes y_{2}+9 x_{2}\otimes h_{2}\otimes y_{2}\\
-18 x_{2}\otimes y_{2}\otimes h_{2}-9x_{2}\otimes y_{2}\otimes h_{1}+27x_{2}\otimes x_{1}\otimes z_{2} \\
-9z_{2}\otimes h_{1}\otimes z_{1}-18 z_{2}\otimes h_{2}\otimes z_{1}+18z_{2}\otimes z_{1}\otimes h_{1}\\
+9z_{2}\otimes z_{1}\otimes h_{2}+27z_{2}\otimes x_{2}\otimes x_{1} \\
+9z_{1}\otimes h_{2}\otimes z_{2}+18 z_{1}\otimes h_{1}\otimes z_{2}-18z_{1}\otimes z_{2}\otimes h_{2}\\
-9z_{1}\otimes z_{2}\otimes h_{1}+27z_{1}\otimes y_{1}\otimes y_{2}. 
\end{array}$$   

Then we use the canonical mapping $\pi:T(\mathfrak{sl}_{3}\C)\rightarrow\mathfrak{U}(\mathfrak{sl}_{3}\C)$ to compute the corresponding element in the universal enveloping algebra. 
$$18\pi(K^{(3)})=\begin{array}{c}
\\
6h_{1}^{3}-6h_{2}^{3}+9h_{1}^{2}h_{2}-9h_{2}^{2}h_{1}-54h_{2}^{2}-27h_{1}h_{2}-54h_{1}-108h_{2}\\
-27h_{2}z_{2}z_{1}+27h_{1}z_{2}z_{1}+81z_{2}x_{1}x_{2}+27h_{1}y_{1}x_{1}+54h_{2}y_{1}x_{1}\\
-54h_{1}y_{2}x_{2}-27h_{2}y_{2}x_{2}+81y_{1}y_{2}z_{1}-162y_{2}x_{2}-81z_{2}z_{1}.
\end{array}$$
Note that we have used the PBW theorem to arrange the terms in two groups. One group consisting of terms that have a raising operator ($x_{1},x_{2}$ or $z_{1}$) to the right and another group that consists of terms of elements in $\h$. With this ordering, let us define a homomorphism $\xi:\mathfrak{U}(\g)\rightarrow\mathfrak{U}(\h)$ that maps each term that does not exclusively consist of elements in $\h$ to zero. Under this map, $18\pi(K^{(3)})$ is mapped to
$$6h_{1}^{3}-6h_{2}^{3}+9h_{1}^{2}h_{2}-9h_{2}^{2}h_{1}-54h_{2}^{2}-27h_{1}h_{2}-54h_{1}-108h_{2}.$$
Moreover let $\eta$ be the homomorphism that maps each $h_{i}$ to $h_{i}-1$, then
$$18\eta(\xi(\pi(K^{(3)})))=6h_{1}^{3}-6h_{2}^{3}-27h_{1}^{2}-27h_{2}^{2}+9h_{1}^{2}h_{2}-9h_{1}h_{2}^{2}-27h_{1}h_{2}+81.$$
Let us write $\varphi=\eta\circ\xi|_{\mathcal{Z}(\g)}$. Every element $z\in\mathcal{Z}(\g)$ in the center of the universal enveloping algebra acts on a highest weight module of highest weight $\lambda$ by a scalar $\chi_{\lambda}(z)\in\C$, called the {\bf central character}.
It is easy to compute (see~\cite{h}) that 
$$\chi_{\lambda}(z)=(\lambda+\rho)(\varphi(z)),\;\;\text{for}\;z\in\mathcal{Z}(\g).$$
Let us write
$$h_{1}=H_{1}-H_{2},\;\;h_{2}=H_{2}-H_{3}.$$
Thus we get
$$
18\varphi(\pi(K^{(3)}))=\begin{array}{c}
6(H_{1}^{3}+H_{2}^{3}+H_{3}^{3})\\
-9(H_{1}^{2}H_{2}+H_{1}^{2}H_{3}+H_{2}^{2}H_{1}+H_{2}^{2}H_{3}+H_{3}^{2}H_{1}+H_{3}^{2}H_{2})\\
+36H_{1}H_{2}H_{3}\\
-27(H_{1}^{2}+H_{2}^{2}+H_{3}^{2})\\
+27(H_{1}H_{2}+H_{1}H_{3}+H_{2}H_{3})\\
+81
\end{array}.$$
As expected, this is a symmetric polynomial of degree 3 in the $H_{i}$'s.
Note that the Casimir operator of $\mathfrak{sl}_{3}\C$ is given by
$$9\varphi(\pi(K^{(2)}))=H_{1}^{2}+H_{2}^{2}+H_{3}^{2}-H_{1}H_{2}-H_{1}H_{3}-H_{2}H_{3}-3.$$
%$$\chi_{\lambda}(\pi(K^{(3)}))=,$$
%where $\lambda+\rho=l_{1}L_{1}+l_{2}L_{2}$ and $L_{i}(h_{j})=\delta_{i,j}$.

\subsubsection{Remark}
More generally, the mapping
$$\varphi:=\eta\circ\xi|_{\mathcal{Z}(\g)}:\mathcal{Z}(\g)\rightarrow S(\h)^{\mathcal{W}},$$
where $S(\h)^{\mathcal{W}}$ is the algebra of elements in the symmetric algebra $S(\h)$ that are fixed by the Weyl group $\mathcal{W}$, is an isomorphism.
For more information about basic generators of $S(\h)^{\mathcal{W}}$ refer to~\cite{hu}.

\subsubsection{Equivalent formulation}\label{generators}
We can equivalently say that $K^{(p)}$ induces a map
$$B^{(p)}:\otimes^{p}\g^{*}\rightarrow\R$$
by firstly identifying $\g^{*}$ with $\g$ with the help of the Killing form and then mapping a simple element $\otimes^{p}\g\ni X_{1}\otimes...\otimes X_{p}$ to $\mathrm{tr}(\phi(X_{1})\circ\cdots\circ\phi(X_{p}))$. This is obviously a $P$-module homomorphism (in fact, it is even a $G$-module homomorphism). Since there are $l$ non-zero $K^{(d_{i})}$, $i=1,...,l$, for the classical Lie algebras and $G_{2}$, see~\cite{o}, the corresponding $B^{(d_{i})}$ are non-zero as well. Apart form the $D$ series, where an extra element needs to be defined, see below, the elements $\pi(K^{(d_{i})})\in\mathcal{Z}(\g)$, for the canonical mapping $\pi:T(\g)\rightarrow\mathfrak{U}(\g)$, generate the center of the universal enveloping algebra of $\g$.

\subsubsection{Examples}
\begin{enumerate}
\item
For $A_{l}$ we can take $\pi(K^{(2)})$, $\pi(K^{(3)})$,...,$\pi(K^{(l+1)})$ as the generators of $\mathcal{Z}(A_{l})$.
\item
For $B_{l}$ (and $C_{l}$) we can take $\pi(K^{(2)})$, $\pi(K^{(4)})$,...,$\pi(K^{(2l)})$ as the generators of $\mathcal{Z}(B_{l})$ (and $\mathcal{Z}(C_{l})$).
\item
For $D_{l}$ we can define another Casimir invariant by
$$\tilde{K}^{(l)}=(-1)^{l(l+2)/2}\frac{1}{2^{l}l!}\epsilon_{\mu_{1}\nu_{1}\mu_{2}\nu_{2}...\mu_{l}\nu_{l}}X_{\mu_{1}\nu_{1}}\otimes\cdots\otimes X_{\mu_{l}\nu_{l}},$$
where $\epsilon_{\mu_{1}...\nu_{l}}$ is the completely antisymmetric Levi-Cevita tensor taking the values $0$ and $\pm 1$ and $X_{\mu\nu}=-X_{\nu\mu}$ are the basis elements of $D_{l}$ ($\mu,\nu=1,..,2l$).
Then the elements $\pi(K^{(2)})$, $\pi(K^{(4)})$,...,$\pi(K^{(2l-2)})$ and $\pi(\tilde{K}^{(l)})$ are the generators of $\mathcal{Z}(D_{l})$.
\item
For $G_{2}$ we can take $\pi(K^{(2)})$ and $\pi(K^{(6)})$ as the generators of $\mathcal{Z}(G_{2})$.
\end{enumerate}
These statements can be found in~\cite{o} and references therein. Note that in each case $|\mathcal{W}|=\prod_{i}d_{i}$.

\begin{definition}
{\rm
Recall the adjoint tractor bundle $\mathcal{A}=\mathcal{G}\times_{P}\g$ and the fundamental derivative
\begin{eqnarray*}
D:\mathcal{O}(\mathcal{G},\mathbb{E})^{P}&\rightarrow &\mathcal{O}(\mathcal{G},\g^{*}\otimes \mathbb{E})^{P}\\
s&\mapsto&X\mapsto \nabla^{\omega}_{X}s,
\end{eqnarray*}
from~\ref{invariantdifferential}. This operator can be iterated to an invariant differential operator
$$D^{p}:\mathcal{O}(\mathcal{G},\mathbb{E})^{P}\rightarrow \mathcal{O}(\mathcal{G},\otimes^{p}\g^{*}\otimes \mathbb{E})^{P},$$
see Lemma~\ref{five}. Properties of the fundamental derivative are discussed in~\cite{cg}, 3.1.
Using the elements $B^{(p)}$ from~\ref{generators}, we can define an invariant linear differential operator
$$\mathcal{C}^{(p)}:\mathcal{O}(\mathcal{G},\mathbb{V})^{P}\stackrel{D^{p}}{\longrightarrow}\mathcal{O}(\mathcal{G},\otimes^{p}\g^{*}\otimes\mathbb{V})^{P}\stackrel{B^{(p)}\otimes\mathrm{id}}{\longrightarrow}\mathcal{O}(\mathcal{G},\mathbb{V})^{P}$$
as the composition $(B^{(p)}\otimes\mathrm{id})\circ D^{p}$.
}\end{definition}

\subsubsection{Example}
For $p=2$ this procedure yields the curved Casimir operator in~\cite{cs3}. The authors prove in this paper that the curved Casimir operator in the flat case is given by
$$\bar{g}^{\mu_{1}\mu_{2}}D^{2}(s_{X_{\mu_{1}}},s_{X_{\mu_{2}}})=\bar{g}^{\mu_{1}\mu_{2}}R_{X_{\mu_{1}}}R_{X_{\mu_{2}}},$$
see~\ref{homogeneousdiff}. In~\cite{cs3} the authors prove the following Lemma~\ref{constant} and Lemma~\ref{viacharacter} for $p=2$ via a direct calculation using an adapted local frame for $\mathcal{A}$.

\begin{lemma}\label{constant}
Let $\mathbb{V}$ be an irreducible representation of $P$, then $\mathcal{C}^{(p)}:\Gamma(V)\rightarrow\Gamma(V)$ has to act by a constant.
\end{lemma}
\begin{proof}
Let $V$ and $W$ be associated vector bundles over $\mathcal{M}$.
An invariant linear differential operator $D:V\rightarrow W$ arises via a composition
$$D:V\stackrel{\nabla^{N}}{\longrightarrow}\otimes^{N}\Lambda^{1}\otimes V\stackrel{\phi}{\longrightarrow} W,$$
where $\phi$ is induced by a $P$-module homomorphism
$$\Phi:\otimes^{N}\p_{+}\otimes\mathbb{V}\rightarrow\mathbb{W}.$$
Now suppose that $\mathbb{V}=\mathbb{W}$ and $\mathbb{V}$ is irreducible. By looking at the action of the grading element one deduces that $N=0$ and Schur's Lemma
forces $\Phi$ to be a constant multiple of the identity. 
\end{proof}

\begin{definition}
{\rm Let $z\in\mathcal{Z}(\g)$ be an arbitrary element in the center of the universal enveloping algebra of $\g$. Then $z$ can be written as a polynomial
$$z=\sum_{i_{1},...,i_{l}}a_{i_{1}...i_{l}}\pi(K^{(d_{1})})^{i_{1}}\cdots\pi(K^{(d_{l})})^{i_{l}},$$
where $K^{(d_{1})},...,K^{(d_{l})}$ are the $l$ non-zero elements of Definition~\ref{generatorsone}. Then we define the {\bf higher order curved Casimir operator} associated to $z$ by
$$\mathcal{C}_{z}=\sum_{i_{1},...,i_{l}}a_{i_{1}...i_{l}}(\mathcal{C}^{(d_{1})})^{i_{1}}\circ\cdots\circ(\mathcal{C}^{(d_{l})})^{i_{l}}.$$
}\end{definition}

\begin{lemma}\label{viacharacter}
If $\mathbb{V}$ is irreducible, then $\mathcal{C}_{z}$ acts on $\Gamma(V)$ by the constant $\chi_{\lambda}(z)$, where $\chi_{\lambda}:\mathcal{Z}(\g)\rightarrow\C$ is the central character associated to $\lambda\in\h^{*}$, the highest weight of $\mathbb{V}^{*}$.
\end{lemma}
\begin{proof}
Lemma~\ref{constant} shows that $\mathcal{C}_{z}$ acts by a constant. We may as well compute this constant for the homogeneous model case $G/P$. For this purpose let us first of all take $V$ to be an arbitrary homogeneous bundle. We will prove the lemma for $z=\pi(K^{(d_{i})})$, the general case then follows from the fact that the central character is an algebra homomorphism. An invariant linear differential operator is determined by a $P$-module homomorphism $J^{k}\mathbb{V}\rightarrow\mathbb{V}$ for some $k$. Following~\cite{cd}, we can view $J^{k}\mathbb{V}$ as a subset of 
$$\oplus_{i=0}^{k}\otimes^{i}\g^{*}\otimes\mathbb{V}.$$
The $P$-module homomorphism that corresponds to $\mathcal{C}_{z}$ is given by
$$B^{(d_{i})}\otimes\mathrm{id}_{\mathbb{V}},$$
with $k=d_{i}$. Dually we obtain a mapping
$$\mathbb{V}^{*}\rightarrow M_{\p}(\mathbb{V})$$
that is given by 
$$v^{*}\mapsto \pi(K^{(d_{i})})\otimes v^{*}=z\otimes v^{*}\in\mathfrak{U}(\g)\otimes_{\mathfrak{U}(\p)}\mathbb{V}^{*}.$$
Finally, this corresponds to a $\g$-module homomorphism 
\begin{eqnarray*}
M_{\p}(\mathbb{V})&\rightarrow&M_{\p}(\mathbb{V})\\
u\otimes v^{*}&\mapsto&zu\otimes v^{*}
\end{eqnarray*}
If $\mathbb{V}$ is irreducible, then this action is by definition given by $\chi_{\lambda}(z)$.
% Since $\mathcal{C}_{z}$ is an invariant differential operator, it corresponds to a homomorphism of $\g$-modules $M_{\p}(\mathbb{V})\rightarrow M_{\p}(\mathbb{V})$. This homomorphism is given by the action of $z$ which is by definition given by $\chi_{\lambda}(z)$.
\end{proof}

\subsubsection{Example}
For the Casimir operator $c$ it is known (\cite{h}) that
$$\chi_{\lambda}(c)=\Vert\lambda+\rho\Vert^{2}-\Vert\rho\Vert^{2}.$$
The constants by which the higher order Casimir operators $K^{(p_{i})}$ act on an irreducible $\g$-module are computed in~\cite{o}. It has to be noted that despite the apparent fractional form of the formulae in~\cite{o}, these formulae can be proved to be symmetric polynomials in the coefficients of $\lambda+\rho$, where $\lambda$ can be an arbitrary weight in $\h^{*}$. For $A_{l}$, for example, these can be found in~\cite{lb}.

\begin{theorem}\label{theoremnine}
Let $\mathbb{V}$ be a representation of $P$ with a $P$-invariant filtration
$$\mathbb{V}=\mathbb{V}^{0}\supset\cdots\supset \mathbb{V}^{N}\supset\{0\},$$
so that each sub-quotient $\mathbb{V}^{i}/\mathbb{V}^{i+1}$ is completely reducible. Let $\mathbb{W}\subset\mathbb{V}^{i}/\mathbb{V}^{i+1}$ be an irreducible component whose dual has highest weight $\lambda$. Moreover, for each $j>i$, let $\mu_{j,k}$ for $k=1,...,n_{j}$ be the highest weights of the irreducible representations which are dual to the irreducible components of $\mathbb{V}^{j}/\mathbb{V}^{j+1}$. 
Suppose that $\chi_{\mu_{j,k}}\not=\chi_{\lambda}$ for all $j,k$, then there exist elements $z_{j,k}\in\mathcal{Z}(\g)$, such that $\chi_{\mu_{j,k}}(z_{j,k})\not=\chi_{\lambda}(z_{j,k})$ for all $j,k$. The operator
$$L=\prod_{j=i+1}^{N}\prod_{k=1}^{n_{j}}(\mathcal{C}_{z_{j,k}}-\chi_{\mu_{j,k}}(z_{j,k}))$$
descends to an operator $\Gamma(W)\rightarrow\Gamma(V^{i})$ that defines an invariant splitting. 
\end{theorem}
\begin{proof}
%The individual operators $L_{j}=\prod_{k=1}^{n_{j}}(\mathcal{C}_{z_{j,k}}-\chi_{\mu_{j,k}}(z_{j,k}))$ and even the individual factors of each $L_{j}$ all commute with each other by the naturality properties of $D$, see~\cite{cg}, Proposition 3.1.  
Let $\sigma\in\Gamma(V^{j})$ and denote the projection
$V^{j}\rightarrow V^{j}/V^{j+1}$ by $\pi_{j}$.  For brevity write $L_{j}=\prod_{k=1}^{n_{j}}(\mathcal{C}_{z_{j,k}}-\chi_{\mu_{j,k}}(z_{j,k}))$. The individual operators $\mathcal{C}_{z_{j,k}}-\chi_{\mu_{j,k}}(z_{j,k})$ commute with the $\pi_{j}$'s by the naturality properties of $D$, see~\cite{cg}, Proposition 3.1. Hence
$$\pi_{j}(L_{j}(\sigma))=\prod_{k=1}^{n_{j}}(\mathcal{C}_{z_{j,k}}-\chi_{\mu_{j,k}}(z_{j,k}))\pi(\sigma)=0.$$
This implies that $L_{j}$ maps $\Gamma(V^{j})$ to $\Gamma(V^{j+1})$ and by induction that $L$ vanishes on $\Gamma(V^{i+1})$. Therefore $L$ descends to an operator $\Gamma(V^{i}/V^{i+1})\rightarrow \Gamma(V^{i})$ that can be restricted to $\Gamma(W)$. Moreover, let $\sigma\in\Gamma(W)$ and choose a representative $\hat{\sigma}\in\Gamma(V^{i})$, then we compute
$$\pi_{i}(L(\hat{\sigma}))=  \prod_{j=i+1}^{N}\prod_{k=1}^{n_{j}}(\chi_{\lambda}(z_{j,k})-\chi_{\mu_{j,k}}(z_{j,k}))\sigma.$$
So if  $\chi_{\lambda}(z_{j,k})\not=\chi_{\mu_{j,k}}(z_{j,k})$ for all $j,k$, then $L(C\hat{\sigma})\in\Gamma(V^{i})$ is an invariant lift for $\sigma$ with $C=\left(\prod_{j=i+1}^{N}\prod_{k=1}^{n_{j}}(\chi_{\lambda}(z_{j,k})-\chi_{\mu_{j,k}}(z_{j,k}))\right)^{-1}$.
\end{proof}  

\subsubsection{Remark}
Theorem~\ref{theoremnine} is a straightforward extension of Theorem 2 in~\cite{cs3}.

\begin{corollary}\label{coreight}
Let $E$ be an irreducible associated bundle, choose $\underline{k}$ and define the $M$-bundle $V_{\underline{M}}(\mathbb{E}^{0})(\underline{k}-\underline{M})$ as in~\ref{mmodule}. Assume that
every generalized Verma module that is associated to the irreducible composition factors of $\mathbb{V}_{\underline{M}}(\mathbb{E}^{0})(\underline{k}-\underline{M})$ has a different central character from $M_{\p}(\mathbb{E})$, then the higher order curved Casimir operators $\mathcal{C}_{z}$ can be used to define an invariant splitting operator
$$L=\prod_{j=i+1}^{N}\prod_{k=1}^{n_{j}}(\mathcal{C}_{z_{j,k}}-\chi_{\mu_{j,k}}(z_{j,k})):\Gamma(E)\rightarrow\Gamma(V_{\underline{M}}(\mathbb{E}^{0})(\underline{k}-\underline{M})),$$
where $\mu_{j,k}$ are the highest weights of the duals of the irreducible components of the composition factors
$$\mathbb{V}_{j}(\underline{k}-\underline{M})\;j=1,...,N$$
of $\mathbb{V}_{\underline{M}}(\mathbb{E}^{0})(\underline{k}-\underline{M})$
and $z_{j,k}\in\mathcal{Z}(\g)$ are such that $\chi_{\mu_{j,k}}(z_{j,k})\not=\chi_{\lambda}(z_{j,k})$.

\end{corollary}
\begin{proof}
$\mathbb{V}_{\underline{M}}(\mathbb{E}^{0})$ is a representation of $\g$ and hence allows a composition series by completely reducible sub-quotients as in Lemma~\ref{lemmathirteen}. Tensoring this representation with $\mathcal{O}(\underline{k}-\underline{M})$ does not change the form of the composition series, it just changes the character by which $\mathfrak{z}(\g_{0})$ acts. Hence we can apply Theorem~\ref{theoremnine}.
\end{proof}

Thus we have proved the next theorem.

\begin{theorem}[Main result 4]\label{mainresultfour}
Let $(\mathcal{M},\mathcal{G},\g,\omega)$ be a regular parabolic geometry of type $(G,P)$ and let $\mathbb{V}$ and $\mathbb{W}$ be two  finite dimensional irreducible $\p$-modules such that the action of $\p$ lifts to an action of $P$ and let $\lambda$ (resp.~$\nu$) be the highest weight of $\mathbb{V}^{*}$ (resp.~$\mathbb{W}^{*}$). Moreover denote the corresponding $\g_{0}^{S}$ modules by $\mathbb{V}^{0}$ and
$\mathbb{W}^{0}$ respectively and define
$$\mathbb{V}_{\underline{M}}(\mathbb{V}^{0})(\underline{k_{V}}-\underline{M})\;\text{and}\;\mathbb{V}_{\underline{M}}(\mathbb{W}^{0})(\underline{k_{W}}-\underline{M})$$
to be the appropriate $M$-modules as in~\ref{mmodule} and Remark~\ref{geometricweights}. The different central characters of the generalized Verma modules associated to the irreducible composition factors of those $M$-modules will be denoted by
$\chi_{\tau_{i,j}}$ ($j=1,...,N_{V}$, $i=1,...,n_{j}$) and $\chi_{\sigma_{k,l}}$ ($k=1,...,N_{W}$, $l=1,...,m_{k}$). If
%$$c(\lambda)\not=(\mu_{\lambda})_{j}^{k}\;\text{and}\;c(\nu)\not=(\mu_{\nu})_{i}^{l}$$
$$\chi_{\lambda}\not=\chi_{\tau_{i,j}}$$
and
$$\chi_{\nu}\not=\chi_{\sigma_{k,l}}$$
for all possible $k,j,i,l$, then there exists an $m$-parameter family of invariant bilinear differential pairings
$$\Gamma(V)\times\Gamma(W)\rightarrow \Gamma(E)$$
for each $\mathbb{E}=\mathbb{E}^{0}(\underline{k_{V}}+\underline{k_{W}}-2\underline{M})$ corresponding to an irreducible component of the $\g_{0}^{S}$ tensor product
$$\mathbb{E}^{0}\subset \mathfrak{U}_{t}(\p_{+})\otimes\mathbb{V}^{0}\otimes\mathbb{W}^{0}$$
of multiplicity $m$ for each $t\leq M=\mathrm{min}_{i\in I}\{M_{i}\}$. This pairing is of weighted order $t$.
\end{theorem}
\begin{proof}
We can use Corollary~\ref{coreight} to define splittings
$$V\rightarrow V_{\underline{M}}(\mathbb{V}^{0})(\underline{k}-\underline{M})\;\text{and}\;W\rightarrow V_{\underline{M}}(\mathbb{W}^{0})(\underline{l}-\underline{M})$$
and Proposition~\ref{propeight} to ensure that there exist the appropriate projections
$$V_{\underline{M}}(\mathbb{V}^{0})(\underline{k_{V}}-\underline{M})\otimes V_{\underline{M}}(\mathbb{W}^{0})(\underline{k_{W}}-\underline{M})\rightarrow E.$$
The weighted order can be determined by looking at the symbol of the differential operator as described in~\ref{bijet}.
\end{proof}

\subsubsection{Remark}
In order to minimize the amount of restrictions $\chi_{\lambda}\not=\chi_{\mu}$, it is best to choose $M_{i}=M$ for all $i\in I$. If one is interested in specific pairings, it might be appropriate to vary the different values of $\underline{M}$ for $V$ and $W$. In those cases one has to examine the corresponding $M$-bundles carefully and exclude appropriate weights in the spirit of the discussion above. This procedure can be a lot more efficient in any specific example (see Chapter~\ref{tractorchapter}) than in the general theory developed above.

\subsection{Splitting operators on the flat model}
On homogeneous spaces $G/P$, we can define splitting operators with the help  of the following three theorems:

\begin{theorem}\label{theoremeleven}
Invariant linear differential operators between sections of homogeneous bundles over a flag manifold $G/P$ are in one-to-one correspondence with $\g$-module homomorphisms of induced modules.
\end{theorem}\begin{proof}
This theorem is proved analogously to Proposition 3 in~\cite{er}, p.~212. It may be noted that the theorem is usually stated in terms of generalized Verma modules (see~\cite{be}, p.~164). The statement, however, remains true for induced modules with identical proof.
% (where the theorem is stated for irreducible homogeneous bundles, but the statement is also true more
%generally with identical proof).
\end{proof}

\begin{theorem}\label{theoremtwelve}
If $M_{\p}(\mathbb{V}_{0}(\underline{k}-\underline{M}))$ has distinct central character from the generalized Verma modules associated to all the other composition factors of $\mathbb{V}_{\underline{M}}(\mathbb{E}^{0})(\underline{k}-\underline{M})$, then it can be canonically split off as a direct summand of $M_{\p}(\mathbb{V}_{\underline{M}}(\mathbb{E}^{0})(\underline{k}-\underline{M}))$.
\end{theorem}
\begin{proof}
The composition series $\mathbb{V}_{\underline{M}}(\mathbb{E}^{0})(\underline{k}-\underline{M})=\mathbb{V}_{0}(\underline{k}-\underline{M})+...+\mathbb{V}_{N}(\underline{k}-\underline{M})$ induces a composition series
$$(\mathbb{V}_{\underline{M}}(\mathbb{E}^{0})(\underline{k}-\underline{M}))^{*}=(\mathbb{V}_{N}(\underline{k}-\underline{M}))^{*}+...+(\mathbb{V}_{0}(\underline{k}-\underline{M}))^{*}$$
of the dual representation. Since the functor that associates to every $\p$-module $\mathbb{V}^{*}$ the corresponding induced module $\mathfrak{U}(\g)\otimes_{\mathfrak{U}(\p)}\mathbb{V}^{*}$ is exact (see~\cite{v}, p.~303, Lemma~6.1.6), we have a filtration 
$$M_{\p}(\mathbb{V}_{\underline{M}}(\mathbb{E}^{0})(\underline{k}-\underline{M}))=M_{\p}(\mathbb{V}_{N}(\underline{k}-\underline{M}))+...+M_{\p}(\mathbb{V}_{0}(\underline{k}-\underline{M}))$$
that induces an injection $M_{\p}(\mathbb{V}_{0}(\underline{k}-\underline{M}))\hookrightarrow M_{\p}(\mathbb{V}_{\underline{M}}(\mathbb{E}^{0})(\underline{k}-\underline{M}))$.
The weight spaces of $M_{\p}(\mathbb{V}_{\underline{M}}(\mathbb{E}^{0})(\underline{k}-\underline{M}))$ can be grouped in terms of central character, so the projection $M_{\p}(\mathbb{V}_{\underline{M}}(\mathbb{E}^{0})(\underline{k}-\underline{M}))\rightarrow M_{\p}(\mathbb{V}_{0}(\underline{k}-\underline{M}))$ may be defined by projecting onto the joint eigenspace of the central character of $M_{\p}(\mathbb{V}_{0}(\underline{k}-\underline{M}))$. Since central character is preserved under the action of $\g$, this projection is indeed a $\g$-module homomorphism and provides a $\g$-module splitting  of $M_{\p}(\mathbb{V}_{\underline{M}}(\mathbb{E}^{0})(\underline{k}-\underline{M}))$.
\end{proof}

\begin{theorem}[Harish-Chandra]\label{theoremthirteen}
Two generalized Verma modules have the same central character if and only if their highest weights are related by the affine action of the Weyl group of $\g$.
\end{theorem}
\begin{proof}
A proof of this theorem can, for example, be found in~\cite{h}, p.~130, Theorem~23.3.
\end{proof}

\subsubsection{Remark}
The irreducible components of the tensor product $\mathfrak{U}_{t}(\p_{+})\otimes\mathbb{V}\otimes\mathbb{W}$ are exactly the possible targets for invariant bilinear differential pairings of weighted order $t$ between sections of $V$ and $W$ that are curved analogues of non-zero pairings on the homogeneous model spaces. This can be seen by looking at the exact sequence
$$
\begin{array}{ccccc}
\oplus_{j=0}^{t}\mathfrak{U}_{j}(\p_{+})\otimes\mathbb{V}\otimes\mathfrak{U}_{t-j}(\p_{+})\otimes\mathbb{W}&\rightarrow& \mathcal{J}^{t}(\mathbb{V},\mathbb{W})&\rightarrow&\mathcal{J}^{t-1}(\mathbb{V},\mathbb{W})\rightarrow 0\\
&\searrow&\downarrow&&\\
&&\mathbb{E}
\end{array}.$$

\subsubsection{Remark}
\begin{enumerate}
\item
Theorems~\ref{theoremeleven},~\ref{theoremtwelve} and~\ref{theoremthirteen} combined are the backbone of the {\bf Jantzen}-{\bf Zuckermann} 
{\bf translation} {\bf functor} as used in~\cite{es} and~\cite{er}.  
\item
In order to define splitting operators we had to exclude weights, i.e.~values of $\underline{k}$, for which the central character of $M_{\p}(\mathbb{V}_{0}(\underline{k}-\underline{M}))$ is the same as the central character of a generalized Verma module associated to another
composition factor of $\mathbb{V}_{\underline{M}}(\mathbb{E}^{0})(\underline{k}-\underline{M})$. A trivial case is $\underline{k}=\underline{M}$, because all the weight spaces of $\mathbb{V}_{\underline{M}}(\mathbb{E}^{0})$ apart from the highest weight space, which lies in $\mathbb{V}_{0}$, have weights $\mu$ so that
$$\Vert \Lambda+\rho\Vert^{2} >\Vert \mu+\rho\Vert^{2}.$$
%with $\rho_{\g}=\frac{1}{2}\sum_{\alpha\in\Delta^{+}(\g)}\alpha$ (see~\cite{h}, p.~114, proposition~21.4 and p.~71, lemma~13.4).
%Since the Weyl group acts by isometries, this implies that $M_{\p}(\mathbb{V}_{0})$ has distinct central character from the generalized Verma modules associated to all the other composition factors of $\mathbb{V}_{\underline{M}}(\mathbb{E}^{0})$.
The pairings that we obtain via our construction in this trivial case are then special cases of the parings $\sqcup_{\eta}$ as defined in~\cite{cd}, p.~13, Theorem~3.6.
\end{enumerate}

\subsection{Comparisons}

\subsubsection{First order operators via splittings and via $\p$-module homomorphisms}
Lemma~\ref{lemmaten} and Paragraph~\ref{obstructionterms2}  show that the results from Chapters three and four are consistent with the results in this chapter. The first method via bi-jet bundles, however, is much more efficient, because we only have to exclude those weights that correspond to operators that actually occur in the pairing. Moreover, the first method shows exactly what happens for excluded weights whereas the second method using splitting operators just fails. On the other hand, the construction in this chapter produces the most general higher order pairings for non-excluded representations.

\subsubsection{Splitting operators}\label{comparison}
Let $\mathbb{V}$, $\mathbb{W}$ be two finite dimensional irreducible representations of $\p$ with the highest weight of $\mathbb{V}^{*}$ (resp.~$\mathbb{W}^{*}$) given by $\lambda$ (resp.~$\nu$). Then we have the following implications.
\begin{eqnarray*}
&&\text{There exists an invariant differential operator}\;d:\Gamma(V)\rightarrow\Gamma(W)\\
&\Rightarrow& M_{\p}(\mathbb{V})\;\text{and}\;M_{\p}(\mathbb{W})\;\text{have the same central character}\\
&\Rightarrow & \Vert\lambda+\rho\Vert^{2}=  \Vert\nu+\rho\Vert^{2},
\end{eqnarray*}
where we only deal with {\bf curved analogues} of flat operators, i.e.~those operators that are non-zero when restricted to flat parabolic geometries.
Therefore, if 
$$\Vert\lambda+\rho\Vert^{2}\not=\Vert\tau_{k,j}+\rho\Vert^{2}$$
for all $j$ and $k$, then none of the irreducible composition factors of $\mathbb{V}_{\underline{M}}(\mathbb{E}^{0})(\underline{k}-\underline{M})$ induces a generalized Verma module that has a central character that is equal to the one of $\mathbb{E}$, so our construction
is more efficient than just using the curved Casimir operator and the splitting operators defined in~\cite{cs3}.
\par
The reverse of the two implications is not true:
%\begin{enumerate}
%\item
%$\xxx{-2}{1}{0}$ and $\xxx{2}{-3}{0}$ have the same central character, but there is no invariant differential operator between those bundles (note that in the Borel case, all %invariant differential operators are give by individual arrows in the BGG sequence).
%\item
%$\xxx{0}{-1}{-1}$ and $\xxx{-2}{-1}{-1}$ have the property that $ \Vert\lambda+\rho\Vert^{2}= \Vert\nu+\rho\Vert^{2}$, but they do not have the same central character.
%\end{enumerate}
%We can, however, restrict our attention to those $\mu$ that arise via a highest weight in the tensor product $\mathfrak{U}_{l}(\p_{+})\otimes\mathbb{V}$ for some $l$.
%It is presently not clear whether there are any splittings in the flat case that do have a curved analogue. This is a deep problem for operators between irreducible bundles already, %see~\cite{gr} and~\cite{gh}.

\subsubsection{Example 1}
Let us look at the following weights:
$$\lambda=\;\xoo{-1}{0}{1}\;\text{and}\;\mu=\;\xoo{-4}{1}{0}.$$
Let $\mathbb{V}^{*}$, $\mathbb{W}^{*}$ be the irreducible representations of $\p$ with highest weights $\lambda$ and $\mu$ respectively. Then 
$$\odot^{2}\g_{1}\otimes\mathbb{V}=\;\xoo{-4}{2}{0}\;\otimes\;\xoo{-1}{0}{1}\;=\;\xoo{-5}{2}{1}\;\oplus\;\xoo{-4}{1}{0},$$
so $\mathbb{W}$ appears as a possible symbol for a differential operator emanating from $V$. The generalized Verma modules $M_{\p}(\mathbb{V})$ and $M_{\p}(\mathbb{W})$ do not have the same central character: in the notation of Section~\ref{branching} we have
$$\lambda+\rho=(4|4,5,7)\;\text{and}\;\mu+\rho=(6|3,5,6)$$
and those numbers do not differ by a permutation. Equivalently, in~\cite{lb} one can see that $K^{(3)}$ acts on $M_{\p}(\mathbb{V})$
and $M_{\p}(\mathbb{W})$ differently. 
However, $\lambda+\rho=-\epsilon_{3}-3\epsilon_{4}$ and $\mu+\rho=3\epsilon_{2}+\epsilon_{3}$ and hence
$$\Vert\lambda+\rho\Vert^{2}=\Vert\mu+\rho\Vert^{2}=\frac{3}{4}.$$
This means that in 
$$\ooo{2}{0}{1}(v-1)=\;\xoo{v+1}{0}{1}\;+\;\begin{array}{c}
\;\xoo{v}{0}{0}\;\\
\oplus\\
\;\xoo{v-1}{1}{1}\;
\end{array}+\begin{array}{c}
\underbrace{\;\xoo{v-2}{1}{0}\;}_{=\mu\;\text{for}\;v=-2}\\
\oplus\\
\;\xoo{v-3}{2}{1}\;
\end{array}+\;\xoo{v-4}{2}{0}\;,$$
we would have to exclude $v=-2$ if we wanted to define a pairing via the curved Casimir operator. However, with the higher order curved Casimir operators, we do not have to exclude this weight (in the next chapter we will see that the weights to exclude are $v=-4,-1,0$), since $\lambda$ and $\mu$ do not have the same central character. In particular, there is no invariant differential operator between the corresponding bundles. 
Finally note that we will prove in the last chapter that the splitting in the general curved case can be written down with the help of tractor calculus, see Theorem~\ref{theoremsixteen}.

%Let us look at the Borel case for $A_{3}$. We have 
%$$\begin{array}{ccccc}
%\p_{+}=\underbrace{\begin{array}{c}
%\xxx{-2}{1}{0}\\
%\oplus\\
%\xxx{1}{-2}{1}\\
%\oplus\\
%\xxx{0}{1}{-2}
%\end{array}}_{\g_{1}}&+&\underbrace{
%\begin{array}{c}
%\xxx{1}{-1}{-1}\\
%\oplus\\
%\xxx{-1}{-1}{1}
%\end{array}}_{\g_{2}}&+&\underbrace{\xxx{-1}{0}{-1}}_{\g_{3}}
%\end{array}$$
%and
%$$\xxx{}{}{}\;=\;\xxx{}{}{}\;\otimes\;\xxx{}{}{}\;\subset\g_{1}\otimes\g_{2}\subset \mathfrak{U}_{3}(\p_{+})\otimes\C.$$
%One can check that the element $\sigma_{\alpha_{2}}\circ\sigma_{\alpha_{3}}$ maps $\lambda+\rho$ to $\mu+\rho$ for $\lambda=\xxx{0}{0}{0}$ and $\mu=\xxx{2}{-3}{0}$, so hese two modules have the same central character. However, there is no

\subsubsection{Example 2}
Let
$$\lambda=\;\oxo{0}{0}{0}\;\text{and}\;\mu=\;\oxo{1}{-4}{1}\;$$
and define $V$, $W$ as above. The generalized Verma modules $M_{\p}(\mathbb{V})$ and $M_{\p}(\mathbb{W})$ have the same central character (see~\cite{er}) and
$$\odot^{3}\g_{1}\otimes\mathbb{V}=(\;\oxo{3}{-6}{3}\;\oplus\;\oxo{1}{-4}{1}\;)\otimes\;\oxo{0}{0}{0}\;=\;\oxo{3}{-6}{3}\;\oplus\;\oxo{1}{-4}{1}\;,$$
hence $\mathbb{W}$ is a possible symbol for a differential operator emanating from $V$. But, in the conformally flat case, the classification in~\cite{er} shows that there is no invariant differential operator
$\Gamma(V)\rightarrow\Gamma(W)$. However, it has to be noted that the weight for 
$$\;\ooo{0}{3}{0}(v-3)=\;\oxo{0}{v}{0}\;+\underbrace{\;\oxo{1}{v-2}{1}}_{(*)}\;+\;\begin{array}{c}
\;\oxo{2}{v-4}{2}\;\\
\oplus\\
\;\oxo{0}{v-2}{0}\;
\end{array}\;+\;\begin{array}{c}
\;\oxo{3}{v-6}{3}\;\\
\oplus\\
\underbrace{\;\oxo{1}{v-4}{1}\;}_{=\mu\;\text{for}\;v=0}
\end{array}\;+\;...$$
that has to be excluded according to the above discussion is $v=0$, which has to be excluded at an even earlier stage $(*)$, because for $v=0$ the exterior derivative
$$d:\;\oxo{0}{0}{0}\;\rightarrow\;\oxo{1}{-2}{1}\;$$
is an invariant operator.  

\subsubsection{Remark}
We have not been able to find an example of a composition series of $\mathbb{V}_{\underline{M}}(\mathbb{V}^{0})(\underline{k}-\underline{M})$ such that
\begin{enumerate}
\item
there exists an irreducible composition factor $\mathbb{W}$ so that the generalized Verma modules $M_{\p}(\mathbb{W})$ and $M_{\p}(\mathbb{V}_{0}(\underline{k}-\underline{M}))$ have the same central character
and
\item
there is no invariant differential operator $d:V\rightarrow E$, where $\mathbb{E}$ is an irreducible composition factor of $\mathbb{V}_{\underline{M}}(\mathbb{V}^{0})(\underline{k}-\underline{M})$ and $V$ is the bundle induced from $\mathbb{V}_{0}(\underline{k}-\underline{M})$.
\end{enumerate}
It is plausible to conjecture that this is not possible. At present, however, this is still an open question.

\section{Higher order pairings for projective geometry}
In this section we will work exclusively on an $n$-dimensional manifold $\mathcal{M}$ with a projective structure that is given in terms of a parabolic geometry $(\mathcal{M},\mathcal{G},\mathfrak{sl}_{n+1}\C,\omega)$ of type $(\mathrm{SL}_{n+1}\R,P)$ with $P$ as given in~\ref{cartanexample} (a) and~\ref{homogeneousexample} (a). For the notation the reader is also advised to refer to Chapter~\ref{tractorchapter}.
\par
Note that we will use Dynkin diagrams for the visualization of four objects: an irreducible representation of $\p$, the corresponding irreducible homogeneous vector bundle, its sections and the generalized Verma module associated to the representation. In every case it should be clear which meaning we refer to and sometimes it is convenient that two meanings are denoted at the same time.
\par
\subsection{Branching rules}
\begin{definition}
\em For every representation 
$$\mathbb{E}=\;\oo{a_{1}}{a_{2}}\;...\;\oo{a_{n-2}}{a_{n-1}}$$
of $\g_{0}^{S}=\mathfrak{sl}_{n}\C$ and every constant $M\geq 1$ we define
$$\mathbb{V}_{M}(\mathbb{E})=\;\ooo{M}{a_{1}}{a_{2}}\;...\;\oo{a_{n-2}}{a_{n-1}},$$
a representation of $\g$, which we also denote by
%The notation
%$$\mathbb{V}_{M}(\mathbb{E})=(0,b_{0},b_{1},b_{2},....,b_{n-1})=\left(0,M,a_{1}+M,a_{1}+a_{2}+M,...,\sum_{i=1}^{n-1}a_{i}+M\right)$$
%will also be used whenever we want to describe the action of the Weyl group $\mathcal{W}$ on the weight, because $\mathcal{W}\cong\mathbb{S}_{n+1}$ and it acts on $(a,b,c,...,d)$ by permutation. When referring to a representation of $\p$, we will also use the notation $(a|b,c,...,d,e,f)$ for $\;\xo{b-a}{c-b}\;...\;\oo{e-d}{f-e}.$
$$\mathbb{V}_{M}(\mathbb{E})=(0,b_{0},b_{1},b_{2},....,b_{n-1})=\left(0,M,a_{1}+M,a_{1}+a_{2}+M,...,\sum_{i=1}^{n-1}a_{i}+M\right).$$
When referring to a representation of $\p$, we will use the notation $(a|b,c,...,d,e,f)$ for $\;\xo{b-a}{c-b}\;...\;\oo{e-d}{f-e}.$
This is important whenever we want to describe the action of the Weyl group $\mathcal{W}$ on the weight, because $\mathcal{W}\cong\mathbb{S}_{n+1}$ and it acts  by permutation (and renormalization to account for the usual ambiguity 
$$(a+m|b+m,c+m,...,d+m,e+m,f+m)=(a|b,c,...,d,e,f)$$
for all $m\in\Z$). 
\end{definition}
%The dominance of a weight can be seen in both notations. 
%In the Dynkin diagram notation it corresponds to the non-negativity of all the integers and in the $(a,b,c,...,d,e)$ notation it corresponds to the case when the sequence $a,b,c,...,d,e$ is non-decreasing.
\par
\vspace{0.2cm} 
The $\g$-module $\mathbb{V}_{M}(\mathbb{E})$ has, as a $\p$-module, a composition series
$$\mathbb{V}_{M}(\mathbb{E})=\mathbb{V}_{0}+\mathbb{V}_{1}+\mathbb{V}_{2}+...+\mathbb{V}_{N},$$
where each $\mathbb{V}_{i}$ decomposes into a direct sum of irreducible $\p$-modules and 
$$\mathbb{V}_{0}=\xooo{M}{a_{1}}{a_{2}}{a_{3}}\;...\;\oo{a_{n-2}}{a_{n-1}}.$$
We may tensor this composition series by $\mathcal{O}(k-M)$ to obtain
$$\mathbb{V}_{0}(k-M)= \xooo{k}{a_{1}}{a_{2}}{a_{3}}\;...\;\oo{a_{n-2}}{a_{n-1}}.$$
This is the $\p$-module that we are interested in and we want to define a mapping 
$$V_{0}(k-M)\rightarrow V_{M}(\mathbb{E})(k-M)$$
using the theorems from the last section. Hence we have to make sure that the generalized Verma modules associated to all the irreducible composition factors of $\mathbb{V}_{M}(\mathbb{E})(k-M)$ have a central character which is different from the central character of~$M_{\p}(\mathbb{V}_{0}(k-M))$.

\subsubsection{Remark}\label{branching}
In the case of  projective geometry, Proposition~\ref{propseven} can be proved directly using Pierie's formula, as in~\cite{fh}, p.~225, for the tensor product $\odot^{l}\g_{1}\otimes\mathbb{E}$ and the branching rules for restrictions of representations of $\mathfrak{sl}_{n+1}\C$ to $\mathfrak{sl}_{n}\C$ as in~\cite{gw}, p.~350. The upshot of this procedure is that we obtain a more precise statement than Proposition~\ref{propseven}, namely that
$\mathbb{V}_{l}(k-M)$ consists of terms $(M-k+l|\tilde{b}_{0},\tilde{b}_{1},...,\tilde{b}_{n-1})$ that {\bf interlace} $(M-k|b_{0},b_{1},....,b_{n-1})$, i.e.
$$0\leq\tilde{b}_{0}\leq b_{0}\leq \tilde{b}_{1}\leq b_{1}\leq\tilde{b}_{2}\leq b_{2}\leq...\leq\tilde{b}_{n-1}\leq b_{n-1}$$
and  $\sum_{i=0}^{n-1}b_{i}-\sum_{i=0}^{n-1}\tilde{b}_{i}=l$.
We can also see that $N=\sum_{i=1}^{n-1}a_{i}+M$, because for $l>N$ it is not possible for any $(M-k+l|\tilde{b}_{0},\tilde{b}_{1},...,\tilde{b}_{n-1})$ to interlace~$(M-k|b_{0},b_{1},...,b_{n-1})$.

\begin{proposition}\label{propeleven}
The only irreducible components of $\mathbb{V}_{l}(k-M)$ that can induce generalized Verma modules with the same central character as $M_{\p}(\mathbb{V}_{0}(k-M))$
are the ones that are of the form
$$(M-k+l|b_{0},b_{1},...,b_{j-1},b_{j}-l,b_{j+1},...,b_{n-1}),$$
for $j=0,1,...,n-1$. If $j\in\{1,...,n-1\}$, then this is only allowed for $a_{j}\geq l$ and if $j=0$, then this is only allowed for $l\leq M$.
In that case the generalized Verma module has the same central character as $M_{\p}(\mathbb{V}_{0}(k-M))$ if and only if
$$k=-\left(\sum_{i=1}^{j}a_{i}+j-l+1\right).$$
For $j=0$, this condition reads $k=l-1$.
\end{proposition}\begin{proof}
Using the remark in Section~\ref{branching}, we know that an arbitrary irreducible component $\mathbb{V}_{l,v}(k-M)$ of $\mathbb{V}_{l}(k-M)$ has to be of the form $(M-k+l|\tilde{b}_{0},...,\tilde{b}_{n-1})$  so that $(M-k+l|\tilde{b}_{0},...,\tilde{b}_{n-1})$ interlaces~$(M-k|b_{0},...,b_{n-1})$.  Let us assume that there are at least two integers $0\leq i<j\leq n-1$ such that $\tilde{b}_{i}<b_{i}$ and~$\tilde{b}_{j}<b_{j}$. We can assume that $i$ is the smallest integer with this property and that $j$ is the biggest integer with this property.
\par
Theorem~\ref{theoremthirteen} implies that the central characters of the generalized Verma modules $M_{\p}(\mathbb{V}_{0}(k-M))$ and $M_{\p}(\mathbb{V}_{l,v}(k-M))$ are identical if and only if there is an element in the Weyl group, i.e.~a permutation, that
maps the weight $(M-k+l|\tilde{b}_{0},...,\tilde{b}_{n-1})+\rho_{\g}$ to $(M-k|b_{0},...,b_{n-1})+\rho_{\g}$. Using $\rho_{\g}=(1,2,...,n,n+1)$, we obtain the condition that the two sets
$$\{M-k+1,b_{0}+2,b_{1}+3,...,b_{i}+i+2,...,b_{j}+j+2,...,b_{n-1}+n+1\}$$
and
$$\{M-k+l+1,\tilde{b}_{0}+2,\tilde{b}_{1}+3,...,\tilde{b}_{i}+i+2,...,\tilde{b}_{j}+j+2,...,\tilde{b}_{n-1}+n+1\}$$
have to be equal.
This is equivalent to
$$\{M-k+1,b_{i}+i+2,...,b_{j}+j+2\}=\{M-k+l+1,\tilde{b}_{i}+i+2,...,\tilde{b}_{j}+j+2\},$$
where the sets contain all those $b_{m}+m+2$, resp. $\tilde{b}_{m}+m+2$, for which~$\tilde{b}_{m}\not=b_{m}$. Furthermore, leaving out $M-k+1$, all numbers in the first set are increasing from left to right.
Since $\tilde{b}_{i}<b_{i}$, $\tilde{b}_{i}+i+2$ is smaller than the second entry in the first set and therefore smaller than everything but the first entry, i.e.~we must
have~$\tilde{b}_{i}+i+2=M-k+1$. %But also $s\geq 1$, so the last entry in the first set is bigger than the last entry in the second set and therefore bigger than everything but the first entry, i.e. we must have $b_{j}+j+2=-k+l+1+M$. 
Moreover $\tilde{b}_{j}<b_{j}$ implies that there has to be an integer $m<j$ so that  
\begin{eqnarray*}
&\tilde{b}_{j}+j+2=b_{m}+m+2
\Rightarrow&\tilde{b}_{j}+j= b_{m}+m.
\end{eqnarray*}
This is not possible, because $\tilde{b}_{j}\geq b_{m}$ and~$j>m$. That proves the first claim.
\par
Let us now assume that $\mathbb{V}_{l,v}(k-M)=(k-M+l|b_{0},b_{1},...,b_{j-1},b_{j}-l,b_{j+1},...,b_{n-1})$. In this case $M_{\p}(\mathbb{V}_{l,v}(k-M))$ has the same central character as $M_{\p}(\mathbb{V}_{0}(k-M))$ if and only if
$$\{M-k+l+1,b_{j}-l+j+2\}\;=\;\{M-k+1,b_{j}+j+2\},$$
which is equivalent to $k=-b_{j}+M-j+l-1=-\left(\sum_{i=1}^{j}a_{i}+j-l+1\right)$.
\end{proof}

%\begin{definition}
%\rm Let $\mathfrak{h}^{S}$ denote the Cartan subalgebra of~$\g_{0}^{S}$. Then we have
%$$(\mathfrak{h}^{S})^{*}=\C\langle L_{1},...,L_{n}\rangle/(L_{1}+...+L_{n}=0),$$
%where $L_{i}(H_{j})=\delta_{i,j}$ and $H_{j}$ denotes the matrix which has a one in the $j$-th diagonal entry and zeros elsewhere as an element in~$\mathfrak{sl}_{n}\C$.
%\end{definition}

\subsection{Excluded weights}

\begin{proposition}\label{proptwelve}
If
$$k=-\left(\sum_{i=1}^{j}a_{i}+j-l+1\right),$$
then there exitst an $l$-th order invariant linear differential operator  
$$\xooo{k}{a_{1}}{a_{2}}{a_{3}}\;...\;\oo{a_{n-2}}{a_{n-1}}\;\rightarrow \xoo{k-l}{a_{1}}{a_{2}}\;...\oo{a_{j}-l}{a_{j+1}+l}\;...\;\oo{a_{n-2}}{a_{n-1}}.$$
\end{proposition}
\begin{proof}
As proved in~\cite{css}, p.~65, Corollary 5.3, the condition for this operator to be invariant is
$$\omega=(\tilde{\alpha}+\epsilon_{n-j},\rho)-\frac{1}{2}(l-1)(|\tilde{\alpha}|^{2}+1)-(-\epsilon_{n-j},\tilde{\lambda}),$$
where $\omega=-\frac{1}{n+1}\left(nk+\sum_{i=1}^{n-1}(n-i)a_{i}\right)$ is the geometric weight of $\xo{k}{a_{1}}\;...\;\oo{a_{n-2}}{a_{n-1}}$, $(.,.)$ is the normalized Killing form as in Corollary 6
and $\tilde{\alpha}=-\epsilon_{n}$ is the highest weight of~$\g_{1}$. Moreover $\rho=\rho_{\mathfrak{sl}_{n}\C}=\sum_{i=1}^{n-1}(n-i)\epsilon_{i}$, $|\alpha|^{2}=(\alpha,\alpha)$ and 
$\tilde{\lambda}=\sum_{i=1}^{n-1}\lambda_{i}\epsilon_{i}$ (we can always assume that $\lambda_{n}=0$, which implies $\lambda_{n-j}=\sum_{i=1}^{j}a_{j}$) is the highest weight 
of~$\mathbb{E}$. 
Using
\begin{eqnarray*}
(\tilde{\alpha}+\epsilon_{n-j},\rho)&=&\frac{nj}{n+1},\\
|\tilde{\alpha}|^{2}&=&\frac{n-1}{n+1},\\
(\epsilon_{n-j},\tilde{\lambda})&=&\frac{n\lambda_{n-j}-\sum_{i=1}^{n}\lambda_{i}}{n+1}
\end{eqnarray*}
and the formula for $\omega$ from above, we see that
$$\omega=(\alpha+\epsilon_{n-j},\rho)-\frac{1}{2}(l-1)(|\tilde{\alpha}|^{2}+1)-(-\epsilon_{n-j},\tilde{\lambda})\Leftrightarrow k=-\left(\sum_{i=1}^{j}a_{i}+j-l+1\right).$$
Note that these calculations for $j\in\{1,...,n-1\}$ make only sense if $a_{j}\geq l$. If $j=0$, then $l$ may be arbitrary.
\par
We could have also proven this proposition utilizing Chapter~\ref{higherorderone}. 
\end{proof}

The problem is, when we look at $M$-th order pairings, we do not really want to exclude weights that correspond to operators that have a higher order. The following 
lemma excludes such a situation at the cost of a restriction on the integers $a_{i}$.

\begin{lemma}\label{lemmasixteen}
Let $M\geq \mathrm{max}_{i}\{a_{i}\}$, then no weights have to be excluded for $l> M$.
\end{lemma}
\begin{proof}
As discussed earlier, an irreducible component of $\mathbb{V}_{l}(k-M)$ that induces a generalized Verma module with the same central character as $M_{\p}(\mathbb{V}_{0}(k-M))$ can only arise by taking
$$(M-k|b_{0},b_{1},...,b_{n-1})$$
and subtracting $l$ from one of the $b_{i}$'s to obtain
$$(M-k+l|\tilde{b}_{0},\tilde{b}_{1},...,\tilde{b}_{n-1})$$
so that $(M-k+l|\tilde{b}_{0},...,\tilde{b}_{n-1})$ interlaces~$(M-k|b_{0},...,b_{n-1})$.
But 
$$b_{i}-b_{i-1}=a_{i}\leq M< l\;\forall\;i=1,...,n-1,$$
so subtracting $l$ from any $b_{i}$, $i\geq 1$, leads to $\tilde{b}_{i}=b_{i}-l<b_{i-1}$, which is not allowed. 
Subtracting $l$ from $b_{0}$ leaves $\tilde{b}_{0}=M-l<0$, which is also not allowed.
Therefore all irreducible components of $\mathbb{V}_{l}(k-M)$, for $l>M$, induce a generalized Verma module that has a central character which is different from the one of $M_{\p}(\mathbb{V}_{0}(k-M))$.
\end{proof}

\subsubsection{Examples}\label{excludedweights}
\begin{description}
\item{(a)}
Let us look at symmetric  two tensors of projective weight $v$, i.e.~sections of the vector bundle $\odot^{2}T\mathcal{M}\otimes\mathcal{O}(v)$ for~$M=2$:
\begin{eqnarray*}
\oo{2}{0}\;...\;\oo{0}{2}(v)\;&=&\;\xo{2+v}{0}\;...\;\oo{0}{2}\;+\begin{array}{c}
\;\xo{1+v}{0}\;...\;\oo{0}{1}\;\\
\oplus\\
\;\xoo{v}{1}{0}\;...\;\oo{0}{2}\;
\end{array}
+\begin{array}{c}
\;\xo{v}{0}\;...\;\oo{0}{0}\;\\
\oplus\\
\;\xoo{v-1}{1}{0}\;...\;\oo{0}{1}\;\\
\oplus\\
\;\xoo{v-2}{2}{0}\;...\;\oo{0}{2}\;
\end{array}\\
&&+\begin{array}{c}
\;\xoo{v-2}{1}{0}\;...\;\oo{0}{0}\;\\
\oplus\\
\;\xoo{v-3}{2}{0}\;...\;\oo{0}{1}\;
\end{array}+
\;\xoo{v-4}{2}{0}\;...\;\oo{0}{0}\;.\end{eqnarray*}
The weights to exclude are
\begin{enumerate}
\item
$v=-2,-(n+3)$ which correspond to the first order invariant linear differential operators $\nabla_{a}V^{bc}-\frac{2}{n+1}\delta_{a}{}^{(b}\nabla_{d}V^{c)d}$ and $\nabla_{a}V^{ab}$ respectively;
\item
$v=-1,-(n+2)$ which correspond to second order invariant linear differential operators $\nabla_{a}\nabla_{b}V^{cd}+P_{ab}V^{cd}-\mathrm{trace}$ and $\nabla_{a}\nabla_{b}V^{ab}+P_{ab}V^{ab}$ respectively.
\end{enumerate}
\item{(b)}
Another example for vector fields of projective weight $v$, i.e.~sections of $T\mathcal{M}\otimes\mathcal{O}(v)$, with~$M=1$:
$$\;\oo{1}{0}\;...\;\oo{0}{1}(v)\;=\;\xo{1+v}{0}\;...\;\oo{0}{1}\;+\begin{array}{c}
\;\xo{v}{0}\;...\;\oo{0}{0}\;\\
\oplus\\
\;\xoo{v-1}{1}{0}\;...\;\oo{0}{1}\;
\end{array}
+\;\xoo{v-2}{1}{0}\;...\;\oo{0}{0}\;.$$
The weights to exclude are $v=-1,-(n+1)$ corresponding to the first order  invariant linear differential operators $\nabla_{a}V^{b}-\frac{1}{n}\delta_{a}{}^{b}\nabla_{c}V^{c}$ and $\nabla_{a}V^{a}$
respectively.
\item{(c)}
The last example deals with weighted functions and a general~$M$:
\begin{eqnarray*}\;\oo{M}{0}\;...\;\oo{0}{0}(w-M)\;&=&\;\xo{w}{0}\;...\;\oo{0}{0}\;+\;\xoo{w-2}{1}{0}\;...\;\oo{0}{0}\;+\;\xoo{w-4}{2}{0}\;...\;\oo{0}{0}\;\\
&&+...+\;\xoo{w-2M}{M}{0}\;...\;\oo{0}{0}\;.\end{eqnarray*}
The weights to exclude are $w=0,1,...M-1$ corresponding to the  invariant linear differential operators $\underbrace{\nabla_{(a}...\nabla_{c)}}_{w+1}f+C.C.T.$ respectively.
\end{description}

\subsection{Classification and Examples} 
To state the main theorem, we have to define precisely what we mean by excluded weights.

\begin{definition}
\rm Let $\;\xoo{k}{a_{1}}{a_{2}}\;...\;\oo{a_{n-2}}{a_{n-1}}\;$ be a representation of $\p$. Then the {\bf excluded} {\bf weights} up to order $M$ consist of 
all $k$ such that there is a $1\leq l\leq M$ and a $0\leq j\leq n-1$ with
$$k=-\left(\sum_{i=1}^{j}a_{i}+j-l+1\right)\quad\text{and}\; a_{j}\geq l.$$
For $j=0$, the excluded weights are $k=l-1$ for~$1\leq l\leq M$.
\end{definition}

\begin{theorem}
Let $\;\xoo{k}{a_{1}}{a_{2}}\;...\;\oo{a_{n-2}}{a_{n-1}}\;$ and $\;\xoo{m}{b_{1}}{b_{2}}\;...\;\oo{b_{n-2}}{b_{n-1}}\;$ be irreducible associated bundles on $\mathcal{M}$.
If $M\geq \max_{i}\{a_{i},b_{i}\}$ and $k$ and $m$ are not equal to one of the excluded weights up to order $M$, then there exists an $r$ parameter family of $M$-th order invariant bilinear differential pairings
$$ \;\xoo{k}{a_{1}}{a_{2}}\;...\;\oo{a_{n-2}}{a_{n-1}}\;\times\;\xoo{m}{b_{1}}{b_{2}}\;...\;\oo{b_{n-2}}{b_{n-1}}\;\rightarrow \;\xoo{s}{c_{1}}{c_{2}}\;...\;\oo{c_{n-2}}{c_{n-1}},\;$$
where $r$ is the multiplicity of $\;\oo{c_{1}}{c_{2}}\;...\;\oo{c_{n-2}}{c_{n-1}}\;$ in 
$$\odot^{M}\g_{1}\otimes\;\oo{a_{1}}{a_{2}}\;...\;\oo{a_{n-2}}{a_{n-1}}\;\otimes\;\oo{b_{1}}{b_{2}}\;...\;\oo{b_{n-2}}{b_{n-1}}\;.$$
Excluded weights correspond to the existence of invariant linear differential operators of order $\leq M$ emanating from the bundles in question.
\end{theorem}
\begin{proof}
If $M\geq \max_{i}\{a_{i},b_{i}\}$ and $k$ and $m$ are not equal to one of the excluded weights up to order $M$, we can use Lemma~\ref{lemmasixteen}, Proposition~\ref{propeleven} and Corollarly~\ref{coreight} to define
invariant differential operators that take $\;\xoo{k}{a_{1}}{a_{2}}\;...\;\oo{a_{n-2}}{a_{n-1}}\;$ and $\;\xoo{m}{b_{1}}{b_{2}}\;...\;\oo{b_{n-2}}{b_{n-1}}\;$ into their $M$-bundles.
Then we decompose the tensor product of the $M$-bundles as described in Proposition~\ref{propeight} and project onto the first composition factor of each of the irreducible components.
That also yields all the invariant pairings of order smaller than $M$, but we may have to exclude more weights than necessary. Moreover there cannot be more invariant bilinear differential pairings that are non-zero in the flat case, because then one would be able to find a linear combination of all those pairings that does not involve the highest order terms ($M$ derivatives) in sections of one of the bundles. But obstruction terms involving $M-1$ derivatives in the sections of that bundle and one $\Upsilon$-term would therefore only occur in $\odot^{M-1}\g_{1}\otimes \mathbb{E}\otimes\g_{1}\otimes \mathbb{F}$ (if $\mathbb{E}$ and $\mathbb{F}$ denote the corresponding $\g_{0}^{S}$-modules as before) and one would not be able to eliminate them, because no operator in the formula is invariant. The last statement follows from Proposition~\ref{proptwelve}.
\end{proof}

\subsubsection{Example}
Let us carry out the described construction for first order pairings between weighted $2$-forms and weighted vector fields for four dimensions. The corresponding $M$ bundles have composition series
$$\;\oooo{1}{0}{1}{0}\;=\;\xooo{1}{0}{1}{0}\;+\begin{array}{c}
\;\xooo{0}{0}{0}{1}\;\\
\oplus\\
\;\xooo{-1}{1}{1}{0}\;
\end{array}
+\;\xooo{-2}{1}{0}{1}\;$$
and
$$\;\oooo{1}{0}{0}{1}\;=\;\xooo{1}{0}{0}{1}\;+\begin{array}{c}
\;\xooo{0}{0}{0}{0}\;\\
\oplus\\
\;\xooo{-1}{1}{0}{1}\;
\end{array}
+\;\xooo{-2}{1}{0}{0}\;.$$
If we tensor these together, we obtain a composition series
$$
\;\oooo{1}{0}{1}{0}\;\otimes\;\oooo{1}{0}{0}{1}\; =
\left(\begin{array}{c}
\;\xooo{2}{0}{1}{1}\;\\
\oplus\\
\;\xooo{2}{1}{0}{0}\;
\end{array}\right)
+\left(\begin{array}{ccc}
4\times\;\xooo{1}{0}{1}{0}\;&\oplus&2\times\;\xooo{0}{1}{1}{1}\;\\
&\oplus&\\
2\times\;\xooo{0}{2}{0}{0}\;&\oplus&2\times\;\xooo{1}{0}{0}{2}\;
\end{array}\right)
$$
$$
+\left(\begin{array}{ccc}
3\times\;\xooo{-1}{1}{0}{2}\;&\oplus&6\times\;\xooo{-1}{1}{1}{0}\;\\
&\oplus&\\
5\times\;\xooo{0}{0}{0}{1}\;&\oplus&\;\xooo{-2}{2}{1}{1}\;\\
&\oplus&\\
\;\xooo{-1}{0}{2}{1}\;&\oplus&\;\xooo{-2}{3}{0}{0}\;
\end{array}\right)
+\left(\begin{array}{ccc}
5\times\;\xooo{-2}{1}{0}{1}\;&\oplus&2\times\;\xooo{-1}{0}{0}{0}\;\\
&\oplus&\\
2\times\;\xooo{-3}{2}{1}{0}\;&\oplus&2\times\;\xooo{-2}{0}{2}{0}\;\\
&\oplus&\\
\;\xooo{-3}{2}{0}{2}\;&\oplus&\;\xooo{-2}{0}{1}{2}\;
\end{array}\right)
$$
$$
+\left(\begin{array}{c}
\;\xooo{-4}{2}{0}{1}\\
\oplus\\
\;\xooo{-3}{0}{1}{1}\;\\
\oplus\\
\;\xooo{0}{1}{0}{0}\;
\end{array}\right).
$$
This composition series can be split up according to
\begin{eqnarray*}
\;\oooo{1}{0}{1}{0}\;\otimes\;\oooo{1}{0}{0}{1}\;& =&\;\oooo{2}{0}{1}{1}\;\oplus\;\oooo{0}{1}{1}{1}\;\oplus\;\oooo{2}{1}{0}{0}\;\\
&&\oplus\;\oooo{1}{0}{0}{2}\;\oplus\;\oooo{0}{2}{0}{0}\;\\
&&\oplus 2\times\;\oooo{1}{0}{1}{0}\;
\;\oplus\;\oooo{0}{0}{0}{1},\;
\end{eqnarray*}
which compose as
$$\;\oooo{2}{0}{1}{1}\;=\;\xooo{2}{0}{1}{1}\;+
\begin{array}{c}
\;\xooo{1}{0}{0}{2}\;\\
\oplus\\
\;\xooo{0}{1}{1}{1}\;\\
\oplus\\
\;\xooo{1}{0}{1}{0}\;
\end{array}
+
\begin{array}{c}
\;\xooo{-1}{1}{0}{2}\;\\
\oplus\\
\;\xooo{-2}{2}{1}{1}\;\\
\oplus\\
\;\xooo{0}{0}{0}{1}\;\\
\oplus\\
\;\xooo{-1}{1}{1}{0}\;
\end{array}
+
\begin{array}{c}
\;\xooo{-2}{1}{0}{1}\;\\
\oplus\\
\;\xooo{-3}{2}{1}{0}\;\\
\oplus\\
\;\xooo{-3}{2}{0}{2}\;
\end{array}
+
\;\xooo{-4}{2}{0}{1}\;,$$

$$\;\oooo{0}{1}{1}{1}\;=\;\xooo{0}{1}{1}{1}\;+
\begin{array}{c}
\;\xooo{-1}{0}{2}{1}\;\\
\oplus\\
\;\xooo{-1}{1}{0}{2}\;\\
\oplus\\
\;\xooo{-1}{1}{1}{0}\;
\end{array}
+
\begin{array}{c}
\;\xooo{-2}{0}{1}{2}\;\\
\oplus\\
\;\xooo{-2}{1}{0}{1}\;\\
\oplus\\
\;\xooo{-2}{0}{2}{0}\;
\end{array}
+
\;\xooo{-3}{0}{1}{1}\;,$$

$$\;\oooo{2}{1}{0}{0}\;=\;\xooo{2}{1}{0}{0}\;+
\begin{array}{c}
\;\xooo{1}{0}{1}{0}\;\\
\oplus\\
\;\xooo{0}{2}{0}{0}\;
\end{array}
+
\begin{array}{c}
\;\xooo{-1}{1}{1}{0}\;\\
\oplus\\
\;\xooo{-2}{3}{0}{0}\;
\end{array}
+
\;\xooo{-3}{2}{1}{0}\;,$$

$$\;\oooo{1}{0}{0}{2}\;=\;\xooo{1}{0}{0}{2}\;+
\begin{array}{c}
\;\xooo{0}{0}{0}{1}\;\\
\oplus\\
\;\xooo{-1}{1}{0}{2}\;
\end{array}
+
\begin{array}{c}
\;\xooo{-1}{0}{0}{0}\;\\
\oplus\\
\;\xooo{-2}{1}{0}{1}\;
\end{array}
+
\;\xooo{0}{1}{0}{0}\;,$$

$$\;\oooo{1}{0}{1}{0}\;=\;\xooo{1}{0}{1}{0}\;+
\begin{array}{c}
\;\xooo{0}{0}{0}{1}\;\\
\oplus\\
\;\xooo{-1}{1}{1}{0}\;
\end{array}
+
\;\xooo{-2}{1}{0}{1}\;$$
and
$$\;\oooo{0}{2}{0}{0}\;=\;\xooo{0}{2}{0}{0}\;+
\;\xooo{-1}{1}{1}{0}\;
+
\;\xooo{-2}{0}{2}{0}\;,
$$

$$\;\oooo{0}{0}{0}{1}\;=\;\xooo{0}{0}{0}{1}\;+
\;\xooo{-1}{0}{0}{0}\;.
$$
There are 5 first order invariant bilinear differential pairings according to the projections onto (including the weights $k=1+v$ for vector fields of projective weight $v$ and $m=w-3$ for 2-forms of projective weight $w$, i.e.~we have to tensor by the line bundle $\mathcal{O}(k-M)\otimes\mathcal{O}(m-M)=\mathcal{O}(v+w-4)$):
$$\;\xooo{v+w-4}{2}{0}{0}\;,\;\xooo{v+w-3}{0}{0}{2}\;,\;\xooo{v+w-4}{1}{1}{1}\;$$
and the two projections onto 
$$\;\xooo{v+w-3}{0}{1}{0}\;=\Omega^{2}(v+w)\;,$$
corresponding to
$$\;\ooo{0}{0}{1}\;\otimes\g_{1}\otimes \;\ooo{0}{1}{0}\;=2\times \;\ooo{0}{1}{0}\;\oplus\;\ooo{1}{1}{1}\;\oplus\;\ooo{2}{0}{0}\;\oplus\;\ooo{0}{0}{2}\;.$$
The concrete formulae for the two projections onto $\Omega^{2}(v+w)$ were given at the end of~\ref{examplesfirstorder}.

\subsection{Weighted functions of excluded geometric weight}
Returning to Example~\ref{excludedweights} (c), let us assume that the central character of the generalized Verma module $M_{\p}(\mathbb{V}_{0}(w-M))$ equals the central character of $M_{\p}(\mathbb{V}_{l}(w-M))$,
i.e.~$0\leq w=l-1\leq M-1$. This corresponds to an $l$-th order invariant differential operator 
$$D:\xooo{w}{0}{0}{0}\;...\;\oo{0}{0}\;\rightarrow \xoo{w-2l}{l}{0}\;...\oo{0}{0}\;...\;\oo{0}{0}.$$
Hence one can invariantly write $D(f)$, for $f\in\mathcal{O}(w)$.
Now we look at the $\p$-module
$$\tilde{\mathbb{V}}_{M,l}(\C)(w-M)=\mathbb{V}_{l}(w-M)+\mathbb{V}_{l+1}(w-M)+...+\mathbb{V}_{M}(w-M).$$
The central character of  $M_{\p}(\mathbb{V}_{l}(k-M))$ is different from the central character of all the other generalized Verma modules, because each $M_{\p}(\mathbb{V}_{s}(w-M))$ has the same central character as $M_{\p}(\mathbb{V}_{0}(w-M))$ if and only if~$w=s-1$. Therefore we can define an invariant differential mapping
$$\mathcal{O}(w)\stackrel{D}{\rightarrow}V_{l}(w-M)\rightarrow \tilde{V}_{M,l}(\C)(w-M)\hookrightarrow V_{M}(\C)(w-M).$$
%In this mapping $D(f)$ is written at the appropriate place and all slots to the left contain zeros only. 
The invariant pairings that we obtain via this construction do not involve derivatives of $f$ of order smaller than~$l$. This is confirmed by the formulae obtained earlier. 
%The pairings we obtain are non-zero only for orders that are bigger or equal to $l$, the other irreducible factors will have zeroes as their first projection.
\par
%In principal this can be done for an arbitrary bundle, however, in this case
%$\mathbb{V}_{l}$ will consist of more than one irreducible factor and when we want to include $D(V)$ into the bundle $\mathbb{V}_{l}(k-M)+\mathbb{V}_{l+1}(k-M)+...+\mathbb{V}_{N}(k-M)$, then we loose all other irreducible factors of $\mathbb{V}_{l}$.  
%\par
These considerations yield:
\begin{corollary}
If $M\geq\mathrm{max}_{i}\{a_{i}\}$ and $k$ does not equal one of the excluded weights up to order $M$ for $V=\;\xoo{k}{a_{1}}{a_{2}}\;...\;\oo{a_{n-2}}{a_{n-1}}\;$, then there exists
a one parameter family of invariant bilinear differential pairings of order $M$ between sections of $V$ and arbitrarily weighted functions onto every bundle that is induced by
an irreducible component of $\odot^{M}\g_{1}\otimes\;\oo{a_{1}}{a_{2}}\;...\;\oo{a_{n-2}}{a_{n-1}}\;$.
\end{corollary}

\subsubsection{Remark}
Using Pierie's formula, it is clear that the tensor product $\odot^{M}\g_{1}\otimes\;\oo{a_{1}}{a_{2}}\;...\;\oo{a_{n-2}}{a_{n-1}}\;$ does not have multiplicities.

\subsubsection{Example}
Let us analyze the example given in~\ref{theproblemwithhigherorderpairings}, where we considered second order pairings $\mathrm{Vect}(\mathcal{M})(v)\times\mathcal{O}(w)\rightarrow\Omega^{1}(v+w)$. For this purpose we decompose
$$\odot^{2}\g_{1}\otimes\oo{0}{0}\;...\;\oo{0}{1}\;=\;\oo{2}{0}\;...\;\oo{0}{1}\oplus\;\oo{1}{0}\;...\;\oo{0}{0}.$$
Therefore if $v\not=-1,-(n+1)$ (for the other projection we also need to exclude $v=0$), then there should be a second order invariant differential pairing. This is true and the formula was given in~\ref{theproblemwithhigherorderpairings}. Moreover one can clearly see which terms vanish in case the weight $w$ is excluded.

\chapter{Explicit formulae: Tractor calculus}\label{tractorchapter}
In this chapter we will review some of the basic properties of projective, conformal and CR geometry in order to construct some explicit splittings with the help of tractor calculus.
These splittings can be used to obtain explicit formulae for invariant differential pairings for the
parabolic geometry in question. All these examples will deal with real manifolds $\mathcal{M}$ and smooth tensor bundles.
\par
To denote the various tensor bundles that occur, we will use Penrose's abstract indices. The tangent space (and sections thereof) will be denoted by $\mathcal{E}^{a}$ and the cotangent space 
(and sections thereof) by $\mathcal{E}_{b}$. All tensor bundles have a description in terms of abstract indices that denote the symmetries of the elements involved. $(\mathcal{E}_{a}{}^{b})_{0}$, 
for example, denotes the bundle of elements $X_{a}{}^{b}$, such that $X_{a}{}^{a}=0$, i.e.~which are trace-free.

\section{Projective geometry}
Throughout this section $\g_{\R}=\mathfrak{sl}_{n+1}\mathbb{R}$ with complexification $\g=A_{n}$ and the grading as given in Example~\ref{grading} (a). The description of the tractor calculus for manifolds with a projective structure follows~\cite{beg}.

\subsection{Projective manifolds}

\begin{definition}
{\rm A {\bf projective structure} on a manifold $\mathcal{M}$ is given by an equivalence class of torsion-free affine connections which have the same (unparametrized) geodesics.}
\end{definition}

\begin{proposition}[\cite{weyl}]
Two torsion free connections $\nabla$, $\hat{\nabla}$ have the same unparametrized geodesics if and only if there is a one form $\Upsilon_{a}$ such that
$$\hat{\nabla}_{a}\omega_{b}=\nabla_{a}\omega_{b}-\Upsilon_{a}\omega_{b}-\Upsilon_{b}\omega_{a}$$
for every one-form $\omega_{b}$.
\end{proposition}
\begin{proof}
A proof may be found in~\cite{e}, Proposition 1.
\end{proof}

\begin{corollary}
Since $\hat{\nabla}_{a}f=\nabla_{a}f$ for every function $f$ and every connection satisfies a Leibniz rule, the difference of $\hat{\nabla}$ and $\nabla$ when acting on vector fields can be deduced.
Again using the Leibniz rule, it is then straightforward to deduce the difference of $\hat{\nabla}$ and $\nabla$ when acting on arbitrary tensor bundles.
\end{corollary}

\begin{definition}
{\rm For every $w\in\R$, let $\mathcal{E}(w)$ denote the line bundle of densities of {\bf projective weight} $w$. This bundle (assuming that $\mathcal{M}$ is oriented) can be defined as $(\Lambda^{n})^{-\frac{w}{n+1}}$, where 
$\Lambda^{n}$ is the line bundle of $n$-forms on the manifold $\mathcal{M}$ of dimension $n$. The tensor product of an arbitrary bundle $\mathcal{E}^{\Phi}$, where $\Phi$ denotes some indices, and $\mathcal{E}(w)$ will be denoted by $\mathcal{E}^{\Phi}(w)$. Note that
$$\hat{\nabla}_{a}f=\nabla_{a}f+w\Upsilon_{a}f$$
for $f\in\mathcal{E}(w)$.}
\end{definition}

\begin{definition}
{\rm The curvature tensor, defined by
$$(\nabla_{a}\nabla_{b}-\nabla_{b}\nabla_{a})V^{c}=R_{ab}{}^{c}{}_{d}V^{d}$$
for every vector field $V^{a}\in\mathcal{E}^{a}$, can be written as
$$R_{ab}{}^{c}{}_{d}=C_{ab}{}^{c}{}_{d}+2\delta_{[a}{}^{c}P_{b]d}+\beta_{ab}\delta_{d}{}^{c},$$
where $C_{ab}{}^{c}{}_{d}$ is the trace-free Weyl tensor, $\beta_{ab}=-2P_{[ab]}$ is skew and $P_{ab}$ is the {\bf Schouten tensor}. This tensor has a transformation law
$$\hat{P}_{ab}=P_{ab}-\nabla_{a}\Upsilon_{b}+\Upsilon_{a}\Upsilon_{b}.$$}
\end{definition}

\subsection{Tractor calculus}

\begin{definition}
{\rm 
Let $\mathbb{A}$ be the standard representation of $\g_{\R}$ on $\R^{n+1}$. The associated bundle $\mathcal{E}^{A}=\mathcal{G}\times_{P}\mathbb{A}$ is called {\bf standard tractor bundle} and it has a composition series
$$\ooo{0}{0}{0}\;...\;\oo{0}{1}=\xoo{0}{0}{0}\;...\;\oo{0}{1}+\xoo{-1}{0}{0}\;...\;\oo{0}{0}.$$
In accordance with~\ref{filtrations}, for every choice of connection $\nabla$, we can write the elements in $\mathcal{E}^{A}$ as
$$\mathcal{E}^{A}\ni V^{A}=\left(\begin{array}{c}
V^{a}\\
\sigma
\end{array}
\right),$$
where $V^{a}\in\mathcal{E}^{a}(-1)$ and $\sigma\in\mathcal{E}(-1)$. 
Under change of affine connection to $\hat{\nabla}$, these elements transform as
$$\widehat{\left(\begin{array}{c}
V^{a}\\
\sigma
\end{array}
\right)}=\left(\begin{array}{c}
V^{a}\\
\sigma-\Upsilon_{a}V^{a}
\end{array}
\right).$$
As explained in  the remark at the end of Section~\ref{tractorconnection}, there is a canonical connection, the {\bf tractor connection}, on the bundle $\mathcal{E}^{A}$ given for each choice of connection by
$$\nabla_{a}\left(\begin{array}{c}
V^{b}\\
\sigma
\end{array}
\right)=\left(\begin{array}{c}
\nabla_{a}V^{b}+\sigma\delta_{a}{}^{b}\\
\nabla_{a}\sigma-P_{ab}V^{b}
\end{array}
\right).$$
Let $\mathbb{A}^{*}$ be the dual of the standard representation of $\g_{\R}$ on $\R^{n+1}$. The associated bundle $\mathcal{E}_{A}=\mathcal{G}\times_{P}\mathbb{A}^{*}$ is  called {\bf standard co-tractor bundle} and it (or rather its complexification) has a composition series
$$\ooo{1}{0}{0}\;...\;\oo{0}{0}=\xoo{1}{0}{0}\;...\;\oo{0}{0}+\xoo{-1}{1}{0}\;...\;\oo{0}{0}.$$
For every choice of connection $\nabla$, we can write the elements in $\mathcal{E}_{A}$ as
$$\mathcal{E}_{A}\ni V_{A}=\left(\begin{array}{c}
\sigma\\
V_{a}
\end{array}
\right),$$
where $V_{a}\in\mathcal{E}_{a}(1)$ and $\sigma\in\mathcal{E}(1)$. 
Under change of connection to $\hat{\nabla}$ these elements transform as
$$\widehat{\left(\begin{array}{c}
\sigma\\
V_{a}
\end{array}
\right)}=\left(\begin{array}{c}
\sigma\\
V_{a}+\Upsilon_{a}\sigma
\end{array}
\right).$$
The tractor connection on the bundle $\mathcal{E}_{A}$ is, for each choice of connection, given by
$$\nabla_{a}\left(\begin{array}{c}
\sigma\\
V_{b}
\end{array}
\right)=\left(\begin{array}{c}
\nabla_{a}\sigma-V_{a}\\
\nabla_{a}V_{b}+P_{ab}\sigma
\end{array}
\right).$$
}\end{definition}

\subsubsection{Remark}\label{tractordescription}
This description of $\mathcal{E}^{A}$ can be used to determine descriptions of all tensor powers of $\mathcal{E}^{A}$ and the corresponding tractor connections by requiring $\nabla$ to satisfy a Leibniz rule. All bundles that are induced from representations of the whole Lie algebra $\g$ are called {\bf tractor bundles} and we use capital letters $A,B,..$ as abstract indices in the same spirit as small letters $a,b,..$ are used as indices for tensor powers of the tangent bundle $\mathcal{E}^{a}$. Moreover we can tensor any tractor bundle with a line bundle $\mathcal{E}(w)$ to change the projective weight as in Remark~\ref{geometricweights}. The tractor connection is not invariant on weighted tractor bundles $\mathcal{E}^{\Phi}(w)$, but has a transformation law $\hat{\nabla}_{a}f=\nabla_{a}f+w\Upsilon_{a}f$, for $f\in\mathcal{E}^{\Phi}(w)$. 
\par 
This remark also applies to the tractor calculus for conformal and CR structures to be presented in the next two sections.
 
\begin{definition}
{\rm Let $f\in\mathcal{E}^{\Phi}(w)$ be a section of a tractor bundle of weight $w$ ($\Phi$ denoting some tractor indices). There exists an invariant operator
\begin{eqnarray*}
D_{A}:\mathcal{E}^{\Phi}(w)&\rightarrow&\mathcal{E}^{\Phi}\otimes\mathcal{E}_{A}(w-1)\\
f&\mapsto& \left(\begin{array}{c}
wf\\
\nabla_{a}f
\end{array}
\right),
\end{eqnarray*}
where $\nabla_{a}$ denotes the appropriate tractor connection on $\mathcal{E}^{\Phi}$.}
\end{definition}

\subsection{Special splittings}

\begin{proposition}\label{propfourteen}
If $V^{i_{1}...i_{k}}\in\mathcal{E}^{(i_{1}...i_{k})}(v)$ and $v\not\in\{-(n+k+\alpha-1)\}_{\alpha=1,...,k}$, then there exists a unique lift to an element $V^{I_{1}...I_{k}}\in\mathcal{E}^{(I_{1}...I_{k})}(v+k)$ such that
$$D_{A}V^{AI_{2}...I_{k}}=0.$$
Each excluded weight $v$ corresponds to the existence of an invariant differential operator.
\end{proposition}
\begin{proof}
We will regard the elements of $\mathcal{E}^{(I_{1}...I_{k})}(v+k)$ as elements in $\mathcal{E}^{I_{1}...I_{k}}(v+k)$ satisfying certain symmetry relations. This can be easily demonstrated for the case $k=2$:
$$\mathcal{E}^{AB}(v+2)=\mathcal{E}^{ab}(v)+\begin{array}{c}
\mathcal{E}^{a}(v)\\
\oplus\\
\mathcal{E}^{b}(v)
\end{array}+\mathcal{E},$$
so that $V^{AB}\in\mathcal{E}^{AB}$ can be written as
$$V^{AB}=\left(\begin{array}{ccccc}
&&V_{1}^{a}&&\\
V^{ab}_{2}&+&\oplus&+&V_{0}\\
&&W_{1}^{b}&&
\end{array}\right),$$
where we will use lower indices to indicate the valence of the corresponding tensor. If $V^{AB}\in\mathcal{E}^{(AB)}(v+2)$, then 
\begin{enumerate}
\item
$V_{2}^{ab}\in\mathcal{E}^{(ab)}(v)$ and
\item
$V_{1}^{a}=W_{1}^{a}\in\mathcal{E}^{a}(v)$.
\end{enumerate}
For higher valence tensors $V^{I_{1}...I_{k}}$, each component $V_{\alpha}^{i_{1}...i_{\alpha}}$ has to be totally symmetric and is equal to all other components with $\alpha$ indices.
\par

We can write $V^{I_{1}I_{2}...I_{k}}\in\mathcal{E}^{(I_{1}I_{2}...I_{k})}(v+k)$ as
$$V^{I_{1}I_{2}...I_{k}}=\left(\begin{array}{c}
V_{k}^{i_{1}I_{2}...I_{k}}\\
V_{k-1}^{I_{2}...I_{k}}
\end{array}\right)$$
and
$$\nabla_{a}V^{I_{1}I_{2}...I_{k}}=\left(\begin{array}{c}
\nabla_{a}V_{k}^{i_{1}I_{2}...I_{k}}+\delta_{a}{}^{i_{1}}V_{k-1}^{I_{2}...I_{k}}\\
\nabla_{a}V_{k-1}^{I_{2}...I_{k}}-V_{k}^{bI_{2}...I_{k}}P_{ab}
\end{array}\right).$$
It is then straightforward to compute
$$D_{I_{1}}V^{I_{1}...I_{k}}=\left(\nabla_{a}V_{k}^{aI_{2}...I_{k}}+(v+k+n)V_{k-1}^{I_{2}...I_{k}}\right).$$
Moreover, for each $l=1,...,\alpha$ and $\alpha=1,...,k$, we can compute
$$\nabla_{a}V_{\alpha}^{i_{1}...i_{l}I_{l+1}I_{l+2}...I_{\alpha}}=\left(\begin{array}{c}
\nabla_{a}V_{\alpha}^{i_{1}...i_{l}i_{l+1}I_{l+2}...I_{\alpha}}+\delta_{a}{}^{i_{l+1}}V_{\alpha-1}^{i_{1}...i_{l}I_{l+2}..I_{\alpha}}\\
\nabla_{a}V_{\alpha-1}^{i_{1}...i_{l}I_{l+2}...I_{\alpha}}-V_{\alpha}^{i_{1}...i_{l}bI_{l+2}...I_{\alpha}}P_{ab}
\end{array}\right)$$
and hence
$$\nabla_{a}V_{\alpha}^{ai_{2}...i_{l}I_{l+1}I_{l+2}...I_{\alpha}}=\left(\begin{array}{c}
\nabla_{a}V_{\alpha}^{ai_{2}...i_{l}i_{l+1}I_{l+2}...I_{\alpha}}+V_{\alpha-1}^{i_{2}...i_{l}i_{l+1}I_{l+2}..I_{\alpha}}\\
\nabla_{a}V_{\alpha-1}^{ai_{2}...i_{l}I_{l+2}...I_{\alpha}}-V_{\alpha}^{ai_{2}...i_{l}bI_{l+2}...I_{\alpha}}P_{ab}
\end{array}\right).$$
Using this equation one can see that $D_{I_{1}}V^{I_{1}...I_{k}}=0$ is equivalent to the following $k$ equations
\begin{equation}\label{projective}
\nabla_{a}V_{\alpha}^{ai_{1}...i_{\alpha-1}}+(v+n+k+\alpha-1)V_{\alpha-1}^{i_{1}...i_{\alpha-1}}-(k-\alpha)V_{\alpha+1}^{abi_{1}...i_{\alpha-1}}P_{ab}=0,
\end{equation}
for $\alpha=1,...,k$.
%These equations originate from 
%$$\nabla_{a}V_{\alpha}^{aI_{1}...I_{\alpha-1}}+(v+n+k)V_{\alpha-1}^{I_{1}...I_{\alpha-1}}-(k-\alpha)V_{\alpha+1}^{abI_{1}...I_{\alpha-1}}P_{ab}=0.$$
%There are in fact $2^{k-1}$ equations but they are all equivalent to the $k$ equations given above.
%To see that note the following:
%\par
%The equation
%$$\nabla_{a}V_{\alpha}^{aI_{1}...I_{\alpha-1}}+(v+n+k)V_{\alpha-1}^{I_{1}...I_{\alpha-1}}-(k-\alpha)V_{\alpha+1}^{abI_{1}...I_{\alpha-1}}P_{ab}=0$$
%is equivalent to
%\begin{eqnarray*}
%\nabla_{a}V_{\alpha}^{ai_{1}I_{2}...I_{\alpha-1}}+(v+n+k+1)V_{\alpha-1}^{i_{1}I_{2}...I_{\alpha-1}}-(k-\alpha)V_{\alpha+1}^{abi_{1}I_{2}...I_{\alpha-1}}P_{ab}&=&0\\
%\nabla_{a}V_{\alpha-1}^{aI_{2}...I_{\alpha-1}}+(v+n+k)V_{\alpha-2}^{I_{2}...I_{\alpha-1}}-(k-(\alpha-1))V_{\alpha}^{abI_{2}...I_{\alpha-1}}P_{ab}&=&0
%\end{eqnarray*}
%and those equations are equivalent to the 3 equations
%\begin{eqnarray*}
%\nabla_{a}V_{\alpha}^{ai_{1}i_{2}I_{3}...I_{\alpha-1}}+(v+n+k+2)V_{\alpha-1}^{i_{1}i_{2}I_{3}...I_{\alpha-1}}-(k-\alpha)V_{\alpha+1}^{abi_{1}i_{2}I_{3}...I_{\alpha-1}}P_{ab}&=&0\\
%\nabla_{a}V_{\alpha-1}^{ai_{1}I_{3}...I_{\alpha-1}}+(v+n+k+1)V_{\alpha-2}^{i_{1}I_{3}...I_{\alpha-1}}-(k-(\alpha-1))V_{\alpha}^{abi_{1}I_{3}...I_{\alpha-1}}P_{ab}&=&0\\
%\nabla_{a}V_{\alpha-2}^{aI_{3}...I_{\alpha-1}}+(v+n+k)V_{\alpha-3}^{I_{3}...I_{\alpha-1}}-(k-(\alpha-2))V_{\alpha-1}^{abI_{3}...I_{\alpha-1}}P_{ab}&=&0.
%\end{eqnarray*}
%These 3 equations are equivalent to 4 further equations and so forth until we obtain $k$ equations.
%\par
If $v\not\in\{-(n+k+\alpha-1)\}_{\alpha=1,...,k}$, then these equations can be uniquely solved starting with $V_{k}^{i_{1}...i_{k}}=V^{i_{1}...i_{k}}$. This shows the uniqueness of the splitting. The
existence can either be shown by explicitly using the transformation rules under change of connection as in Remark~\ref{tractordescription} or by using the general theory from the last chapter.
\par
If $v=-(n+k+\alpha-1)$, then we can use Proposition~\ref{proptwelve} from the last chapter to see that
% then we can solve the equations
%$$\nabla_{a}V_{\alpha}^{ai_{1}...i_{\alpha-1}}+(v+n+k+\alpha-1)V_{\alpha-1}^{i_{1}...i_{\alpha-1}}-(k-\alpha)V_{\alpha+1}^{abi_{1}...i_{\alpha-1}}P_{ab}=0$$
%for $\alpha=k,...,\alpha_{0}+1$. The next equation
%$$\nabla_{a}V_{\alpha_{0}^{ai_{1}...i_{\alpha_{0}-1}}-(k-\alpha_{0})V_{\alpha_{0}+1}^{abi_{1}...i_{\alpha_{0}-1}}P_{ab}$$
%can then be written as a sum of  $\nabla_{i_{1}}...\nabla_{i_{k+1-\alpha_{0}}}V^{i_{1}...i_{k}}$ and terms involving $P_{ab}$ and lower order derivatives of $V_{k}^{i_{1}...i_{k}}$.
%The projection of $D
the differential operator
$$V^{i_{1}...i_{k}}\mapsto \nabla_{i_{1}}...\nabla_{i_{k+1-\alpha}}V^{i_{1}...i_{k}}+\mathrm{C.C.T},$$
where $\mathrm{C.T.T.}$ stands for curvature correction terms, is projectively invariant.

\end{proof}

\subsubsection{Remark}
This proposition was first proved in~\cite {f1}, Proposition 2.1, without explicit use of tractors. The elements $\tilde{V}^{I_{1}...I_{k}}$ in this paper are written down as elements in 
$\mathcal{E}^{(I_{1}...I_{k})}(v+k)$ and so are related to our elements $V^{I_{1}...I_{k}}$ by the projection $\mathcal{E}^{I_{1}...I_{k}}\rightarrow \mathcal{E}^{(I_{1}...I_{k})}$. More explicitly
$$\left(\tilde{V}_{\alpha}^{i_{1}...i_{\alpha}}\right)_{\text{Fox}}=\binom{k}{\alpha}\left(V_{\alpha}^{i_{1}...i_{\alpha}}\right)_{\text{here}}.$$

\begin{corollary}
For $v=-(n+k+\alpha_{0}-1)$, it is possible to write down the full form (including curvature correction terms) of the corresponding invariant differential operator in Proposition~\ref{propfourteen} using the equations~(\ref{projective}) from the proof of the proposition.
\end{corollary}
\begin{proof}
We can solve the equation
$$
\nabla_{a}V_{\alpha}^{ai_{1}...i_{\alpha-1}}+(\alpha-\alpha_{0})V_{\alpha-1}^{i_{1}...i_{\alpha-1}}-(k-\alpha)V_{\alpha+1}^{abi_{1}...i_{\alpha-1}}P_{ab}=0,
$$
for $\alpha=k,...,\alpha_{0}+1$. The next line
$$
\nabla_{a}V_{\alpha_{0}}^{ai_{1}...i_{\alpha_{0}-1}}-(k-\alpha_{0})V_{\alpha_{0}+1}^{abi_{1}...i_{\alpha_{0}-1}}P_{ab}=0
$$
can then be written as $\nabla_{i_{1}}...\nabla_{i_{k-\alpha_{0}+1}}V^{i_{1}...i_{k}}+C.C.T.$. This line is the first non-zero projection of $D_{A}V^{AI_{2}...I_{k}}$ and hence an invariant expression.
\end{proof}

\subsubsection{Example}
Let us take $\alpha_{0}=k-2$, then the first non-zero entry of $D_{A}V^{AI_{1}...I_{k-1}}$ is
$$\nabla_{a}\nabla_{b}\nabla_{c}V^{abci_{1}...i_{k-3}}+2\nabla_{a}(P_{bc}V^{abci_{1}...i_{k-3}})+2P_{ab}\nabla_{c}V^{abci_{1}...i_{k-3}},$$
for $V^{i_{1}...i_{k}}\in\mathcal{E}^{(i_{1}...i_{k})}(-n-2k+3)$, in accordance with the expressions for the curvature correction terms given in Chapter~\ref{higherorderone}. Note that the sign convention of~\cite{f1} for $P_{ab}$ is different.

\begin{theorem}\label{theoremsixteen}
If 
$$v\not\in\{-(n+k+\alpha-1)\}_{\alpha=1,...,k}\cup\{\alpha-k\}_{\alpha=0,...,M-1},$$
then there exists a splitting
$$\mathcal{E}^{(i_{1}...i_{k})}(v)\rightarrow \left(\mathcal{E}^{(I_{1}...I_{k})}_{(J_{1}...J_{M})}\right)_{0}(v+k-M)=\;\ooo{M}{0}{0}\;...\;\oo{0}{k}(v+k-M).$$
Each excluded weight corresponds to the existence of an invariant differential operator:
\begin{enumerate}
\item
If $v=-(n+k+\alpha-1),\; \alpha=1,...,k,$ then
\begin{eqnarray*}
\xoo{v+k}{0}{0}\;...\;\oo{0}{k}&\rightarrow&\xoo{v+\alpha-1}{0}{0}\;...\;\oo{0}{\alpha-1}\\
V^{i_{1}...i_{k}}&\mapsto& \nabla_{i_{1}}...\nabla_{i_{k+1-\alpha}}V^{i_{1}...i_{k}}+\mathrm{C.C.T}
\end{eqnarray*}
is projectively invariant and
\item
if $v=-(k+\alpha),\; \alpha=0,...,M-1$, then
\begin{eqnarray*}
\xoo{v+k}{0}{0}\;...\;\oo{0}{k}&\rightarrow&\xoo{v+k-2(\alpha+1)}{\alpha+1}{0}\;...\;\oo{0}{k}\\
V^{i_{1}...i_{k}}&\mapsto& \nabla_{(j_{1}}...\nabla_{j_{\alpha+1})}V^{i_{1}...i_{k}}-\mathrm{trace}+\mathrm{C.C.T}
\end{eqnarray*}
is projectively invariant.
\end{enumerate}
These splittings are given explicitly by tractor formulae.
\end{theorem}
\begin{proof}
Let $V^{i_{1}...i_{k}}\in\mathcal{E}^{i_{1}...i_{k}}(v)$, then we can use the splitting of Proposition~\ref{propfourteen} to obtain an element $V^{I_{1}...I_{k}}\in\mathcal{E}^{I_{1}...I_{k}}(v+k)$.
The element 
$$D_{(J_{1}}...D_{J_{M})}V^{I_{1}...I_{k}}\in(\mathcal{E}^{(I_{1}...I_{k})}_{(J_{1}...J_{M})})_{0}(v+k-M)$$
is trace-free by construction of $V^{I_{1}...I_{k}}$ and has as its projection onto $\mathcal{E}^{i_{1}...i_{k}}(v)$ the element
$$\prod_{\alpha=0}^{M-1}(v+k-\alpha)V^{i_{1}...i_{k}}.$$
Proposition 12 from the last chapter ensures that the excluded weights correspond to the existence of invariant linear differential operators. Since those operators can be constructed as Ricci-corrected derivatives (as we will see in the Appendix), the machinery in Chapter~\ref{higherorderone} produces explicit formulae for the curvature correction terms.
\end{proof}

\begin{corollary}
This theorem can be used to explicitly write down the formula for every invariant bilinear differential pairing between sections of
$\mathcal{E}^{(i_{1}...i_{k_{1}})}(v_{1})$ and $\mathcal{E}^{(i_{1}...i_{k_{2}})}(v_{2})$. In practice this can be rather tedious, since the expressions for elements in certain tractor bundles  can be rather complicated. This is due to the fact that the tractor bundles above encode the information about {\bf all} invariant bilinear differential pairings. In specific cases, however, one can use certain tricks to make the computations easier. We will demonstrate this in Example 2 below. 
\end{corollary}

\subsection{Examples}

\subsubsection{Example 1}
Let us have a look at weighted vector fields. 
$$\mathcal{E}^{A}{}_{B}(v)=\mathcal{E}^{a}(v)+\begin{array}{c}
\mathcal{E}^{a}{}_{b}(v)\\
\oplus\\
\mathcal{E}(v)
\end{array}+\mathcal{E}_{b}(v).$$
For $X^{a}\in\mathcal{E}^{a}(v)$ and $v\not\in\{-1,-(n+1)\}$, we can define $X^{A}\in\mathcal{E}^{A}(v+1)$ as in Proposition~\ref{propfourteen} and then compute
\begin{eqnarray*}
D_{B}X^{A}&=&\left(\begin{array}{ccc}
&\nabla_{b}X^{a}-\frac{1}{n+v+1}\delta_{b}{}^{a}\nabla_{c}X^{c}&\\
(v+1)X^{a}&&-\frac{1}{n+v+1}\nabla_{b}\nabla_{a}X^{a}-X^{a}P_{ab}\\
&-\frac{v+1}{n+v+1}\nabla_{a}X^{a}&
\end{array}\right)\\
&\in&(\mathcal{E}^{A}{}_{B})_{0}(v)=\;\ooo{1}{0}{0}\;\cdots\;\oo{0}{1}(v).
\end{eqnarray*}
This can be used to determine the exact form of the two invariant bilinear differential pairings 
\begin{eqnarray*}
\mathcal{E}^{a}(v)\times\mathcal{E}^{b}(w)&\rightarrow&\mathcal{E}^{b}(v+w)\\
(X^{a},Y^{b})&\mapsto&(D_{C}X^{A})(D_{A}Y^{B})\;\text{or}\;(D_{C}Y^{A})(D_{A}X^{B})
\end{eqnarray*}
given by
$$(v+1)X^{a}(\nabla_{a}Y^{b}-\frac{1}{n+w+1}\delta_{a}{}^{b}\nabla_{c}Y^{c})-\frac{(w+1)(v+1)}{n+v+1}Y^{b}\nabla_{a}X^{a}$$
and
$$(w+1)Y^{a}(\nabla_{a}X^{b}-\frac{1}{n+v+1}\delta_{a}{}^{b}\nabla_{c}X^{c})-\frac{(v+1)(w+1)}{n+w+1}X^{b}\nabla_{a}Y^{a}.$$
Note that by writing down the formulae in terms of tractors we implicitly project onto the first slot in the composition series.
\par
This result may be contrasted with the theory of natural bilinear differential pairings. The only natural bilinear differential pairings $\Gamma(T\mathcal{M})\times\Gamma(T\mathcal{M})\rightarrow\Gamma(T\mathcal{M})$ are constant multiples of the Lie bracket, see~\cite{kms}, Remark 30.5. We can obtain the Lie bracket for $v=w=0$ by taking the first minus the second pairing from above.

\subsubsection{Example 2} 
The pairing
$$\mathcal{E}^{(ab)}(v)\times\mathcal{E}(w)\rightarrow\mathcal{E}(v+w)$$
can be computed by
$$V^{AB}D_{A}D_{B}f,$$
where $V^{AB}$ is the lift of $V^{ab}\in\mathcal{E}^{(ab)}(v)$ as in Proposition~\ref{propfourteen}. The exact formula is given by
\begin{eqnarray*}
&&V^{ab}\nabla_{a}\nabla_{b}f-\frac{2(w-1)}{n+v+3}(\nabla_{a}V^{ab})(\nabla_{b})f+\frac{w(w-1)}{(n+v+3)(n+v+2)}f\nabla_{a}\nabla_{b}V^{ab}\\
&+&\frac{w(w+v+n+1)}{v+n+2}fV^{ab}P_{ab},
\end{eqnarray*}
in accordance with~\ref{examplepairings} (c).

\section{Conformal geometry}
In this section $\g$ will denote $\mathfrak{so}_{n+2}\C$ and in the Dynkin diagram notation we will have to distinguish $n$ even (in which case we use $D_{m}$) and $n$ odd (in which case we use $B_{m}$).
%In the Dynkin diagram notation we will use the familiar $D_{m}=\mathfrak{so}_{2m}\C$ but keep in mind that the considerations are equally valid for 
%$B_{m}=\mathfrak{so}_{2m+1}\C$.
The description of the tractor calculus for conformal manifolds follows~\cite{beg}.

\subsection{Conformal manifolds}
\begin{definition}
{\rm 
A {\bf conformal manifold} is a pair $(\mathcal{M}, [g])$, where $\mathcal{M}$ is an $n$-dimensional manifold and $[g]$ is an equivalence class of metrics with equivalence given by
$$\hat{g}\sim g\Leftrightarrow \hat{g}=\Omega^{2}g$$
for some smooth nowhere vanishing function $\Omega$. 
}\end{definition}

\begin{definition}
{\rm Let $(\mathcal{M},[g])$ be a conformal manifold and define $Q$ to be the bundle of metrics, which is a subbundle of $\odot^{2}T^{*}\mathcal{M}$ with fiber $\R^{+}$. For every $w\in \R$, write
$\mathcal{E}[w]$ for the line bundle that is associated to the principal fiber bundle $Q$ via the representation $\R^{+}\ni x\mapsto x^{-\frac{w}{2}}\in\mathfrak{gl}(\R)$.
Analogous to the projective case, we will write $\mathcal{E}^{\Phi}[w]$ for the tensor product of $\mathcal{E}[w]$ with an arbitrary bundle $\mathcal{E}^{\Phi}$.}
\end{definition}

\begin{proposition}\label{conformaltransformation}
Let $\mathcal{M}$ be a manifold endowed with an equivalence class of metrics $[g_{ab}]$. Any two metrics are related by $\hat{g}_{ab}=\Omega^{2}g_{ab}$ for some smooth nowhere vanishing function $\Omega$.
Let $\Upsilon_{a}=\Omega^{-1}\nabla_{a}\Omega=\nabla_{a}\log\Omega$, then the Levi-Civita connections $\hat{\nabla}$ and $\nabla$ associated to $\hat{g}_{ab}$ and
$g_{ab}$ are related by
\begin{eqnarray*}
\hat{\nabla}_{a}f&=&\nabla_{a}f+w\Upsilon_{a},\;f\in\mathcal{E}[w]\\
\hat{\nabla}_{a}X^{b}=&=&\nabla_{a}X^{b}+(w+1)\Upsilon_{a}X^{b}-X_{a}\Upsilon^{b}+X^{c}\Upsilon_{c}\delta_{a}{}^{b},\;X^{b}\in\mathcal{E}^{b}[w]\\
\hat{\nabla}_{a}\omega_{b}=&=&\nabla_{a}\omega_{b}+(w-1)\Upsilon_{a}\omega_{b}-\Upsilon_{b}\omega_{a}+\Upsilon^{c}\omega_{c}g_{ab},\;\omega_{b}\in\mathcal{E}_{b}[w]
\end{eqnarray*}
\end{proposition}
\begin{proof}
A proof may be found in~\cite{pr}.
\end{proof}

%\subsubsection{Remark}
%A {\bf conformal Weyl structure} actually consists of a conformal class of metrics $[g]$ together with a distinguished class of all those torsion-free connections, the Weyl connections, that preserve the conformal structure.
%When changing from one Weyl connection to another, the formulae in Proposition~\ref{conformaltransformation} still hold with $\Upsilon_{a}$ being an arbitrary (not necessarily closed) one-form.

\begin{definition}
{\rm The Riemann curvature tensor, defined by
$$(\nabla_{a}\nabla_{b}-\nabla_{b}\nabla_{a})V^{c}=R_{ab}{}^{c}{}_{d}V^{d},$$
can be written as
$$R_{abcd}=C_{abcd}+2g_{c[a}P_{b]d}+2g_{d[b}P_{a]c},$$
where $C_{abcd}$ is the conformally invariant totally trace-free Weyl curvature and $P_{ab}$ is the {\bf Rho-tensor}.
Let us write $P=P^{a}{}_{a}$. The Rho-tensor has a conformal transformation law
$$\hat{P}_{ab}=P_{ab}-\nabla_{a}\Upsilon_{b}+\Upsilon_{a}\Upsilon_{b}-\frac{1}{2}\Upsilon_{c}\Upsilon^{c}g_{ab}.$$}
\end{definition}

\subsection{Tractor calculus}
\begin{definition}
{\rm 
Let $\mathbb{A}$ be the standard representation of $\g$ on $\C^{2n}$ (or $\C^{2n+1}$) and denote by $\mathcal{E}^{A}$ the sheaf of sections of the associated {\bf standard tractor bundle}.
This bundle has a composition series
$$\mathcal{E}^{A}=\mathcal{E}[1]+\mathcal{E}^{a}[-1]+\mathcal{E}[-1].$$
This composition series is a result of the composition series for $\mathbb{A}$ as a $\p$-module. For $D_{m}$ this composition series has the form
\begin{eqnarray*}
&&\quad\quad\quad\Ddo{1}{0}{0}{0}{0}{0}{0}\\
&&\\
&=&\quad\quad\quad\Ddd{1}{0}{0}{0}{0}{0}{0}+\quad\quad\quad\Ddd{-1}{1}{0}{0}{0}{0}{0}+\quad\quad\quad\Ddd{-1}{0}{0}{0}{0}{0}{0}\\
&&
\end{eqnarray*}
and for $B_{m}$ it takes the form
$$\ooo{1}{0}{0}\;...\;\btwo{0}{0}\;=\;\xoo{1}{0}{0}\;...\;\btwo{0}{0}\;+\;\xoo{-1}{1}{0}\;...\;\btwo{0}{0}\;+\;\xoo{-1}{0}{0}\;...\;\btwo{0}{0}\;.$$

For every choice of metric, elements in $\mathcal{E}^{A}$ can be identified with tuples
$$\left(\begin{array}{c}
\sigma\\
\mu^{a}\\
\rho
\end{array}\right),$$
where $\sigma\in\mathcal{E}[1]$, $\mu^{a}\in\mathcal{E}^{a}[-1]$ and $\rho\in\mathcal{E}[-1]$. Under change of metric these elements transform as
$$\widehat{\left(\begin{array}{c}
\sigma\\
\mu^{a}\\
\rho
\end{array}\right)}=\left(\begin{array}{c}
\sigma\\
\mu^{a}+\Upsilon^{a}\sigma\\
\rho-\Upsilon_{a}\mu^{a}-\frac{1}{2}\Upsilon_{a}\Upsilon^{a}\sigma
\end{array}\right).$$
There is a canonical connection, the {\bf tractor connection}, on the bundle $\mathcal{E}^{A}$. For each choice of metric, this is given by
$$\nabla_{b}\left(\begin{array}{c}
\sigma\\
\mu^{a}\\
\rho
\end{array}\right)=\left(\begin{array}{c}
\nabla_{b}\sigma-\mu_{b}\\
\nabla_{b}\mu^{a}+\delta_{b}{}^{a}\rho+P_{b}{}^{a}\sigma\\
\nabla_{b}\rho-P_{ba}\mu^{a}
\end{array}\right),$$
where the $\nabla$ within the bracket is the Levi-Civita connection.}
\end{definition}

\begin{corollary}
There is a canonical section $\mathbb{X}^{A}\in\mathcal{E}^{A}[1]$ given by
$$\mathbb{X}^{A}=\left(\begin{array}{c}
0\\
0\\
1
\end{array}\right),$$
for each choice of metric. 
\end{corollary}

\begin{definition}
{\rm 
The bundle $\mathcal{E}^{A}$ carries a non-degenerate symmetric form $h_{AB}$, the {\bf tractor metric}, which is, for each choice of metric, defined by
$$h_{AB}X^{A}Y^{B}=X^{a}Y_{a}+\sigma\beta+\rho\alpha,$$
where
$$X^{A}=\left(\begin{array}{c}
\sigma\\
X^{a}\\
\rho
\end{array}\right)\;\text{and}\; Y^{B}= \left(\begin{array}{c}
\alpha\\
Y^{a}\\
\beta
\end{array}\right).$$
It is easy to see that this definition is independent of the choice of metric within the conformal class.
This form can be used to identify $\mathcal{E}^{A}$ with its dual $\mathcal{E}_{A}$. In particular, we will use $h_{AB}$ and its inverse $h^{AB}$ to raise and lower tractor indices analogous 
to normal tensor indices. 
}\end{definition}

\subsubsection{Remark}
As in the projective case, these definitions can be used to determine the explicit descriptions of all tractor bundles and the corresponding tractor connections.
We will use capital letters $A,B,..$ to denote tractor indices.

\begin{definition}\label{tractorDconformal}
{\rm There exists an invariant differential operator
\begin{eqnarray*}
D_{A}:\mathcal{E}^{\Phi}[w]&\rightarrow&\mathcal{E}^{\Phi}_{A}[w-1]\\
f&\mapsto & \left(\begin{array}{c}
w(n+2w-2)f\\
(n+2w-2)\nabla_{a}f\\
-(\Delta+wP)f
\end{array}\right),
\end{eqnarray*}
where $\Phi$ denotes arbitrary tractor indices, $\nabla$ is the corresponding tractor connection and $\Delta=\nabla_{a}\nabla^{a}$ is the tractor Laplacian.
}\end{definition}

\subsection{Special splittings}

\begin{proposition}\label{propsixteen}
Let $V^{i_{1}...i_{k}}\in\mathcal{E}^{(i_{1}...i_{k})}_{0}[v]$ be a totally symmetric and totally trace-free tensor of conformal weight $v$. If 
$$v\not\in\{-(n+k+\alpha-2)\}_{\alpha=1,...,k},$$
then there exists a unique lift to an element 
$$V^{I_{1}...I_{k}}\in\mathcal{E}^{(I_{1}...I_{k})}_{0}[v+k]$$
such that 
$$\mathbb{X}_{A}V^{AI_{2}...I_{k}}=0\;\text{and}\;D_{A}V^{AI_{2}...I_{k}}=0.$$
The excluded weights $v$ correspond to conformally invariant operators 
$$V^{i_{1}...i_{k}}\mapsto\nabla_{i_{1}}...\nabla_{i_{\alpha}}V^{i_{1}...i_{k}}+\mathrm{C.C.T.},$$
for $\alpha=1,...,k$.
\end{proposition}
\begin{proof}
As in the projective case, we will view elements in $\mathcal{E}^{(I_{1}...I_{k})}_{0}[v+k]$ as elements in $\mathcal{E}^{I_{1}...I_{k}}[v+k]$ that satisfy certain symmetry and trace conditions. More specifically, each component $V_{\alpha}^{i_{1}...i_{\alpha}}$ of $V^{I_{1}...I_{k}}\in\mathcal{E}^{(I_{1}...I_{k})}_{0}[v+k]$ is totally symmetric, totally trace-free and equal to every other component with $\alpha$ indices and of the same conformal weight. The equation $\mathbb{X}_{A}V^{AI_{2}...I_{k}}=0$ ensures that $V^{I_{1}...I_{k}}\in\mathcal{E}^{(I_{1}...I_{k})}_{0}[v+k]$ has $k+1$ independent components, one for each number of indices.
\par
We can write every $V^{I_{1}I_{2}...I_{k}}\in\mathcal{E}^{(I_{1}I_{2}...I_{k})}_{0}[v+k]$ with $\mathbb{X}_{I_{1}}V^{I_{1}I_{2}...I_{k}}=0$ as
$$V^{I_{1}I_{2}...I_{k}}=\left(\begin{array}{c}
0\\
V_{k}^{i_{1}I_{2}...I_{k}}\\
V_{k-1}^{I_{2}...I_{k}}
\end{array}\right)$$
and compute
$$\nabla_{a}V^{I_{1}I_{2}...I_{k}}=\left(\begin{array}{c}
-(V_{k})_{a}{}^{I_{2}...I_{k}}\\
\nabla_{a}V_{k}^{i_{1}I_{2}...I_{k}}+\delta_{a}{}^{i_{1}}V_{k-1}^{I_{2}...I_{k}}\\
\nabla_{a}V_{k-1}^{I_{2}...I_{k}}-V_{k}^{bI_{2}...I_{k}}P_{ab}
\end{array}\right).$$
It follows that
$$\nabla_{b}\nabla_{a}V^{I_{1}I_{2}...I_{k}}=\left(\begin{array}{c}
-\nabla_{b}(V_{k})_{a}{}^{I_{2}...I_{k}}-\nabla_{a}(V_{k})_{b}{}^{I_{2}...I_{k}}-g_{ab}V_{k-1}^{I_{2}...I_{k}}\\
*\\
*
\end{array}\right)$$
and hence $$D_{A}V^{I_{1}...I_{k}}=$$
%$$\left(\begin{array}{ccc}
%0&(v+k)(n+2(v+k-1))V_{k}^{i_{1}I_{2}...I_{k}}&(v+k)(n+2(v+k-1))V_{k-1}^{I_{2}...I_{k}}\\
%-(n+2(v+k-1))(V_{k})_{a}{}^{I_{2}...I_{k}}&(n+2(v+k-1))\left(\nabla_{a}V_{k}^{i_{1}I_{2}...I_{k}}+\delta_{a}{}^{i_{1}}V_{k-1}^{I_{2}...I_{k}}\right)&(n+2(v+k-1))\nabla_{a}V_{k-1}^{I_%{2}...I_{k}}-V_{k}^{bI_{2}...I_{k}}P_{ab}\\
%2\nabla_{a}V^{aI_{2}...I_{k}}+nV_{k-1}^{I_{2}...I_{k}}&*&*
%\end{array}\right).$$
\scriptsize
$$\left(\begin{array}{ccc}
0&w(n+2(w-1))V_{k}^{i_{1}I_{2}...I_{k}}&w(n+2(w-1))V_{k-1}^{I_{2}...I_{k}}\\
*&(n+2(w-1))\left(\nabla_{a}V_{k}^{i_{1}I_{2}...I_{k}}+\delta_{a}{}^{i_{1}}V_{k-1}^{I_{2}...I_{k}}\right)&*\\
2\nabla_{a}V^{aI_{2}...I_{k}}+nV_{k-1}^{I_{2}...I_{k}}&*&*
\end{array}\right),$$
\normalsize
with $w=v+k$.
Taking the trace yields
$$D_{I_{1}}V^{I_{1}...I_{k}}=(n+2(v+k))\left(\nabla_{a}V_{k}^{aI_{2}...I_{k}}+(v+k+n-1)V_{k-1}^{I_{2}...I_{k}}\right).$$
The expression $\nabla_{a}V_{k}^{aI_{2}...I_{k}}+(v+k+n-1)V_{k-1}^{I_{2}...I_{k}}$ is in itself invariant and therefore can serve as the defining equation. That means that we do not have to exclude the case that $n+2(v+k)=0$.
\par
Moreover, for every $\alpha=1,...,k$ and $l=2,...,\alpha$, we have
$$\nabla_{a}V_{\alpha}^{i_{1}...i_{l}I_{l+1}I_{l+2}...I_{\alpha}}=\left(\begin{array}{c}
-V_{\alpha}^{i_{1}...i_{l}}{}_{a}{}^{I_{l+2}...I_{\alpha}}\\
\nabla_{a}V_{\alpha}^{i_{1}...i_{l}i_{l+1}I_{l+2}...I_{\alpha}}+\delta_{a}{}^{i_{l+1}}V_{\alpha-1}^{i_{1}...i_{l}I_{l+2}..I_{\alpha}}\\
\nabla_{a}V_{\alpha-1}^{i_{1}...i_{l}I_{l+2}...I_{\alpha}}-V_{\alpha}^{i_{1}...i_{l}bI_{l+2}...I_{\alpha}}P_{ab}
\end{array}\right)$$
and hence
$$\nabla_{a}V_{\alpha}^{ai_{2}...i_{l}I_{l+1}I_{l+2}...I_{\alpha}}=\left(\begin{array}{c}
0\\
\nabla_{a}V_{\alpha}^{ai_{2}...i_{l}i_{l+1}I_{l+2}...I_{\alpha}}+V_{\alpha-1}^{i_{2}...i_{l}i_{l+1}I_{l+2}..I_{\alpha}}\\
\nabla_{a}V_{\alpha-1}^{ai_{2}...i_{l}I_{l+2}...I_{\alpha}}-V_{\alpha}^{ai_{2}...i_{l}bI_{l+2}...I_{\alpha}}P_{ab}
\end{array}\right),$$
since $V^{I_{1}...I_{k}}$ is totally trace-free.
Iterating this rule shows that $D_{A}V^{AI_{2}...I_{k}}=0$ is equivalent to the following  $k$ equations
\begin{equation}\label{conformalsplitting}
\nabla_{a}V_{\alpha}^{ai_{1}...i_{\alpha-1}}+(v+n+k+\alpha-2)V_{\alpha-1}^{i_{1}...i_{\alpha-1}}-(k-\alpha)V_{\alpha+1}^{abi_{1}...i_{\alpha-1}}P_{ab}=0,
\end{equation}
for $\alpha=1,...,k$. In order to see this, note that for $\alpha=k,...,1$ the equation
$$\nabla_{a}V_{\alpha}^{aI_{1}...I_{\alpha-1}}+(v+n+k-1)V_{\alpha-1}^{I_{1}...I_{\alpha-1}}-(k-\alpha)V_{\alpha+1}^{abI_{1}...I_{\alpha-1}}P_{ab}=0$$
is equivalent to the two equations
\begin{eqnarray*}
\nabla_{a}V_{\alpha}^{ai_{1}I_{2}...I_{\alpha-1}}+(v+n+k)V_{\alpha-1}^{i_{1}I_{2}...I_{\alpha-1}}-(k-\alpha)V_{\alpha+1}^{abi_{1}I_{2}...I_{\alpha-1}}P_{ab}&=&0,\\
\nabla_{a}V_{\alpha-1}^{aI_{2}...I_{\alpha-1}}+(v+n+k-1)V_{\alpha-2}^{I_{2}...I_{\alpha-1}}-(k-(\alpha-1))V_{\alpha}^{abI_{2}...I_{\alpha-1}}P_{ab}&=&0
\end{eqnarray*}
and those equations are equivalent to the three equations
\begin{eqnarray*}
\nabla_{a}V_{\alpha}^{ai_{1}i_{2}I_{3}...I_{\alpha-1}}+(v+n+k+1)V_{\alpha-1}^{i_{1}i_{2}I_{3}...I_{\alpha-1}}-(k-\alpha)V_{\alpha+1}^{abi_{1}i_{2}I_{3}...I_{\alpha-1}}P_{ab}&=&0,\\
\nabla_{a}V_{\alpha-1}^{ai_{1}I_{3}...I_{\alpha-1}}+(v+n+k)V_{\alpha-2}^{i_{1}I_{3}...I_{\alpha-1}}-(k-(\alpha-1))V_{\alpha}^{abi_{1}I_{3}...I_{\alpha-1}}P_{ab}&=&0,\\
\nabla_{a}V_{\alpha-2}^{aI_{3}...I_{\alpha-1}}+(v+n+k-1)V_{\alpha-3}^{I_{3}...I_{\alpha-1}}-(k-(\alpha-2))V_{\alpha-1}^{abI_{3}...I_{\alpha-1}}P_{ab}&=&0.
\end{eqnarray*}
These three equations are equivalent to four further equations and so forth until we obtain $\alpha$ equations. Carrying out this procedure starting with $\alpha=k$ yields the $k$ independent equations as given
above. The reason for obtaining $k$ equations is as follows: $D_{A}V^{AI_{2}...I_{k}}\in\mathcal{E}^{(I_{2}...I_{k})}_{0}[v+k-1]$ also has the property that
$$\mathbb{X}_{A}D_{B}V^{BAI_{3}...I_{k}}=0,$$
so $D_{A}V^{AI_{2}...I_{k}}$ has $k$ independent components.
\par
If $v\not\in\{v+n+k+\alpha-2\}_{\alpha=1,...,k}$, then those equations can be uniquely solved starting with $V_{k}^{i_{1}...i_{k}}=V^{i_{1}...i_{k}}$. This shows the uniqueness.
The existence can either be deduced by considering the explicit description (behavior under change of metric) of $\mathcal{E}^{(I_{1}...I_{k})}_{0}[v+k]$ or by the general theory of the last chapter.
\par
If $v=-(n+k+\alpha-2)$, then one can use the theory developed in~\cite{css1} to deduce that the (standard) invariant differential operator 
$$V^{i_{1}...i_{k}}\mapsto\nabla_{i_{1}}...\nabla_{i_{k-\alpha+1}}V^{i_{1}...i_{k}}$$
has a curved analogue. 
\end{proof}

\subsubsection{Remark}
This splitting was first written down in~\cite{e2}, p.~1658, for the case $v=0$. The components $\sigma_{\alpha}^{i_{1}...i_{\alpha}}$ in this article differ, however, from our components by a constant. To be more precise, the components $\sigma_{k-\alpha}^{i_{1}...i_{k-\alpha}}$ in~\cite{e2} are given by
$$V^{i_{1}...i_{k-\alpha}}_{k-\alpha}=C(\alpha)\sigma^{i_{1}...i_{k-\alpha}}_{k-\alpha}$$
with
\begin{eqnarray*}
C(0)&=&1\;\text{and for}\; 1\leq \alpha\leq k\\
C(\alpha)&=&\alpha C(\alpha-1).
\end{eqnarray*}

\begin{corollary}
For each excluded weight $v$ in Proposition~\ref{propsixteen}, the explicit form of the corresponding invariant differential operator can be explicitly determined using~(\ref{conformalsplitting}). 
\end{corollary}
\begin{proof}
Let $v=-(n+k+\alpha_{0}-2)$, then the equations
$$\nabla_{a}V_{\alpha}^{ai_{1}...i_{\alpha-1}}+(\alpha-\alpha_{0})V_{\alpha-1}^{i_{1}...i_{\alpha-1}}-(k-\alpha)V_{\alpha+1}^{abi_{1}...i_{\alpha-1}}P_{ab}=0$$
can be uniquely solved for $\alpha=k,...,\alpha_{0}+1$. The next line
$$\nabla_{a}V_{\alpha_{0}}^{ai_{1}...i_{\alpha_{0}-1}}-(k-\alpha_{0})V_{\alpha_{0}+1}^{abi_{1}...i_{\alpha-1}}P_{ab}$$
can be written as
$$\nabla_{i_{1}}...\nabla_{i_{k-\alpha_{0}+1}}V^{i_{1}...i_{k}}+\mathrm{L.O.T},$$
where $\mathrm{L.O.T.}$ stands for lower order terms combined with curvature terms $P_{ab}$.
This is the first non-zero part of $D_{A}V^{AI_{2}...I_{k}}$ and hence an invariant expression.
\end{proof}

\subsubsection{Example}
The linear differential operator
\begin{eqnarray*}
\nabla_{a}\nabla_{b}\nabla_{c}\nabla_{d}V^{abcdi_{1}...i_{k-4}}&+&3\nabla_{a}\nabla_{b}\left(V^{abcdi_{1}...i_{k-4}}P_{cd}\right)\\
&+&4\nabla_{a}\left(P_{bc}\nabla_{d}V^{abcdi_{1}...i_{k-4}}\right)\\
&+&3P_{ab}\left(\nabla_{c}\nabla_{d}V^{abcdi_{1}...i_{k-4}}\right)\\
&+&9P_{ab}P_{cd}V^{abcdi_{1}...i_{k-4}}
\end{eqnarray*}
is invariant in case $v+n+2k-5=0$ in accordance with~(\ref{cct}).

\begin{theorem}\label{theoremseventeen}
If we exclude certain weights as follows:
\begin{enumerate}
\item
$v\not\in\{-(n+k+\alpha-2)\}_{\alpha=1,...,k}$,
\item
$v\not\in\{\alpha-k-1\}_{\alpha=0,...,k-1}$,
\item
$v\not\in\{\alpha\}_{\alpha=0,...,M-1}$ and
\item
$v\not\in\{1-\frac{n}{2}-k+\alpha\}_{\alpha=0,...,k+M-1}$,
\end{enumerate}
then there exists a splitting
$$\mathcal{E}^{(i_{1}...i_{k})}_{0}[v]\rightarrow \quad\quad\quad\Ddo{M}{k}{0}{0}{0}{0}{0}[v-M]$$
for $n$ even or
$$\mathcal{E}^{(i_{1}...i_{k})}_{0}[v]\rightarrow\;\ooo{M}{k}{0}\;...\;\btwo{0}{0}[v-M]\;$$
for $n$ odd, given by an explicit tractor formula. Each excluded weight $v$ corresponds to the existence of an invariant linear differential operator on the flat model space $\mathbb{S}^{n}$.
\end{theorem}
\begin{proof}
Let $V^{i_{1}...I_{k}}\in\mathcal{E}^{(i_{1}...i_{k})}_{0}[v]$. Then we can define the splitting in three steps, each time excluding appropriate weights:
\begin{enumerate}
\item
Define 
$$V^{I_{1}...I_{k}}\in\mathcal{E}^{(I_{1}...I_{k})}_{0}[v+k]=\quad\quad\quad\Ddo{k}{0}{0}{0}{0}{0}{0}[v+k]$$
or 
$$V^{I_{1}...I_{k}}\in\mathcal{E}^{(I_{1}...I_{k})}_{0}[v+k]=\;\ooo{k}{0}{0}\;...\;\btwo{0}{0}[v+k]$$
as in Proposition~\ref{propsixteen}. The weights to exclude are 
$$v\not\in\{-(n+k+\alpha-2)\}_{\alpha=1,...,k}$$
and correspond to the existence of invariant linear differential operators
\begin{eqnarray*}
\Dd{-(n+k+\alpha-2)}{k}{0}{0}{0}{0}{0}&\rightarrow&\hspace{3cm}\Dd{-(n+k+\alpha-2)}{\alpha-1}{0}{0}{0}{0}{0}\\
&&\\
\text{or}&&\\
\;\xoo{-(n+k+\alpha-2)}{k}{0}\;...\;\btwo{0}{0}&\rightarrow&\;\xoo{-(n+k+\alpha-2)}{\alpha-1}{0}\;...\;\btwo{0}{0}\\
V^{i_{1}...i_{k}}&\mapsto& \nabla_{i_{1}}...\nabla_{i_{k-\alpha+1}}V^{i_{1}...i_{k}}.
\end{eqnarray*}
\item
Set $T^{BQCR...DS}=\;\text{pair skew}(D^{B}D^{C}...D^{D}V^{QR...S})$, where, following~\cite{g}, p.~223, \lq pair skew\rq~means to simultaneously take the skew part over each of the index pairs $BQ,CR,...,DS$. Then
$$T^{BQCR...DS}\in\quad\quad\quad\Ddo{0}{k}{0}{0}{0}{0}{0}[v]$$
or
$$T^{BQCR...DS}\in\;\ooo{0}{k}{0}\;...\;\btwo{0}{0}[v].$$
From the proof of Proposition~\ref{propsixteen} one can see that the weights to exclude are 
$$v\not\in\{1-\frac{n}{2}+\alpha-k\}_{\alpha=0,...,k-1}\cup\{\alpha-k-1\}_{\alpha=0,...,k-1},$$
because the projection onto the first factor in $\text{pair skew}(D^{B}D^{C}...D^{D}V^{QR...S})$ will be (modulo a non-zero scalar):
$$\prod_{\alpha=0}^{k-1}(n+2(v+k-\alpha)-2)(v+k-\alpha+1)V^{i_{1}...i_{k}}.$$
These excluded weights correspond to the existence of invariant linear differential operators
\begin{eqnarray*}
\Dd{\alpha-k-1}{k}{0}{0}{0}{0}{0}&\rightarrow&\hspace{3cm}\Dd{-k-2}{k-\alpha-1}{\alpha+1}{0}{0}{0}{0}\\
&&\\
\text{or}\\
\;\xoo{\alpha-k-1}{k}{0}\;...\;\btwo{0}{0}&\rightarrow&\;\xoo{-k-2}{k-\alpha-1}{\alpha+1}\;...\;\btwo{0}{0}\\
V^{i_{1}...i_{k}}&\mapsto&\;\text{pair skew}\nabla_{j_{1}}...\nabla_{j_{\alpha+1}}V_{i_{1}...i_{k}},
\end{eqnarray*}
which are standard operators, and
\begin{eqnarray*}
\Dd{1-\frac{n}{2}+\alpha-k}{k}{0}{0}{0}{0}{0}&\rightarrow&\hspace{3cm}\Dd{-1-\frac{n}{2}-\alpha-k}{k}{0}{0}{0}{0}{0}\\
&&\\
\text{or}&&\\
\;\xoo{1-\frac{n}{2}+\alpha-k}{k}{0}\;...\;\btwo{0}{0}&\rightarrow&\;\xoo{-1-\frac{n}{2}-\alpha-k}{k}{0}\;...\;\btwo{0}{0}\\
V^{i_{1}...i_{k}}&\mapsto&\Delta^{\alpha+1}V^{i_{1}...i_{k}}+...\;,
\end{eqnarray*}
which are non-standard operators.
%where there are lower order terms, curvature correction terms and even terms of the order $2(\alpha+1)$ that we leave out.
\item
Define 
$$D_{(J_{1}}...D_{J_{M})}T^{BQCR...DS}-\mathrm{trace}\in \quad\quad\quad\Ddo{M}{k}{0}{0}{0}{0}{0}[v-M]$$
or
$$D_{(J_{1}}...D_{J_{M})}T^{BQCR...DS}-\mathrm{trace}\in\;\ooo{M}{k}{0}\;...\;\btwo{0}{0}[v-M].$$
The weights to exclude here are 
$$v\not\in\{\alpha,1-\frac{n}{2}+\alpha\}_{\alpha=0,...,M-1}$$
and correspond to the existence of invariant linear differential operators
\begin{eqnarray*}
\Dd{\alpha}{k}{0}{0}{0}{0}{0}&\rightarrow&\hspace{3cm}\Dd{-\alpha-2}{k+\alpha+1}{0}{0}{0}{0}{0}\\
&&\\
\text{or}&&\\
\;\xoo{\alpha}{k}{0}\;...\;\btwo{0}{0}&\rightarrow&\;\xoo{-\alpha-2}{k+\alpha+1}{0}\;...\;\btwo{0}{0}\\
V^{i_{1}...i_{k}}&\mapsto&\nabla_{(i_{1}}...\nabla_{i_{\alpha+1}}V_{i_{\alpha+2}...i_{k+\alpha+1})}-\mathrm{trace},
\end{eqnarray*}
which are standard operators, and
\begin{eqnarray*}
\Dd{1-\frac{n}{2}+\alpha}{k}{0}{0}{0}{0}{0}&\rightarrow&\hspace{3cm}\Dd{-1-\frac{n}{2}-\alpha-2k}{k}{0}{0}{0}{0}{0}\\
&&\\
\text{or}&&\\
\;\xoo{1-\frac{n}{2}+\alpha}{k}{0}\;...\;\btwo{0}{0}&\rightarrow&\;\xoo{-1-\frac{n}{2}-\alpha-2k}{k}{0}\;...\;\btwo{0}{0}\\
V^{i_{1}...i_{k}}&\mapsto&\Delta^{\alpha+k+1}V^{i_{1}...i_{k}}+...\;,
\end{eqnarray*}
which are non-standard operators.
\end{enumerate}
\end{proof}

\subsubsection{Remark}
The existence of the invariant linear differential operators for the various excluded weights in Theorem~\ref{theoremseventeen} follows from the classification of invariant linear differential operators on the homogeneous model space 
$\mathbb{S}^{n}$ as in~\cite{bc1} and~\cite{bc2}. All the so-called standard operators (that correspond to individual arrows in the BGG resolution) have curved analogues (see~\cite{css2}), but not all higher powers of the Laplacian in even dimensions allow curved analogues, see~\cite{gr} and~\cite{gh}. 

\subsection{Examples}

\subsubsection{Example 1}
Let us look at weighted vector fields $X^{a}\in\mathcal{E}^{a}[v]$:
$$X^{A}=\left(\begin{array}{c}
0\\
X^{a}\\
-\frac{1}{n+v}\nabla_{c}X^{c}
\end{array}\right)$$
and hence
$$D_{[A}X_{B]}=\left(\begin{array}{c}
(v+2)(n+2v)X_{b}\\
(n+2v)\nabla_{[a}X_{b]}\quad-\frac{(v+2)(n+2v)}{n+v}\nabla_{c}X^{c}\\
\Delta X_{b}-\frac{n+2v-2}{n+v}\nabla_{b}\nabla_{c}X^{c}+(v+1)PX_{b}-(n+2v+2)P_{ba}X^{a}
\end{array}\right).$$
Now we can easily use the formula in Definition~\ref{tractorDconformal} for $D_{A}$ to obtain the explicit description of $D_{(J_{1}}...D_{J_{M})}D^{[A}X^{B]}-\mathrm{trace}$.
The excluded weights (up to $M=1$) have the following meaning (on $\mathbb{S}^{n}$):
\begin{enumerate}
\item
If $v+n=0$, then the semi-invariant operator $X^{a}\mapsto\nabla_{a}X^{a}$ is invariant.
\item
If $v+2=0$, then $X^{a}\mapsto\nabla_{[a}X_{b]}$ is invariant.
\item
If $n+2v=0$, then $X^{a}\mapsto\Delta X_{a}-\frac{4}{n}\nabla_{a}\nabla_{b}X^{b}+\frac{2-n}{2}PX_{a}-2P_{ab}X^{b}$ is invariant.
\item
If $v=0$, then $X^{a}\mapsto \nabla_{(a}X_{b)}-\frac{1}{n}(\nabla_{c}X^{c})g_{ab}$ is invariant.
\item
If $n+2v-2=0$, then $X^{a}\mapsto \Delta^{2}X^{a}+...$ is invariant.
%$D^{A}\Box D_{A}X^{C}=-(n-4)\Delta^{2}X^{C}$ is invariant. Here $\Delta$ is the tractor connection and we have to assume that $n\not 4$ if we want that operator to have a curved analogue.
\end{enumerate}

These are exactly the same weights that have to be excluded when we carry out the construction in the last chapter and compare the central characters:
\begin{eqnarray*}
\hspace{1.4cm}\Ddo{1}{1}{0}{0}{0}{0}{0}[v-1]&=&\hspace{1.2cm}\Dd{v}{1}{0}{0}{0}{0}{0}+\hspace{.8cm}\begin{array}{c}
\hspace{1cm}\Dd{v}{0}{0}{0}{0}{0}{0}\\
\oplus\\
\hspace{1cm}\Dd{v-1}{0}{1}{0}{0}{0}{0}\\
\oplus\\
\hspace{1cm}\Dd{v-2}{2}{0}{0}{0}{0}{0}
\end{array}\\
&+&\hspace{.3cm} \begin{array}{c}
\hspace{1cm}\Dd{v-3}{1}{1}{0}{0}{0}{0}\\
\oplus\\
\hspace{1cm}\Dd{v-2}{1}{0}{0}{0}{0}{0}\\
\oplus\\
\hspace{1cm}\Dd{v-2}{1}{0}{0}{0}{0}{0}
\end{array}
+\hspace{.3cm}\begin{array}{c}
\hspace{1cm}\Dd{v-3}{0}{1}{0}{0}{0}{0}\\
\oplus\\
\hspace{1cm}\Dd{v-4}{2}{0}{0}{0}{0}{0}\\
\oplus\\
\hspace{1cm}\Dd{v-2}{0}{0}{0}{0}{0}{0}
\end{array}\\
&+&\hspace{1.4cm}\Dd{v-4}{1}{0}{0}{0}{0}{0}.
\end{eqnarray*}
A similar composition series holds for $B_{m}$.
\par
This allows us, in theory, to write down the three invariant bilinear differential pairings
$$\mathcal{E}^{a}[v]\times\mathcal{E}^{b}[w]\rightarrow\mathcal{E}^{a}[v+w]$$
according to
$$3\times\hspace{1.2cm}\Ddo{2}{1}{0}{0}{0}{0}{0}[v+w-2]\subset\hspace{1cm}\Ddo{1}{1}{0}{0}{0}{0}{0}\otimes\hspace{1.2cm}\Ddo{1}{1}{0}{0}{0}{0}{0}[v+w-2]$$
or
$$3\times\ooo{2}{1}{0}\;...\;\btwo{0}{0}[v+w-2]\subset\;\ooo{1}{1}{0}\;...\;\btwo{0}{0}\;\otimes\;\ooo{1}{1}{0}\;...\;\btwo{0}{0}\;[v+w-2]$$
and the projections
$$\hspace{1.5cm}\Ddo{2}{1}{0}{0}{0}{0}{0}[v+w-2]\rightarrow\hspace{1.5cm}\Ddd{v+w}{1}{0}{0}{0}{0}{0}$$
or
$$\;\ooo{2}{1}{0}\;...\;\btwo{0}{0}[v+w-2]\;\rightarrow\;\xoo{v+w}{1}{0}\;...\;\btwo{0}{0}.$$
These pairings can be computed as
\begin{description}
\item{(a)}
$$\begin{array}{l}
v(v+2)(v+n)X^{a}\nabla_{b}Y^{b}-\frac{(w+n)v(v+2)}{n}Y^{a}\nabla_{b}X^{b}+(w+n)v(v+n)Y_{b}\nabla^{[a}X^{b]}\\
-(w+n)(v+n)(v+2)Y_{b}(\nabla^{(a}X^{b)}-\frac{1}{n}g^{ab}\nabla_{c}X^{c})
\end{array},$$
\item{(b)}
$$\begin{array}{l}
v(v+2)(v+n)X_{b}\nabla^{[a}Y^{b]}+\frac{(w+2)v(v+2)(n-1)}{2n}Y^{a}\nabla_{b}X^{b}-\frac{1}{2}(w+2)v(v+n)Y_{b}\nabla^{[a}X^{b]}\\
-\frac{1}{2}(w+2)(v+2)(v+n)Y_{b}(\nabla^{(a}X^{b)}-\frac{1}{n}g^{ab}\nabla_{c}X^{c})
\end{array}
$$
and
\item{(c)}
$$\begin{array}{l}
v(v+2)(v+n)X_{b}(\nabla^{(a}Y^{b)}-\frac{1}{n}g^{ab}\nabla_{c}Y^{c})-\frac{wv(v+2)(n+2)(n-1)}{2n^{2}}Y^{a}\nabla_{b}X^{b}\\
-\frac{wv(v+n)(n+2)}{2n}Y_{b}\nabla^{[a}X^{b]}+\frac{w(v+2)(v+n)(2-n)}{2n}Y_{b}(\nabla^{(a}X^{b)}-\frac{1}{n}g^{ab}\nabla_{c}X^{c})
\end{array}.$$
\end{description}

\subsubsection{Example 2}
Let us look at the second order pairing
$$\mathcal{E}^{(ab)}_{0}[v]\times\mathcal{E}[w]\rightarrow\mathcal{E}[v+w].$$
This pairing is given by 
\begin{eqnarray*}
&&V^{ab}\nabla_{a}\nabla_{b}f-\frac{2(w-1)}{n+v+2}(\nabla_{a}V^{ab})(\nabla_{b}f)+\frac{w(w-1)}{(v+n+2)(v+n+1)}f\nabla_{a}\nabla_{b}V^{ab}\\
&+&w\frac{v+w+n}{v+n+1}P_{ab}V^{ab}f
\end{eqnarray*}
and can be written as 
$$V^{AB}\tilde{D}_{A}\tilde{D}_{B}f,$$
where 
$$\tilde{D}_{A}f=\left(\begin{array}{c}
wf\\
\nabla_{a}f\\
0
\end{array}\right).$$
This operator is not generally invariant, however $\hat{\tilde{D}}_{A}f=\tilde{D}_{A}f+\mathbb{X}_{A}\Phi(\Upsilon)$, where $\Phi$ is just some function of $\Upsilon$, see~\cite{g}.
So we can use the fact that $V^{AB}\mathbb{X}_{A}=0$ for the unique lift $V^{AB}$ of $V^{ab}$, to deduce that $V^{AB}\tilde{D}_{A}\tilde{D}_{B}f$ is invariant. This formula is in accordance with~\ref{examplepairings} (d).
% In fact, $\tilde{D}_{A}$ is intimately related to the Levi-Civita connection of the ambient metric defined by F

\section{CR-geometry}
In order to be consistent with~\cite{gg}, the complexification of the underlying Lie algebra of this parabolic geometry is $A_{n+1}$. The description of the geometry and the tractor calculus of CR-manifolds in this section follows closely~\cite{gg} and the reader is advised to consult this article for more details.

\subsection{CR-manifolds}
Throughout this section $\mathcal{M}$  will be a real $2n+1$ dimensional  manifold with a subbundle $T^{1,0}\subset \C\otimes T\mathcal{M}$ of the complexified tangent bundle such that
$$
\mathrm{dim}_{\C}T^{1,0}=n,\;\; T^{1,0}\cap T^{0,1}=\{0\},\;\text{where}\; T^{0,1}=\overline{T^{1,0}},$$
and
$$[\Gamma(T^{1,0}),\Gamma(T^{1,0})]\subset \Gamma(T^{1,0}).$$
Define a complex line bundle (the canonical bundle) by $\mathcal{K}=\Lambda^{n+1}((T^{0,1})^{\perp})$. Following~\cite{gg}, we will assume that $\mathcal{K}$ allows an $(n+2)^{\mathrm{nd}}$ root and set $\mathcal{E}(1,0)=\mathcal{K}^{-\frac{1}{n+2}}$. The bundle $\mathcal{E}(w,w')$ is then defined to be the associated bundle of the $\C^{*}$ principal bundle $\mathcal{E}(1,0)\backslash\{0\}$ corresponding to the representation 
$$\C^{*}\ni \lambda\mapsto \lambda^{w}\bar{\lambda}^{w'}=\vert\lambda\vert^{2w'}\exp((w-w')\log(\lambda))\in\mathfrak{gl}(\C).$$
The representation is defined for $w,w'\in\C$ with $w-w'\in\Z$ and it will always be implicitly assumed that this is the case.
We will use the notation $\mathcal{E}^{\alpha}$ (resp.~$\mathcal{E}^{\bar{\alpha}}$) for the sheaf of sections of $T^{1,0}$ (resp.~$T^{0,1}$) and $\mathcal{E}_{\alpha}$ (resp.~$\mathcal{E}_{\bar{\alpha}}$) for its dual. 

%\subsubsection{Remark}
%The notations and definitions for CR-manifolds are taken from the case where $\mathcal{M}$ is a codimension 1 real submanifold of a complex manifold. $T^{1,0}$ then  corresponds to tangent vectors of type $(1,0)$ in $\C\otimes T\mathcal{M}$, see~\cite{t1} and the canonical bundle is isomorphic to the restriction of the canonical bundle of $(n+1,0)$ forms on the complex manifold.

\begin{definition}
{\rm
A {\bf Levi form} is a Hermitian form $h=h_{\alpha\bar{\beta}}=h_{\bar{\beta}\alpha}$ on $T^{1,0}$ defined by $h(Z,\bar{W})=-2id\theta(Z,\bar{W})$ for a non-vanishing global section $\theta$ of 
$H^{\perp}\subset T^{*}\mathcal{M}$ (assuming as usual that $\mathcal{M}$ is orientable), where $H=(T^{1,0}\oplus T^{0,1})\cap T\mathcal{M}\subset T\mathcal{M}$ has the property that its complexification
$\C\otimes H$ is $T^{1,0}\oplus T^{0,1}$.
We will only deal with non-degenerate Levi forms of arbitrary signature $(p,q)$, $p+q=n$. The choice of section $\theta$ is called {\bf pseudohermitian structure}.
It can be shown that a pseudohermitian structure determines a connection, the {\bf Tanaka-Webster} connection, see~\cite{w} and~\cite{t1}. This connection satisfies $\nabla h=0$,
so raising an lowering indices commutes with differentiation.
}\end{definition}

\subsubsection{Remark}
The notations and definitions for CR-manifolds are taken from the case where $\mathcal{M}$ is a codimension 1 real submanifold of a complex manifold. Then we can define
for every $x\in\mathcal{M}$:
$$H_{x}=\{v\in T_{x}\mathcal{M}\;:\;\sqrt{-1}v\in T_{x}\mathcal{M}\}$$
and obtain a complex structure $J_{x}:H_{x}\rightarrow H_{x}$ given by $J_{x}(v)=\sqrt{-1}v.$ This allows one to define $T_{x}^{1,0}$ and $T_{x}^{0,1}$ as the $+\sqrt{-1}$ and
$-\sqrt{-1}$ eigenspaces of $J_{x}$ when acting on $\C\otimes H_{x}$.
\par
The canonical bundle is isomorphic to the restriction of the canonical bundle of $(n+1,0)$ forms on the complex manifold, see~\cite{lee}.

\subsubsection{Remark}
The two filtrations
$$\C\otimes T\mathcal{M}\supset \begin{array}{c}
T^{1,0}\\
\oplus\\
T^{0,1}
\end{array}\supset 0\;\text{and}\;T\mathcal{M}\supset H\supset 0$$
of the complexified tangent bundle and the tangent bundle are reflected in the fact that the complexified Lie algebra $\g=A_{n+1}$ allows a $|2|$-grading, where $\g_{-1}$ is the direct sum of two 
irreducible $\g_{0}$-modules, whereas the original Lie algebra $\g_{\R}=\mathfrak{su}(p+1,q+1)$ has a $|2|$-grading as in~\cite{cs}, 4.14, where it is clear that $(\g_{\R})_{-1}$ is an irreducible
$\g_{0}$-module.

\begin{proposition}
The Tanaka -Webster connection preserves the splitting
$$\C\otimes T\mathcal{M}=T^{1,0}\oplus T^{0,1}\oplus \mathrm{span}\;T,$$
where $T$ is the unique vector field on $\mathcal{M}$ satisfying $\theta(T)=1$ and $\iota_{T}d\theta=0$.
\end{proposition}
\begin{proof}
Details and a proof may be found in~\cite{t1}. 
\end{proof}
This implies that we can compute the covariant derivative componentwise. To denote these components we will use the notation $\nabla_{\alpha},\nabla_{\bar{\alpha}}$ and $\nabla_{0}$.

\begin{proposition}
Under change of pseudohermitian structure $\hat{\theta}=\exp(\Upsilon)\theta$, the connection transforms as
\begin{eqnarray*}
\hat{\nabla}_{\alpha}f&=&\nabla_{\alpha}f+w\Upsilon_{\alpha}f\\
\hat{\nabla}_{\bar{\alpha}}f&=&\nabla_{\bar{\alpha}}f+w'\Upsilon_{\bar{\alpha}}f\\
\hat{\nabla}_{0}f&=&\nabla_{0}f+i\Upsilon^{\bar{\gamma}}\nabla_{\bar{\gamma}}f-i\Upsilon^{\gamma}\nabla_{\gamma}f\\
&&+\frac{1}{n+2}\left((w+w')\Upsilon_{0}+iw\Upsilon^{\gamma}{}_{\gamma}-iw\Upsilon^{\bar{\gamma}}{}_{\bar{\gamma}}+i(w'-w)\Upsilon^{\gamma}\Upsilon_{\gamma}\right)f
\end{eqnarray*}
for $f\in\mathcal{E}(w,w')$ and
\begin{eqnarray*}
\hat{\nabla}_{\alpha}\tau_{\beta}&=&\nabla_{\alpha}\tau_{\beta}+(w-1)\Upsilon_{\alpha}\tau_{\beta}-\Upsilon_{\beta}\tau_{\alpha}\\
\hat{\nabla}_{\bar{\alpha}}\tau_{\beta}&=&\nabla_{\bar{\alpha}}\tau_{\beta}+w'\Upsilon_{\bar{\alpha}}\tau_{\beta}+h_{\beta\bar{\alpha}}\Upsilon^{\gamma}\tau_{\gamma}\\
\hat{\nabla}_{0}\tau_{\beta}&=&\nabla_{0}\tau_{\beta}+i\Upsilon^{\bar{\gamma}}\nabla_{\bar{\gamma}}\tau_{\beta}-i\Upsilon^{\gamma}\nabla_{\gamma}\tau_{\beta}-i(\Upsilon^{\gamma}{}_{\beta}-\Upsilon^{\gamma}\Upsilon_{\beta})\tau_{\gamma}\\
&&+\frac{1}{n+2}\left((w+w')\Upsilon_{0}+iw\Upsilon^{\gamma}{}_{\gamma}-iw\Upsilon^{\bar{\gamma}}{}_{\bar{\gamma}}+i(w'-w)\Upsilon^{\gamma}\Upsilon_{\gamma}\right)\tau_{\beta}
\end{eqnarray*}
for $\tau_{\alpha}\in\mathcal{E}_{\alpha}(w,w')$. We have used the abbreviation $\Upsilon_{\alpha}=\nabla_{\alpha}\Upsilon$, $\Upsilon_{\bar{\alpha}}=\nabla_{\bar{\alpha}}\Upsilon$ and so forth to denote the components of the derivative of $\Upsilon$. Using the Leibniz rule (and complex conjugation) one can easily determine the transformation law on an arbitrary tensor bundle. 
\end{proposition}
\begin{proof}
These explicit formulae can be found in~\cite{gg}, Proposition 2.3. Note that, as far as invariance and obstruction terms are concerned, $\nabla_{0}$ should be treated as a weighted second order operator. 
\end{proof}

\begin{definition}
{\rm There are several quantities related to the torsion and curvature of $\nabla$:
\begin{enumerate}
\item
The {\bf pseudohermitian curvature tensor} $R_{\alpha\bar{\beta}\gamma\bar{\delta}}\in\mathcal{E}_{\alpha\bar{\beta}\gamma\bar{\delta}}(1,1)$,
\item
the {\bf pseudohermitian torsion tensor} $A_{\alpha\beta}\in\mathcal{E}_{(\alpha\beta)}$,
\item
the Webster-Ricci tensor $R_{\alpha\bar{\beta}}=R^{\gamma}{}_{\gamma\alpha\bar{\beta}}\in\mathcal{E}_{\alpha\bar{\beta}}$,
\item
the Webster scalar curvature $R=R_{\alpha}{}^{\alpha}\in\mathcal{E}(-1,-1)$,
\item
$P_{\alpha\bar{\beta}}=\frac{1}{n+2}\left(R_{\alpha\bar{\beta}}-\frac{1}{2(n+1)}Rh_{\alpha\bar{\beta}}\right)\in\mathcal{E}_{\alpha\bar{\beta}}$,
\item
$P=P_{\alpha}{}^{\alpha}\in\mathcal{E}(-1,-1)$,
\item
$T_{\alpha}=\frac{1}{n+2}(\nabla_{\alpha}P-i\nabla^{\beta}A_{\alpha\beta})\in\mathcal{E}_{\alpha}(-1,-1)$ and
\item
$S=-\frac{1}{n}\left(\nabla^{\alpha}T_{\alpha}+\nabla^{\bar{\alpha}}T_{\bar{\alpha}}+P_{\alpha\bar{\beta}}P^{\alpha\bar{\beta}}-A_{\alpha\beta}A^{\alpha\beta}\right)\in\mathcal{E}(-2,-2)$.
\end{enumerate}
The transformation laws for these objects can be found in~\cite{gg}.
}\end{definition}

\subsection{Tractor calculus}

\begin{definition}
{\rm
Let $\mathcal{E}_{A}$ be the standard representation of $\mathfrak{su}(p+1,q+1)$ on $\R^{n+2}$, then the (complexified) representation of $A_{n+1}$ on $\mathbb{C}^{n+2}$ has, as a $\p$-module,
a composition series.
$$\mathcal{E}_{A}=\ooo{1}{0}{0}\;...\;\oo{0}{0}=\xoo{1}{0}{0}\;...\;\oo{0}{0}\;+\;\xoo{-1}{1}{0}\;...\;\ox{0}{0}\;+\;\xoo{0}{0}{0}\;...\;\ox{0}{-1}.$$
In terms of sheaves,
$$\mathcal{E}_{A}=\mathcal{E}(1,0)+\mathcal{E}_{\alpha}(1,0)+\mathcal{E}(0,-1).$$
The elements can be identified with tuples
$$\mathcal{E}_{A}\ni v_{A}=\left(\begin{array}{c}
\sigma\\
v_{\alpha}\\
\rho
\end{array}\right),$$
where $\sigma\in\mathcal{E}(1,0)$, $v_{\alpha}\in\mathcal{E}_{\alpha}(1,0)$ and $\rho\in\mathcal{E}(0,-1)$. Under change of pseudohermitian structure, $\hat{\theta}=\exp(\Upsilon)\theta$, these elements transform as
$$\widehat{\left(\begin{array}{c}
\sigma\\
v_{\alpha}\\
\rho
\end{array}\right)}=\left(\begin{array}{c}
\sigma\\
v_{\alpha}+\Upsilon_{\alpha}\sigma\\
\rho-\Upsilon^{\beta}v_{\beta}-\frac{1}{2}(\Upsilon^{\beta}\Upsilon_{\beta}+i\Upsilon_{0})\sigma
\end{array}\right).$$
Let us denote the dual bundle by $\mathcal{E}^{A}$ and the conjugate bundle by $\mathcal{E}_{\bar{A}}$. There is an invariant Hermitian metric $h_{A\bar{B}}$ on $\mathcal{E}_{A}$ induced by $h_{\alpha\bar{\beta}}$ that we can use to lower (and its inverse to raise) tractor indices in the same way that $h_{\alpha\bar{\beta}}$ is used. 
\par

Via $h_{\bar{A}B}$ we can identify 
$$\mathcal{E}_{\bar{A}}=\mathcal{E}^{B}=\mathcal{E}(0,1)+\mathcal{E}_{\bar{\alpha}}(0,1)\;(=\mathcal{E}^{\beta}(-1,0))+\mathcal{E}(-1,0).$$
We can represent elements in this bundle as 
$$\mathcal{E}_{\bar{A}}\ni w_{\bar{A}}=\left(\begin{array}{c}
\tau\\
w_{\bar{\alpha}}\\
\kappa
\end{array}\right),$$
where $\tau\in\mathcal{E}(0,1)$, $w_{\bar{\alpha}}\in\mathcal{E}_{\bar{\alpha}}(0,1)$ and $\kappa\in\mathcal{E}(-1,0)$. Under change of pseudohermitian structure, $\hat{\theta}=\exp(\Upsilon)\theta$, these elements transform as
$$\widehat{\left(\begin{array}{c}
\tau\\
w_{\bar{\alpha}}\\
\kappa
\end{array}\right)}=\left(\begin{array}{c}
\tau\\
w_{\bar{\alpha}}+\Upsilon_{\bar{\alpha}}\tau\\
\kappa-\Upsilon^{\bar{\beta}}w_{\bar{\beta}}-\frac{1}{2}(\Upsilon^{\bar{\beta}}\Upsilon_{\bar{\beta}}-i\Upsilon_{0})\tau
\end{array}\right).$$
Then the mapping 
\begin{eqnarray*}
\mathcal{E}_{A}\times\mathcal{E}_{\bar{B}}&\rightarrow&\mathcal{E}\\
(v_{A},w_{\bar{B}})&\mapsto& h^{A\bar{B}}v_{A}w_{\bar{B}}=\sigma\kappa+v_{\alpha}w_{\bar{\beta}}h^{\alpha\bar{\beta}}+\rho\tau
\end{eqnarray*}
is invariant.

\par
There is a canonical section $\mathbb{X}^{A}\in\mathcal{E}^{A}(1,0)$ that is (for each choice of pseudohermitian structure) given by 
$$\mathbb{X}^{A}=\left(\begin{array}{c}
0\\
0\\
1
\end{array}\right).$$
%If
%$$v_{A}=\left(\begin{array}{c}
%\sigma\\
%\tau_{\alpha}\\
%\rho
%\end{array}\right)\in\mathcal{E}^{A}(w,w'),$$
%then
%$$\mathbb{X}^{A}v_{A}=\sigma.$$

 }
\end{definition}

\begin{proposition}
There are two invariant operators defined on arbitrary weighted tractor bundles.
\begin{eqnarray*}
D_{\bar{A}}:\mathcal{E}^{\Phi}(w,w')&\rightarrow &\mathcal{E}_{\bar{A}}^{\Phi}(w,w'-1)\\
D_{\bar{A}}f&=&\left(\begin{array}{c}
w'(n+w+w')f\\
(n+w+w')\nabla_{\bar{\alpha}}f\\
-\left(\nabla^{\bar{\beta}}\nabla_{\bar{\beta}}f-iw'\nabla_{0}f+w'\left(1+\frac{w-w'}{n+2}\right)Pf\right)
\end{array}\right)
\end{eqnarray*}
and
\begin{eqnarray*}
D_{A}:\mathcal{E}^{\Phi}(w,w')&\rightarrow &\mathcal{E}_{A}^{\Phi}(w-1,w)\\
D_{A}f&=&\left(\begin{array}{c}
w(n+w+w')f\\
(n+w+w')\nabla_{\alpha}f\\
-\left(\nabla^{\beta}\nabla_{\beta}f+iw\nabla_{0}f+w\left(1+\frac{w'-w}{n+2}\right)Pf\right)
\end{array}\right),
\end{eqnarray*}
where the connections $\nabla_{\alpha},\nabla_{\bar{\alpha}}$ and $\nabla_{0}$ are the appropriate tractor connections uniquely determined by the formulae on $\mathcal{E}_{A}$ and $\mathcal{E}_{\bar{A}}$ given by \scriptsize
\begin{eqnarray*}
\nabla_{\bar{\beta}}\left(\begin{array}{c}
\sigma\\
v_{\alpha}\\
\rho
\end{array}\right)=
\left(\begin{array}{c}
\nabla_{\bar{\beta}}\sigma\\
\nabla_{\bar{\beta}}v_{\alpha}+h_{\alpha\bar{\beta}}\rho+P_{\alpha\bar{\beta}}\sigma\\
\nabla_{\bar{\beta}}\rho+iA_{\bar{\beta}}{}^{\alpha}v_{\alpha}-T_{\bar{\beta}}\sigma
\end{array}\right)
&,&
\nabla_{\bar{\beta}}\left(\begin{array}{c}
\tau\\
w_{\bar{\alpha}}\\
\kappa
\end{array}\right)=
\left(\begin{array}{c}
\nabla_{\bar{\beta}}\tau-w_{\bar{\beta}}\\
\nabla_{\bar{\beta}}w_{\bar{\alpha}}-iA_{\bar{\alpha}\bar{\beta}}\tau\\
\nabla_{\bar{\beta}}\kappa-P_{\bar{\beta}}{}^{\bar{\alpha}}w_{\bar{\alpha}}+T_{\bar{\beta}}\tau
\end{array}\right),\\
\nabla_{\beta}\left(\begin{array}{c}
\sigma\\
v_{\alpha}\\
\rho
\end{array}\right)=
\left(\begin{array}{c}
\nabla_{\beta}\sigma-\tau_{\beta}\\
\nabla_{\beta}v_{\alpha}+iA_{\alpha\beta}\sigma\\
\nabla_{\beta}\rho-P_{\beta}{}^{\alpha}v_{\alpha}+T_{\beta}\sigma
\end{array}\right)
&,&
\nabla_{\beta}\left(\begin{array}{c}
\tau\\
w_{\bar{\alpha}}\\
\kappa
\end{array}\right)=
\left(\begin{array}{c}
\nabla_{\beta}\tau\\
\nabla_{\beta}w_{\bar{\alpha}}+h_{\bar{\alpha}\beta}\kappa+P_{\bar{\alpha}\beta}\tau\\
\nabla_{\beta}\kappa-iA_{\beta}{}^{\bar{\alpha}}w_{\bar{\alpha}}-T_{\beta}\tau
\end{array}\right)
\end{eqnarray*}
\normalsize and
\begin{eqnarray*}
\nabla_{0}\left(\begin{array}{c}
\sigma\\
\tau_{\alpha}\\
\rho
\end{array}\right)&=&
\left(\begin{array}{c}
\nabla_{0}\sigma+\frac{i}{n+2}P\sigma-i\rho\\
\nabla_{0}\tau_{\alpha}-iP_{\alpha}{}^{\beta}\tau_{\beta}+\frac{i}{n+2}P\tau_{\alpha}+2iT_{\alpha}\sigma\\
\nabla_{0}\rho+\frac{i}{n+2}P\rho+2iT^{\alpha}\tau_{\alpha}+iS\sigma
\end{array}\right),\\
\nabla_{0}\left(\begin{array}{c}
\tau\\
w_{\bar{\alpha}}\\
\kappa
\end{array}\right)&=&
\left(\begin{array}{c}
\nabla_{0}\tau-\frac{i}{n+2}P\tau+i\kappa\\
\nabla_{0}w_{\bar{\alpha}}+iP_{\bar{\alpha}}{}^{\bar{\beta}}w_{\bar{\beta}}-\frac{i}{n+2}Pw_{\bar{\alpha}}-2iT_{\bar{\alpha}}\tau\\
\nabla_{0}\kappa-\frac{i}{n+2}P\kappa-2iT^{\bar{\alpha}}w_{\bar{\alpha}}-iS\tau
\end{array}\right).
\end{eqnarray*}
\end{proposition}
\begin{proof}
These operators are defined and the invariance is verified in~\cite{gg}.
\end{proof}

%Let $v_{A}\in\mathcal{E}_{A}(w-1,w')$, then
%$$D_{\bar{B}}v_{A}=\left(\begin{array}{ccc}
%*&*&w'(n+w-1+w')\rho\\
%*&(n+w-1+w')(\nabla_{\bar{\beta}}\tau_{\alpha}+h_{\alpha\bar{\beta}}\rho)&*\\
%\nabla^{\gamma}\tau_{\gamma}+(n+w')\rho&*&*
%\end{array}\right)$$
%$$D^{A}\tau_{A}=(n+w+w')\left(\nabla^{\gamma}\tau_{\gamma}+(w'+n)\rho\right).$$

%$$w_{\bar{A}}=\left(\begin{array}{c}
%\tau\\
%w_{\bar{\alpha}}\\
%\kappa
%\end{array}\right).$$

\subsection{Special splittings}

\begin{proposition}\label{proptwenty}
If 
$$w'\not\in\{-(n+s)\}_{s=0,...,k-1},$$
then there exists a unique lift for every element  $v_{i_{1}...i_{k}}\in\mathcal{E}_{(i_{1}...i_{k})}(w,w')$ to an element $v_{I_{1}...I_{k}}\in\mathcal{E}_{(I_{1}...I_{k})}(w-k,w')$, so that
$$D^{A}v_{AI_{2}...I_{k}}=0$$ and
$$\mathbb{X}^{A} v_{AI_{2}...I_{k}}=0.$$
Each excluded weight $w'$ corresponds to the existence of an invariant linear differential operator.
\end{proposition}
\begin{proof}
Let $v_{I_{1}...I_{k}}\in\mathcal{E}_{(I_{1}...I_{k})}(w-k,w')$ be an element with $\mathbb{X}^{A}v_{AI_{2}...I_{k}}=0$, looked at as an element in  $\mathcal{E}_{I_{1}...I_{k}}(w-k,w')$ with special symmetries. To be more precise, there are $k+1$ independent components realized as totally symmetric tensors $v_{i_{1}...i_{s}}$ with $s$ indices, $s=0,...,k$, and where the tensor with $k$ indices is exactly $v_{i_{1}...i_{k}}$.
\par
We can write
$$v_{I_{1}...I_{k}}=\left(\begin{array}{c}
0\\
v_{i_{1}I_{2}...I_{k}}\\
v_{I_{2}...I_{k}}
\end{array}\right).$$
Then we have 
$$D^{A}v_{AI_{2}...I_{k}}=(n+w+w'-k+1)\left(\nabla^{\beta}v_{\beta I_{2}...I_{k}}+(n+w')v_{I_{2}...I_{k}}\right).$$
Using
$$\nabla^{\beta}v_{\beta I_{2}...I_{k}}=\left(\begin{array}{c}
0\\
\nabla^{\beta}v_{\beta i_{2}I_{3}...I_{k}}+v_{i_{2}I_{3}...I_{k}}\\ 
\nabla^{\beta}v_{\beta I_{3}...I_{k}}+iA^{\alpha\beta}v_{\alpha\beta I_{3}...I_{k}}
\end{array}\right)$$
and the symmetries of $v_{I_{1}...I_{k}}$ it is easy to show that $D^{A}v_{AI_{2}...I_{k}}=0$ is equivalent to the two equations
\begin{eqnarray*}
\nabla^{\beta}v_{\beta i_{2}I_{3}...I_{k}}+(n+w'+1)v_{i_{2}I_{3}...I_{k}}&=&0\\
\nabla^{\beta}v_{\beta I_{3}...I_{k}}+(n+w')v_{I_{3}...I_{k}}+iA^{\alpha\beta}v_{\alpha\beta I_{3}...I_{k}}&=&0.
\end{eqnarray*}
Iterating this process (in the same way as for projective geometry and conformal geometry) yields the following $k$ equations:
\begin{equation}\label{splittingcr}
\nabla^{\beta}v_{\beta i_{2}...i_{s}}+(n+w'+s)v_{i_{1}...i_{s}}+(k-1-s)iA^{\alpha\beta}v_{\alpha\beta i_{1}...i_{s}},\;s=k-1,...,0.
\end{equation}
This proves uniqueness. Existence can either be proved by using the explicit description of the transformation laws or by the general theory in the last chapter.
\par
If $w'+n+k-l=0$, $l=1,...,k$, then there exists an invariant differential operator
\begin{eqnarray*}
\xoo{w-2k}{k}{0}\;...\;\ox{0}{w'}&\rightarrow &\xoo{w-2k+l}{k-l}{0}\;...\;\ox{0}{-l+w'}\\
v_{i_{1}...i_{k}}&\mapsto& \nabla^{i_{1}}...\nabla^{i_{l}}v_{i_{1}...i_{k}}+C.C.T.\;.
\end{eqnarray*}
\end{proof}

\subsubsection{Remark}
$$\xoo{w-2k}{k}{0}\;...\;\ox{0}{w'}$$
has the same central character as
$$\xoo{w-2k+l}{k-l}{0}\;...\;\ox{0}{-l+w'}$$
if and only if $w'+n+k-l=0$, $l=1,...,k$. 

\begin{proposition}\label{proptwentyone}
If 
$$w\not\in\{-(n+s)\}_{s=0,...,k-1},$$
then there exists a unique lift for every element  $v_{\bar{i}_{1}...\bar{i}_{k}}\in\mathcal{E}_{(\bar{i}_{1}...\bar{i}_{k})}(w,w')$ to an element $v_{\bar{I}_{1}...\bar{I}_{k}}\in\mathcal{E}_{(\bar{I}_{1}...\bar{I}_{k})}(w,w'-k)$, so that
$$D^{\bar{A}}v_{\bar{A}\bar{I}_{2}...\bar{I}_{k}}=0$$ and
$$\mathbb{X}^{\bar{A}} v_{\bar{A}\bar{I}_{2}...\bar{I}_{k}}=0.$$
Each excluded weight $w$ corresponds to the existence of an invariant linear differential operator.
\end{proposition}
\begin{proof}
This is proved exactly as Proposition~\ref{proptwenty}. The equations to solve are
$$\nabla^{\bar{\beta}}v_{\bar{\beta} \bar{i}_{2}...\bar{i}_{s}}+(n+w+s)v_{\bar{i}_{1}...\bar{i}_{s}}-(k-1-s)iA^{\bar{\alpha}\bar{\beta}}v_{\bar{\alpha}\bar{\beta} \bar{i}_{1}...\bar{i}_{s}},\;s=k-1,...,0$$
and excluded weights correspond to invariant operators
\begin{eqnarray*}
\xoo{w}{0}{0}\;...\;\oox{0}{k}{w'-2k}&\rightarrow &\xoo{w-l}{0}{0}\;...\;\oox{0}{k-l}{w'-2k+l}\\
v_{\bar{i}_{1}...\bar{i}_{k}}&\mapsto& \nabla^{\bar{i}_{1}}...\nabla^{\bar{i}_{l}}v_{\bar{i}_{1}...\bar{i}_{k}}+C.C.T.\;.
\end{eqnarray*}
\end{proof}
 
\begin{corollary}
For $w'=-(n+s)$, the corresponding invariant linear differential operator in  Proposition~\ref{proptwenty} can be written down explicitly by using the equations~(\ref{splittingcr}) in the proof of that proposition.
An analogous statement holds for $w=-(n+s)$ in Proposition~\ref{proptwentyone}.
\end{corollary}
\begin{proof}
This is done exactly as in the projective and conformal case.
\end{proof}

\subsubsection{Example}
Let us take $s=k-2$, then 
$$\nabla^{\alpha}\nabla^{\beta}v_{\alpha\beta i_{1}...i_{k-2}}-iA^{\alpha\beta}v_{\alpha\beta i_{1}...i_{k-2}}$$
is an invariant differential operator for $v_{i_{1}...i_{k}}\in\mathcal{E}_{(i_{1}...i_{k})}(w,-n-k+2)$.
\par
\vspace{0.4cm}

Now we can use the operators $D_{A}$ and $D_{\bar{A}}$ to include $v_{I_{1}...I_{k}}$ in the appropriate $M$-bundle just like in the projective and conformal case. 
We will demonstrate this procedure for an example where $k=1$.

\subsection{Examples}

\subsubsection{Example 1}

For first order splittings of $v_{\alpha}\in\mathcal{E}_{\alpha}(w,w')$ we have to compute
$$\xoo{w-2}{1}{0}\;...\;\ox{0}{w'}=\mathcal{E}_{\alpha}(w,w')\ni v_{\alpha}\rightarrow D_{\bar{A}}D_{B}D_{[C}v_{D]}\in\;\ooo{1}{1}{0}\;...\;\oo{0}{1}\;(w-3,w'-1).$$
Note that $D^{A}D_{A}f=0$ for every tractor field $f$ of arbitrary weight. We proceed in several steps:
%The excluded weights have the following meaning: (on the flat model space $\mathcal{H}$):
\begin{enumerate}
\item
First of all, we use Propostion~\ref{proptwenty} to define
$$v_{A}=\left(\begin{array}{c}0\\ v_{\alpha}\\ -\frac{1}{w'+n}\nabla^{\gamma}v_{\gamma}\end{array}\right)\in\mathcal{E}_{A}(w-1,w').$$
The weight to be excluded is $w'+n=0$. If $w'+n=0$, then the linear differential operator $v_{\alpha}\mapsto\nabla^{\gamma}v_{\gamma}$ is invariant. 
\item
Then we map $v_{A}$ to $D_{[A}v_{B]}\in\mathcal{E}_{[AB]}(w-2,w')$. This can be computed by
$$D_{A}v_{B}=\left(\begin{array}{ccc}
*&(w-1)(n+w+w'-1)v_{\beta}&*\\
-(n+w+w'-1)v_{\beta}&*&*\\
*&*&*
\end{array}\right)$$ 
and hence
\begin{eqnarray*}
D_{[A}v_{B]}&=&\left(\begin{array}{ccc}
&*&\\
w(n+w+w'-1)v_{\beta}&&*\\
&*&
\end{array}\right)\\
&\in&\begin{array}{ccc}
&\mathcal{E}_{[\alpha\beta]}(w,w')&\\
\mathcal{E}_{\alpha}(w,w')+&\oplus&+\mathcal{E}_{\alpha}(w-1,w'-1)\\
&\mathcal{E}(w-1,w'-1)&
\end{array}.
\end{eqnarray*}
The following weights have to be excluded:
\begin{enumerate}
\item
If $w=0$, then $v_{\alpha}\mapsto\nabla_{[\alpha}v_{\beta]}$ is invariant. 
\item
If $n+w+w'-1=0$, then $v_{\alpha}\mapsto\Delta v_{\alpha}+...$ is invariant, where the operator is given by $\Delta=-(\nabla^{\beta}\nabla_{\beta}+\nabla^{\bar{\beta}}\nabla_{\bar{\beta}})$.
\end{enumerate}
\item
Next, we compute $D_{A}D_{[B}v_{C]}\in\;\ooo{1}{1}{0}\;...\;\oo{0}{0}(w-3,w').$ The following weights have to be excluded:

\begin{enumerate}
\item
If $w-2=0$, then $v_{\alpha}\mapsto\nabla_{(\beta}v_{\alpha)}$ is invariant.
\item
If $w+w'+n-2=0$, then $v_{\alpha}\mapsto\Delta^{2}v_{\alpha}+...$ is invariant.
\end{enumerate}
\item
Finally, we compute $D_{\bar{A}}D_{B}D_{[C}v_{D]}\in\;\ooo{1}{1}{0}\;...\;\oo{0}{1}(w-3,w'-1)$ and note that the following weights have to be excluded:
\begin{enumerate}
\item
If $w'=0$, then $v_{\alpha}\mapsto\nabla_{\bar{\beta}}v_{\alpha}-\frac{1}{n}\nabla^{\gamma}v_{\gamma}h_{\bar{\beta}\alpha}$ is invariant.
\item
If $n+w+w'-3=0$, then $v_{\alpha}\mapsto\Delta^{3}v_{\alpha}+...$ is invariant.
\end{enumerate}
\end{enumerate}
The existence of the differential operators given above can only be guaranteed for the flat model space $\mathcal{H}$.
If the operator in question is standard (as are all cases (a) from above),~\cite{css2} guarantees that it allows a curved analogue. 
\par
This can be used, for example, as follows:
we can compute the two invariant bilinear differential pairings
\begin{eqnarray*}
\mathcal{E}_{\alpha}(v,v')\times\mathcal{E}_{\beta}(w,w')&\rightarrow&\mathcal{E}_{\alpha}(v+w-1,v'+w'-1)\\
(\omega_{\alpha},\tau_{\beta})&\mapsto&\tau^{\bar{B}}D_{\bar{B}}\omega_{A}\;\text{or}\;\omega^{\bar{B}}D_{\bar{B}}\omega_{A}
\end{eqnarray*}
as
\begin{description}
\item{(a)}
$$\tau^{\bar{\beta}}(\nabla_{\bar{\beta}}\omega_{\alpha}-\frac{1}{v'+n}h_{\bar{\beta}\alpha}\nabla^{\gamma}\omega_{\gamma})-\frac{v'}{w'+n}\omega_{\alpha}\nabla^{\beta}\tau_{\beta}$$
and
\item{(b)}
$$\omega^{\bar{\beta}}(\nabla_{\bar{\beta}}\tau_{\alpha}-\frac{1}{w'+n}h_{\bar{\beta}\alpha}\nabla^{\gamma}\tau_{\gamma})-\frac{w'}{v'+n}\tau_{\alpha}\nabla^{\beta}\omega_{\beta}.$$
\end{description}
Note that we have divided the pairings by $n+v+v'-1$ and $n+w+w'-1$ respectively.
% \begin{description}
%\item{(a)}
%$$v'(v'+n)\omega_{\alpha}\nabla^{\beta}\tau_{\beta}-(w'+n)(v'+n)\tau^{\bar{\beta}}(\nabla_{\bar{\beta}}\omega_{\alpha}-\frac{1}{n}h_{\bar{\beta}\alpha}\nabla^{\gamma}\omega_{\gamma})-\frac{1}{n}(w'+n)v'\tau_{\alpha}\nabla^{\gamma}\omega_{\gamma}$$
%and
%\item{(b)}
%$$v'(v'+n)\omega^{\bar{\beta}}(\nabla_{\bar{\beta}}\tau_{\alpha}-\frac{1}{n}h_{\bar{\beta}\alpha}\nabla^{\gamma}\tau_{\gamma})+\frac{1}{n}w'(v'+n)\tau^{\bar{\beta}}(\nabla_{\bar{\beta}}\omega_{\alpha}-\frac{1}{n}h_{\bar{\beta}\alpha}\nabla^{\gamma}\omega_{\gamma})+\frac{1-n^{2}}{n^{2}}w'v'\tau_{\alpha}\nabla^{\gamma}\omega_{\gamma}.$$
%\end{description}
\subsubsection{Example 2}
Let $v_{\alpha\beta}\in\mathcal{E}_{(\alpha\beta)}(w,w')$ and $f\in\mathcal{E}(v,v')$, then the pairing
\begin{eqnarray*}
&&v_{\alpha\beta}\nabla^{\alpha}\nabla^{\beta}f-\frac{2(v'-1)}{w'+n+1}(\nabla^{\alpha}v_{\alpha\beta})(\nabla^{\beta}f)+\frac{v'(v'-1)}{(w'+n+1)(w'+n)}f\nabla^{\alpha}\nabla^{\beta}v_{\alpha\beta}\\ 
&-&\frac{v'(v'+w'+n-1)}{w'+n}iA^{\alpha\beta}v_{\alpha\beta}f
\end{eqnarray*}
is invariant. It can also be written as
$$v_{AB}\tilde{D}^{A}\tilde{D}^{B}f,$$
where 
$$\tilde{D}^{A}f=\left(\begin{array}{c}
v'f\\
\nabla^{\alpha}f\\
0
\end{array}\right).$$
 This operator is not invariant in general, but has an easy  transformation law given by $\hat{\tilde{D}}^{A}f=\tilde{D}^{A}f+\mathbb{X}^{A}\Phi(\Upsilon),$ where $\Phi$ is some function of $\Upsilon$. Now we can use the fact that $\mathbb{X}^{A}v_{AB}=0$ to deduce that $v_{AB}\tilde{D}^{A}\tilde{D}^{B}f$ is invariant. This is in accordance with the formulae given in~\ref{examplepairings} (e).

% \section{a little bit of contemplation about the original question}
% $\g$ is an irreducible $\g$-module under the adjoint representation. The composition series of $\g$, when considered as a $\p$-module (for an arbitrary parabolic subgroup $\p\subset\g$) is exactly the grading
% $$\g=\g_{-k_{0}}+...+\g_{k_{0}}.$$
% Since for every weight $\mu\not=\gamma$ (where $\gamma$ is the highest weight of $\g$) of $\g$, we have $\Vert\gamma+\rho\Vert^{2}\g>\Vert\nu+\rho\Vert^{2}$, there is an invariant splitting operator 
%$$\mathcal{O}(\mathcal{G},\g_{-k_{0}})^{P}\rightarrow\mathcal{O}(\mathcal{G},\g).$$
%Let $\g\odot\g=\mathbb{V}_{1}\oplus\cdots\mathbb{V}_{j}$ be the decomposition into irreducible $\g$-modules. Then the $P$-module homomorphisms 
%$$\pi_{i}:\g\otimes\g\rightarrow\mathbb{V}_{i}$$
%define invariant pairings
%$$P_{i}:\mathcal{O}(\mathcal{G},\g)\times\mathcal{O}(\mathcal{G},\g)\rightarrow\mathcal{O}(\mathcal{G},\mathbb{V}_{i}).$$
%Taking these maps together yields an invariant bilinear differential pairing
%$$P_{i}:\mathcal{O}(\mathcal{G},\g_{-k_{0}})^{P}\times\mathcal{O}(\mathcal{G},\g_{-k_{0}})^{P}\rightarrow\mathcal{O}(\mathcal{G},\mathbb{V}_{i})^{P}.$$
%Claim 1:
%\par
%There is always an invariant bilinear differential pairing
%$$L_{i}:\mathcal{O}(\mathcal{G},\mathbb{V}_{i})^{P}\times\mathcal{O}(\mathcal{G},\mathbb{V})^{P}\rightarrow\mathcal{O}(\mathcal{G},\mathbb{V})^{P}$$
%Claim 2
%$$(D_{s}D_{t}+D_{t}D_{s})f =L_{1}(P_{1}(s,t),f)+...+L_{j}(P_{j}(s,t),f).$$
%Together with $(D_{s}D_{t}-D_{t}D_{s})f=D_{[s,t]}f$ this yields the decomposition of $D_{s}D_{t}$.

 \chapter{Appendix}
 
 \section{The BGG sequence}

For the convenience of the reader, we will write down the BGG sequences for the three geometries that we have considered throughout this paper: projective geometry, conformal geometry and CR geometry.
We will emphasize which of the maps in the BGG sequence give rise to curved analogues via the construction in Chapter~\ref{higherorderone} as Ricci corrected derivatives $\mathcal{D}_{j}$. In this case we will say that an operator arises via extremal roots. These operators can then be paired in the canonical way described in Chapter~\ref{higherorderone}.

\subsubsection{The general theory}
The BGG sequences are resolutions of finite dimensional irreducible representations of $G$ via invariant linear differential operators on $G/P$. They arise dually via resolutions of generalized Verma modules. These resolutions first appeared for Verma modules (which are generalized Verma modules for the special case that $\p=\B$ is a borel subalgebra) in~\cite{bgg} and were generalized in~\cite{l}. The interpretation of these resolutions in terms of differential operators first appeared in~\cite{er} and more details of this construction can be found in~\cite{be}. In~\cite{css2} these sequences were generalized to general curved parabolic geometries. In particular, this shows that all standard differential operators (i.e.~those which occur in a BGG sequence) have curved analogues. This is not true for non-standard operators as mentioned earlier. In fact, the problem of determining all non-standard differential operators on $G/P$ is still not solved. Only in certain cases (including conformal and projective geometry) this has been solved in~\cite{bc1} and~\cite{bc2}.

\subsection{$A_{n}$: The projective case}
Let us have a look at $\;\xoo{\;\;}{\;\;}{\;\;}\;...\;\oo{\;\;}{\;\;}\;$. All roots in $\g_{-1}$ belong to the same orbit of the Weyl group. Hence all roots in $\g_{-1}$ are extremal and we denote them by 
%$\theta_{i}$ such that $\g_{\theta_{i}}\in\g_{-1}$ 
$$\{\theta_{1},...,\theta_{n}\},\;\theta_{i}=-(\epsilon_{1}-\epsilon_{i+1})=-\sum_{j=1}^{i}\alpha_{i},$$
in the Dynkin diagram notation
$$\begin{array}{l}
\theta_{1}=\xoo{-2}{1}{0}\;...\;\oo{0}{0},\;\theta_{2}=\xoo{-1}{-1}{1}\;...\;\oo{0}{0},\;\theta_{3}=\xooo{-1}{0}{-1}{1}\;...\;\oo{0}{0},\;...,\\
\theta_{n}=\xoo{-1}{0}{0}\;...\;\oo{0}{-1}\;\end{array}.$$
%Thus for
%$$V=\xoo{k}{a_{1}}{a_{2}}\;...\;\oo{a_{n-2}}{a_{n-1}},$$
%we can go in the direction $\theta_{1}$ indefinitely, whereas we can go in the direction $\theta_{i}$ exactly $a_{i-1}$ steps. 
%So for $\;\xoo{\;\;}{\;\;}{\;\;}\;...\;\oo{\;\;}{\;\;}\;$ all invariant operators arise this way, which can be seen in the
Schematically, the BGG resolution looks like this:
$$V_{1}\rightarrow V_{2}\rightarrow V_{3}\rightarrow\cdots\rightarrow V_{n}\rightarrow V_{n+1},$$
more precisely
\begin{eqnarray*}
\;\xoo{a}{b}{c}\;...\;\oo{e}{f}\;&\stackrel{(a+1)\theta_{1}}{\longrightarrow}&\;\xoo{-a-2}{a+b+1}{c}\;...\;\oo{e}{f}\;\\
&\stackrel{(b+1)\theta_{2}}{\longrightarrow}&\;\xoo{-a-b-3}{a}{c+b+1}\;...\;\oo{e}{f}\;\\
%&\stackrel{(c+1)\theta_{3}}{\longrightarrow}&\\
&\vdots&\\
&\stackrel{(e+1)\theta_{n-1}}{\longrightarrow}&\;\xoo{x+f+1}{a}{b}\;...\;\oo{d}{e+f+1}\\
&\;\stackrel{(f+1)\theta_{n}}{\longrightarrow}&\;\xoo{x}{a}{b}\;...\;\oo{d}{e}\;,
\end{eqnarray*}
with $x=-(a+b+c+...+e+f+n+1)$. The meaning of an arrow $V \stackrel{m\theta_{i}}{\longrightarrow}W$ is as follows: the highest weight of $\mathbb{W}^{*}$ is equal to the highest weight of $\mathbb{V}^{*}$ plus $m\theta_{i}$, i.e.~the differential operator between those bundles is of order $m$ and arises via the extremal root $\theta_{i}$. One can see that all standard invariant linear differential operators (which in fact include all invariant linear differential operators in projective geometry) arise via extremal roots.

\subsection{$B_{l}$: The odd dimensional conformal case}
Let us now look at $\;\xoo{\;\;}{\;\;}{\;\;}\;...\;\btwo{\;\;}{\;\;}\;$. $\alpha_{l}$ is the only short simple root. The roots in $\g_{-1}$ can be grouped into two orbits of the Weyl group. One only contains $\alpha_{l}=\epsilon_{l}$ and the other one contains all the others, including $-\alpha_{1}$. In fact
$$\g_{-1}=\oplus_{\alpha\in\Delta_{-1}}\g_{\alpha}\;\text{with}\;\Delta_{-1}=\{-(\epsilon_{1}\pm\epsilon_{i})\}_{i=2,...,l}\cup\{-\epsilon_{1}\}.$$
We denote the roots in $\g_{-1}$ by 
$$\theta_{i}=-\epsilon_{1}+\epsilon_{i+1}=-\sum_{j=1}^{i}\alpha_{j},\;i=1,...,l-1,$$
$$\theta_{l}=-\epsilon_{1}=-\sum_{j=1}^{l}\alpha_{j}$$
and
$$\theta_{l+i}=-\epsilon_{1}-\epsilon_{l+1-i}=-\sum_{j=1}^{l-i}\alpha_{j}-2\sum_{j=l+1-i}^{l}\alpha_{j},\;i=1,...,l-1.$$
In pictures
\begin{eqnarray*}
\theta_{1}&=&\Bfive{-2}{1}{0}{0}{0},\;\theta_{2}=\Bfive{-1}{-1}{1}{0}{0},\;\theta_{3}=\Bfive{-1}{0}{-1}{0}{0}\\
&\vdots&\\
\theta_{l-1}&=&\Bfive{-1}{0}{0}{-1}{2},\;\theta_{l}=\Bfive{-1}{0}{0}{0}{0},\;\theta_{l+1}=\Bfive{-1}{0}{0}{1}{-2}\\
&\vdots&\\
\theta_{2l-3}&=&\Bfive{-1}{0}{1}{0}{0},\;\theta_{2l-2}=\Bfive{-1}{1}{-1}{0}{0},\;\theta_{2l-1}=\Bfive{0}{-1}{0}{0}{0}.
\end{eqnarray*}
Note that all $\theta_{i}$ apart from $\theta_{l}$ are extremal roots.
The BGG schematically looks like this
$$V_{1}\stackrel{\theta_{1}}{\longrightarrow}\cdots\stackrel{\theta_{l-1}}{\longrightarrow}V_{l}\stackrel{\theta_{l}}{\longrightarrow}V_{l+1}\stackrel{\theta_{l+1}}{\longrightarrow}
\cdots\stackrel{\theta_{2l-1}}{\longrightarrow}V_{2l},$$
more precisely
\begin{eqnarray*}
\Bfive{a}{b}{c}{e}{f}&\stackrel{(a+1)\theta_{1}}{\longrightarrow}&\Bfive{-a-2}{a+b+1}{c}{e}{f}\\
&\stackrel{(b+1)\theta_{2}}{\longrightarrow}&\Bfive{-a-b-3}{a}{c+b+1}{e}{f}\\
&\vdots&\\
&\stackrel{(e+1)\theta_{l-1}}{\longrightarrow}&\Bfive{y}{a}{b}{d}{f+2e+2}\\
&\stackrel{(f+1)\theta_{l}}{\longrightarrow}&\Bfive{y-f-1}{a}{b}{d}{f+2e+2}\\
&\stackrel{(e+1)\theta_{l+1}}{\longrightarrow}&\Bfive{y-f-e-2}{a}{b}{d+e+1}{f}\\
&\vdots&\\
&\stackrel{(a+1)\theta_{2l-1}}{\longrightarrow}&\Bfive{x}{b}{c}{e}{f},
\end{eqnarray*}
where $y=-(a+b+c+...+e)-l$ and $x=-(a+2b+2c+...+2e+f)-2l$. So all standard invariant linear differential operators apart from the one that occurs in the middle of the BGG resolution arise via extremal roots. The form of the operator that occurs in the middle of the BGG resolution is indeed different from the standard form of Ricci-corrected derivatives owing to the fact that the target representation occurs in $\otimes^{k}\g_{1}\otimes\mathbb{V}$ with multiplicity. An example is given in~\cite{css}, Remark 6.6.

\subsection{$D_{l}$: The even dimensional conformal case}
For conformal geometry in even dimensions all roots in $\g_{-1}$ belong to the same orbit of the Weyl group, so all roots in $\g_{-1}$ are extremal. We have
$\g_{-1}=\oplus_{\theta\in\Delta_{-1}}\g_{\theta}$ with
$$\Delta_{-1}=\{-(\epsilon_{1}\pm\epsilon_{i})\}_{i=2,...,l}.$$
In terms of simple roots these can be written as
$$\theta_{i}=-(\epsilon_{1}-\epsilon_{i+1})=-\sum_{j=1}^{i}\alpha_{j}\;,i=1,...,l-1$$
and
$$\theta_{2l-i-1}=-(\epsilon_{1}+\epsilon_{i+1})=-\sum_{j=1}^{i}\alpha_{j}-2\sum_{j=i+1}^{l-2}\alpha_{j}-\alpha_{l-1}-\alpha_{l},\;i=1,...,l-1.$$
In the Dynkin diagram notation \scriptsize
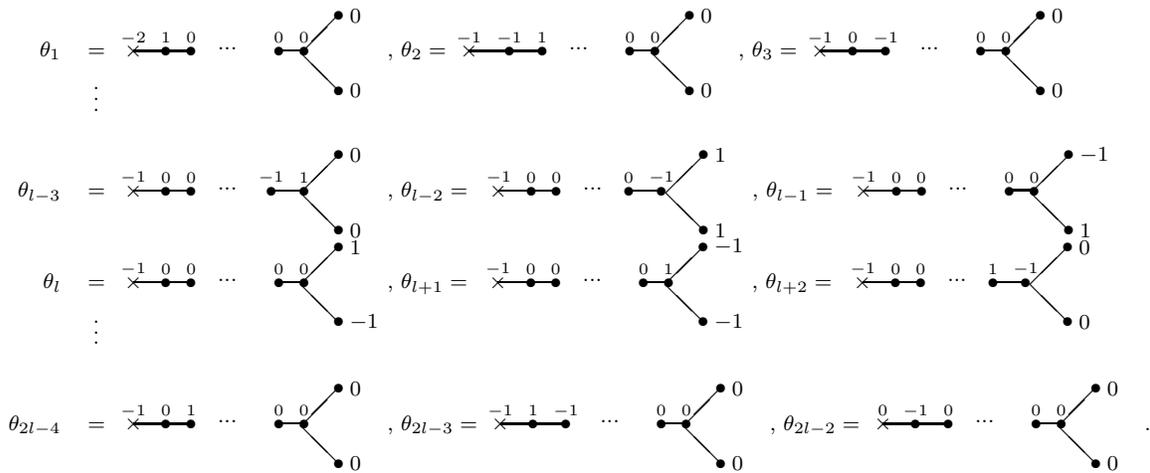
\begin{eqnarray*}
\theta_{1}&=&\hspace{.9cm}\Ddd{-2}{1}{0}{0}{0}{0}{0},\,\theta_{2}=\hspace{1.3cm}\Ddd{-1}{-1}{1}{0}{0}{0}{0},\,\theta_{3}=\hspace{1.3cm}\Ddd{-1}{0}{-1}{0}{0}{0}{0}\\
&\vdots&\\
\theta_{l-3}&=&\hspace{.9cm}\Ddd{-1}{0}{0}{-1}{1}{0}{0},\,\theta_{l-2}=\hspace{1.2cm}\Ddd{-1}{0}{0}{0}{-1}{1}{1},\,\theta_{l-1}=\hspace{1.2cm}\Ddd{-1}{0}{0}{0}{0}{-1}{1}\\
\theta_{l}&=&\hspace{.9cm}\Ddd{-1}{0}{0}{0}{0}{1}{-1},\,\theta_{l+1}=\hspace{1.2cm}\Ddd{-1}{0}{0}{0}{1}{-1}{-1},\,\theta_{l+2}=\hspace{1.2cm}\Ddd{-1}{0}{0}{1}{-1}{0}{0}\\
&\vdots&\\
\theta_{2l-4}&=&\hspace{.9cm}\Ddd{-1}{0}{1}{0}{0}{0}{0},\,\theta_{2l-3}=\hspace{1.3cm}\Ddd{-1}{1}{-1}{0}{0}{0}{0},\,\theta_{2l-2}=\hspace{1.2cm}\Ddd{0}{-1}{0}{0}{0}{0}{0}.
\end{eqnarray*}
\normalsize This fits into the BGG sequence as follows:

$$\Ddd{a}{b}{c}{e}{f}{g}{h}\stackrel{(a+1)\theta_{1}}{\longrightarrow}\hspace{2.5cm}\Ddd{-a-2}{a+b+1}{c}{e}{f}{g}{h}\stackrel{(b+1)\theta_{2}}{\longrightarrow}\cdots
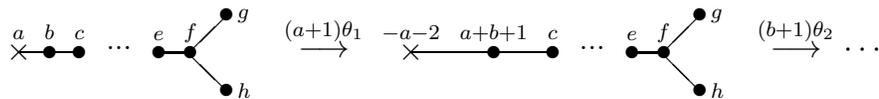
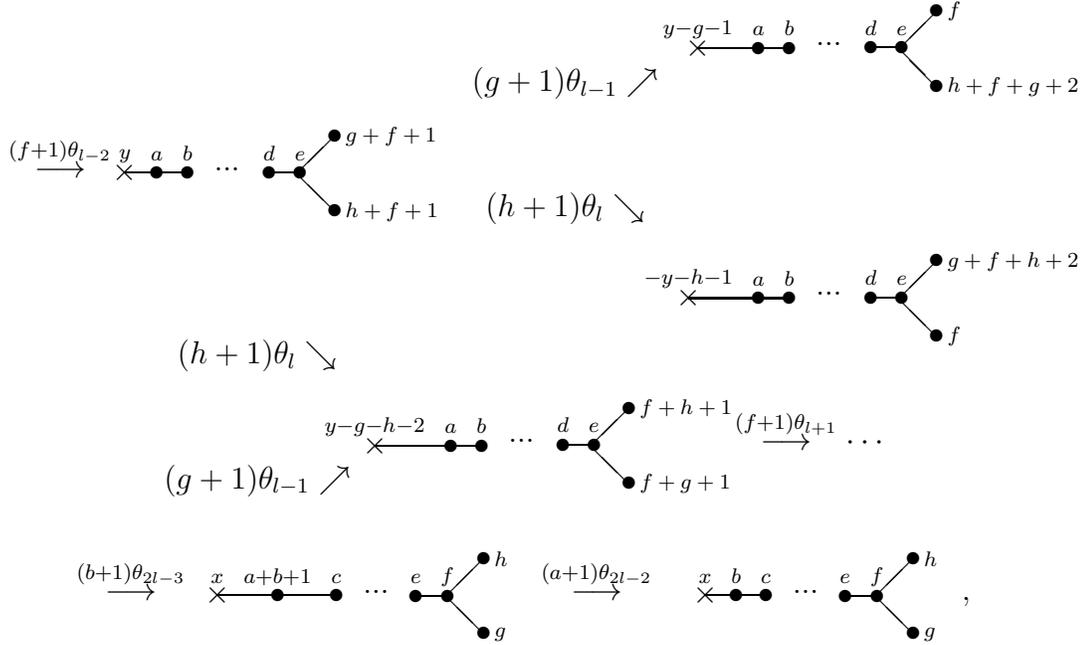
$$
$$\begin{array}{ccccc}
&&\hspace{1.5cm}\Ddd{y-g-1}{a}{b}{d}{e}{f}{h+f+g+2}&&\\
&\quad(g+1)\theta_{l-1}\nearrow&&&\\
\stackrel{(f+1)\theta_{l-2}}{\longrightarrow}\hspace{1cm}\Ddd{y}{a}{b}{d}{e}{g+f+1}{h+f+1}\quad&&&&\\
&\quad(h+1)\theta_{l}\searrow&&&\\
&&\hspace{1.5cm}\Ddd{-y-h-1}{a}{b}{d}{e}{g+f+h+2}{f}&&\\
\end{array}$$
$$\begin{array}{ccc}
(h+1)\theta_{l}\searrow&&\\
&\hspace{1.5cm}\Ddd{y-g-h-2}{a}{b}{d}{e}{f+h+1}{f+g+1}&\quad\stackrel{(f+1)\theta_{l+1}}{\longrightarrow}\cdots\\
(g+1)\theta_{l-1}\nearrow&&
\end{array}$$
\par
%\vspace{0.5cm}
$$\stackrel{(b+1)\theta_{2l-3}}{\longrightarrow}\hspace{2cm}\Ddd{x}{a+b+1}{c}{e}{f}{h}{g}\stackrel{(a+1)\theta_{2l-2}}{\longrightarrow}\hspace{1.5cm}\Ddd{x}{b}{c}{e}{f}{h}{g},$$
\par
\vspace{1cm}
where $y=-(a+b+c+...+e+f)-l+1$ and $x=-2(b+c+...+e+f)-g-h-a-2(l-1)$. The general pattern looks like this:
$$\begin{array}{ccccccccccccccccc}
&&&&&&V_{l}&&&&&&\\
&&&&&\theta_{l-1}\nearrow&&\theta_{l}\searrow&&&&&\\
V_{1}&\stackrel{\theta_{1}}{\rightarrow}&\cdots&\stackrel{\theta_{l-2}}{\rightarrow}&V_{l-1}&&&&V_{l+2}&\stackrel{\theta_{l+1}}{\rightarrow}&\cdots&\stackrel{\theta_{2l-2}}{\rightarrow}&V_{2l}\\
&&&&&\theta_{l}\searrow&&\theta_{l-1}\nearrow&&&&&\\
&&&&&&V_{l+1}&&&&&&
\end{array}.$$
So we see that all standard invariant linear differential operators arise via extremal roots.

\subsection{$A_{n+1}$: The CR case}
In CR geometry $\g=A_{n+1}$, so all roots have the same length. Therefore all roots $\theta_{i}$ with $\g_{\theta_{i}}\in\g_{-1}$ are extremal. In the BGG resolution, however, only those operators arise via extremal roots that do not involve the $\g_{-2}$ part. More precisely, for $\xoo{\;\;}{\;\;}{\;\;}\;...\;\ox{\;\;}{\;\;}$, we have
$$\g_{-}=\oplus_{\alpha\in\Delta_{-2}}\g_{\alpha}+\begin{array}{c}
\oplus_{\alpha\in\Delta_{-1}^{1}}\g_{\alpha}\\
\oplus\\
\oplus_{\alpha\in\Delta_{-1}^{2}}\g_{\alpha}
\end{array}
=\g_{-2}+\begin{array}{c}
\g_{-1}^{1}\\
\oplus\\
\g_{-1}^{2}
\end{array}$$
with
$$\Delta_{-1}^{1}=\{-(\epsilon_{1}-\epsilon_{i})\}_{i=2,...,n+1},\;\;\Delta_{-1}^{2}=\{\epsilon_{n+2}-\epsilon_{i}\}_{i=2,...,n+1}\;\text{and}\;\Delta_{-2}=\{-\epsilon_{1}+\epsilon_{n+2}\}.$$
Let us write $\theta_{i}=-(\epsilon_{1}-\epsilon_{i+1})$ and $\Theta_{i}=\epsilon_{n+2}-\epsilon_{i+1}$ for $i=1,...,n$.
Schematically the BGG resolution looks this:

%$$\begin{array}{ccccccc}
%              &               &               &         &            &              &          &V_{n+1,1}&\rightarrow&W_{n+1,1}\\
 %            &               &             &           &             &\nearrow&\cdots& \hdots&    \searrow&\hdots\\
 %            &             &              &             &V_{3,1}&              &\cdots& \hdots&                   &\\
  %           &             &             &\nearrow&             &\searrow&\cdots& \hdots&\\
 %           &\nearrow&V_{2,1}&                &            &\nearrow&\cdots& \hdots&\\
%V_{1}   &              &            & \searrow &V_{3,2}&              &\cdots& \hdots&\\
 %           &\searrow&V_{2,2} & \nearrow &             &\searrow&\cdots& \hdots&\\
 %           &              &             &\searrow&              &\nearrow&\cdots& \hdots&\\
 %            &             &             &              &V_{3,3}  &              &\cdots& \hdots&\\
 %             &            &              &             &               &\searrow&\cdots& \hdots&\nearrrow\\
  %            &            &              &              &               &              &        &  V_{n+1,n+1}&\rightarrow
%\end{array}$$
\par
$$\begin{picture}(10,10)
\put(-10,5){$V_{1,1}$}
\put(-7,6){$V_{2,1}$}
\put(-7,4){$V_{2,2}$}
\put(-4,7){$V_{3,1}$}
\put(-4,5){$V_{3,2}$}
\put(-4,3){$V_{3,3}$}
\put(1,10){$V_{n+1,1}$}
\put(1,8){$V_{n+1,2}$}
\put(1,2){$V_{n+1,n}$}
\put(1,0){$V_{n+1,n+1}$}
\put(2,5){$\vdots$}
\put(4,5){$\vdots$}
\put(6,5){$\vdots$}
\put(17,5){$W_{1,1}$}
\put(14,6){$W_{2,1}$}
\put(14,4){$W_{2,2}$}
\put(11,7){$W_{3,1}$}
\put(11,5){$W_{3,2}$}
\put(11,3){$W_{3,3}$}
\put(5,10){$W_{n+1,1}$}
\put(5,8){$W_{n+1,2}$}
\put(5,2){$W_{n+1,n}$}
\put(5,0){$W_{n+1,n+1}$}
\put(-8.5,5){\vector(1,1){1}}
\put(-8.5,5){\vector(1,-1){1}}
\put(-5.5,6){\vector(1,1){1}}
\put(-5.5,6){\vector(1,-1){1}}
\put(-5.5,4){\vector(1,1){1}}
\put(-5.5,4){\vector(1,-1){1}}
\put(-2.5,7){\vector(1,1){1}}
\put(-2.5,7){\vector(1,-1){1}}
\put(-2.5,5){\vector(1,1){1}}
\put(-2.5,5){\vector(1,-1){1}}
\put(-2.5,3){\vector(1,1){1}}
\put(-2.5,3){\vector(1,-1){1}}
\put(-1,8){$\Ddots$}
\put(-1,1){$\ddots$}
\put(8,9){$\ddots$}
\put(8,1){$\Ddots$}
\put(12.75,7.25){\vector(1,-1){1}}
\put(12.75,5.25){\vector(1,-1){1}}
\put(15.75,6.25){\vector(1,-1){1}}
\put(15.75,4.25){\vector(1,1){1}}
\put(12.75,3.25){\vector(1,1){1}}
\put(12.75,5.25){\vector(1,1){1}}
\put(9.75,8.25){\vector(1,-1){1}}
\put(9.75,6.25){\vector(1,-1){1}}
\put(9.75,4.25){\vector(1,1){1}}
\put(9.75,6.25){\vector(1,1){1}}
\put(9.75,4.25){\vector(1,-1){1}}
\put(9.75,2.25){\vector(1,1){1}}
\put(3.5,8.25){\vector(1,1){1.5}}
\put(3.5,2.25){\vector(1,1){1.5}}
\put(3.5,0.25){\vector(1,1){1.5}}
\put(3.5,10.25){\vector(1,-1){1.5}}
\put(3.5,8.25){\vector(1,-1){1.5}}
\put(3.5,2.25){\vector(1,-1){1.5}}

\put(3.5,10.25){$....$}
\put(1.75,9.71){$\arrowhead$}
\put(3.5,8.25){$....$}
\put(1.75,7.71){$\arrowhead$}
\put(3.5,2.25){$....$}
\put(1.75,1.71){$\arrowhead$}
\put(3.5,0.25){$....$}
\put(1.75,-.29){$\arrowhead$}
%\put(3.5,10.25){\line(1,0){1}}
%\put(3.5,8.25){\line(1,0){1}}
%\put(3.5,2.25){\line(1,0){1}}
%\put(3.5,0.25){\line(1,0){1}}
\end{picture}.$$
All the arrows, apart from the horizontal ones that are drawn with dotted lines, arise via extremal roots. For 
$$V_{1,1}=\xoo{a_{1}}{a_{2}}{a_{3}}\;...\;\ox{a_{n}}{a_{n+1}},$$
for example,
\begin{enumerate}
\item
the map $V_{i,j}\rightarrow V_{i+1,j}$ arises via $\theta_{i}$ (the order of the operator is $a_{i}+1$),
\item
the map $V_{i,j}\rightarrow V_{i+1,j+1}$ arises via $\Theta_{n+1-i}$ (the order of the operators is $a_{n+2-i}+1$),
\item
the map $W_{i,j}\rightarrow W_{i-1,j}$ arises via $\Theta_{i-1}$ (the order of the operator is $a_{i}+1$) and
\item
the map $W_{i,j}\rightarrow W_{i-1,j-1}$ arises via $\theta_{n+2-i}$ (the order of the operator is $a_{n+2-i}+1$).
\end{enumerate}

 \section{Outlook}
There are a few points that remain open and we will list them in no particular order.
\begin{enumerate} 
\item 
The most obvious open question is the conjecture stated in~\ref{comparison}. This conjecture would associate a meaning to every excluded representation (geometric weight).
\item
Are there invariant bilinear differential pairings on a homogeneous space $G/P$ that do not allow a curved analogue? This happens for invariant linear differential operators and the obstruction is given by interesting tensors like the Bach tensor in conformal geometry. We have not encountered such a situation as yet, so this is a totally open question.
\item
Other Cartan geometries (not necessarily parabolic). This is a field which has not been studied very much at all. It would be interesting to see how many concepts carry over from the parabolic case. One might, for example, find a parabolic geometry that lies above a given Cartan geometry and then one would be able use tools such as the BGG sequence.
%A first example of this situation can be encountered for conformal symplectic structures that are closely related to contact projective structures. 
This is a whole new interesting research area.
\end{enumerate}


\begin{thebibliography}{XX}
 
 \bibitem{beg} T.~N.~Bailey, M.~G.~Eastwood, and R.~Gover,
{\em Thomas's structure bundle for conformal, projective, and related
structures},
Rocky Mountain J. Math.  {\bf 24}  (1994),  no. 4, 1191--1217.

\bibitem{b} R.~J.~Baston,
{\em Almost Hermitian Symmetric Manifolds 2, Differential Invariants},
Duke Math. J. {\bf 63}  (1991),  no. 1, 113--138.
 

\bibitem{be} R.~J.~Baston and M.~G.~Eastwood,
{\em The Penrose transform. Its interaction with representation theory},
Oxford Mathematical Monographs. Oxford Science Publications. The Clarendon Press, Oxford University Press, New York, (1989). xvi+213 pp. ISBN: 0-19-853565-1. 

\bibitem{bgg} I.~N.~Bernstein, I.~M.~Gelfand, and S.~I.~Gelfand,
{\em Differential operators on the base affine space and a study of $\g$-modules},
Lie groups and their representations (Proc. Summer School, Bolyai Janos Math. Soc., Budapest, 1971), Halsted, New York, (1975) 21--64.


\bibitem{bc1} B.~D.~Boe, and D.~H.~Collingwood,
{\em A comparison theory for the structure on induced representations 1},
 J. Algebra  {\bf 94}  (1985),  no. 2, 511--545. 

\bibitem{bc2} B.~D.~Boe, and D.~H.~Collingwood,
{\em A comparison theory for the structure on induced representations 2},
 Math. Z.  {\bf 190}  (1985),  no. 1, 1--11.

\bibitem{bceg} T.~Branson, A.~\v{C}ap, M.~G.~Eastwood, and R.~Gover,
{\em Prolongations of Geometric Overdetermined Systems},
Int. Jour. Math. {\bf 17}, no. 6, (2006) 641--664.

\bibitem{ca} E.~Cartan,
{\em Sur la g\'eom\'etrie pseudo-conforme des hypersurfaces de l'espace de deux variables complexes 1 and 2},
\newline
1. Ann. Mat. Pura Appl.  {\bf 11}  (1932),  no. 1, 17--90.\newline
%{\em Sur la g\'eom\'etrie pseudo-conforme des hypersurfaces de deux variables complexes},
2. Ann. Scuola Norm. Sup. Pisa Cl. Sci. (2)  {\bf 1}  (1932),  no. 4, 333--354.
%2. Ann. Scuola Norm. Sup. Pisa (2), {\bf 1} (1932), 333--354

\bibitem{ca1} E.~Cartan,
{\em Les espaces \'a connexion conforme},
Ann. Soc. Pol. Math. {\bf 2} (1923), 171--221.

\bibitem{ca2} E.~Cartan,
{\em Sur les vari\'et\'es \'a connexion projective},
 Bull. Soc. Math. France  {\bf 52}  (1924), 205--241. 

\bibitem{cd} D.~M.~J.~Calderbank and T.~Diemer,
{\em Differential invariants and curved Bernstein-Gelfand-Gelfand sequences},
J. Reine Angew. Math. {\bf 537} (2001), 67--103.

\bibitem{cds} D.~M.~J.~Calderbank, T.~Diemer and V.~Sou\v{c}ek,
{\em Ricci-corrected derivatives and invariant differential operators},
Differential Geom. Appl. {\bf 23}  (2005),  no.~2, 149--175. 

\bibitem{cg} A.~\v{C}ap, R.~Gover,
{\em Tractor calculi for parabolic geometries},
Trans. Amer. Math. Soc.  {\bf 354}  (2002),  no.~4, 1511--1548. 

\bibitem{css1} A.~\v{C}ap, J.~Slov\'ak, and V.~Sou\v{c}ek,
{\em Invariant operators on manifolds with almost Hermitian symmetric structures, 1. Invariant Differentiation},
Acta Math. Univ. Comenian. {\bf 66} (1997) 33--69. 

\bibitem{css4} A.~\v{C}ap, J.~Slov\'ak, and V.~Sou\v{c}ek,
{\em Invariant operators on manifolds with almost Hermitian symmetric structures, 2. Normal Cartan Connections},
Acta Math. Univ. Comenian. {\bf 66}  (1997),  no.~2, 203--220. 

\bibitem{css} A.~\v{C}ap, J.~Slov\'ak, and V.~Sou\v{c}ek,
{\em Invariant operators on manifolds with almost Hermitian symmetric structures, 3. Standard operators},
 Differential Geom. Appl.  {\bf 12}  (2000),  no. 1, 51--84.

\bibitem{css2} A.~\v{C}ap, J.~Slov\'ak, and V.~Sou\v{c}ek,
{\em Bernstein-Gelfand-Gelfand Sequences},
 Ann. of Math. (2)  {\bf 154}  (2001),  no.~1, 97--113.

\bibitem{cs} A.~\v{C}ap, and H.~Schichl,
{\em Parabolic geometries and canonical Cartan connections},
Hokkaido Math.~J.~{\bf 29}  (2000),  no.~3, 453--505.

\bibitem{cs1} A.~\v{C}ap, and J.~Slov\'ak,
{\em Parabolic Geometries},
to appear in Mathematical Surveys and Monographs, American Mathematical Society.

\bibitem{cs2} A.~\v{C}ap and J.~Slov\'ak,
{\em Weyl structures for parabolic geometries},
Math. Scand. {\bf 93}  (2003),  no.~1, 53--90. 

\bibitem{cs3} A.~\v{C}ap and V.~Sou\v{c}ek,
{\em Curved Casimir operators and the BGG machinery},
SIGMA Symmetry Integrability Geom. Methods Appl. {\bf 3}  (2007), Paper 111, 17 pp.
 
\bibitem{cm} S.~S.~Chern, and J.~Moser,
{\em Real hypersurfaces in complex manifolds},
Acta Math. {\bf 133} (1974), 219--271.


\bibitem{d} V.~K.~Dobrev,
{\em New generalized Verma modules and multilinear intertwining differential operators},
J. Geom. Phys. {\bf 25} (1998), no.~1-2, 1--28.

\bibitem{e} M.~G.~Eastwood,
{\em Notes on Projective Differential Geometry},
Symmetries and Overdetermined Systems of Partial Differential Equations,
IMA Volumes in Mathematics and its Applications, {\bf 144} (2007), Springer, 41--60.

\bibitem{e1} M.~G.~Eastwood,
{\em The Cartan Product},
Bull. Belg. Math. Soc. Simon Stevin {\bf 11}  (2004),  no.~5, 641--651.

\bibitem{e2} M.~G.~Eastwood,
{\em Higher symmetries of the Laplacian},
 Ann. of Math. (2)  {\bf 161}  (2005),  no.~3, 1645--1665.

\bibitem{er} M.~G.~Eastwood and J.~W.~Rice,
{\em Conformally invariant differential operators on Minkowski space 
and their curved analogues},
Commun. Math. Phys. {\bf 109} (1987), no. 2, 207--228. 
Erratum, Commun. Math. Phys. {\bf 144} (1992), no. 1, 213.

\bibitem{es} M.~G.~Eastwood, and J.~Slov\'ak,
{\em Semi-holonomic Verma Modules},
Jour. Alg. {\bf 197}  (1997),  no.~2, 424--448.

\bibitem{f} H.~D.~Fegan,
{\em Conformally invariant first order differential operators},
Quart. J. Math. Oxford. Ser. (2) {\bf 27} (1976), no. 107, 371--378.

\bibitem{f1} D.~J.~ F.~Fox,
{\em Projectively invariant star products},
IMRP Int. Math. Res. Pap. (2005), no.~{\bf 9}, 461--510.

\bibitem{f2} D.~J.~F.~Fox,
{\em Contact projective structures},
Indiana Univ. Math. J. {\bf 54}  (2005),  no.~6, 1547--1598. 

\bibitem{fh} W.~Fulton, and J.~Harris,
{\em Representation Theory. A First Course},
Graduate Texts in Mathematics, {\bf 129},
Readings in Mathematics. Springer-Verlag, New York, (1991). xvi+551 pp. ISBN: 0-387-97527-6; 0-387-97495-4. 

\bibitem{gw} R.~Goodman, and N.~R.~Wallach,
{\em Representations and Invariants of the Classical Groups},
Encyclopedia of Mathematics and its Applications, {\bf 68},
Cambridge University Press, Cambridge, (1998). xvi+685 pp. ISBN: 0-521-58273-3; 0-521-66348-2. 

\bibitem{g} R.~Gover,
{\em Aspects of Parabolic Invariant Theory},
Rend. Circ. Mat. Palermo (2) Suppl. No. 59 (1999), 25--47.

\bibitem{g2} R.~Gover,
{\em Conformally invariant operators of standard type},
Quart. J. Math. Oxford Ser. (2)  {\bf 40}  (1989),  no.~158, 197--207.

\bibitem{gh} R.~Gover, and K.~Hirachi,
{\em Conformally invariant powers of the Laplacian---a complete nonexistence theorem}
J. Amer. Math. Soc. {\bf 17}  (2004),  no.~2, 389--405.

\bibitem{gg} R.~Gover, and R.~Graham,
{\em CR invariant powers of the sub-Laplacian},
J. Reine Angew. Math. {\bf 583} (2005), 1--27.

\bibitem{gs} R.~Gover, and J.~Slov\'ak,
{\em Invariant local twistor calculus for quaternionic structures and related geometries}
J. Geom. Phys. {\bf 32} (1999), no.~1, 14--56. 

\bibitem{gr} R.~Graham,
{\em \ Conformally invariant powers of the Laplacian II. Nonexistence},
J. London Math. Soc. (2)  {\bf 46}  (1992),  no.~3, 566--576.


\bibitem{h} J.~E.~Humphreys,
{\em Introduction to Lie Algebras and Representation Theory},
Graduate Texts in Mathematics, Vol. {\bf  9}. Springer-Verlag, New York-Berlin, (1972). xii+169 pp.

\bibitem{hu} J.~E.~Humphreys,
{\em Reflection Groups and Coxeter Groups},
Cambridge Studies in Advanced Mathematics, {\bf 29}. Cambridge University Press, Cambridge, (1990). xii+204 pp. ISBN: 0-521-37510-X. 

\bibitem{kob} S.~Kobayashi,
{\em Transformations groups in Differential Geometry},
Ergebnisse der Mathematik und ihrer Grenzgebiete, Band {\bf 70}. Springer-Verlag, New York-Heidelberg, (1972). viii+182 pp. 

\bibitem{kms} I.~Kol\'a\v{r}, W.~Michor, and J.~Slov\'ak,
{\em Natural Operations in Differential Geometry},
Springer-Verlag, Berlin, (1993). vi+434 pp. ISBN: 3-540-56235-4. 

\bibitem{ko} B.~Kostant,
{\em Lie algebra cohomology and the generalized Borel-Weil theorem},
Ann. of Math. (2)  {\bf 74}  (1961) 329--387.

\bibitem{k} J.~Kroeske,
{\em Invariant differential pairings},
to appear in Acta Math. Univ. Comenian. {\bf 77} (2008), no.~2.

\bibitem{kn} A.~W.~Knapp,
{\em Lie groups, Lie algebras and cohomology},
Mathematical Notes, {\bf 34}. Princeton University Press, Princeton, NJ, (1988). xii+510 pp. ISBN: 0-691-08498-X

\bibitem{ku} S.~Kumar,
{\em Proof of the Parthasarathy-Ranga-Rao-Varadarajan conjecture},
Invent. Math. {\bf 93} (1988), no. 1,  117--130.

\bibitem{lb} J.~,D.~Louck and L.~C.~Biedenharn,
{\em Canonical unit adjoint tensor operators in $U(n)$}
J. Mathematical Phys. {\bf 11} (1970), 2368--2414. 

\bibitem{lee} J.~M.~Lee,
{\em The Fefferman metric and pseudo-Hermitian invariants},
Trans. Amer. Math. Soc. {\bf 296} (1986), no.~1, 411--429.

\bibitem{l} J.~Lepowsky,
{\em A generalization of the Bernstein-Gelfand-Gelfand resolution},
Jour. Alg. {\bf 49} (1977) 496--511.

\bibitem{m} T.~Morimoto,
{\em Lie algebras, geometric structures and differential equations on filtered manifolds},
Lie groups, geometric structures and differential equations---one hundred years after Sophus Lie (Kyoto/Nara, 1999), 
Adv. Stud. Pure Math., {\bf 37}, Math. Soc. Japan, Tokyo, (2002) 205--252.

\bibitem{o} S.~Okubo,
{\em Casimir invariants and vector operators in simple and classical Lie algebras}
Journal of Mathematical Physics, Vol. {\bf 18}, No.~12, (1977), 2382--2394.

\bibitem{ol} P.~J.~Olver,
{\em Equivalence, Invariants, and Symmetry},
Cambridge University Press, Cambridge, (1995). xvi+525 pp. ISBN: 0-521-47811-1. 

\bibitem{pr} R.~Penrose and W.~Rindler,
{\em Spinors and Space-time, Volume 1, Two-spinor calculus and relativistic fields},
Cambridge Monographs on Mathematical Physics. Cambridge University Press, Cambridge, (1984). x+458 pp. ISBN: 0-521-24527-3.

\bibitem{sh} R.~W.~Sharpe,
{\em Differential geometry. Cartan's generalization of Klein's Erlangen program},
Graduate Texts in Mathematics, {\bf 166}. Springer-Verlag, New York, (1997). xx+421 pp. ISBN: 0-387-94732-9.

\bibitem{sl} J.~Slov\'ak,
{\em On the geometry of almost Hermitian symmetric structures},
Proceedings of the Conference Differential Geometry and Applications, Brno, 1995 (Masaryk University in Brno, 1996) 191--206.

\bibitem{slov} J.~Slov\'ak,
{\em Parabolic geometries},
Research Lecture Notes, Part of DrSc.~Dissertation, Preprint  IGA 11/97, electronically available at www.maths.adelaide.edu.au, 70pp.


\bibitem{ss} J.~Slov\'ak and V.~Sou\v{c}ek,
{\em Invariant operators of the first order on manifolds with a given parabolic structure},
Global analysis and harmonic analysis (Marseille-Luminy, 1999),  251--276,
S\'emin. Congr., {\bf 4}, Soc. Math. France, Paris, 2000. 

\bibitem{s} D.~C.~Spencer,
{\em Overdetermined systems of linear partial differential equations},
Bull. A.M.S. {\bf 75} (1969) 179--239.

\bibitem{t} N.~Tanaka,
{\em On non-degenerate real hypersurfaces, graded Lie algebras and Cartan connections},
Japan J. Math. {\bf 2} (1976), no. 1, 131--190.

\bibitem{t1} N.~Tanaka,
{\em A differential geometric study on strongly pseudo-convex manifolds},
Lectures in Mathematics, Department of Mathematics, Kyoto University, No. {\bf 9}. Kinokuniya Book-Store Co., Ltd., Tokyo, (1975). iv+158 pp.

\bibitem{v} D.~A.~Vogan, Jr.,
{\em Representations of Real Reductive Lie Groups},
Progress in Mathematics, {\bf 15}.  Birkh$\mathrm{\ddot{a}}$user, Boston, Mass., (1981). xvii+754 pp. ISBN: 3-7643-3037-6.


\bibitem{w} S.~M.~Webster,
{\em Pseudo-Hermitian structures on a real hypersurface},
J. Differential Geom.~{\bf 13}  (1978), no.~1, 25--41. 

\bibitem{weyl} H.~Weyl,
{\em Zur Infinitesimalgeometrie: Einordnung der projectiven und konformen Auffassung},
Nachrichten von der Gesellschaft der Wissenschaften zu Goettingen, Mathematisch-Physikalische Klasse (1921), 99--112.

\bibitem{y} K.~Yamaguchi,
{\em Differential systems associated with simple graded Lie algebras},
Adv. Stud. Pure Math., {\bf 22}, Math.~Soc.~Japan, Tokyo, (1993). 


\end{thebibliography}
\end{document}